\documentclass[reqno,10pt,centertags]{amsart}
\usepackage{amsmath,amsthm,amscd,amssymb,latexsym,upref} 

\usepackage{hyperref}

\newcommand{\bbN}{{\mathbb{N}}}
\newcommand{\bbR}{{\mathbb{R}}}

\newcommand{\bbC}{{\mathbb{C}}}

\newcommand{\cA}{{\mathcal A}}
\newcommand{\cB}{{\mathcal B}}

\newcommand{\cD}{{\mathcal D}}

\newcommand{\cH}{{\mathcal H}}

\newcommand{\cN}{{\mathcal N}}

\newcommand{\cS}{{\mathcal S}}

\newcommand{\cW}{{\mathcal W}}
\newcommand{\cX}{{\mathcal X}}

\newcommand{\dott}{\,\cdot\,}
\newcommand{\no}{\notag}
\newcommand{\lb}{\label}
\newcommand{\f}{\frac}

\newcommand{\ol}{\overline}

\newcommand{\wti}{\widetilde}

\newcommand{\loc}{\text{\rm{loc}}}

\newcommand{\ran}{\text{\rm{ran}}}

\newcommand{\dom}{\text{\rm{dom}}}

\newcommand{\supp}{\text{\rm{supp}}}

\newcommand{\bi}{\bibitem}
\newcommand{\hatt}{\widehat}
\newcommand{\beq}{\begin{equation}}
\newcommand{\eeq}{\end{equation}}
\newcommand{\ba}{\begin{align}}
\newcommand{\ea}{\end{align}}

\newcommand{\tr}{\text{\rm{tr}}}

\newcommand{\abs}[1]{\lvert#1\rvert}

\renewcommand{\Re}{\text{\rm Re}}
\renewcommand{\Im}{\text{\rm Im}}
\renewcommand{\ln}{\text{\rm ln}}
\renewcommand{\ge}{\geqslant}
\renewcommand{\le}{\leqslant}

\newcommand{\norm}[1]{\left\Vert#1\right\Vert}

\newcommand{\Om}{\Omega}
\newcommand{\dOm}{{\partial\Omega}}
\newcommand{\si}{\sigma}

\newcommand{\ga}{\gamma}

\newcommand{\eps}{\varepsilon}

\newcommand{\LOm}{L^2(\Om;d^nx)}
\newcommand{\LdOm}{L^2(\dOm;d^{n-1} \omega)}

\allowdisplaybreaks \numberwithin{equation}{section}

\newtheorem{theorem}{Theorem}[section]

\newtheorem{lemma}[theorem]{Lemma}
\newtheorem{corollary}[theorem]{Corollary}
\newtheorem{definition}[theorem]{Definition}
\newtheorem{hypothesis}[theorem]{Hypothesis}

\theoremstyle{remark}
\newtheorem{remark}[theorem]{Remark}

\begin{document}

\title[Self-Adjoint Extensions of the Laplacian and Krein Formulas]
{A Description of All Self-Adjoint Extensions of the Laplacian and 
Krein-Type Resolvent Formulas \\ on Non-smooth Domains}
\author[F.\ Gesztesy and M.\ Mitrea]{Fritz Gesztesy and Marius Mitrea}
\address{Department of Mathematics,
University of Missouri, Columbia, MO 65211, USA}
\email{gesztesyf@missouri.edu}
\urladdr{http://www.math.missouri.edu/personnel/faculty/gesztesyf.html}
\address{Department of Mathematics, University of
Missouri, Columbia, MO 65211, USA}
\email{mitream@missouri.edu}
\urladdr{http://www.math.missouri.edu/personnel/faculty/mitream.html} 
\thanks{Based upon work partially supported by the US National Science
Foundation under Grant No.\ DMS-0400639}
\dedicatory{Dedicated to the memory of M.\ Sh.\ Birman (1928--2009)}
\thanks{{\it J. Analyse Math.} {\bf 113}, 53--172 (2011).}

\begin{abstract}
This paper has two main goals. First, we are concerned with a description 
of all self-adjoint extensions of the Laplacian $-\Delta\big|_{C^\infty_0(\Omega)}$ 
in $L^2(\Omega; d^n x)$. Here, the domain 
$\Omega$ belongs to a subclass of bounded Lipschitz domains (which we 
term quasi-convex domains), which contains all convex domains, as well as 
all domains of class $C^{1,r}$, for $r > 1/2$. Second, we establish 
Krein-type formulas for the resolvents of the various self-adjoint 
extensions of the Laplacian in quasi-convex domains and study the well-posedness 
of boundary value problems for the Laplacian, as well as basic properties of the corresponding Weyl--Titchmarsh operators (or energy-dependent 
Dirichlet-to-Neumann maps). 

One significant innovation in this paper is an extension of the classical 
boundary trace theory for functions in spaces which lack Sobolev 
regularity in a traditional sense, but are suitably adapted to the Laplacian. 
\end{abstract}

\maketitle

{\scriptsize{\tableofcontents}}
\normalsize

\section{Introduction}\label{s1}

One fundamental problem, with far-reaching implications in functional 
analysis, spectral theory, mathematical physics, etc., is that of 
characterizing all self-adjoint extensions of a given symmetric, densely 
defined, closed, unbounded operator $S$ in a separable, complex Hilbert space $\cH$. 
Historically, J.\ von Neumann pioneered the theory of self-adjoint extensions of densely defined, 
closed symmetric operators. Interestingly, in \cite{Ne29}, he was also the first to produce what 
we now call the Krein--von-Neumann self-adjoint extension of a given  
densely defined and strictly positive operator $S$. The construction of this extension in the 
general case, where the underlying densely defined operator $S$ is only nonnegative, 
was given by M.\ Krein \cite{Kr47} in 1947. Krein's construction turned out to be extremal 
among all nonnegative self-adjoint extensions in the following sense: A second distinguished  
extension had earlier been found by K.\ Friedrichs (cf.\ \cite{Fr34}) in 1934, and Krein 
(cf.\ \cite{Kr47}) proved in 1947 that all other nonnegative self-adjoint extensions of $S$ 
necessarily lie in between 
the Krein--von-Neumann and Friedrichs extensions in the sense of order between 
semibounded self-adjoint operators in $\cH$ (cf., the comments surrounding \eqref{AleqB}), 
illustrating the extremal roles played by these two extensions. 

Building on the earlier work of Vishik, Birman, and Lions-Magenes, 
G.\ Grubb in \cite{Gr68} characterized all self-adjoint extensions of 
the minimal realization of a symmetric properly elliptic even-order differential operator  
in $L^2(\Om;d^nx)$  in the case in which the underlying domain $\Om$ is smooth. Actually, in this smooth setting, the results in \cite{Gr68} 
also cover the non-symmetric case (cf.\ also \cite{Gr70} for related results), but in this paper 
we will focus on the symmetric case. 
The setting in \cite{Gr68} is that of a differential operator of order 
$2m$, $m\in\bbN$, 
\begin{align} \label{Wey-3EE} 
\begin{split} 
& \cA = \sum_{0 \leq |\alpha| \leq 2m} a_{\alpha}(\cdot) D^{\alpha},   \\
& D^{\alpha} = (-i \partial/\partial x_1)^{\alpha_1} \cdots  
(-i\partial/\partial x_n)^{\alpha_n}, 
\quad \alpha =(\alpha_1,\dots,\alpha_n) \in \bbN_0^n,
\end{split}
\end{align} 
whose coefficients belong to $C^\infty(\ol{\Om})$, and where
$\Omega\subset\bbR^n$ is a bounded $C^\infty$ domain. In addition, for some of the results 
derived in \cite{Gr68} it is assumed that $\cA$ is symmetric, that is, 
\begin{equation} \label{Wey-4EE}
(\cA u,v)_{L^2(\Om;d^nx)}=(u, \cA v)_{L^2(\Om;d^nx)},
\quad u,v\in C^\infty_0(\Om),
\end{equation} 
has a strictly positive lower bound, that is, there exists $\kappa_0>0$ such that
\begin{equation} \label{Wey-4bEE}
(\cA u,u)_{L^2(\Om;d^nx)}\geq\kappa_0\,\|u\|^2_{L^2(\Om;d^nx)},
\quad u\in C^\infty_0(\Om),
\end{equation} 
and is strongly elliptic, in the sense that there exists $\kappa_1>0$ such that 
\begin{equation} \label{Wey-5EE}
a^0(x,\xi):= \Re\bigg(\sum_{|\alpha|=2m}
a_{\alpha}(x) \xi^{\alpha}\bigg) \geq \kappa_1\,|\xi|^{2m},
\quad x\in\ol{\Om}, \; \xi \in\bbR^n.
\end{equation} 
Denote by $A_{min,\Om}$, $A_{max,\Om}$, and $A_{c,\Om}$, respectively, 
the $L^2(\Om;d^nx)$-realizations of $\cA$ with domains 
\begin{align}\label{Wey-6EE} 
\dom (A_{min,\Om})&:=H^{2m}_0(\Om),\quad \dom (A_{c,\Om}):=C^\infty_0(\Om),
\\
\dom (A_{max,\Om})&:
=\big\{u\in L^2(\Omega;d^nx)\,\big|\,\cA u\in L^2(\Omega;d^nx)\big\}, 
\end{align}
where $H^s(\Om)$ stands for the $L^2$-based Sobolev space of order 
$s\in\bbR$ in $\Om$, and $H_0^{s}(\Omega)$ is the subspace of 
$H^{s}(\bbR^n)$ consisting of distributions supported in $\ol{\Om}$.  
For a domain $\Om$ which is smooth, elliptic regularity implies
\begin{equation} \label{Kre-DEE}
(A_{min,\Om})^*=A_{max,\Om}\, \mbox{ and }\, \ol{A_{c,\Om}}=A_{min,\Om}.
\end{equation} 
This is a crucial ingredient in Grubb's work. Another basic result readily 
available in the context of smooth domains is a powerful, well-developed 
Sobolev boundary trace theory which, among many other things, allows for 
the characterization 
\begin{equation} \label{Wey-8EE}
H^{2m}_0(\Omega)
=\big\{u\in H^{2m}(\Omega)\,\big|\,\gamma^{2m-1}_D u=0\big\}, 
\end{equation} 
where $\gamma^{2m-1}_D u:=\bigl(\gamma_N^ju\bigr)_{0\leq j\leq 2m-1}$ is 
the Dirichlet trace operator of order $2m-1$ (with $\gamma_N^j$ denoting 
the $j$-th normal derivative on $\partial\Omega$).   

Our paper has several principal aims, namely a description of all 
self-adjoint extensions of the Laplacian $-\Delta\big|_{C^\infty_0(\Omega)}$ in 
$L^2(\Om; d^n x)$, establishing Krein-type  
formulas for the resolvents of these self-adjoint extensions, investigating the 
well-posedness of boundary value problems for the Laplacian, and  
studying some properties of the corresponding Weyl--Titchmarsh operators.   
In contrast with Grubb's work mentioned earlier, our goal is to go beyond 
the smooth setting and allow domains with irregular boundaries.
In part, this is motivated by the applications of this theory to spectral 
analysis. For example, in the paper \cite{AGMT09}, we make essential 
use of the results developed here in order to prove Weyl-type asymptotic formulas 
for the eigenvalue counting function of 
perturbed Krein Laplacians in non-smooth domains. In this connection it is worth 
pointing out that, generally speaking, the smoothness of the boundary of 
the underlying domain $\Omega$ fundamentally affects the nature of the remainder in 
the Weyl asymptotics, as well as the types of differential 
operators and boundary conditions for which such an asymptotic formula holds. 
Indeed, understanding the interplay between these structures has became 
a central theme of research in this area. 
Before proceeding to describing our main results, we will now 
highlight some of the subtleties and difficulties which the presence of 
boundary singularities entail. 

In the case of an irregular domain $\Omega\subset\bbR^n$, several new 
significant difficulties are encountered:

$\bullet$ First, there is the issue of the 
(global) Sobolev-regularity exhibited by functions belonging to the domains of 
$-\Delta_{D,\Om}$, $-\Delta_{N,\Om}$, $-\Delta_{K,\Om}$ (the Dirichlet, Neumann 
and Krein Laplacian, respectively). If $\Omega$ has a $C^\infty$-smooth 
boundary, then it is well-known that the domains of the aforementioned 
operators are subspaces of $H^2(\Om)$. For this particular regularity result  
the smoothness requirement on $\partial\Om$ can be substantially reduced: 
e.g., $C^{1,r}$ with $r>1/2$ will do. On the other hand, if $\Omega$ has a 
Lipschitz boundary then $\dom (- \Delta_{D,\Om})\subset H^{3/2}(\Om)$ and this 
result is optimal, as simple examples of (two-dimensional) domains with 
re-entrant corners show. At a more sophisticated level, B.\ Dahlberg, 
D.\ Jerison and C.\ Kenig (cf.\ the discussion in \cite{JK95}) have constructed 
a bounded $C^1$ domain $\Omega\subset\bbR^n$ and $f\in C^\infty(\ol{\Om})$ 
such that 
\begin{eqnarray}\lb{DJK-1}
(-\Delta_{D,\Om})^{-1}f\notin W^{2,1}(\Om),
\end{eqnarray}
where $W^{k,p}(\Om)$, $1\leq p\leq\infty$, stands for the $L^p$-based Sobolev 
space of order $k\in\bbN$. This shows that the endpoint $p=1$ of the 
implication 
\begin{align}\lb{DJK-1X}
\begin{split} 
& \Omega\mbox{ a bounded $C^{1,r}$ domain with $1>r>1-1/p$} \\
& \quad \text{implies } \, (-\Delta_{D,\Om})^{-1}\in 
\cB\big(L^p(\Om;d^nx),W^{2,p}(\Om)\big),
\end{split}
\end{align}
which (as our arguments show) is valid whenever $p\in(1,\infty)$, is sharp. 
In fact, the results in Section\ 7 of \cite{MS85} show that \eqref{DJK-1X} is
in the nature of best possible (as far as the regularity of $\dOm$, measured
on the H\"older scale, is concerned). 
Nonetheless, if a bounded Lipschitz domain has all its singularities directed 
outwardly (more precisely, locally, it satisfies either a uniform exterior 
ball condition, or is of class $C^{1,r}$ for some $r>1/2$), then the inclusion 
$\dom (- \Delta_{D,\Om})\subset H^2(\Om)$ holds. 

In this paper we shall work with (what we term) the class of quasi-convex 
domains, which is a hybrid of the two categories of domains mentioned above. 
More specifically, an open subset of $\bbR^n$ is called a quasi-convex domain 
if it behaves locally either like a $C^{1,r}$ domain (with $r>1/2$) or 
like a Lipschitz domain satisfying a uniform exterior ball condition. 

$\bullet$ Second, the nature of the Dirichlet and Neumann trace operators, $\ga_D$ and
$\ga_N$, changes fundamentally in the
presence of boundary irregularities. To shed some light on this phenomenon, we 
recall that if $\Omega$ is a $C^\infty$ domain, then the second-order boundary 
trace operator 
\begin{eqnarray}\lb{DJK-3}
\gamma_2=(\gamma_D,\gamma_N): H^2(\Om)\to  
H^{3/2}(\dOm)\times H^{1/2}(\dOm)
\end{eqnarray}
is well-defined, bounded, and has a linear, continuous right-inverse. 
The problem with the case where $\Omega$ is only Lipschitz is that 
the range of $\gamma_2$ acting on $H^2(\Om)$ no longer decouples into 
a Cartesian product of boundary Sobolev spaces. In fact, it was rather 
recently that $\gamma_2(H^2(\Om))$ has been identified in \cite{MMS05}  
(where, in fact, higher smoothness Sobolev spaces are considered) as 
\begin{eqnarray}\lb{DJK-4}
\bigl\{(g_0,g_1)\in H^1(\partial\Omega)
\dotplus L^2(\partial\Omega;d^{n-1}\omega) \,\big|\, 
\nabla_{tan}g_0+g_1\nu\in \bigl(H^{1/2}(\partial\Omega)\bigr)^n\bigl\},
\end{eqnarray}
where $\nabla_{tan}$ and $\nu$ are the tangential gradient and outward
unit normal on $\dOm$, respectively. As opposed to the case of smoother 
domains, in the situation where $\Omega$ is merely Lipschitz,  
no optimal (Sobolev) smoothness conditions on the individual functions 
$g_0,g_1$ can be inferred from the knowledge that $g_0\in H^1(\partial\Omega)$ 
and $g_1\in L^2(\partial\Omega;d^{n-1}\omega)$ are such that 
$\nabla_{tan}g_0+g_1\nu\in \bigl(H^{1/2}(\partial\Omega)\bigr)^n$.

In addition to the aforementioned issue, we are now forced to consider 
Dirichlet and Neumann traces for functions in $\dom (- \Delta_{max})
:= \big\{u\in L^2(\Om;d^nx) \,\big|\, \Delta u\in L^2(\Om;d^nx)\big\}$. 
When $\Omega$ is a bounded Lipschitz domain, it can be shown that  
$\gamma_D$ and $\gamma_N$ extend as linear bounded maps 
\begin{eqnarray}\lb{DJK-5}
&& \wti\gamma_D:\bigl\{u\in H^{1/2}(\Omega) \,\big|\, 
\Delta u\in L^2(\Omega;d^nx)\bigr\}\to L^2(\partial\Omega;d^{n-1}\omega), 
\\[4pt]
&& \wti\gamma_N:\bigl\{u\in H^{3/2}(\Omega) \,\big|\, 
\,\Delta u\in L^2(\Omega;d^nx)\bigr\}\to L^2(\partial\Omega;d^{n-1}\omega). 
\lb{DJK-6}
\end{eqnarray}
However, it should be noted that in general,  
\begin{eqnarray}\lb{M-M1}
\mbox{the inclusion }\,
\dom (- \Delta_{max})\subseteq H^s(\Om) \,\mbox{ fails for every }\,s>0.
\end{eqnarray}
Indeed, the function $u(z):={\rm Re}\,(z^{-\alpha})
=r^{-\alpha}\cos\,(\alpha\,\theta)$ if $z=re^{i\theta}$ is
harmonic and square integrable in the (smooth) domain  
$\Omega:=\{z\in\bbC \,|\, |z-1|<1\}$ if $\alpha<1$,  
but fails to be in $H^s(\Om)$ if $\alpha>1-s$. Thus, in light of the fact 
that functions in $\dom (- \Delta_{max})$ do not, 
generally speaking, exhibit any global Sobolev regularity besides
mere square integrability, the trace results \eqref{DJK-5}, \eqref{DJK-6}
are not satisfactory for our goals. 
Nonetheless, we are able to augment \eqref{DJK-5}, \eqref{DJK-6}
by proving that $\wti\gamma_D$, $\wti\gamma_N$ above,  
further extend as bounded maps in the following contexts: 
\begin{align}\lb{DJK-7}
& \widehat\gamma_D: \dom (- \Delta_{max})\to 
\bigl(N^{1/2}(\dOm)\bigr)^*,
\\
& \widehat\gamma_N: \dom (- \Delta_{max})\to 
\bigl(N^{3/2}(\dOm)\bigr)^*,
\lb{DJK-8}
\end{align}
where 
\begin{align}\lb{DJK-9}
& N^{1/2}(\partial\Omega)
:=\bigl\{g\in L^2(\partial\Omega;d^{n-1}\omega) \,\big|\, 
g\nu\in \bigl(H^{1/2}(\partial\Omega)\bigr)^n\bigl\}, 
\\
& N^{3/2}(\partial\Omega):=\bigl\{g\in H^1(\partial\Omega) \,\big|\, 
\nabla_{tan}g\in \bigl(H^{1/2}(\partial\Omega)\bigr)^n\bigl\}.
\lb{DJK-10}
\end{align}
These ``exotic'' spaces provide the natural context for describing 
the mapping properties for the Dirichlet and Neumann trace operators
acting on $\dom (- \Delta_{max})$. They are natural, in the sense
that $N^{1/2}(\partial\Omega)=H^{1/2}(\partial\Omega)$ and
$N^{3/2}(\partial\Omega)=H^{3/2}(\partial\Omega)$ if 
$\Omega$ is a $C^{1,r}$ domain for some $r>1/2$.

Dealing with the aforementioned topics occupies the bulk of the next six sections 
of the paper, where we develop a trace theory which goes considerably 
beyond the scope of the traditional treatment of (reasonably) smooth domains. 
Having dealt with this host of issues in Sections \ref{s2}--\ref{s8}, we then 
proceed to the next item on our agenda, namely the task of 
characterizing all self-adjoint extensions of the minimal Laplacian 
in quasi-convex domains. Based on our trace theory and an abstract, 
functional analytic result of Grubb \cite{Gr68}, we prove in Section \ref{s14} the following
theorem, which provides a universal parametrization of all 
self-adjoint extensions of $-\Delta\big|_{C^\infty_0(\Om)}$ in $L^2(\Om; d^n x)$:

\begin{theorem}\lb{CC.wII} 
Assume that $\Omega\subseteq{\mathbb{R}}^n$ is a $($bounded\,$)$quasi-convex domain 
and let $z\in\bbR\backslash\si(-\Delta_{D,\Om})$. Suppose that $X$ 
is a closed subspace of $\bigl(N^{1/2}(\partial\Omega)\bigr)^*$ and denote 
by $X^*$ the conjugate dual space of $X$. In addition, consider a  
 self-adjoint operator 
\begin{eqnarray}\lb{4.Aw1I}
L:\dom (L)\subseteq X\to  X^*, 
\end{eqnarray} 
and define the linear operator 
$- \Delta^D_{X,L,z}:\dom (- \Delta^D_{X,L,z})\subset L^2(\Om;d^nx) \to  L^2(\Om;d^nx)$, 
by taking 
\begin{align} \lb{4.Aw3I}
\begin{split}
& -\Delta^D_{X,L,z} u:=(-\Delta-z )u,    \\ 
& \; u \in \dom (- \Delta^D_{X,L,z}):= \big\{
v\in \dom (- \Delta_{max})\,\big|\, \widehat\gamma_D v \in \dom (L),\,
\tau^N_z v  \big|_{X}=L\big(\widehat\gamma_D v \big)\big\}.     \\ 
\end{split}
\end{align} 
Above, $\tau^N_z$ is a regularized Neumann trace operator 
$($for details see \eqref{3.Aw1}, \eqref{3.Aw2}$)$, and the boundary condition 
$\tau^N_z u  \big|_{X}=L\bigl(\widehat\gamma_D u \bigr)$ is interpreted as
\begin{eqnarray}\lb{4.Aw4BI}
{}_{N^{1/2}(\dOm)}\langle\tau^N_z u  ,f\rangle_{(N^{1/2}(\dOm))^*}
=\ol{{}_{X}\langle f,L(\widehat\gamma_D u )\rangle_{X^*}}, 
\quad  f \in X.
\end{eqnarray}
Then 
\begin{equation}\lb{4.Aw4I}
- \Delta^D_{X,L,z}\,\mbox{ is self-adjoint in $L^2(\Om;d^nx)$}, 
\end{equation}
and
\begin{equation}
-\Delta_{min}-zI_{\Om}\subsetneqq -\Delta^D_{X,L,z}\subsetneqq -\Delta_{max}-zI_{\Om}.
\end{equation}

Conversely, if 
\begin{eqnarray}\lb{4.Aw5I}
\wti S:\dom \big(\wti S\big)\subseteq L^2(\Om;d^nx)\to  L^2(\Om;d^nx) 
\end{eqnarray}
is a self-adjoint operator with the property that 
\begin{eqnarray}\lb{4.Aw6I}
-\Delta_{min}-zI_{\Om}\subseteq \wti S\subseteq -\Delta_{max}-zI_{\Om}, 
\end{eqnarray}
then there exist $X$, a closed subspace of $\bigl(N^{1/2}(\dOm)\bigr)^*$, 
and $L:\dom (L)\subseteq X\to X^*$, a self-adjoint operator, such that 
\begin{eqnarray}\lb{4.Aw7I}
\wti S = -\Delta^D_{X,L,z}.
\end{eqnarray}

In the above scheme, the operator $\wti S$ and the pair $X,L$ correspond 
uniquely to each other. In fact, 
\begin{eqnarray}\lb{4.Aw8I}
\dom (L)=\widehat\gamma_D\big(\dom \big(\wti S\big)\big),\quad
X=\ol{\widehat\gamma_D\big(\dom \big(\wti S)\big)} \quad \text{ $\big($with closure in 
$\bigl(N^{1/2}(\dOm)\bigr)^*$$\big)$}.
\end{eqnarray}
\end{theorem}

In the context of Theorem \ref{CC.wII}, it is illuminating to indicate 
how various distinguished self-adjoint extensions of the Laplacian occur 
in this scheme. For example, for every  
$z_0\in\bbR\backslash \si(-\Delta_{D,\Om})$ one has 
\begin{eqnarray}\lb{4.A.Y}
X:=\dom (L):=\{0\}\, \mbox{ and }\,  
L:=0 \, \text{ imply } \, - \Delta^D_{X,L,z_0} + z_0I_{\Om} = - \Delta_{D,\Om},
\end{eqnarray}
the Dirichlet Laplacian, and 
\begin{eqnarray}\lb{4.A.U}
\left.
\begin{array}{c}
X:=\bigl(N^{1/2}(\partial\Omega)\bigr)^*\mbox{ and }
L:=M^{(0)}_{D,N,\Om}(z_0)
\\[4pt]
\mbox{ with }\dom (L):=N^{3/2}(\partial\Omega)
\end{array}
\right\}   \text{ imply } \, - \Delta^D_{X,L,z_0} + z_0I_{\Om} = - \Delta_{N,\Om},
\end{eqnarray}
the Neumann Laplacian. Furthermore, 
\begin{eqnarray}\lb{4.A.Z}
X:=\dom (L):=\bigl(N^{1/2}(\partial\Omega)\bigr)^* \, \mbox{ and }\, 
L:=0  \, \text{ imply } \,  - \Delta^D_{X,L,z_0} = - \Delta_{K,\Om,z_0},
\end{eqnarray}
the Krein Laplacian. Related versions of these results (starting from a 
different ``reference'' operator) can be found in Theorem \ref{CC.wF}
and Corollary \ref{rE.1F}. 

As an interesting application of the material developed in this paper, the final Section \ref{s16} is devoted to proving Krein-type formulas for the resolvents of various self-adjoint 
extensions of the Laplacian in quasi-convex domains, and to studying the 
properties of the corresponding spectral parameter dependent Dirichlet-to-Neumann maps 
(i.e., Weyl--Titchmarsh operators). To put this application in a proper perspective requires some preparation and so we will next recall some pertinent facts derived in \cite{GMT98} (see also 
\cite{GKMT01} and \cite{GT00}) on abstract 
Krein-type resolvent formulas. For hints to the literature on this vast and currently very active area of research, we refer to Section \ref{s16}. 

Let $S$ be a closed, symmetric operator in a separable, complex Hilbert space $\cH$ with equal deficiency indices ${\rm def} (S)=(r,r)$, $r\in \bbN\cup\{\infty\}$, and    
suppose that $S_\ell$, $\ell=1,2$, are two distinct, relatively prime self-adjoint extensions 
of $S$, that is, $\dom (S_1)\cap \dom (S_2) = \dom (S)$. At this instance none of $S$ and 
$S_j$, $j=1,2$,  need to be bounded from below and we emphasize that $S_1$ and $S_2$ are chosen to be relatively prime for simplicity only.  Then the deficiency subspaces $\cN_{\pm}$ of $S$ are given by
\begin{equation}
\cN_{\pm}=\ker({S}^*\mp i I_{\cH}), \quad \dim (\cN_{\pm})= {\rm def} (S) = r, 
\lb{3.1}
\end{equation}
and for any self-adjoint extension $\wti S$ of $S$ in $\cH$ ,
the corresponding Cayley transform $C_{\wti S}$ in $\cH$ is defined by
\begin{equation}
C_{\wti S}= \big(\wti S+i I_{\cH}\big) \big(\wti S-i I_{\cH}\big)^{-1},
\lb{3.2}
\end{equation}
implying
\begin{equation}
C_{\wti S}\cN_-=\cN_+.
\lb{3.3}
\end{equation} 

Given a self-adjoint extension $\wti S$ of $S$ and a closed linear subspace 
$\cN$ of $\cN_+$, $\cN\subseteq \cN_+$, the Donoghue-type Weyl--Titchmarsh 
operator $M_{\wti S,\cN}(z) \in \cB(\cN)$ associated with the pair $(\wti S,\cN)$  is defined by
\begin{align}
M_{\wti S,\cN}(z)&=P_\cN \big(z\wti S+I_\cH\big)\big(\wti S-z I_{\cH}\big)^{-1} P_\cN\big\vert_\cN  \no \\
&=zI_\cN+(1+z^2)P_\cN \big(\wti S-z I_{\cH}\big)^{-1} P_\cN\big\vert_\cN\,, \quad  
z\in \bbC\backslash \bbR,
\lb{3.5}
\end{align}
with $I_\cN$ the identity operator in $\cN$, and $P_\cN$ the orthogonal projection in 
$\cH$ onto $\cN$.

Following Saakjan \cite{Sa65} (in a version presented in Theorem 5 and Corollary 6 in \cite{GMT98}), and using the notions introduced in \eqref{3.1}--\eqref{3.5}, Krein's formula for the difference of the resolvents of $S_1$ and $S_2$ then reads as follows (cf.\ \cite{GMT98}, \cite{Sa65}):
\begin{align}
(S_2-z I_{\cH})^{-1}&=(S_1-z I_{\cH})^{-1}   
+(S_1-i I_{\cH})(S_1-z I_{\cH})^{-1}P_{1,2}(z)
(S_1+i I_{\cH})(S_1-z I_{\cH})^{-1}  \no \\
&=(S_1-z I_{\cH})^{-1}+(S_1-i I_{\cH})(S_1-z I_{\cH})^{-1} P_{\cN_+}   \no \\
&\quad \times (\tan (\alpha_{1,2})- M_{S_1,\cN_+}(z))^{-1}
P_{\cN_+}(S_1+i I_{\cH})(S_1-z I_{\cH})^{-1}, \quad z\in \rho(S_1)\cap\rho(S_2),   \label{3.8b}  
\end{align}
where
\begin{equation}
e^{-2i \alpha_{1,2}} = - C_{S_2} C_{S_1}^{-1}\big|_{\cN_+}.    \lb{3.9}
\end{equation}

One can show that $M_{S_1,\cN_+}(\cdot)$ (and hence 
$(\tan (\alpha_{1,2})- M_{S_1,\cN_+}(\cdot))^{-1}$) is an operator-valued Herglotz function, that is, 
$M_{S_1,\cN_+}(\cdot)$ is analytic on $\bbC_+ = \{z\in\bbC \,|\, \Im(z)>0\}$ and 
$\Im(M_{S_1,\cN_+}(z)) \geq 0$, $z\in\bbC_+$. In addition, $M_{S_1,\cN_+}(\cdot)$ has the 
symmetry property,
\begin{equation}
[M_{S_1,\cN_+}(z)]^*  = M_{S_1,\cN_+}(\ol z),   \quad z \in\bbC\backslash\bbR.    \lb{3.10}
\end{equation}

Next, assume in addition that $S\ge \varepsilon I_{\cH}$ for some $\varepsilon>0$, and denote by 
$S_K$ the Krein--von Neumann extension, and by $S_F$ the Friedrichs extension of $S$, respectively 
(cf.\ Section \ref{s9} for more details). Since $S_K$ and $S_F$  are relatively prime (see, e.g., 
\cite[Lemma\ 2.8]{AGMT09}), \eqref{3.9} yields the following version of Krein's formula 
connecting the resolvents of $S_K$ and $S_F$: 
\begin{align}
(S_K-z I_{\cH})^{-1} &=(S_F-z I_{\cH})^{-1}   \no \\ 
& \quad +(S_F-i I_{\cH})(S_F-z I_{\cH})^{-1} P_{\cN_+}
\big[M_{S_F,\cN_+}(0) - M_{S_F,\cN_+}(z)\big]^{-1}    \label{3.8c} \\
&\qquad \times P_{\cN_+}(S_F+i I_{\cH})(S_F-z I_{\cH})^{-1},  \quad 
z \in \rho(S_K)\cap\rho(S_F).   \no 
\end{align}
Equation \eqref{3.8c} follows by combining  Theorem\ 4.7 and Corollary\ 4.8 in \cite{GKMT01} and will be further discussed in more detail elsewhere.  

In a nutshell, this represents the approach to Krein-type resolvent formulas with emphasis on the deficiency subspace $\cN_+$, and so the only Hilbert spaces involved are $\cN_+$ and $\cH$ (with $\cN_+ \subset \cH$). However, in connection with PDE applications, where, for example, $S$ is generated by a suitable second-order strongly elliptic differential operator on the domain $\Om \subset \bbR^n$, one would naturally prefer to replace the pair of spaces $(\cN_+, \cH)$ by the pair 
$\big(L^2(\partial\Om; d^{n-1} \omega), L^2(\Om; d^n x)\big)$, and at the same time, replace the abstract Donoghue-type Weyl--Titchmarsh operator $M_{S_1,\cN_+}(\cdot)$ by appropriate energy-dependent Dirichlet-to-Neumann maps. This change of emphasis then also necessitates the introduction of appropriate (extensions of) Dirichlet and Neumann boundary trace operators $\gamma_D$ and $\gamma_N$. 

We note that in the PDE context these ideas were realized only very recently in the 
work of Amrein and Pearson \cite{AP04}, 
Behrndt and Langer \cite{BL07}, Brown, Grubb, and Wood \cite{BGW09}, 
Brown, Hinchliffe, Marletta, Naboko, and Wood \cite{BHMNW09}, 
Brown, Marletta, Naboko, and Wood \cite{BMNW08}, 
Gesztesy and Mitrea \cite{GM08}, \cite{GM09}, 
Grubb \cite{Gr08a} (including a discussion of non-self-adjoint extensions), 
Posilicano \cite{Po08}, Posilicano and Raimondi \cite{PR09}, and Ryzhov \cite{Ry10}.  With the exception of Grubb,  Posilicano and Raimondi, and Ryzhov, who treat $C^{1,1}$ domains $\Om$, the remaining authors are dealing with the case of $C^\infty$-smooth domains $\Om$. 

One of our motivations for introducing the class of non-smooth, quasi-convex domains $\Om$, and the associated boundary trace theory, was precisely to be able to prove Krein-type resolvent formulas for Laplacians on such non-smooth domains. As a typical result we prove in Section \ref{s16} we thus mention the 
following theorem:  

\begin{theorem}\lb{T.Na-qI}
Assume that $\Omega$ is a $($bounded\,$)$ quasi-convex domain and suppose that 
$z_0\in\bbR\backslash\si(-\Delta_{D,\Om})$. In addition, consider two bounded,  
self-adjoint operators 
\begin{eqnarray}\lb{AG-1.1I}
L_j:\bigl(N^{1/2}(\partial\Omega)\bigr)^*\to  N^{1/2}(\partial\Omega),
\quad j=1,2.
\end{eqnarray} 
Associated with these, define the  
operators $-\Delta^D_{X_j,L_j,z_0}$, $j=1,2$, as in \eqref{4.Aw3I} 
corresponding to $z=z_0$ and $X_1=X_2=\bigl(N^{1/2}(\partial\Omega)\bigr)^*$.
Finally, fix a complex number
$z\in\bbC\backslash \big(\si(-\Delta^D_{X_1,L_1,z_0})\cup
\si(-\Delta^D_{X_2,L_2,z_0})\big)$.
Then the following Krein-type  resolvent formula holds on $L^2(\Omega;d^nx)$, 
\begin{align}\lb{N-Bbb1.1I}
&  
\big(-\Delta^D_{X_2,L_2,z_0}-zI_\Om\big)^{-1} 
=\big(-\Delta^D_{X_1,L_1,z_0}-zI_\Om\big)^{-1}  
\\
& \quad
+\big[\widehat\gamma_D\big(-\Delta^D_{X_1,L_1,z_0}-zI_\Om\big)^{-1}\big]^*
M^D_{L_1,L_2,z_0}(z)
\big[\widehat\gamma_D\big(-\Delta^D_{X_1,L_1,z_0}-zI_\Om\big)^{-1}\big],  \no
\end{align}
where 
\begin{align}\lb{N-Bbb1.2I}
& M^D_{L_1,L_2,z_0}(z):= 
(L_2-L_1)\big[I_\Om + \big[M^{(0)}_{D,N,\Om}(z_0)
-M^{(0)}_{D,N,\Om}(z+z_0)-L_2\big]^{-1} 
(L_2-L_1)\big]
\no \\
& \quad
=(L_2-L_1)\big[M^{(0)}_{D,N,\Om}(z_0)-M^{(0)}_{D,N,\Om}(z+z_0)-L_2\big]^{-1} 
\no \\
& \qquad \times \big[M^{(0)}_{D,N,\Om}(z_0)-M^{(0)}_{D,N,\Om}(z+z_0)-L_1\big], 
 \quad z \in \bbC\backslash\bbR, 
\end{align}
has the property that 
\begin{eqnarray}\lb{N-Bbb1.3I}
\bbC_+ \ni z\mapsto M^D_{L_1,L_2,z_0}(z)
\in \cB\bigl((N^{1/2}(\dOm))^*,N^{1/2}(\dOm)\bigr)
\end{eqnarray}
is an operator-valued Herglotz function which satisfies 
\begin{eqnarray}\lb{N-Bbb1.4I}
\big[M^D_{L_1,L_2,z_0}(z)\big]^*=M^D_{L_1,L_2,z_0}(\ol{z}). 
\end{eqnarray}
\end{theorem}

Here $M^{(0)}_{D,N,\Om}(\cdot)$ denotes the Dirichlet-to-Neumann map for 
the Laplacian in the domain $\Omega$ described in detail in Section \ref{s5}.  

Finally, we emphasize that an application of the principal results of this paper to Weyl-type spectral asymptotics for perturbed Krein Laplacian on non-smooth domains $\Om$ has been presented 
in \cite{AGMT09}.  

We close this section by briefly elaborating on the most basic notational 
conventions used throughout this paper. Let $\cH$ be a separable complex 
Hilbert space, $(\dott,\dott)_{\cH}$ the inner product in $\cH$ 
(linear in the second factor), and $I_{\cH}$ the identity operator in $\cH$.
Next, let $T$ be a linear operator mapping (a subspace of) a Banach space 
into another, with $\dom(T)$, $\ran(T)$, and $\ker(T)$ denoting the domain, range, 
and the kernel (null space) of $T$. The closure of a closable operator 
$S$ is denoted by $\ol S$. The spectrum of a closed linear operator in $\cH$ will be 
denoted 
by $\sigma(\cdot)$. The Banach spaces of bounded and compact linear operators 
on $\cH$ are denoted by $\cB(\cH)$ and $\cB_\infty(\cH)$, respectively. 
Similarly, $\cB(\cX_1,\cX_2)$ and $\cB_\infty (\cX_1,\cX_2)$ 
will be used for bounded and compact operators between two Banach spaces 
$\cX_1$ and $\cX_2$. Moreover, $\cX_1 \hookrightarrow \cX_2$ denotes the 
continuous embedding of the Banach space $\cX_1$ into the Banach space 
$\cX_2$. In addition, $U_1 \dotplus U_2$ denotes the direct sum of the 
subspaces $U_1$ and $U_2$ of a Banach space $\cX$; and $V_1 \oplus V_2$ represents the orthogonal direct sum of the subspaces $V_j$, $j=1,2$, of a 
Hilbert space $\cH$. We shall employ the notation $\|\cdot\|_1 \approx \|\cdot\|_2$ 
in order to indicate that two norms $\|\cdot\|_1$ and $\|\cdot\|_2$ on a vector space 
are equivalent.  

Throughout this manuscript, if $X$ denotes a Banach space, $X^*$ 
denotes the {\it adjoint space} of continuous conjugate linear functionals 
on $X$, that is, the {\it conjugate dual space} of $X$ (rather than the usual 
dual space of continuous linear functionals 
on $X$). This avoids the well-known awkward distinction between adjoint 
operators in Banach and Hilbert spaces (cf., e.g., the pertinent discussion 
in \cite[p.\ 3, 4]{EE89}). 

We denote by $I_\Om$ the identity operator in $L^2(\Om;d^nx)$, and similarly, 
by $I_{\partial\Om}$ the identity operator in $L^2(\partial\Omega;d^{n-1}\omega)$.

Finally, a notational comment: For obvious reasons in connection 
with quantum mechanical applications, we will, with a slight abuse of 
notation, dub $-\Delta$ (rather than $\Delta$) as the ``Laplacian'' in 
this paper.

\section{Sobolev Spaces on Lipschitz and $C^{1,r}$ Domains}
\label{s2}

In this section we summarize some fundamental results on Lipschitz and $C^{1,r}$ domains $\Om \subset \bbR^n$, $n\in\bbN$, and on the corresponding Sobolev spaces 
$H^s(\Omega)$ and $H^r(\partial\Omega)$ needed in the remainder of this paper. 

Before we focus primarily on bounded Lipschitz domains, we briefly recall some basic facts in connection with Sobolev spaces corresponding to open sets 
$\Om\subseteq\bbR^n$, $n\in\bbN$: For an arbitrary $m\in\bbN\cup\{0\}$, we follow the customary way of defining $L^2$-Sobolev spaces of order $\pm m$ in $\Om$ as
\begin{align}\label{hGi-1}
H^m(\Om) &:=\big\{u\in L^2(\Om;d^nx)\,\big|\,\partial^\alpha u\in L^2(\Om;d^nx), 
\, 0 \leq |\alpha|\leq m\big\}, \\
H^{-m}(\Om) &:=\biggl\{u\in\cD^{\prime}(\Om)\,\bigg|\,u=\sum_{0 \leq |\alpha|\leq m}
\partial^\alpha u_{\alpha}, \mbox{ with }u_\alpha\in L^2(\Om;d^nx), 
\,  0 \leq |\alpha|\leq m\biggr\},
\label{hGi-2}
\end{align}
equipped with natural norms (cf., e.g., \cite[Ch.\ 3]{AF03}, \cite[Ch.\ 1]{Ma85}). 
Here $\cD^\prime(\Om)$ denotes the usual set of distributions on 
$\Omega\subseteq \bbR^n$ (i.e., the dual of $\cD(\Om)$, the space of test functions 
$C_0^\infty (\Om)$, equipped with the usual inductive limit topology). Then we set
\begin{equation}\label{hGi-3}
H^m_0(\Om):=\,\mbox{the closure of $C^\infty_0(\Om)$ in $H^m(\Om)$}, 
\quad m \in \bbN\cup\{0\}.
\end{equation}
As is well-known, all three spaces above are Banach, reflexive and,
in addition,
\begin{equation} \label{hGi-4}
\bigl(H^m_0(\Om)\bigr)^*=H^{-m}(\Om).
\end{equation} 
Again, see, for instance, \cite[Ch.\ 3]{AF03}, \cite[Ch.\ 1]{Ma85}.

We recall that an open, nonempty set $\Omega\subseteq\bbR^n$ is
called a {\it Lipschitz domain} if the following property holds: 
There exists an open covering $\{{\mathcal O}_j\}_{1\leq j\leq N}$
of the boundary $\partial\Omega$ of $\Om$ such that for every
$j\in\{1,...,N\}$, ${\mathcal O}_j\cap\Omega$ coincides with the portion
of ${\mathcal O}_j$ lying in the over-graph of a Lipschitz function
$\varphi_j:\bbR^{n-1}\to\bbR$ $($considered in a new system of coordinates
obtained from the original one via a rigid motion$)$. The number
$\max\,\{\|\nabla\varphi_j\|_{L^\infty (\bbR^{n-1};d^{n-1}x')^{n-1}}\,|\,1\leq j\leq N\}$
is said to represent the {\it Lipschitz character} of $\Omega$.

As regards $L^2$-based Sobolev spaces of fractional order $s\in\bbR$,
on arbitrary Lipschitz domains $\Om\subseteq\bbR^n$, we introduce
\begin{align}\label{HH-h1}
H^{s}(\bbR^n) &:=\bigg\{U\in \cS^\prime(\bbR^n)\,\bigg|\,
\norm{U}_{H^{s}(\bbR^n)}^2 = \int_{\bbR^n}d^n\xi\,
\big|\hatt U(\xi)\big|^2\big(1+\abs{\xi}^{2s}\big)<\infty \bigg\},
\\
H^{s}(\Om) &:=\big\{u\in \cD^\prime(\Om)\,\big|\,u=U|_\Om\text{ for some }
U\in H^{s}(\bbR^n)\big\} = R_{\Om} \, H^s(\bbR^n),
\label{HH-h2}
\end{align} 
where $R_{\Om}$ denotes the restriction operator (i.e., $R_{\Om} \, U=U|_{\Om}$, 
$U\in H^{s}(\bbR^n)$),   
$\cS^\prime(\bbR^n)$ is the space of tempered distributions on $\bbR^n$,
and $\hatt U$ denotes the Fourier transform of $U\in\cS^\prime(\bbR^n)$.
These definitions are consistent with \eqref{hGi-1}, \eqref{hGi-2}. 
Next, retaining that $\Om\subseteq \bbR^n$ is an arbitrary Lipschitz domain, we introduce
\begin{equation}\label{incl-xxx}
H^{s}_0(\Omega):=\big\{u\in H^{s}(\bbR^n)\,\big|\,\supp (u)\subseteq\ol{\Omega}\big\}, 
\quad s\in\bbR,
\end{equation} 
equipped with the natural norm induced by $H^{s}(\bbR^n)$. The space 
$H^{s}_0(\Omega)$ is reflexive, being a closed subspace of $H^{s}(\bbR^n)$. 
Finally, we introduce for all $s\in\bbR$,
\begin{align} 
\mathring{H}^{s} (\Omega) &= \mbox{the closure of $C^\infty_0(\Omega)$ in $H^s(\Omega)$},   \\
H^{s}_{z} (\Om) &= R_{\Om} \, H^{s}_0(\Omega).
\end{align}

Assuming from now on that $\Om\subset\bbR^n$ is a Lipschitz domain with a compact boundary, we recall the existence of a universal linear extension operator 
$E_{\Om}:\cD^\prime (\Om) \to \cS^\prime (\bbR^n)$ such that 
$E_{\Om}: H^s(\Om) \to H^s(\bbR^n)$ is bounded for all $s\in\bbR$, and 
$R_{\Om}  E_{\Om}=I_{H^s(\Om)}$ (cf.\ \cite{Ry99}). If $\widetilde{C_0^\infty(\Om)}$ denotes the set of $C_0^\infty(\Om)$-functions extended to all of $\bbR^n$ by setting functions zero outside of $\Omega$, then for all $s\in\bbR$, 
$\widetilde{C_0^\infty(\Om)} \hookrightarrow H^s_0(\Om)$ densely.  

Moreover, one has
\begin{equation}\label{incl-Ya}
\big(H^{s}_0(\Omega)\big)^*=H^{-s}(\Omega),  \quad s\in\bbR.
\end{equation}
(cf., e.g., \cite{JK95}) consistent with \eqref{hGi-3}, and also, 
\begin{equation}
\big(H^s(\Om)\big)^* = H^{-s}_0(\Om),  \quad s\in\bbR, 
\end{equation}
in particular, $H^s(\Om)$ is a reflexive Banach space. We shall also use the fact that  
for a Lipschitz domain $\Om\subset\bbR^n$ with compact boundary, the space 
$\mathring{H}^{s} (\Omega)$ 
satisfies 
\begin{equation}\label{incl-Yb}
\mathring{H}^{s} (\Omega) = H^s_{z}(\Om)
\, \mbox{ if } \, s > -1/2,\,\,s\notin{\textstyle\big\{{\frac12}}+\bbN_0\big\}. 
\end{equation}
For a Lipschitz domain $\Omega\subseteq\bbR^n$ with compact boundary it is also known that
\begin{equation}\label{dual-xxx}
\bigl(H^{s}(\Omega)\bigr)^*=H^{-s}(\Omega), \quad - 1/2 <s< 1/2.
\end{equation}
See \cite{Tr02} for this and other related properties. Throughout this paper 
we agree to use the {\it adjoint} (rather than the dual) space $X^*$ of a Banach space 
$X$. 

From this point on (and unless explicitly stated otherwise) we will always make at least the following assumption on the set $\Omega$. (This will be strengthened later on in Section \ref{s8}):  

\begin{hypothesis}\lb{h2.1}
Let $n\in\bbN$, $n\geq 2$, and assume that $\Om\subset{\bbR}^n$ is
a bounded Lipschitz domain.
\end{hypothesis}

At times we will invoke also the notion of $C^{1,r}$ domains and then introduce 
the following stronger hypothesis on $\Om$: 

\begin{hypothesis}\lb{h2.1C}
Let $n\in\bbN$, $n\geq 2$, and assume that $\Om\subset{\bbR}^n$ is
a bounded domain of class $C^{1,r}$, $r\in(1/2,1)$. 
\end{hypothesis}

The definition of a domain of class $C^{1,r}$, $0<r<1$, is similar to that 
of a Lipschitz domain, except that, this time, the functions $\varphi_j$ 
used to locally describe the boundary of $\Omega$ are such that 
$\nabla\varphi_j$ is H\"older continuous of order $r$.

Parenthetically, we note that a $C^{1,r}$ domain $\Omega\subset\bbR^n$, 
with $r\in(0,1)$, is characterized by the following set of conditions: 
\begin{align}\lb{FLP-1}
\begin{array}{l}
(i)\,\,
\Omega\mbox{ is of finite local perimeter}, (\mbox{i.e., $\nabla\chi_{\Om}$ 
is a Radon measure in $\bbR^n$}), 
\\
(ii)\,\,
\mbox{$\Om$ lies on only one side of its topological boundary}, (\mbox{i.e., }
\partial\ol{\Om}=\partial\Om), 
\\
(iii)\,\,
\mbox{the outward unit normal (in the sense of Federer \cite{Fe69}) belongs to } 
C^r(\dOm).
\end{array}
\end{align}
For a more detailed discussion in this regard, the reader is referred to 
\cite{HMT07}. 

The classical theorem of Rademacher on almost everywhere differentiability of Lipschitz
functions ensures that for any  Lipschitz domain $\Omega$, the
surface measure $d^{n-1} \omega$ is well-defined on  $\partial\Omega$ and
that there exists an outward  pointing unit normal vector $\nu$ at
almost every point of $\partial\Omega$. 

In the case where $\Omega\subset\bbR^n$ is the domain lying above the graph of 
a function $\varphi\colon\bbR^{n-1}\to\bbR$ of class $C^{1,r}$, $r\in(0,1)$, we 
define the Sobolev space $H^s(\partial\Omega)$ 
for $0\leq s<1+r$, as the space 
of functions $f\in L^2(\partial\Omega;d^{n-1}\omega)$ with the property that 
$f(x',\varphi(x'))$, as a function of $x'\in\bbR^{n-1}$, belongs to 
$H^s(\bbR^{n-1})$. This definition is easily adapted to the case where 
$\Omega$ is a domain of class $C^{1,r}$, $r\in(0,1)$, whose boundary is compact,
by using a smooth partition of unity. Finally, for $-1-r<s<0$, we set
$H^s(\partial\Omega)=\big(H^{-s}(\partial\Omega)\big)^*$. 
The same construction concerning $H^s(\partial\Omega)$ applies in the case where  
$\Om\subset\bbR^n$ is a Lipschitz domain (i.e., $\varphi\colon\bbR^{n-1}\to\bbR$
is only Lipschitz) provided $0\le s \le 1$. In this scenario we set 
\begin{equation}
H^s(\dOm) = \big(H^{-s}(\dOm)\big)^*, \quad -1 \le s \le 0.   \lb{A.6}
\end{equation}
It is useful to observe that, in the Lipschitz upper-graph case we are 
currently considering, this entails ($\approx$ denoting equivalent norms) 
\begin{equation}\label{Pk-D2}
\|f\|_{H^{-s}(\partial\Omega)}\approx
\|\sqrt{1+|\nabla\varphi(\cdot)|^2}f(\cdot,\varphi(\cdot))\|
_{H^{-s}(\bbR^{n-1})},\quad 0\leq s\leq 1.
\end{equation}

To define $H^s(\dOm)$, $0\leq s \le 1$, when $\Om$ is a Lipschitz domain with 
compact boundary, we use a smooth partition of unity to reduce matters to the 
graph case. More precisely, if $0\leq s\leq 1$ then $f\in H^s(\partial\Omega)$ 
if and only if the assignment 
${\mathbb{R}}^{n-1}\ni x'\mapsto (\psi f)(x',\varphi(x'))$ is in 
$H^s({\mathbb{R}}^{n-1})$ whenever $\psi\in C^\infty_0({\mathbb{R}}^n)$
and $\varphi\colon {\mathbb{R}}^{n-1}\to{\mathbb{R}}$ is a Lipschitz function
with the property that if $\Sigma$ is an appropriate rotation and
translation of $\{(x',\varphi(x'))\in\bbR^n \,|\,x'\in{\mathbb{R}}^{n-1}\}$, 
then $(\supp (\psi) \cap\partial\Omega)\subset\Sigma$ (this appears to 
be folklore, but a proof will appear in \cite[Proposition 2.4]{MM07}). 
Then Sobolev spaces with a negative amount of smoothness are defined as 
in \eqref{A.6} above. 

From the above characterization of $H^s(\partial\Omega)$ it follows that 
any property of Sobolev spaces (of order $s\in[-1,1]$) defined in Euclidean 
domains, which are invariant under multiplication by smooth, compactly 
supported functions as well as composition by bi-Lipschitz diffeomorphisms, 
readily extends to the setting of $H^s(\partial\Omega)$ (via localization and
pullback). As a concrete example, for each Lipschitz domain $\Omega$ 
with compact boundary, one has  
\begin{equation} \label{EQ1}
H^s(\partial\Omega)\hookrightarrow L^2(\partial\Omega;d^{n-1} \omega)
\, \text{ compactly if }\,0<s\leq 1.  
\end{equation}
For additional background 
information in this context we refer, for instance, to \cite{Au04}, 
\cite{Au06}, \cite[Chs.\ V, VI]{EE89}, \cite[Ch.\ 1]{Gr85}, 
\cite[Ch.\ 3]{Mc00}, \cite[Sect.\ I.4.2]{Wl87}.

For a Lipschitz domain $\Om\subset\bbR^n$ with compact boundary, an equivalent 
definition of the Sobolev space $H^1(\partial\Omega)$ is the collection of 
functions in $L^2(\partial\Omega;d^{n-1}\omega)$ with the property that the
(pointwise, Euclidean) norm of their tangential gradient belongs to 
$L^2(\partial\Omega;d^{n-1}\omega)$. To make this precise, 
consider the first-order tangential derivative operators 
$\partial/\partial\tau_{j,k}$, $1\leq j,k\leq n$, acting on a function 
$\psi$ of class $C^1$ in a neighborhood of $\partial\Omega$ by 
\begin{eqnarray}\label{def-TAU}
\partial\psi/\partial\tau_{j,k}=\nu_j(\partial_k\psi)\Bigl|_{\partial\Omega}
-\nu_k(\partial_j\psi)\Bigl|_{\partial\Omega}.
\end{eqnarray}
For every $f\in L^1(\partial\Omega)$ define the functional 
$\partial f/\partial\tau_{j,k}$ by setting 
\begin{eqnarray}\label{IBP-tau}
\partial f/\partial\tau_{j,k}:C^1({\mathbb{R}}^{n})\ni\psi\mapsto
\int_{\partial\Omega}d^{n-1}\omega\,f\,(\partial\psi/\partial\tau_{k,j}). 
\end{eqnarray}
When $f\in L^1(\partial\Omega;d^{n-1}\omega)$ has 
$\partial f/\partial\tau_{j,k}\in L^1(\partial\Omega;d^{n-1}\omega)$, 
the following integration by parts formula holds: 
\begin{eqnarray}\label{IBP-t2}
\int_{\partial\Omega} d^{n-1}\omega\,f\,(\partial\psi/\partial\tau_{k,j})
=\int_{\partial\Omega} d^{n-1}\omega\,(\partial f/\partial\tau_{j,k})\,\psi, 
\quad \psi\in C^1({\mathbb{R}}^{n}).
\end{eqnarray}
One then has the Sobolev-type description of $H^1(\partial\Omega)$: 
\begin{equation}\label{M1.1}
H^{1}(\dOm) = \big\{f\in\LdOm \,\big|\, \partial f/\partial\tau_{j,k}
\in\LdOm, \; j,k = 1,\dots,n\big\}, 
\end{equation}
with 
\begin{equation}\label{M1.1y}
\|f\|_{H^{1}(\dOm)}\approx\|f\|_{\LdOm}+\sum_{j,k = 1}^n
\|\partial f/\partial\tau_{j,k}\|_{\LdOm}, 
\end{equation}
or equivalently,
\begin{align}
H^{1}(\dOm) &= \bigg\{f\in\LdOm \,\bigg|\, \, 
\text{there exists a constant $c>0$ such that}   \no \\ 
& \hspace*{6.4cm} \;  \text{for every $v\in C_0^\infty(\bbR^n)$,}   \lb{A.64} \\
& \qquad \;\;\, \bigg|\int_\dOm d^{n-1} \omega f\,\partial
v/\partial\tau_{j,k}\bigg| \leq c \norm{v}_{\LdOm},\;
j,k=1,\dots,n\bigg\}. \no 
\end{align}

We also point out that if $\Omega\subset\bbR^n$ is a bounded
Lipschitz domain, then for any $j,k\in\{1,...,n\}$, the operator 
\begin{equation}\label{Pf-2}
\partial/\partial\tau_{j,k}:H^s(\partial\Omega)\to H^{s-1}(\partial\Omega),
\quad 0\leq s\leq 1, 
\end{equation}
is well-defined, linear, and bounded. This is proved by interpolating 
the case $s=1$ and its dual version. In fact, the following more general
result (extending \eqref{M1.1}) holds: 

\begin{lemma}\label{Lg-T}
Assume Hypothesis \ref{h2.1}. Then for every $s\in[0,1]$, 
\begin{equation}\label{Pf-2.3}
H^s(\partial\Omega)=\{f\in L^2(\partial\Omega;d^{n-1}\omega)\,|\,
\partial f/\partial\tau_{j,k}\in H^{s-1}(\partial\Omega), \, 
1\leq j,k\leq n\}
\end{equation}
and 
\begin{equation}\label{Pf-2.4}
\|f\|_{H^s(\partial\Omega)}\approx \|f\|_{L^2(\partial\Omega;d^{n-1}\omega)}
+\sum_{j,k=1}^n\|\partial f/\partial\tau_{j,k}\|_{H^{s-1}(\partial\Omega)}.
\end{equation}
\end{lemma}

A proof can be found in \cite{GM08}, where the following result is also 
established. 

\begin{lemma}\label{K-t1}
Assume Hypothesis \ref{h2.1C}. Then 
\begin{equation}\lb{M.2}
H^{3/2}(\dOm)=\{f\in H^1(\dOm)\,|\,\partial f/\partial\tau_{j,k}\in 
H^{1/2}(\dOm),\,\,1\leq j,k\leq n\}
\end{equation}
and
\begin{equation}\lb{M.2x}
\|f\|_{H^{3/2}(\dOm)}\approx\|f\|_{H^1(\dOm)}
+\sum_{j,k=1}^n\|\partial f/\partial\tau_{j,k}\|_{H^{1/2}(\dOm)}. 
\end{equation}
\end{lemma}

In the sequel, the sesquilinear form
\begin{equation}\label{PPL}
\langle \dott, \dott \rangle_{s}={}_{H^{s}(\dOm)}\langle\dott,\dott
\rangle_{(H^{s}(\dOm))^*}\colon H^{s}(\dOm)\times \big(H^{s}(\dOm)\big)^*
\to \bbC,  
\end{equation}
(antilinear in the first, linear in the second factor), will occasionally denote the duality
pairing between $H^s(\dOm)$ and $\big(H^s(\dOm)\big)^*$ for appropriate $s \geq 0$. 
In particular, 
\begin{equation} 
\langle f,g\rangle_s=\int_{\dOm} d^{n-1}\omega(\xi)\,\ol{f(\xi)}g(\xi) 
= (f,g)_{L^2(\dOm;d^{n-1}\omega)},  \quad 
f\in H^s(\dOm),\,g\in L^2(\dOm;d^{n-1}\omega),   \lb{2.4}
\end{equation}
and
\begin{equation}
L^2(\dOm;d^{n-1}\omega)\hookrightarrow 
\big(H^{s}(\dOm)\big)^* = H^{-s}(\dOm), \quad  s\in [0,1],
\end{equation}
where, as before, $d^{n-1}\omega$ stands for the surface measure on $\dOm$.

\medskip 

In the final part of this section we wish to further comment on the nature 
of the spaces $H^s(\dOm)$, $-1\leq s\leq 1$, in the case where $\Om\subseteq\bbR^n$ 
is a Lipschitz domain with compact boundary. Specifically, the goal is to 
indicate that the aforementioned Sobolev spaces have a canonical Hilbert 
space structure. Describing it, requires some preparations. Let $\nabla_{tan}$ 
denote the tangential gradient operator on $\partial\Omega$, 
mapping scalar-valued functions to vector fields, defined as 
\begin{equation}\label{EEs-5}
\nabla_{tan}:=\bigg(\sum_{k=1}^n\nu_k\partial/\partial\tau_{k,1},\cdots,
\sum_{k=1}^n\nu_k\partial/\partial\tau_{k,n}\bigg),
\end{equation}
where $\nu=(\nu_1,...,\nu_n)$ denotes the outward unit normal to $\Omega$.  
Hence, if we consider the space of tangential vector fields with 
square-integrable components on $\partial\Omega$, that is, 
\begin{equation}\label{Agg3}
L^2_{tan}(\partial\Omega;d^{n-1}\omega)
:= \big\{f=(f_1,...,f_n) \,\big|\, f_j\in L^2(\partial\Omega;d^{n-1}\omega), \,
1\leq j\leq n,\,\,\nu\cdot f=0 
\mbox{ $\omega$-a.e.\ on } \partial\Omega\big\}, 
\end{equation}
and equip it with the norm inherited from 
$\big[L^2(\partial\Omega;d^{n-1}\omega)\big]^n$, then 
\begin{equation}\label{Aggab}
\nabla_{tan}:H^1(\partial\Omega) \to  
L^2_{tan}(\partial\Omega;d^{n-1}\omega),
\end{equation}
is a well-defined and bounded operator. 

\begin{theorem}\label{TH-GMMM}
Let $\Omega\subseteq{\mathbb{R}}^n$, $n\geq 2$, be a Lipschitz domain with 
compact boundary. Define the Laplace--Beltrami operator on the Lipschitz surface 
$\dOm$ as the linear unbounded operator in $L^2(\partial\Omega;d^{n-1}\omega)$ 
given by 
\begin{align}
& -\Delta_{\partial\Omega}f:=g,\quad 
f\in{\rm dom}\,(-\Delta_{\partial\Omega}), 
\;  g\mbox{ as in \eqref{Tha-2}}, 
\label{Tha-3} \\
& \,\, {\rm dom}\,(-\Delta_{\partial\Omega}):=\bigg\{f\in H^1(\partial\Omega)\,\bigg|\,
\mbox{there exists $g\in L^2(\partial\Omega;d^{n-1}\omega)$ such that}  
\nonumber\\
& \hspace*{2.8cm} 
\int_{\partial\Omega}d^{n-1}\omega\,
\overline{\nabla_{tan}f}\,\nabla_{tan}h
=\int_{\partial\Omega}d^{n-1}\omega\,\overline{g}h 
\mbox{ for all $h\in H^1(\partial\Omega)$}\bigg\},
\label{Tha-2}
\end{align}
This is a nonnegative self-adjoint operator in $L^2(\partial\Omega;d^{n-1}\omega)$
which has the following additional properties:

\begin{enumerate}
\item[$(i)$] For every $s\in[0,1]$, one has 
\begin{eqnarray}\label{maRf-1}
{\rm dom}\,\bigl((-\Delta_{\partial\Omega}+ I_{\partial\Om})^{s/2}\bigr)
=H^s(\partial\Omega)
\end{eqnarray}
and 
\begin{eqnarray}\label{maRf-2}
\begin{array}{l}
(-\Delta_{\partial\Omega}+ I_{\partial\Om})^{s/2}\in\cB\big(H^s(\partial\Omega), 
L^2(\partial\Omega;d^{n-1}\omega)\big) \, \mbox{ is an isomorphism}, 
\\[6pt]
\mbox{with inverse } \, 
(-\Delta_{\partial\Omega}+ I_{\partial\Om})^{-s/2}
\in\cB\big(L^2(\partial\Omega;d^{n-1}\omega), 
H^s(\partial\Omega)\big).
\end{array}
\end{eqnarray}
As a consequence, the isomorphism 
\begin{eqnarray}\label{ma-2SZ}
(-\Delta_{\partial\Omega}+ I_{\partial\Om})^{-s/2}=
\bigl[(-\Delta_{\partial\Omega}+ I_{\partial\Om})^{s/2}\bigr]^{-1}
:L^2(\partial\Omega;d^{n-1}\omega)
\to H^s(\partial\Omega),\quad 0\leq s\leq 1,
\end{eqnarray}
can be thought of as a lifting $($or smoothing\,$)$ operator of order $s$. 

\item[$(ii)$] More generally, for every $r, s \in[0,1]$, 
\begin{eqnarray}\label{maRf-2K}
\begin{array}{l}
(-\Delta_{\partial\Omega}+ I_{\partial\Om})^{r/2}\in\cB\big(
H^s(\partial\Omega)\,,\,H^{s-r}(\partial\Omega)\big) \, 
\mbox{ is an isomorphism}, 
\\[6pt]
\mbox{with inverse } \, 
(-\Delta_{\partial\Omega}+ I_{\partial\Om})^{-r/2}\in\cB\big(H^{s-r}(\partial\Omega), 
H^s(\partial\Omega)\big).
\end{array}
\end{eqnarray}

\item[$(iii)$] For each $s\in[-1,1]$, the norm induced by the inner product 
\begin{eqnarray}\label{EEs-3X}
(f,g)_{H^s(\partial\Omega)}:=
\int_{\partial\Omega}d^{n-1}\omega\,
\overline{(-\Delta_{\partial\Omega}+ I_{\partial\Om})^{s/2}f}\,
(-\Delta_{\partial\Omega}+ I_{\partial\Om})^{s/2}g,  \quad f,g\in H^s(\partial\Omega),
\end{eqnarray}
is equivalent with the original norm on $H^s(\partial\Omega)$. Hence 
$H^s(\partial\Omega)$ has a canonical Hilbert space structure for 
every $s\in[-1,1]$.
\end{enumerate}
\end{theorem}

A detailed discussion of Theorem \ref{TH-GMMM} (as well as other related results)
can be found in \cite{GMMM10}. We wish to stress, however, that throughout the 
present paper we prefer to work with a duality pairing between 
$H^s(\dOm)$ and $(H^s(\dOm))^*$ which is compatible with the canonical 
pairing in $L^2(\dOm;d^{n-1}\omega)$ (cf. \eqref{PPL}, \eqref{2.4}).

\section{The Dirichlet and Neumann Trace Operators 
on Lipschitz Domains}
\label{s3}

We briefly recall basic properties of Dirichlet and Neumann trace operators in 
Lipschitz domains. 

Assuming Hypothesis \ref{h2.1}, we introduce the boundary trace 
operator $\ga_D^0$ (the Dirichlet trace) by
\begin{equation}
\ga_D^0\colon C(\ol{\Om})\to C(\dOm), \quad \ga_D^0 u = u|_\dOm.   \lb{2.5}
\end{equation}
Then there exists a bounded, linear operator $\gamma_D$ (cf., e.g., 
\cite[Theorem 3.38]{Mc00}),
\begin{align}
\begin{split}
& \ga_D\colon H^{s}(\Om)\to H^{s-(1/2)}(\dOm) \hookrightarrow \LdOm,
\quad 1/2<s<3/2, \lb{2.6}  \\
& \ga_D\colon H^{3/2}(\Om)\to H^{1-\varepsilon}(\dOm) \hookrightarrow \LdOm,
\quad \varepsilon \in (0,1), 
\end{split}
\end{align}
whose action is compatible with that of $\ga_D^0$. That is, the two
Dirichlet trace  operators coincide on the intersection of their
domains. Moreover, we recall that 
\begin{equation}\lb{2.6a}
\ga_D\colon H^{s}(\Om)\to H^{s-(1/2)}(\dOm) \, \text{ is onto for  
$1/2<s<3/2$}, 
\end{equation}
and point out that the adjoint of \eqref{2.6} maps boundedly as follows
\begin{equation}\lb{ga*}
\ga_D^* \colon \big(H^{s-1/2}(\dOm)\big)^* \to (H^{s}(\Om)\big)^*,
\quad 1/2<s<3/2. 
\end{equation} 

While in the class of bounded Lipschitz subdomains in $\bbR^n$  
the endpoint cases $s=1/2$ and $s=3/2$ of 
$\gamma_D\in\cB\bigl(H^{s}(\Om),H^{s-(1/2)}(\dOm)\bigr)$ fail, we nonetheless
have
\begin{eqnarray}\label{A.62x}
\ga_D\in \cB\big(H^{(3/2)+\eps}(\Om), H^{1}(\dOm)\big), \quad \eps>0.
\end{eqnarray}
See \cite{GM08} for a proof. It is useful to augment this with the 
following result, also proved in \cite{GM08}:

\begin{lemma}\label{Gam-L1}
Assume Hypothesis \ref{h2.1}. Then for each $s>-3/2$, $\gamma_D^{0}$ in \eqref{2.5} extends to a linear operator 
\begin{eqnarray}\label{Mam-1}
\gamma_D:\bigl\{u\in H^{1/2}(\Omega)\,|\,\Delta u\in H^{s}(\Omega)\bigr\}
\to L^2(\partial\Omega;d^{n-1}\omega),
\end{eqnarray}
is compatible with \eqref{2.6}, and is bounded when 
$\{u\in H^{1/2}(\Omega)\,|\,\Delta u\in H^{s}(\Omega)\bigr\}$ is equipped
with the natural graph norm $u\mapsto \|u\|_{H^{1/2}(\Omega)}
+\|\Delta u\|_{H^{s}(\Omega)}$. In addition, this operator has a
linear, bounded right-inverse $($thus, in particular, it is onto\,$)$. 

Furthermore, for each $s>-3/2$, $\gamma_D^{0}$ in \eqref{2.5} also extends to a linear operator 
\begin{eqnarray}\label{Mam-2}
\gamma_D:\bigl\{u\in H^{3/2}(\Omega)\,|\,\Delta u\in H^{1+s}(\Omega)\bigr\}
\to H^1(\partial\Omega), 
\end{eqnarray}
which again is compatible with \eqref{2.6}, and is bounded when 
$\{u\in H^{3/2}(\Omega)\,|\,\Delta u\in H^{1+s}(\Omega)\bigr\}$ is equipped
with the natural graph norm $u\mapsto \|u\|_{H^{3/2}(\Omega)}
+\|\Delta u\|_{H^{1+s}(\Omega)}$. Once again, this operator has a
linear, bounded right-inverse $($hence, in particular, it is onto\,$)$. 
\end{lemma}

For later purposes, let us record here the following useful version 
of the Divergence Theorem:
\begin{align}\label{Ht-r3}
& \left.
\begin{array}{l}
\Om\, \mbox{ as in }\, \mbox{Hypothesis \ref{h2.1},}
\\
G\in\bigl\{u\in H^{1/2}(\Omega)\,|\,\Delta u\in H^{s}(\Omega)\bigr\}^n,\,\,
s>-{\textstyle\frac32}, 
\\
\mbox{with the property that }\,\,{\rm div} (G)\in L^1(\Omega; d^n x)
\end{array}  \right\}  
\, \text{ imply } \,  
\int_{\Omega}dx^n\,{\rm div} (G)
=\int_{\partial\Omega}d^{n-1}\omega\,\nu\cdot\wti\gamma_D G.  
\end{align} 
A proof can be found in \cite{GM08}.

Next, retaining Hypothesis \ref{h2.1}, we introduce the operator 
$\ga_N$ (the strong Neumann trace) by
\begin{align}\lb{2.7} 
\ga_N = \nu\cdot\ga_D\nabla \colon H^{s+1}(\Om)\to \LdOm, \quad 1/2<s<3/2, 
\end{align}
where $\nu$ denotes the outward pointing normal unit vector to
$\partial\Om$. It follows from \eqref{2.6} that $\ga_N$ is also a
bounded operator. We seek to extend the action of the Neumann trace
operator \eqref{2.7} to other (related) settings. To set the stage, 
assume Hypothesis \ref{h2.1} and observe that the inclusion 
\begin{equation}\lb{inc-1}
\iota:H^{s_0}(\Omega)\hookrightarrow \bigl(H^r(\Omega)\bigr)^*, \quad 
s_0>-1/2,\; r>1/2, 
\end{equation}
is well-defined and bounded. One then introduces the weak Neumann trace 
operator 
\begin{equation}\lb{2.8}
\wti\ga_N\colon\big\{u\in H^{s+1/2}(\Om)\,\big|\,\Delta u\in H^{s_0}(\Om)\big\} 
\to H^{s-1}(\dOm),\quad s\in(0,1),\; s_0>-1/2,  
\end{equation}
as follows: Given $u\in H^{s+1/2}(\Om)$ with $\Delta u \in H^{s_0}(\Om)$ 
for some $s\in(0,1)$ and $s_0>-1/2$, we set (with $\iota$ as in 
\eqref{inc-1} for $r:=3/2-s>1/2$)
\begin{align}\lb{2.9}
\langle \phi, \wti\ga_N u \rangle_{1-s}
={}_{H^{1/2-s}(\Om)}\langle \nabla \Phi,\nabla u\rangle_{(H^{1/2-s}(\Om))^*}
+ {}_{H^{3/2-s}(\Om)}\langle \Phi, \iota(\Delta u)\rangle_{(H^{3/2-s}(\Om))^*}, 
\end{align}
for all $\phi\in H^{1-s}(\dOm)$ and $\Phi\in H^{3/2-s}(\Om)$ such that
$\ga_D\Phi = \phi$. We note that the first pairing on the right-hand side 
above is meaningful since 
\begin{eqnarray}\lb{2.9JJ}
\bigl(H^{1/2-s}(\Om)\bigr)^*=H^{s-1/2}(\Om),\quad s\in (0,1),
\end{eqnarray}
that the definition \eqref{2.9} is independent of the particular 
extension $\Phi$ of $\phi$, and that $\wti\ga_N$ is a bounded extension 
of the Neumann trace operator $\ga_N$ defined in \eqref{2.7}.

Corresponding to the endpoint cases $s=0,1$ of \eqref{2.8}, the 
following result has been established in \cite{GM08}:  

\begin{lemma}\label{Neu-tr}
Assume Hypothesis \ref{h2.1}. 
Then the Neumann trace operator \eqref{2.7} also extends to
\begin{eqnarray}\label{MaX-1}
\wti\gamma_N:\bigl\{u\in H^{3/2}(\Omega)\,|\,\Delta u\in L^2(\Omega;d^nx)\bigr\}
\to L^2(\partial\Omega;d^{n-1}\omega) 
\end{eqnarray}
in a bounded fashion when the space 
$\{u\in H^{3/2}(\Omega)\,|\,\Delta u\in L^2(\Omega;d^nx)\bigr\}$ is equipped
with the natural graph norm $u\mapsto \|u\|_{H^{3/2}(\Omega)}
+\|\Delta u\|_{L^2(\Omega;d^nx)}$. This extension is compatible 
with \eqref{2.8} and has a linear, bounded right-inverse $($hence, it is onto\,$)$.  

Moreover, the Neumann trace operator \eqref{2.7} further extends to
\begin{eqnarray}\label{MaX-1U}
\wti\gamma_N:\bigl\{u\in H^{1/2}(\Omega)\,|\,\Delta u\in L^2(\Omega;d^nx)\bigr\}
\to H^{-1}(\dOm) 
\end{eqnarray}
in a bounded fashion when the space 
$\{u\in H^{1/2}(\Omega)\,|\,\Delta u\in L^2(\Omega;d^nx)\bigr\}$ is equipped
with the natural graph norm $u\mapsto \|u\|_{H^{1/2}(\Omega)}
+\|\Delta u\|_{L^2(\Omega;d^nx)}$. Once again, this extension is compatible 
with \eqref{2.8} and has a linear, bounded right-inverse $($thus, in particular,
it is onto\,$)$.  
\end{lemma}
 
For future purposes, we shall need yet another extension of the concept of a 
Neumann trace. This requires some preparation  
(throughout, Hypothesis \ref{h2.1} is enforced). First, we recall that, 
as is well-known (see, e.g., \cite{JK95}), one has the natural identification 
\begin{eqnarray}\lb{jk-9}
\big(H^{1}(\Om)\big)^*\equiv
\big\{u\in H^{-1}(\bbR^n)\,\big|\, \supp (u) \subseteq\overline{\Omega}\big\}.
\end{eqnarray} 
We note that the latter is a closed subspace of $H^{-1}(\bbR^n)$. 
In particular, if $R_{\Omega}u:=u|_{\Omega}$ denotes the operator 
of restriction to $\Omega$ (considered in the sense of distributions), then 
\begin{eqnarray}\lb{jk-10}
R_{\Omega}:\big(H^{1}(\Om)\big)^*\to  H^{-1}(\Om)
\end{eqnarray}
is well-defined, linear, and bounded. Furthermore, the 
composition of $R_\Omega$ in \eqref{jk-10} with $\iota$ in \eqref{inc-1}
is the natural inclusion of $H^s(\Om)$ into $H^{-1}(\Om)$. 
Next, given $z\in\bbC$, set 
\begin{eqnarray}\lb{2.88X}
W_z(\Om):=\bigl\{(u,f)\in H^1(\Om)\times\bigl(H^1(\Om)\bigr)^* \,\big|\,
(-\Delta-z)u=f|_{\Omega}\mbox{ in }\mathcal{D}'(\Om)\bigr\}, 
\end{eqnarray}
equipped with the norm inherited from 
$H^1(\Om)\times\bigl(H^1(\Om)\bigr)^*$. We then denote by
\begin{equation}\lb{2.8X}
\wti\ga_{\cN}\colon W_z(\Om)\to  H^{-1/2}(\dOm)  
\end{equation}
the ultra weak Neumann trace operator defined by
\begin{align}\lb{2.9X}
\begin{split}
& \langle\phi,\wti\ga_{\cN} (u,f)\rangle_{1/2}
:= \int_\Om d^n x\,\ol{\nabla \Phi(x)} \cdot \nabla u(x)  \\
& \quad -z\,\int_\Om d^n x\,\ol{\Phi(x)}u(x)  
-{}_{H^1(\Om)}\langle \Phi, f\rangle_{(H^1(\Om))^*}, \quad 
(u,f)\in W_z(\Om), 
\end{split} 
\end{align}
for all $\phi\in H^{1/2}(\dOm)$ and $\Phi\in H^1(\Om)$ such that 
$\ga_D\Phi=\phi$. Once again, this definition is independent of the 
particular extension $\Phi$ of $\phi$. As a corollary, the following Green's 
formula
\begin{equation}
\langle \ga_D\Phi, \wti\gamma_{\cN}(u,f)\rangle_{1/2} 
= (\nabla \Phi, \nabla u)_{\LOm^n} 
+  {}_{H^1(\Om)}\langle \Phi, f\rangle_{(H^1(\Om))^*},
\lb{wGreen}
\end{equation}
is valid, granted Hypothesis \ref{h2.1}, for any $u\in H^{1}(\Om)$, 
$f\in \big(H^{1}(\Om)\big)^*$ with $\Delta u=f|_{\Om}$, and any 
$\Phi\in H^{1}(\Om)$. The pairing on the left-hand side of \eqref{wGreen} 
is between functionals in $\big(H^{1/2}(\dOm)\big)^*$ and elements in
$H^{1/2}(\dOm)$, whereas the last pairing on the right-hand side
is between functionals in $\big(H^{1}(\Om)\big)^*$ and elements in
$H^{1}(\Om)$. Furthermore, as in the case of the Dirichlet trace, the 
ultra weak Neumann trace operator \eqref{2.8X}, \eqref{2.9X} is onto 
(this is a corollary of Theorem \ref{t3.XV}).

The relationship between the ultra weak Neumann trace operator 
\eqref{2.8X}, \eqref{2.9X} and the weak Neumann trace operator 
\eqref{2.8}, \eqref{2.9} can be described as follows: Given 
$s>-1/2$ and $z\in\bbC$, denote by 
\begin{eqnarray}\lb{2.10X}
j_z:\{u\in H^1(\Om)\,\big|\,\Delta u\in H^s(\Om)\big\} \to W_z(\Om)
\end{eqnarray}
the injection 
\begin{eqnarray}\lb{2.11X}
j_z(u):=(u,\iota(-\Delta u -zu)),\quad u\in H^1(\Om),
\,\,\,\Delta u\in H^s(\Om),
\end{eqnarray}
where $\iota$ is as in \eqref{inc-1}. Then 
\begin{eqnarray}\lb{2.12X}
\wti\gamma_{\cN}  j_z=\wti\ga_N.
\end{eqnarray}
Thus, from this particular perspective, $\wti\ga_{\cN}$ can also be regarded 
as a bounded extension of the Neumann trace operator $\ga_N$ defined 
in \eqref{2.7}.

\begin{remark}\lb{GaNu}
Since the ultra weak Neumann trace $\wti\ga_{\cN}(u,f)$ is defined for 
a class of (pairs of) functions $(u,f)\in W_z(\Om)$ for which the notion 
of the strong Neumann trace $\ga_N u$ is utterly ill-defined, it is 
appropriate to remark that $(u,f)\mapsto \wti\ga_{\cN}(u,f)$ is not an 
extension of the operation of taking the trace $u\mapsto\ga_N u$ in an 
ordinary sense. In fact, it is more appropriate to regard the former as a 
``renormalization'' of the latter trace, in a fashion that depends strongly on 
the choice of $f$. 

To further shed light on this issue, we recall that for $u\in H^1(\Omega)$, 
$\Delta u$ is naturally defined as 
a linear functional in $(H^1_0(\Omega))^*$, where $H^1_0(\Omega)$ is 
the closure of $C_0^\infty(\Omega)$ in $H^1(\Omega)$. The choice of $f$ 
is the choice of an extension of this linear functional to a functional in 
$(H^1(\Omega))^*$ (a space which is naturally identified with 
$H^{-1}_0(\Om):= 
\big\{g\in H^{-1}(\bbR^n) \,\big|\, \supp (g) \subseteq\ol{\Omega}\big\}$).  
As an example, consider $u\in H^1(\Omega)$ and suppose that actually  
$u\in H^2(\Omega)$, so $\ga_N u \in L^2(\dOm;d^{n-1}\omega)$ is well defined.  
In this case, $\Delta u\in L^2(\Om;d^nx)$ has a ``natural''  extension 
$f_0\in H^{-1}_0(\Omega)$ (i.e., $f_0$ is the extension of $\Delta u$ to 
$\bbR^n$ by setting this equal zero outside $\Om$). Any other extension  
$f_1\in H^{-1}_0(\Omega)$ differs from $f_0$ by a distribution   
$\eta\in H^{-1}(\bbR^n)$ supported on $\dOm$. We have  
\begin{eqnarray}\lb{gjjh}
\wti\ga_{\cN}(u,f_0)=\ga_N u,  
\end{eqnarray}
but if $\eta\neq 0$ then $\wti\ga_{\cN}(u,f_1)$ is not equal to 
$\wti\ga_{\cN}(u,f_0)$. Indeed, by linearity one concludes that
$\wti\ga_{\cN}(u,f_1)=\wti\ga_{\cN}((u,f_0)+(0,\eta))
=\wti\ga_{\cN}(u,f_0)+\wti\ga_{\cN}(0,\eta)$ and \eqref{2.9X} shows that
\begin{eqnarray}\lb{demo-1}
\langle\phi,\wti\ga_{\cN}(0,\eta)\rangle_{1/2}
=-{}_{H^1(\Om)}\langle \Phi,\eta\rangle_{(H^1(\Om))^*}, 
\end{eqnarray}
for each $\phi\in H^{1/2}(\dOm)$ and $\Phi\in H^1(\Om)$ such that 
$\ga_D\Phi=\phi$. Consequently, $\wti\ga_{\cN}(0,\eta)\not=0$ if $\eta\not=0$.
\end{remark}

We conclude this section by recording the following result, proved in 
\cite{GM08}:  

\begin{lemma} \lb{lA.6}
Assume Hypothesis \ref{h2.1}. Then for each $r\in(1/2,1)$, the H\"older 
space $C^r(\partial\Omega)$ is a module over $H^{1/2}(\partial\Omega)$. 
More precisely, if $M_f$ denotes the operator of multiplication by $f$, then 
there exists $C=C(\Omega,r)>0$ such that
\begin{equation}\lb{M.x1}
M_f\in{\mathcal{B}}\big(H^{1/2}(\partial\Omega)\big)\mbox{ and } \, 
\|M_f\|_{{\mathcal{B}}\bigl(H^{1/2}(\partial\Omega)\bigl)}
\leq C\|f\|_{C^r(\partial\Omega)}, \quad f\in C^r(\partial\Omega).
\end{equation}
As a consequence, if Hypothesis \ref{h2.1C} is enforced, 
then the Neumann and Dirichlet trace operators $\ga_N$, $\ga_D$ satisfy
\begin{equation}\lb{A.62}
\ga_N\in \cB\big(H^{2}(\Om), H^{1/2}(\dOm)\big), \quad 
\ga_D\in \cB\big(H^{2}(\Om), H^{3/2}(\dOm)\big). 
\end{equation}
\end{lemma}

\section{Boundary Layer Potential Operators on Lipschitz Domains}
\label{s4}

We recall the fundamental solution of the Helmholtz equation and some fundamental 
properties of boundary layer potential operators on Lipschitz domains.

Let $E_n(z;x)$ be the fundamental solution associated with the Helmholtz differential operator  
$(-\Delta -z)$ in $\bbR^n$, $n\in\bbN$, $n\geq 2$, that is,
\begin{align}
& E_n(z;x) = \begin{cases}
(i/4) \big(2\pi |x|/z^{1/2}\big)^{(2-n)/2} H^{(1)}_{(n-2)/2} 
\big(z^{1/2}|x|\big), & n\geq 2, 
\; z\in\bbC\backslash \{0\}, \\
\f{-1}{2\pi} \ln(|x|), & n=2, \; z=0, \\ 
\f{1}{(n-2)\omega_{n-1}}|x|^{2-n}, & n\geq 3, \; z=0, 
\end{cases}    \no  \\
& \hspace*{6.8cm} \Im\big(z^{1/2}\big)\geq 0,\; x\in\bbR^n\backslash\{0\}. 
 \lb{2.52}
\end{align}
Here $H^{(1)}_{\nu}(\cdot)$ denotes the Hankel function of the first kind 
with index $\nu\geq 0$ (cf.\ \cite[Sect.\ 9.1]{AS72}). 

Assuming Hypothesis \ref{h2.1}, we next introduce the  
single layer potential by setting 
\begin{equation}\label{sing-layer}
({\mathcal S_z}g)(x)=\int_{\partial\Omega}d^{n-1}\omega(y)\,E_n(z;x-y)g(y),
\quad x\in\Omega, \; z\in\bbC, 
\end{equation}
where $g$ is an arbitrary measurable function on $\partial\Omega$. 
We shall also be interested in the adjoint double layer on 
$\partial\Omega$, given by 
\begin{equation}\label{Ksharp}
(K^{\#}_zg)(x)={\rm p.v.}\int_{\partial\Omega}d^{n-1} \omega(y)\,
\partial_{\nu_x}E_n(z;x-y)g(y),
\quad x\in\partial\Omega, \; z\in\bbC. 
\end{equation}

We denote by $I_{\dOm}$ the identity operator in the various spaces of 
functions (and distributions) on $\dOm$. In \cite{GM08}, the following 
results boundedness and jump-relations have been established: 

\begin{lemma}\label{L-Kb}
Assume Hypothesis \ref{h2.1} and fix $z\in\bbC$. Then 
\begin{align}
& K^{\#}_z\in\cB\big(\LdOm\big),      \lb{MM-1.8} \\
& \gamma_D\cS_{z}\in \cB\big(\LdOm,H^1(\dOm)\big),  \lb{MM-1.9X} \\
& \wti\gamma_N {\mathcal S_{z}}g
=\big({\textstyle{-\frac12}}I_\dOm+K^{\#}_{z}\big)g, \quad g\in\LdOm.
\label{jump}
\end{align}
\end{lemma}

\begin{lemma}\label{K-CPT}
Assume Hypothesis \ref{h2.1C}. Then 
\begin{equation}\label{Ks-4B}
K^{\#}_z\in\cB_\infty\big(H^{1/2}(\dOm)\big),\quad z\in\bbC.
\end{equation}
Furthermore, 
\begin{equation}\label{goal-2}
{\mathcal S_{z}}\in \cB\big(H^{1/2}(\partial\Omega), H^{2}(\Omega)\big),
\quad z\in\bbC.
\end{equation}
\end{lemma}

\section{Dirichlet and Neumann Boundary Value Problems 
on Lipschitz Domains}
\label{s5}

This section is devoted to Dirichlet and Neumann Laplacians and to Dirichlet and 
Neumann boundary value problems. We also briefly recall some results on nonlocal 
Robin Laplacians. 

We start by reviewing the Dirichlet Laplacian $-\Delta_{D,\Om}$ associated with 
a domain $\Om$ in $\bbR^n$. 

\begin{theorem} \lb{t2.5}
Assume Hypothesis \ref{h2.1}. Then the Dirichlet Laplacian, 
$-\Delta_{D,\Om}$, defined by 
\begin{align}
& -\Delta_{D,\Om} = -\Delta,   \no \\ 
& \; \dom(-\Delta_{D,\Om}) = 
\big\{u\in H^1(\Om)\,\big|\, \Delta u \in L^2(\Om;d^n x); \, 
\gamma_D u =0 \text{ in $H^{1/2}(\dOm)$}\big\}   \no \\
& \hspace*{2.13cm} = \big\{u\in H_0^1(\Om)\,\big|\, \Delta u \in L^2(\Om;d^n x)\big\},  \lb{2.39}
\end{align}
is self-adjoint and strictly positive in $L^2(\Om;d^nx)$. Moreover,
\begin{equation}
\dom\big((-\Delta_{D,\Om})^{1/2}\big) = H^1_0(\Om).   \lb{2.40}
\end{equation}
\end{theorem}

This is essentially known; we refer to \cite{GM08} for a discussion  
where also the following result concerning the Neumann Laplacian $-\Delta_{N,\Om}$ 
associated with $\Om$ can be found: 

\begin{theorem}\lb{t2.3}
Assume Hypothesis \ref{h2.1}. Then the Neumann Laplacian, 
$-\Delta_{N,\Om}$, defined by 
\begin{align}
& -\Delta_{N,\Om} = -\Delta,   \lb{2.20} \\ 
& \; \dom(-\Delta_{N,\Om}) = 
\big\{u\in H^1(\Om)\,\big|\, \Delta u \in L^2(\Om;d^nx); \, 
\wti\gamma_N u =0 \text{ in $H^{-1/2}(\dOm)$}\big\},  \no 
\end{align}
is self-adjoint and bounded from below in $L^2(\Om;d^nx)$. Moreover,
\begin{equation}
\dom\big(|-\Delta_{N,\Om}|^{1/2}\big) = H^1(\Om).   \lb{2.21}
\end{equation}
\end{theorem}

In our next two theorems, we review well-posedness results for 
the Dirichlet and Neumann problems. To state these we will denote the 
identity operator in the various spaces of functions (and distributions) 
in $\Om$ as $I_{\Om}$. 

\begin{theorem} \lb{t3.3}
Assume Hypothesis \ref{h2.1} and suppose that 
$z\in\bbC\backslash\si(-\Delta_{D,\Om})$. Then for every $f \in H^s(\dOm)$, 
$0\leq s\leq 1$, the following Dirichlet boundary value problem,
\begin{equation} \lb{3.31}
\begin{cases}
(-\Delta - z)u = 0 \text{ in }\, \Om, \quad u \in H^{s+1/2}(\Om), \\
\ga_D u = f \text{ on }\, \dOm,
\end{cases}
\end{equation}
has a unique solution $u=u_D$. This solution $u_D$ satisfies 
\begin{equation}\lb{Hh.3} 
\wti\ga_N u_D \in H^{s-1}(\dOm)\, \mbox{ and }\,  
\|\wti\ga_N u_D\|_{H^{s-1}(\dOm)}\leq C_D\|f\|_{H^s(\dOm)},
\end{equation}
for some constant $C_D=C_D(\Omega,s,z)>0$. Moreover, 
\begin{equation}\lb{3.33}
\|u_D\|_{H^{s+1/2}(\Omega)} \leq C_D \|f\|_{H^s(\partial\Omega)}.  
\end{equation} 
Finally,      
\begin{equation}
\big[\wti\ga_N (-\Delta_{D,\Om}-{\ol z}I_\Om)^{-1}\big]^* \in 
\cB\big(H^s(\dOm), H^{s+1/2}(\Om)\big),   \lb{3.34} 
\end{equation}
and the solution $u_D$ is given by the formula
\begin{equation}
u_D = -\big[\wti\ga_N (-\Delta_{D,\Om}-\ol{z}I_\Om)^{-1}\big]^*f. 
\lb{3.35}
\end{equation}
\end{theorem}

\begin{theorem}\lb{t3.2H} 
Assume Hypothesis \ref{h2.1} and suppose that 
$z\in\bbC\backslash\si(-\Delta_{N,\Om})$. Then for every 
$g\in H^{s-1}(\dOm)$, $0\leq s\leq 1$, the following Neumann boundary 
value problem,
\begin{equation} \lb{3.6H}
\begin{cases}
(-\Delta - z)u = 0 \text{ in }\,\Om,\quad u \in H^{s+1/2}(\Om), \\
\wti\ga_N u = g \text{ on } \,\dOm,
\end{cases}
\end{equation}
has a unique solution  $u=u_N$. This solution $u_N$ satisfies
\begin{eqnarray}\lb{3.6aH}
\begin{array}{l}
\ga_D u_N\in H^s(\dOm)\mbox{ and }
\|\ga_D u_N\|_{H^s(\dOm)}\leq C\|g\|_{H^{s-1}(\dOm)}, 
\end{array}
\end{eqnarray}
as well as 
\begin{equation}\lb{3.7H}
\|u_N\|_{H^{s+1/2}(\Omega)} \leq C\|g\|_{H^{s-1}(\dOm)}, 
\end{equation}
for some constant constant $C= C(\Omega,s,z)>0$. Finally,    
\begin{equation}\lb{3.8H}
\big[\ga_D (-\Delta_{N,\Om}-\ol{z}I_\Om)^{-1}\big]^* \in
\cB\big(H^{s-1}(\dOm), H^{s+1/2}(\Om)\big),   
\end{equation}
and the solution $u_N$ is given by the formula  
\begin{equation}\lb{3.9H}
u_N= \big(\ga_D (-\Delta_{N,\Om}-\ol{z}I_\Om)^{-1}\big)^*g. 
\end{equation}
\end{theorem}

When $s\in\{0,1\}$, the above results have been established in 
\cite{GM08}. The more general case $0\leq s\leq 1$ is, however, 
proved along very similar lines. 

We shall now review well-posedness results for the inhomogeneous 
Dirichlet and Neumann problems (again, see \cite{GM08} for proofs). 
In the following we denote by $\wti I_{\Om}$ the continuous inclusion 
(embedding) map of $H^1(\Omega)$ into $\bigl(H^1(\Omega)\bigr)^*$. 
By a slight abuse of notation, we also denote the continuous inclusion 
map of $H^1_0(\Omega)$ into $\bigl(H^1_0(\Omega)\bigr)^*$ by the same 
symbol $\wti I_{\Om}$. We recall the ultra weak Neumann trace operator 
$\wti\ga_{\cN}$ in \eqref{2.8X}, \eqref{2.9X}. Finally, assuming 
Hypothesis \ref{h2.1}, we denote by 
\begin{equation} \lb{3.JqY}
- \wti \Delta_{N,\Om}\in\cB\big(H^1(\Om),\big(H^1(\Om)\big)^*\big)
\end{equation}
the operator defined uniquely by the requirement that 
\begin{equation} \lb{3.JqZ}
{}_{H^1(\Om)}\langle u,- \wti \Delta_{N,\Om}v\rangle_{(H^1(\Om))^*}
=\int_{\Om}d^nx\,\ol{\nabla u(x)}\cdot\nabla v(x)
= (\nabla u, \nabla v)_{L^2(\Om; d^n x)},\quad u,v\in H^1(\Om).
\end{equation}
We note that $- \Delta_{N,\Om}$ then becomes the part of 
$- \wti \Delta_{N,\Om}$ in $L^2(\Om;d^nx)$. 

\begin{theorem} \lb{t3.XV} 
Assume Hypothesis \ref{h2.1} and suppose that 
$z\in\bbC\backslash\si(-\Delta_{N,\Om})$. Then for every 
$w\in (H^{1}(\Omega))^*$, the following generalized inhomogeneous Neumann 
problem,
\begin{equation} \lb{3.Jq}
\begin{cases}
(-\Delta - z)u = w|_{\Om} \text{ in }\,{\mathcal{D}}'(\Om),
\quad u \in H^{1}(\Om), \\
\wti\ga_{\cN} (u,w)= 0 \text{ on } \,\dOm,
\end{cases}
\end{equation}
has a unique solution $u=u_{N,w}$. Moreover, there exists a constant 
$C= C(\Om,z)>0$ such that
\begin{equation}\lb{3.Jq2}
\|u_{N,w}\|_{H^{1}(\Omega)} \leq C\|w\|_{(H^{1}(\partial\Omega))^*}.  
\end{equation}
In particular, the operator $(-\Delta_{N,\Om}-zI_\Om)^{-1}$,
$z\in\bbC\backslash\si(-\Delta_{N,\Om})$, originally defined as a bounded 
operator on $\LOm$,  
\begin{align}\label{faH}
(-\Delta_{N,\Om}-zI_\Om)^{-1} \in \cB\big(L^2(\Om;d^nx)\big),
\end{align}
can be extended to a map in $\cB\big(\big(H^{1}(\Om)\big)^*,H^1(\Om)\big)$, 
which in fact coincides with 
\begin{equation}\label{fcH}
\big(- \wti \Delta_{N,\Om} -z \wti I_\Om\big)^{-1}
\in\cB\big(\big(H^{1}(\Om)\big)^*,H^1(\Om)\big).
\end{equation} 
\end{theorem}

Continuing to retain Hypothesis \ref{h2.1}, we denote by 
\begin{equation}\lb{3.JqD}
- \wti \Delta_{D,\Om}\in\cB\big(H^{-1}(\Om),H^1_0(\Om)\big)
\end{equation}
the operator defined uniquely by the requirement that 
\begin{equation} \lb{3.JqD2}
{}_{H^1_0(\Om)}\langle u,- \wti \Delta_{D,\Om}v\rangle_{(H^1_0(\Om))^*}
=\int_{\Om}d^nx\,\ol{\nabla u(x)}\cdot\nabla v(x)
= (\nabla u, \nabla v)_{L^2(\Om; d^n x)},\quad u,v\in H^1_0(\Om).
\end{equation}
In this context, $-\Delta_{D,\Om}$ then becomes the part of 
$- \wti \Delta_{D,\Om}$ in $L^2(\Om;d^nx)$. 

\begin{theorem} \lb{t3.XD} 
Assume Hypothesis \ref{h2.1} and suppose that 
$z\in\bbC\backslash\si(-\Delta_{D,\Om})$. Then for every 
$w\in H^{-1}(\Omega)$, the following inhomogeneous Dirichlet problem,
\begin{equation} \lb{3.JqD3} 
(-\Delta - z)u = w\text{ in }\,\Om,
\quad u \in H^1_0(\Om), 
\end{equation}
has a unique solution $u=u_{D,w}$. Moreover, there exists a constant 
$C= C(\Om,z)>0$ such that
\begin{equation}\lb{3.JqD4}
\|u_{D,w}\|_{H^{1}(\Omega)} \leq C\|w\|_{H^{-1}(\partial\Omega)}.  
\end{equation}
In particular, the operator $(-\Delta_{D,\Om}-zI_\Om)^{-1}$,
$z\in\bbC\backslash\si(-\Delta_{D,\Om})$, originally defined as a bounded 
operator on $\LOm$,  
\begin{align}\label{faH-D5}
(-\Delta_{D,\Om}-zI_\Om)^{-1} \in \cB\big(L^2(\Om;d^nx)\big),
\end{align}
can be extended to a map in $\cB\big(H^{-1}(\Om),H^1_0(\Om)\big)$, 
which in fact coincides with 
\begin{equation}\label{fcH-D6}
\big(- \wti \Delta_{D,\Om} -z \wti I_\Om\big)^{-1}
\in\cB\big(H^{-1}(\Om),H^1_0(\Om)\big).
\end{equation} 
\end{theorem}

Assuming Hypothesis \ref{h2.1}, we now introduce the Dirichlet-to-Neumann 
map $M_{D,N,\Om}^{(0)}(z)$ associated with $(-\Delta-z)$ on $\Om$, as follows,
\begin{align}
M_{D,N,\Om}^{(0)}(z) \colon
\begin{cases}
H^1(\dOm) \to \LdOm,  \\
\hspace*{10mm} f \mapsto - \wti\ga_N (u_D),
\end{cases}  \quad z\in\bbC\backslash\si(-\Delta_{D,\Om}), \lb{3.44}
\end{align}
where $u_D$ is the unique solution of
\begin{align}
(-\Delta-z)u = 0 \,\text{ in }\Om, \quad u \in
H^{3/2}(\Om), \quad \ga_D u = f \,\text{ on }\dOm.   \lb{3.45}
\end{align}

Still assuming Hypothesis \ref{h2.1}, we next introduce the 
Neumann-to-Dirichlet map $M_{N,D,\Om}^{(0)}(z)$ associated with 
$(-\Delta-z)$ on $\Om$, as follows,
\begin{align}
M_{N,D,\Om}^{(0)}(z) \colon \begin{cases} \LdOm \to H^1(\dOm),
\\
\hspace*{20.8mm} g \mapsto \ga_D u_N, \end{cases}  \quad
z\in\bbC\backslash\si(-\Delta_{N,\Om}), \lb{3.48}
\end{align}
where $u_N$ is the unique solution of
\begin{align}
(-\Delta-z)u = 0 \,\text{ in }\Om, \quad u \in
H^{3/2}(\Om), \quad\wti\ga_N u  = g 
\,\text{ on }\dOm.  \lb{3.49}
\end{align}

\begin{theorem} \lb{t3.5} 
Assume Hypothesis \ref{h2.1}. Then 
\begin{equation}
M_{D,N,\Om}^{(0)}(z) \in \cB\big(H^1(\partial\Om), \LdOm \big), \quad
z\in\bbC\backslash\si(-\Delta_{D,\Om}),   \lb{3.46}
\end{equation}
and 
\begin{equation}
M_{D,N,\Om}^{(0)}(z) =
\wti\gamma_N\big[\wti\gamma_N(-\Delta_{D,\Om} - \ol{z}I_\Om)^{-1}\big]^*, 
\quad z\in\bbC\backslash\si(-\Delta_{D,\Om}). \lb{3.47}
\end{equation}
Moreover, 
\begin{equation}
M_{N,D,\Om}^{(0)}(z) \in \cB\big(\LdOm, H^1(\partial\Om) \big), \quad 
z\in\bbC\backslash\si(-\Delta_{N,\Om}),    \lb{3.50}
\end{equation}
hence, in particular, 
\begin{equation}
M_{N,D,\Om}^{(0)}(z) \in \cB_\infty\big(\LdOm\big), \quad 
z\in\bbC\backslash\si(-\Delta_{N,\Om}).  \lb{3.51}
\end{equation}
In addition, 
\begin{equation}
M_{N,D,\Om}^{(0)}(z) = \gamma_D\big[\gamma_D
(-\Delta_{N,\Om} - \ol{z}I_\Om)^{-1}\big]^*, \quad
z\in\bbC\backslash\si(-\Delta_{N,\Om}). \lb{3.52}
\end{equation} 

Finally, let 
$z\in\bbC\backslash(\si(-\Delta_{D,\Om})\cup\si(-\Delta_{N,\Om}))$. Then
\begin{equation}
M_{N,D,\Om}^{(0)}(z) = - M_{D,N,\Om}^{(0)}(z)^{-1}.   \lb{3.53}  
\end{equation}
\end{theorem}

For closely related recent work on Weyl--Titchmarsh operators associated with non-smooth 
domains we refer to \cite{GM08}, \cite{GM09}, \cite{GM09a}, and \cite{GMZ07}. For an 
extensive list of references on $z$-dependent Dirichlet-to-Neumann maps predominantly 
associated with smooth domains we also refer, for instance, to \cite{Ag03}, \cite{ABMN05}, 
\cite{AP04}, \cite{BL07}, \cite{BMN02}, 
\cite{BGW09}, \cite{BHMNW09}, \cite{BMNW08}, \cite{BGP08}, 
\cite{DM91}, \cite{DM95}, \cite{GLMZ05}--\cite{GMZ07}, \cite{Gr08a}, \cite{Po08}, 
\cite{Ry07}, \cite{Ry09}, \cite{Ry10}. We will return to this topic in Section \ref{s11} 
in the context of quasi-convex domains.

In the last part of this section we include a brief discussion of 
nonlocal Robin Laplacians in bounded Lipschitz subdomains of $\bbR^n$. 
Concretely, we describe a family of self-adjoint Laplace operators 
$-\Delta_{\Theta,\Om}$ in $L^2(\Om;d^nx)$ indexed by the boundary operator 
$\Theta$. We will refer to $-\Delta_{\Theta,\Om}$ as the (nonlocal) 
Robin Laplacian. To facilitate the presentation, we isolate a technical 
condition in the hypothesis below:

\begin{hypothesis} \lb{h2.2FF}
Assume Hypothesis \ref{h2.1}, suppose that $\delta>0$ is a given number,
and assume that $\Theta\in\cB\big(H^{1/2}(\dOm),H^{-1/2}(\dOm)\big)$
is a self-adjoint operator which can be written as
\begin{equation}\label{Filo-1}
\Theta=\Theta_1+\Theta_2+\Theta_3,
\end{equation}
where the operators $\Theta_j$, $j=1,2,3$, have the following properties: There exists a 
closed sesquilinear form $a_{\Theta_0}$ in $\LdOm$, with domain 
$H^{1/2}(\dOm)\times H^{1/2}(\dOm)$, bounded from below
by $c_{\Theta_0}\in\bbR$ $($hence, $a_{\Theta_0}$
is symmetric$)$ such that if $\Theta_0\ge c_{\Theta}I_{\dOm}$ denotes
the self-adjoint operator in $\LdOm$ uniquely associated with $a_{\Theta_0}$, 
then $\Theta_1$ is the extension of $\Theta_0$ to an operator in 
$\cB\big(H^{1/2}(\dOm),H^{-1/2}(\dOm)\big)$. In addition,
\begin{equation}\label{Filo-2}
\Theta_2\in\cB_{\infty}\big(H^{1/2}(\dOm),H^{-1/2}(\dOm)\big),
\end{equation}
whereas $\Theta_3\in\cB\big(H^{1/2}(\dOm),H^{-1/2}(\dOm)\big)$ satisfies
\begin{equation}\label{Filo-3}
\|\Theta_3\|_{\cB(H^{1/2}(\dOm),H^{-1/2}(\dOm))}<\delta.
\end{equation}
\end{hypothesis}

The following result has been proved in \cite{GM09a}:
 
\begin{theorem}\lb{t2.3FF}
Assume Hypothesis~\ref{h2.2FF}, where the number $\delta>0$ is taken
to be sufficiently small relative to the Lipschitz character of $\Om$.
Then the nonlocal Robin Laplacian, $-\Delta_{\Theta,\Om}$, defined by
\begin{align}\lb{2.20FF}
& -\Delta_{\Theta,\Om}=-\Delta,    \\
& \; \dom(-\Delta_{\Theta,\Om})=
\big\{u\in H^1(\Om) \,\big|\, \Delta u\in L^2(\Om;d^nx),\,
\big(\wti\gamma_N+\Theta \gamma_D\big)u=0\text{ in $H^{-1/2}(\dOm)$}\big\} \no
\end{align}
is self-adjoint and bounded from below in $L^2(\Om;d^nx)$. Moreover,
\begin{equation}\lb{2.21FF}
\dom\big(|-\Delta_{\Theta,\Om}|^{1/2}\big)=H^1(\Om),
\end{equation}
and $-\Delta_{\Theta,\Om}$, has purely discrete spectrum bounded from below,
in particular,
\begin{equation}\lb{2.21aFF}
\sigma_{\rm ess}(-\Delta_{\Theta,\Om})=\emptyset.
\end{equation}
\end{theorem}

\section{Higher-Order Smoothness Spaces on Lipschitz Surfaces}
\lb{s6}

This section presents some new results on higher-order smoothness spaces on 
Lipschitz surfaces.

Assuming Hypothesis \ref{h2.1} we recall the tangential derivative 
operators $\partial/\partial\tau_{j,k}$ in \eqref{Pf-2}. We can then 
define the tangential gradient operator 
\begin{align}\lb{Tan-C1} 
\nabla_{tan}:\begin{cases} H^1(\partial\Omega)\to  
\big(L^2(\partial\Omega;d^{n-1}\omega)\big)^n   \\[1mm]
\hspace*{1cm}
f \mapsto \nabla_{tan}f=\Big(\sum_{k=1}^n\nu_k\frac{\partial f}{\partial\tau_{k,j}}
\Big)_{1\leq j\leq n}.
\end{cases}
\end{align}
The following result has been proved in \cite{MMS05}.

\begin{theorem}\lb{T-MMS}
Assume Hypothesis \ref{h2.1} and denote by $\nu$ the outward unit normal 
to $\partial\Omega$. Then the operator 
\begin{align} 
\gamma_2: \begin{cases} H^2(\Omega)\to \big\{(g_0,g_1)\in H^1(\partial\Omega)
\dotplus    L^2(\partial\Omega;d^{n-1}\omega)\,\big|\, 
\nabla_{tan}g_0  
+g_1\nu\in \big(H^{1/2}(\partial\Omega)\big)^n\big\}    \\
\hspace*{8mm} u \mapsto \gamma_2 u =\big(\gamma_D u,\gamma_N u \big)
\end{cases}   \lb{Tan-C2} 
\end{align}
is well-defined, linear, bounded, onto, and has a linear, bounded 
right-inverse. In \eqref{Tan-C2}, the space $\bigl\{(g_0,g_1)\in H^1(\partial\Omega)
\dotplus    L^2(\partial\Omega;d^{n-1}\omega) \,\big|\, 
\nabla_{tan}g_0+g_1\nu\in \bigl(H^{1/2}(\partial\Omega)\bigr)^n\bigl\}$ 
is considered equipped with the natural norm 
\begin{eqnarray}\lb{NoRw-1}
(g_0,g_1)\mapsto \|g_0\|_{H^1(\partial\Omega)}
+\|g_1\|_{L^2(\partial\Omega;d^{n-1}\omega)}
+\|\nabla_{tan}g_0+g_1\nu\|_{(H^{1/2}(\partial\Omega))^n}.
\end{eqnarray}
Furthermore, the null space of the operator \eqref{Tan-C2}
is given by 
\begin{eqnarray}\lb{Tan-C3} 
\ker (\gamma_2) = 
\big\{u\in H^2(\Omega) \,\big|\, \gamma_D u =\gamma_N u =0\big\}
=H^2_0(\Omega),
\end{eqnarray}
with the latter space denoting the closure of $C^\infty_0(\Omega)$ 
in $H^2(\Omega)$.
\end{theorem}

Assuming Hypothesis \ref{h2.1}, we now introduce 
\begin{equation}   
N^{1/2}(\partial\Omega):=\bigl\{g\in L^2(\partial\Omega;d^{n-1}\omega) \,\big|\, 
g\nu_j\in H^{1/2}(\partial\Omega),\,\,1\leq j\leq n\bigl\},   \lb{Tan-C4} 
\end{equation}
where the $\nu_j$'s are the components of $\nu$. We equip this space with the 
natural norm 
\begin{eqnarray}\lb{Tan-C4B} 
\|g\|_{N^{1/2}(\partial\Omega)}
:=\sum_{j=1}^n\|g\nu_j\|_{H^{1/2}(\partial\Omega)}. 
\end{eqnarray} 

\begin{lemma}\lb{L-refN}
Assuming Hypothesis \ref{h2.1}. Then $N^{1/2}(\partial\Omega)$ is a 
reflexive Banach space which embeds continuously into 
$L^2(\partial\Omega;d^{n-1}\omega)$. Furthermore, 
\begin{equation}\lb{Tan-C5} 
N^{1/2}(\partial\Omega)=H^{1/2}(\partial\Omega), \, 
\mbox{ whenever $\Omega$ is a bounded $C^{1,r}$ domain with $r>1/2$}. 
\end{equation}
\end{lemma}
\begin{proof}
Obviously we have 
\begin{equation}\lb{Ob-1}
g=\sum_{j=1}^n\nu_j(g\nu_j)\,\mbox{ for any function }\,
g\in L^2(\partial\Omega;d^{n-1}\omega) 
\end{equation}
so that, in particular, $\|g\|_{L^2(\partial\Omega;d^{n-1}\omega)}
\leq n\|g\|_{N^{1/2}(\partial\Omega)}$. This proves that the natural inclusion 
$N^{1/2}(\partial\Omega)\hookrightarrow L^2(\partial\Omega;d^{n-1}\omega)$ is
bounded. If $\{g_k\}_{k\in{\mathbb{N}}}$ is a Cauchy sequence in 
$N^{1/2}(\partial\Omega)$ then, for each $j\in\{1,...,n\}$, 
$\{g_k\nu_j\}_{k\in{\mathbb{N}}}$ is a Cauchy sequence in 
$H^{1/2}(\partial\Omega)$ and, from what we have proved so far, 
$\{g_k\}_{k\in\bbN}$ converges in $L^2(\partial\Omega;d^{n-1}\omega)$ 
to some $g\in L^2(\partial\Omega;d^{n-1}\omega)$. It follows
that $\{g_k\nu_j\}_{k\in{\mathbb{N}}}$ converges in 
$L^2(\partial\Omega;d^{n-1}\omega)$ to $g\nu_j$ for each $j\in\{1,...,n\}$.
With this at hand, it is then easy to conclude that $g$ is the 
limit of $\{g_k\}_{k\in{\mathbb{N}}}$ in $N^{1/2}(\partial\Omega)$.
This proves that $N^{1/2}(\dOm)$ is a Banach space. 

Next, by relying on the same simple identity in \eqref{Ob-1}
and Lemma \ref{lA.6}, we also see that \eqref{Tan-C5} holds. If we now consider
\begin{eqnarray}\lb{Ob-2}
\Phi:N^{1/2}(\partial\Omega)\to  \bigl[H^{1/2}(\dOm)\bigr]^n,
\quad \Phi(g):=(g\nu_j)_{1\leq j\leq n},
\end{eqnarray}
it follows that $\Phi$ is an isometric embedding, which allows identifying 
$N^{1/2}(\partial\Omega)$ with a closed subspace of the reflexive space
$\bigl[H^{1/2}(\partial\Omega)\bigr]^n$, implying that 
$N^{1/2}(\partial\Omega)$ is also reflexive. 
\end{proof}

It should be mentioned that, in spite of \eqref{Tan-C5}, the spaces 
$H^{1/2}(\partial\Omega)$ and $N^{1/2}(\partial\Omega)$ can be quite 
different for an arbitrary Lipschitz domain $\Omega$. 
Our interest in the latter space stems from the fact that this 
arises naturally when considering the Neumann trace operator acting on  
\begin{eqnarray}\lb{Tan-C6} 
\big\{u\in H^2(\Omega) \,\big|\, \gamma_D u =0\big\}=H^2(\Omega)\cap H^1_0(\Omega), 
\end{eqnarray}
considered as a closed subspace of $H^2(\Omega)$ (hence, a Banach space
when equipped with the $H^2$-norm). More specifically, we have the following result: 

\begin{lemma}\lb{Lo-Tx}
Assume Hypothesis \ref{h2.1}. Then the Neumann trace operator $\gamma_N$ 
considered in the context 
\begin{eqnarray}\lb{Tan-C7} 
\gamma_N:H^2(\Omega)\cap H^1_0(\Omega)\to N^{1/2}(\partial\Omega)
\end{eqnarray}
is well-defined, linear, bounded, onto, and with a linear, bounded 
right-inverse. In addition, the null space of $\gamma_N$ in \eqref{Tan-C7} is 
precisely $H^2_0(\Omega)$, the closure of $C^\infty_0(\Omega)$ in 
$H^2(\Omega)$.
\end{lemma}
\begin{proof} To prove that \eqref{Tan-C7} is well-defined, we note that if 
$u\in H^2(\Omega)\cap H^1_0(\Omega)$, then Theorem \ref{T-MMS} yields
\begin{align}\lb{Tan-C8} 
& (0,\gamma_N u )=\gamma_2 u   
\in\bigl\{(g_0,g_1)\in H^1(\partial\Omega)
\dotplus    L^2(\partial\Omega;d^{n-1}\omega) \,\big|\, 
\nabla_{tan}g_0+g_1\nu\in \bigl(H^{1/2}(\partial\Omega)\bigr)^n\bigl\}, 
\end{align}
implying that $\gamma_N u $ belongs to $N^{1/2}(\partial\Omega)$ 
and $\|\gamma_N u \|_{N^{1/2}(\partial\Omega)}\leq C\|u\|_{H^2(\Omega)}$
for some $C=C(\Omega)>0$ independent of $u$. This shows that \eqref{Tan-C7} 
is well-defined, linear, and bounded. Next, denote by 
${\mathcal{E}}_2$ a linear, bounded right-inverse for $\gamma_2$
in \eqref{Tan-C2}. Then, if 
\begin{align}\lb{Tan-C9} 
& \iota: N^{1/2}(\partial\Omega)\to 
\bigl\{(g_0,g_1)\in H^1(\partial\Omega)
\dotplus    L^2(\partial\Omega;d^{n-1}\omega) \,\big|\, 
\nabla_{tan}g_0  
+g_1\nu\in \bigl(H^{1/2}(\partial\Omega)\bigr)^n\bigl\}  
\end{align}
is the injection given by $\iota(g):=(0,g)$, for every 
$g\in N^{1/2}(\partial\Omega)$, it follows that the composition 
${\mathcal{E}}_2 \iota: N^{1/2}(\partial\Omega)\to 
H^2(\Omega)\cap H^1_0(\Omega)$ is a linear, bounded right-inverse for the 
operator $\gamma_N$ in \eqref{Tan-C7}. Consequently, this operator is 
onto. Finally, the fact that the null space of $\gamma_N$ in \eqref{Tan-C7} 
is precisely $H^2_0(\Omega)$ follows from its definition and the last part in 
the statement of Theorem \ref{T-MMS}.  
\end{proof}

Our goal is to use the above Neumann trace result in order to extend the 
action of the Dirichlet trace operator \eqref{2.6} to $\dom(- \Delta_{max})$,
the domain of the maximal Laplacian, that is,  
$\big\{u\in L^2(\Omega;d^nx) \,\big|\, \Delta u\in L^2(\Omega;d^nx)\big\}$, 
which we consider equipped with the graph norm 
$u\mapsto \|u\|_{L^2(\Omega;d^nx)}+\|\Delta u\|_{L^2(\Omega;d^nx)}$.
Specifically, with $\bigl(N^{1/2}(\partial\Omega)\bigr)^*$ denoting the 
conjugate dual space of $N^{1/2}(\partial\Omega)$, we have the following result:
 
\begin{theorem}\lb{New-T-tr}
Assume Hypothesis \ref{h2.1}. Then there exists a unique linear, bounded 
operator 
\begin{eqnarray}\lb{Tan-C10} 
\widehat{\gamma}_D: 
\big\{u\in L^2(\Omega;d^nx) \,\big|\, \Delta u\in L^2(\Omega;d^nx)\big\}
\to  \bigl(N^{1/2}(\partial\Omega)\bigr)^*
\end{eqnarray}
which is compatible with the Dirichlet trace, introduced in \eqref{2.6}
and further extended in Lemma~\ref{Gam-L1}, in the sense that 
for each $s\geq 1/2$ one has  
\begin{eqnarray}\lb{Tan-C11}
\widehat{\gamma}_D u =\gamma_D u \, \mbox{ for every $u\in H^s(\Omega)$ 
with $\Delta u\in L^2(\Omega;d^nx)$}.
\end{eqnarray}
Furthermore, this extension of the Dirichlet trace operator 
has dense range and allows for the following generalized integration 
by parts formula
\begin{eqnarray}\lb{Tan-C12} 
{}_{N^{1/2}(\partial\Omega)}\langle\gamma_N w,\widehat{\gamma}_D u 
\rangle_{(N^{1/2}(\partial\Omega))^*}
= (\Delta w,u)_{L^2(\Om;d^nx)}
- (w,\Delta u)_{L^2(\Om;d^nx)},
\end{eqnarray}
valid for every $u\in L^2(\Omega;d^nx)$ with $\Delta u\in L^2(\Omega;d^nx)$ 
and every $w\in H^2(\Omega)\cap H^1_0(\Omega)$. 
\end{theorem}
\begin{proof}
Let $u\in L^2(\Omega;d^nx)$ be such that $\Delta u\in L^2(\Omega;d^nx)$. 
We attempt to define a functional 
$\widehat{\gamma}_D u \in\bigl(N^{1/2}(\partial\Omega)\bigr)^*$ 
as follows: Assume that $g\in N^{1/2}(\partial\Omega)$ 
has been given arbitrarily. By Lemma \ref{Lo-Tx} there exists a  
$w\in H^2(\Omega)\cap H^1_0(\Omega)$ such that $\gamma_N w =g$ and 
$\|w\|_{H^2(\Omega)}\leq C\|g\|_{N^{1/2}(\partial\Omega)}$
for some finite constant $C=C(\Omega)>0$, independent of $g$. 
We then set
\begin{eqnarray}\lb{Tan-C13} 
{}_{N^{1/2}(\partial\Omega)}\langle g,\widehat{\gamma}_D u 
\rangle_{(N^{1/2}(\partial\Omega))^*}
:= (\Delta w,u)_{L^2(\Om;d^nx)}
- (w,\Delta u)_{L^2(\Om;d^nx)}.
\end{eqnarray}
First we need to show that the above definition 
does not depend on the particular choice of $w$, with the properties 
listed above. By linearity, this comes down to proving the following claim: 
If $u$ is as before and $w\in  H^2(\Omega)\cap H^1_0(\Omega)$ is such that 
$\gamma_N w                 =0$ then 
\begin{eqnarray}\lb{Tan-C14} 
(\Delta w,u)_{L^2(\Om;d^nx)}
= (w,\Delta u)_{L^2(\Om;d^nx)}.
\end{eqnarray}
However, since Lemma \ref{Lo-Tx} yields that $w\in H^2_0(\Omega)$, 
formula \eqref{Tan-C14} readily follows by approximating $w$ with 
functions form $C^\infty_0(\Omega)$ in the norm of $H^2(\Omega)$. 
Thus, formula \eqref{Tan-C13} yields a well-defined, linear, and bounded
operator in the context of \eqref{Tan-C13}. By definition, this 
operator will satisfy \eqref{Tan-C12}. 

Next, we will show that \eqref{Tan-C11} is valid for each $s\geq 1/2$. 
Fix $s\in [1/2,1]$ along with some function $u\in H^s(\Omega)$ 
with $\Delta u\in L^2(\Omega;d^nx)$. In particular, 
$\gamma_D u \in H^{s-(1/2)}(\partial\Omega)$ and we select a sequence 
$\{f_j\}_{j\in{\mathbb{N}}}$ such that 
\begin{eqnarray}\lb{Tan-C15} 
f_j\in H^{1/2}(\partial\Omega), \; 
j\in{\mathbb{N}}, \, \text{ and }\, f_j\to \gamma_D u 
\,\mbox{ in $H^{s-(1/2)}(\partial\Omega)$ as }j\to\infty.  
\end{eqnarray}
Next, for each $j\in{\mathbb{N}}$, select a function $u_j$ such that 
\begin{eqnarray}\lb{Tan-C16}
\Delta u_j=\Delta u\,\mbox{ in }\Omega,\quad u_j\in H^{1}(\Omega), 
\; \gamma_D u_j = f_j\,\mbox{ on }\partial\Omega.   
\end{eqnarray}
From \eqref{Tan-C15} and the continuous dependence of the solution on the 
data, we may conclude that 
\begin{eqnarray}\lb{Tan-C17}
u_j\to u\, \mbox{ in $H^{s}(\partial\Omega)$ as $j\to\infty$}.
\end{eqnarray}
In addition, given an arbitrary $w\in H^2(\Omega)\cap H^1_0(\Omega)$, 
we also select a sequence $\{w_k\}_{k\in{\mathbb{N}}}$ such that 
\begin{eqnarray}\lb{Tan-C18} 
w_k\in C^\infty(\ol{\Omega}), \; k\in{\mathbb{N}}, \, \text{ and }\, 
w_k\to w\, \mbox{ in $H^2(\Omega)$ as $k\to\infty$}.
\end{eqnarray}
With $j$ fixed, we now introduce the vector fields 
$G_k:=\ol{u_j}\nabla w_k\in \bigl(H^{1}(\Omega)\bigr)^n$, $k\in{\mathbb{N}}$, 
and obtain  
\begin{eqnarray}\lb{Tan-C19} 
\begin{array}{l}
{\rm div} (G_k)=\ol{u_j}\Delta w_k+\ol{\nabla u_j}\cdot\nabla w_k
\in L^2(\Omega;d^n)\hookrightarrow L^1(\Omega;d^nx),
\\[6pt]
\Delta G_k=\ol{\Delta u_j}\nabla w_k+\ol{u_j}\nabla\Delta w_k
+2\sum_{i=1}^n\ol{\nabla u_j}\cdot\nabla\partial_i w_k\in L^2(\Omega;d^nx),
\\[6pt]
\nu\cdot\wti\gamma_D(G_k)=\nu\cdot\gamma_D(G_k)
=\ol{\gamma_D u_j}\gamma_N w_k=\ol{f_j}\gamma_N w_k .
\end{array}
\end{eqnarray}
Based on this and \eqref{Ht-r3} we may then write 
\begin{align}\lb{Tan-C20}
(u_j,\Delta w)_{L^2(\Om;d^nx)} & =  
\lim_{k\to\infty} (u_j,\Delta w_k)_{L^2(\Om;d^nx)}  \no \\
& =\lim_{k\to\infty}\bigg(\int_{\Om}d^nx\,{\rm div} (G_k)
- (\nabla u_j,\nabla w_k)_{(L^2(\Om;d^nx))^n}\bigg)
\nonumber\\
&= \lim_{k\to\infty}\bigg(\int_{\partial\Omega}d\omega^{n-1}\,
\ol{f_j}{\gamma}_N w_k 
- (\nabla u_j,\nabla w_k)_{(L^2(\Om;d^nx))^n}\bigg)
\nonumber\\
&= \int_{\partial\Omega}d\omega^{n-1}\,\ol{f_j}{\gamma}_N w                  
- (\nabla u_j,\nabla w)_{(L^2(\Om;d^nx))^n}. 
\end{align}
In order to continue, we now select a sequence 
$\{v_k\}_{k\in{\mathbb{N}}}$ such that 
\begin{eqnarray}\lb{Tan-C21} 
v_k\in C^\infty_0(\Omega), \; k\in{\mathbb{N}}, \, \text{ and }\, 
v_k\to w\, \mbox{ in $H^1(\Omega)$ as $k\to\infty$}.
\end{eqnarray}
Then we can further transform the last integral in \eqref{Tan-C20} into 
\begin{align}\lb{Tan-C22} 
(\nabla u_j,\nabla w)_{(L^2(\Om;d^nx))^n}
&= \lim_{k\to\infty}(\nabla u_j,\nabla v_k)_{(L^2(\Om;d^nx))^n}  
=-\lim_{k\to\infty}(\Delta u_j,v_k)_{L^2(\Om;d^nx)}
\nonumber\\
&= - (\Delta u_j,w)_{L^2(\Om;d^nx)}
= (\Delta u,w)_{L^2(\Om;d^nx)},  
\end{align}
where the last equality utilizes the fact that $\Delta u_j=\Delta u$ 
(cf.\eqref{Tan-C16}). Together, \eqref{Tan-C20} and \eqref{Tan-C22} 
yield that 
\begin{eqnarray}\lb{Tan-C23} 
(u_j,\Delta w)_{L^2(\Om;d^nx)} 
=\int_{\partial\Omega}d\omega^{n-1}\,\ol{f_j}{\gamma}_N w                  
+ (\Delta u,w)_{L^2(\Om;d^nx)}, 
\end{eqnarray}
for every $j\in{\mathbb{N}}$. By passing to the limit as $j\to\infty$  
in \eqref{Tan-C23}, we arrive, on account of \eqref{Tan-C15} and 
\eqref{Tan-C17}, at 
\begin{eqnarray}\lb{Tan-C24} 
(u,\Delta w)_{L^2(\Om;d^nx)}
=\int_{\partial\Omega}d\omega^{n-1}\,\ol{\gamma_D u }{\gamma}_N w                  
+ (\Delta u,w)_{L^2(\Om;d^nx)}. 
\end{eqnarray}
Consequently, 
\begin{align}\lb{Tan-C25} 
\int_{\partial\Omega}d\omega^{n-1}\,\ol{\gamma_D u }{\gamma}_N w                  
&= (u,\Delta w)_{L^2(\Om;d^nx)} 
- (\Delta u,w)_{L^2(\Om;d^nx)}
\nonumber\\
&= {}_{N^{1/2}(\partial\Omega)}\langle\gamma_N w                 ,\widehat{\gamma}_D u 
\rangle_{(N^{1/2}(\partial\Omega))^*}. 
\end{align}
We recall that this formula 
is valid for any $w\in H^2(\Omega)\cap H^1_0(\Omega)$ and 
the operator \eqref{Tan-C7} maps this space onto $N^{1/2}(\partial\Omega)$. 
We may therefore deduce from \eqref{Tan-C25} that 
$\widehat{\gamma}_D u $, originally viewed as a functional on 
$\bigl(N^{1/2}(\partial\Omega)\bigr)^*$, is given by integration 
against ${\gamma}_D u \in L^2(\partial\Omega;d^{n-1}\omega)$ in the 
case where $u\in H^s(\Omega)$ satisfies $\Delta u\in L^2(\Omega;d^nx)$. 
This concludes the proof of \eqref{Tan-C11}. 

Next, we wish to establish the uniqueness of an operator satisfying 
\eqref{Tan-C10}, \eqref{Tan-C11}. For this it suffices to show
that $C^\infty(\ol{\Omega})$ embeds densely into 
$\big\{u\in L^2(\Omega;d^nx) \,\big|\, \Delta u\in L^2(\Omega;d^nx)\big\}$. 
In fact, a more general result holds, namely, 
\begin{eqnarray}\lb{Tan-C26}
C^\infty(\ol{\Omega})\hookrightarrow 
\big\{u\in H^s(\Omega) \,\big|\, \Delta u\in L^2(\Omega;d^nx)\big\}\,\mbox{ densely, 
whenever $s<2$},
\end{eqnarray}
where the latter space is equipped with the natural graph norm 
$u\mapsto \|u\|_{H^s(\Omega)}+\|\Delta u\|_{L^2(\Omega;d^nx)}$. 
When $s=1$, this appears as Lemma 1.5.3.9 on p.\ 60 of \cite{Gr85}, 
and the extension to $s<2$ has been worked out, along similar lines, in 
\cite{CD98}.  

Finally, we shall show that the Dirichlet trace operator 
\eqref{Tan-C10} has dense range. To this end, granted \eqref{Tan-C26}, it 
suffices to prove that 
\begin{eqnarray}\lb{Nh-1B}
\bigl\{u|_{\partial\Omega} \,\big|\, u\in C^\infty(\ol{\Om})\bigr\}
\, \mbox{ is a dense subspace of }\, 
\bigl(N^{1/2}(\partial\Omega)\bigr)^*.
\end{eqnarray}
In turn, \eqref{Nh-1B} will follow as soon as we show that
\begin{eqnarray}\lb{NSaZ.1}
\mbox{if $\Phi\in\big(\bigl(N^{1/2}(\partial\Omega)\bigr)^*\big)^*$ 
vanishes on $\gamma_D[C^\infty(\overline{\Omega})]$, then necessarily 
$\Phi=0$}.
\end{eqnarray}
With this goal in mind, we fix a functional $\Phi$ as in the first 
part of \eqref{NSaZ.1} and note that since $N^{1/2}(\partial\Omega)$ 
is a reflexive Banach space, continuously embedded 
into $L^2(\partial\Omega;d^{n-1}\omega)$ (cf.\ Lemma \ref{L-refN}), 
we may conclude that $\Phi\in N^{1/2}(\partial\Omega)\hookrightarrow 
L^2(\partial\Omega;d^{n-1}\omega)$. Together with Lemma \ref{Lo-Tx} this 
shows that there exists $w\in H^2(\Omega)\cap H^1_0(\Omega)$ with the 
property that $\gamma_N(w)=\Phi$. In particular, 
\begin{eqnarray}\lb{NSaZ.2}
{}_{N^{1/2}(\partial\Omega)}\langle\gamma_N w,
\gamma_D u \rangle_{(N^{1/2}(\partial\Omega))^*}=0 \, \text{ for all } \,
u\in C^\infty(\overline{\Omega}). 
\end{eqnarray}
Having established \eqref{NSaZ.2}, the integration by parts 
formula \eqref{Tan-C12} in Theorem \ref{New-T-tr} yields that 
\begin{eqnarray}\lb{NSaZ.3} 
(\Delta w,u)_{L^2(\Om;d^nx)}=(w,\Delta u)_{L^2(\Om;d^nx)}, \quad 
u\in C^\infty(\overline{\Omega}). 
\end{eqnarray}
On the other hand, since $w\in H^2(\Omega)$, for every 
$u\in C^\infty(\overline{\Omega})$ we may write 
\begin{eqnarray}\lb{NSaZ.4} 
(\Delta w,u)_{L^2(\Om;d^nx)}
=\int_{\partial\Omega}d\omega^{n-1}\,\ol{{\gamma}_N w}\gamma_D u   
-\int_{\partial\Omega}d\omega^{n-1}\,\ol{\gamma_D w}{\gamma}_N u               
+(w,\Delta u)_{L^2(\Om;d^nx)}. 
\end{eqnarray}
Upon recalling that we also have $w\in H^1_0(\Omega)$ and 
$\Phi=\gamma_N w$, it follows from \eqref{NSaZ.3}--\eqref{NSaZ.4} that 
\begin{eqnarray}\lb{NSaZ.5} 
\int_{\partial\Omega}d\omega^{n-1}\,\ol{\Phi}\gamma_D u=0, \quad 
u\in C^\infty(\overline{\Omega}). 
\end{eqnarray}
Thus, if 
\begin{eqnarray}\lb{NSaZ.5T}
\mbox{$\gamma_D[C^\infty(\overline{\Omega})]$ is dense in 
$L^2(\partial\Omega;d^{n-1}\omega)$}, 
\end{eqnarray}
we see from \eqref{NSaZ.5} that $\Phi=0$ in 
$L^2(\partial\Omega;d^{n-1}\omega)$. Hence, \eqref{NSaZ.1} 
follows, completing the proof of the theorem, modulo the justification of 
\eqref{NSaZ.5T}. Finally, as far as the claim \eqref{NSaZ.5T} is concerned, we 
start by recalling that on any metric measure space (such as $\partial\Omega$, 
equipped with the surface measure and the Euclidean distance) Lipschitz
functions are dense in $L^2$. Hence, if suffices to further approximate
(in the uniform norm) a given Lipschitz function $f$ on $\partial\Omega$ with 
restrictions of functions from $C^\infty({\mathbb{R}}^n)$ to $\partial\Omega$.
This, however, can be achieved by first extending $f$ to a Lipschitz
function in the entire Euclidean space (which can be done even with 
preservation of the Lipschitz constant), and then mollifying this extension.
\end{proof}

Our next result underscores the point that the space $N^{1/2}(\partial\Omega)$
is quite rich.

\begin{corollary}\lb{L-Den}
Assume Hypothesis \ref{h2.1}. Then 
\begin{eqnarray}\lb{Nyyy}
N^{1/2}(\partial\Omega)\hookrightarrow
L^2(\partial\Omega;d^{n-1}\omega)\hookrightarrow
\bigl(N^{1/2}(\partial\Omega)\bigr)^*\, \mbox{ continuously and 
densely in each case}. 
\end{eqnarray}
Moreover, the duality paring between 
$N^{1/2}(\partial\Omega)$ and $\bigl(N^{1/2}(\partial\Omega)\bigr)^*$ 
is compatible with the natural integral paring 
in $L^2(\partial\Omega;d^{n-1}\omega)$.
\end{corollary}
\begin{proof}
We recall that given two reflexive Banach spaces $X,Y$, with $Y\subseteq X$, 
such that the inclusion map $\iota:Y\hookrightarrow X$ is continuous with 
dense range, it follows that its adjoint, that is, $\iota^\ast:X^\ast\to Y^\ast$, 
is also one-to-one, continuous, and with dense range. This can be interpreted 
as saying that $X^\ast$ embeds continuously and densely into $Y^\ast$. 
Furthermore, in the case when $X$ is actually a Hilbert space (so that 
$X^*$ is canonically identified with $X$) then the duality paring between 
$Y$ and $Y^\ast$ is compatible with the inner product in $X$. 
In light of this general result and Lemma \ref{L-refN}, it suffices to 
show that the second inclusion in \eqref{Nyyy} is well-defined, continuous,  
and with dense range.
To this end, if $f\in L^2(\partial\Omega;d^{n-1}\omega)$ is arbitrary and
$u\in H^{1/2}(\Omega)$ is such that $\Delta u=0$ in $\Omega$ and $\gamma_Du=f$
(cf.\ Theorem \ref{t3.3}) then, by virtue of \eqref{Tan-C11} and \eqref{Tan-C10}, 
$f=\widehat{\gamma}_Du\in \bigl(N^{1/2}(\partial\Omega)\bigr)^*$ and 
$\|f\|_{\bigl(N^{1/2}(\partial\Omega)\bigr)^*}
\leq C\|f\|_{L^2(\partial\Omega;d^{n-1}\omega)}$. This proves that the
second inclusion in \eqref{Nyyy} is well-defined and continuous, and we are left 
with proving that it also has dense range. The latter property is 
an immediate consequence of \eqref{Tan-C26}, \eqref{Tan-C11}
and the fact that the map \eqref{Tan-C10} has dense range.
\end{proof}

For smoother domains, Theorem \ref{New-T-tr} takes a somewhat more familiar 
form (compare with \cite{LM72} and \cite{Gr68}):

\begin{corollary}\lb{Neq-T2}
Assume that $\Omega\subset{\mathbb{R}}^n$, $n\geq 2$, is a bounded 
$C^{1,r}$ domain with $r>1/2$. Then the Dirichlet trace $\gamma_D$ in \eqref{2.6} 
extends in a unique fashion to a linear, bounded operator 
\begin{eqnarray}\lb{Tan-C10X} 
\widehat{\gamma}_D: 
\big\{u\in L^2(\Omega;d^nx) \,\big|\, \Delta u\in L^2(\Omega;d^nx)\big\}
\to  H^{-1/2}(\partial\Omega). 
\end{eqnarray}
Moreover, for this extension of the Dirichlet trace operator one has
the following generalized integration by parts formula
\begin{eqnarray}\lb{Tan-C12X} 
\langle\gamma_N w                 ,\widehat{\gamma}_D u \rangle_{1/2}
= (\Delta w,u)_{L^2(\Om;d^nx)}
- (w,\Delta u)_{L^2(\Om;d^nx)},
\end{eqnarray}
whenever $u\in L^2(\Omega;d^nx)$ satisfies $\Delta u\in L^2(\Omega;d^nx)$,  
and $w\in H^2(\Omega)\cap H^1_0(\Omega)$. 
\end{corollary}
\begin{proof}
This is an immediate consequence of Theorem \ref{New-T-tr} and equation \eqref{Tan-C5}. 
\end{proof}

We next turn our attention to case of the Neumann trace, whose action 
we would like to extend to $\dom (- \Delta_{max})$. To this end, we need 
to address a number of preliminary matters. First, assuming 
Hypothesis \ref{h2.1}, we make the following definition 
(compare with \eqref{Tan-C4}):  
\begin{equation}   
N^{3/2}(\partial\Omega):=\bigl\{g\in H^1(\partial\Omega) \,\big|\, 
\nabla_{tan}g\in \bigl(H^{1/2}(\partial\Omega)\bigr)^n\bigl\},  \lb{3an-C4}
\end{equation}
equipped with the natural norm 
\begin{eqnarray}\lb{3an-C4B} 
\|g\|_{N^{3/2}(\partial\Omega)}
:=\|g\|_{L^2(\partial\Omega;d^{n-1}\omega)}+
\|\nabla_{tan}g\|_{(H^{1/2}(\partial\Omega))^n}.  
\end{eqnarray} 

\begin{lemma}\lb{L-refNN}
Assume Hypothesis \ref{h2.1}. Then $N^{3/2}(\partial\Omega)$ is a 
reflexive Banach space which embeds continuously into $H^1(\partial\Omega)$. 
\end{lemma}
\begin{proof}
That the natural injection $N^{3/2}(\partial\Omega)\hookrightarrow 
H^1(\partial\Omega)$ is bounded 
is clear from \eqref{3an-C4B}, \eqref{Tan-C1} and \eqref{Pf-2.4}. 
Next, let $\{g_m\}_{m\in{\mathbb{N}}}$ be a Cauchy sequence in 
$N^{3/2}(\partial\Omega)$. Then $\{g_m\}_{m\in\bbN}$ converges in 
$H^1(\dOm)$ to some $g\in H^1(\dOm)$. Consequently, for each 
$j,k\in\{1,...,n\}$, the sequence 
$\{\partial g_m/\partial\tau_{j,k}\}_{m\in\bbN}$ converges to 
$\partial g/\partial\tau_{j,k}$ in $L^2(\dOm;d^{n-1}\omega)$. 
This and \eqref{Tan-C1} then imply that 
$\{\nabla_{tan}g_m\}_{m\in\bbN}$ converges to 
$\nabla g$ in $\bigl[L^2(\dOm;d^{n-1}\omega)\bigr]^n$. 
Since $\{\nabla_{tan}g_m\}_{m\in\bbN}$ is also known to be Cauchy in 
$H^1(\dOm)$, we may finally conclude that $g\in N^{3/2}(\dOm)$ and 
$\{g_m\}_{m\in\bbN}$ converges to $g$ in $N^{3/2}(\dOm)$. 
This proves that $N^{3/2}(\dOm)$ is a Banach space. Next, one observes that 
\begin{eqnarray}\lb{Ob-2X}
\Phi:N^{3/2}(\partial\Omega)\to  
L^2(\dOm;d^{n-1}\omega)\dotplus   \bigl[H^{1/2}(\dOm)\bigr]^n,
\quad \Phi(g):=(g,\nabla_{tan}g), 
\end{eqnarray}
is an isometric embedding, allowing for the identification of 
$N^{3/2}(\partial\Omega)$ with a closed subspace of the reflexive space
$L^2(\dOm;d^{n-1}\omega)\dotplus   \bigl[H^{1/2}(\dOm)\bigr]^n$. 
Consequently, $N^{3/2}(\partial\Omega)$ is also reflexive. 
\end{proof}

Our next result shows that this is a natural substitute for the 
more familiar space $H^{3/2}(\partial\Omega)$ in the case where $\Omega$ 
is sufficiently smooth. Concretely, we have the following result:
 
\begin{lemma}\lb{Lk-t1}
Let $\Omega\subset{\mathbb{R}}^n$, $n\geq 2$, be a bounded $C^{1,r}$ domain 
with $r>1/2$. Then  
\begin{eqnarray}\lb{3an-C5} 
N^{3/2}(\partial\Omega)=H^{3/2}(\partial\Omega), 
\end{eqnarray}
as vector spaces with equivalent norms.
\end{lemma}
\begin{proof}
Denote by $(\nu_1,...,\nu_n)$ the components of the outward unit normal 
$\nu$ to $\partial\Omega$. If $g\in H^{1}(\partial\Omega)$ then \eqref{Tan-C1} 
and elementary algebra yield that 
\begin{eqnarray}\label{D-pv4}
\frac{\partial g}{\partial\tau_{j,k}}
=\nu_j(\nabla_{tan}g)_k-\nu_k(\nabla_{tan}g)_j,\quad j,k=1,...,n,
\end{eqnarray}
where $(\nabla_{tan}g)_r$ stands for the $r$-th component of the vector 
field $\nabla_{tan}g$. As a consequence of this identity and Lemma \ref{lA.6}, 
if $g\in N^{3/2}(\partial\Omega)$, then 
$\partial g/\partial\tau_{j,k}\in H^{1/2}(\partial\Omega)$. 
Thus, Lemma \ref{K-t1} yields that
$g\in H^{3/2}(\partial\Omega)$ with 
$\|g\|_{H^{3/2}(\partial\Omega)}\approx\|g\|_{N^{3/2}(\partial\Omega)}$,
whenever $g\in N^{3/2}(\partial\Omega)$. This proves the left-to-right 
inclusion in \eqref{3an-C5}. 

Conversely, if $g\in H^{3/2}(\partial\Omega)$, then 
$g\in H^{1}(\partial\Omega)$ and 
$\partial g/\partial\tau_{j,k}\in H^{1/2}(\partial\Omega)$ for 
$1\leq j,k\leq n$, by Lemma \ref{K-t1} (and a natural equivalence of norms). 
Based on this and \eqref{Tan-C1}, we may then conclude that 
$g\in N^{3/2}(\partial\Omega)$ with 
$\|g\|_{N^{3/2}(\partial\Omega)}\approx\|g\|_{H^{3/2}(\partial\Omega)}$, 
whenever $g\in H^{3/2}(\partial\Omega)$. This proves the right-to-left 
inclusion in \eqref{3an-C5}. 
\end{proof}


The reason we are interested in $N^{3/2}(\partial\Omega)$ is that this space
arises naturally when considering the Dirichlet trace operator acting on 
\begin{eqnarray}\lb{3an-C6} 
\big\{u\in H^2(\Omega) \,\big|\, \gamma_N u = 0\big\}, 
\end{eqnarray}
considered as a closed subspace of $H^2(\Omega)$ (thus, a Banach space
when equipped with the norm inherited from $H^2(\Omega)$). 
More precisely, we have the following result:
 
\begin{lemma}\lb{3o-Tx}
Assume Hypothesis \ref{h2.1}. Then the Dirichlet trace operator $\gamma_D$ 
considered in the context 
\begin{eqnarray}\lb{3an-C7} 
\gamma_D: \big\{u\in H^2(\Omega) \,\big|\, \gamma_N u = 0\big\}
\to  N^{3/2}(\partial\Omega)
\end{eqnarray}
is well-defined, linear, bounded, onto, and with a linear, bounded 
right-inverse. In addition, the null space of $\gamma_D$ in \eqref{3an-C7} is 
precisely $H^2_0(\Omega)$, the closure of $C^\infty_0(\Omega)$ in 
$H^2(\Omega)$.
\end{lemma}
\begin{proof} If $u\in H^2(\Omega)$ is such that $\gamma_N u =0$, then 
Theorem \ref{T-MMS} yields
\begin{align}\lb{3an-C8} 
& (\gamma_D u ,0)=\gamma_2 u   
\in\big\{(g_0,g_1) \in H^1(\partial\Omega)
\dotplus    L^2(\partial\Omega;d^{n-1}\omega) \,\big|\, 
\nabla_{tan}g_0+g_1\nu\in \bigl(H^{1/2}(\partial\Omega)\bigr)^n\big\}
\end{align}
which entails $\gamma_D u \in N^{3/2}(\partial\Omega)$. 
Furthermore, there exists $C=C(\Omega)>0$, independent of $u$, such that 
$\|\gamma_D u \|_{N^{3/2}(\partial\Omega)}\leq C\|u\|_{H^2(\Omega)}$. 
As a consequence, the operator \eqref{Tan-C7} is well-defined, linear, and 
bounded. Next, we recall that ${\mathcal{E}}_2$ stands for a linear, bounded  
right-inverse for $\gamma_2$ in \eqref{Tan-C2}. Then, if 
\begin{align}\lb{3an-C9} 
& \iota': N^{3/2}(\partial\Omega)\to 
\big\{(g_0,g_1)\in H^1(\partial\Omega)
\dotplus    L^2(\partial\Omega;d^{n-1}\omega) \,\big|\, 
\nabla_{tan}g_0  
+g_1\nu\in \big(H^{1/2}(\partial\Omega)\big)^n\big\}  
\end{align}
is the injection given by $\iota'(g):=(g,0)$, for every 
$g\in N^{3/2}(\partial\Omega)$, it follows that the composition 
${\mathcal{E}}_2 \iota': N^{3/2}(\partial\Omega)\to 
\big\{u\in H^2(\Omega) \,\big|\, \gamma_N u = 0\big\}$ is a linear, bounded right-inverse 
for the operator $\gamma_D$ in \eqref{3an-C7}. Hence, this operator is also
onto. Finally, the fact that the null space of $\gamma_D$ in \eqref{3an-C7} 
is precisely $H^2_0(\Omega)$ follows from its definition and the last part in 
the statement of Theorem \ref{T-MMS}.  
\end{proof}

Next, we shall use the Neumann trace result in Lemma \ref{3o-Tx}   
to extend the action of the Neumann trace operator \eqref{2.7} to 
$\dom (- \Delta_{max}) 
= \big\{u\in L^2(\Omega;d^nx) \,\big|\, \Delta u\in L^2(\Omega;d^nx)\big\}$.
As before, this space is equipped with the natural graph norm. 
We denote by $\bigl(N^{3/2}(\partial\Omega)\bigr)^*$ the 
conjugate dual space of $N^{3/2}(\partial\Omega)$. 

\begin{theorem}\lb{3ew-T-tr}
Assume Hypothesis \ref{h2.1}. Then there exists a unique linear and bounded 
operator 
\begin{eqnarray}\lb{3an-C10} 
\widehat{\gamma}_N: 
\big\{u\in L^2(\Omega;d^nx) \,\big|\, \Delta u\in L^2(\Omega;d^nx)\big\}
\to  \bigl(N^{3/2}(\partial\Omega)\bigr)^*
\end{eqnarray}
which is compatible with the Neumann trace, originally introduced in \eqref{2.7}
and then further extended in \eqref{MaX-1}, in the sense that, for each 
$s\geq 3/2$, one has  
\begin{eqnarray}\lb{3an-C11}
\widehat{\gamma}_N u=\wti\gamma_N u \, \mbox{ for every $u\in H^s(\Omega)$ 
with $\Delta u\in L^2(\Omega;d^nx)$}.
\end{eqnarray}
Furthermore, this extension of the Neumann trace operator has dense range 
and allows for the following generalized integration by parts formula
\begin{eqnarray}\lb{3an-C12} 
{}_{N^{3/2}(\partial\Omega)}\langle\gamma_D w,
\widehat\gamma_N u\rangle_{(N^{3/2}(\partial\Omega))^*}
=(w,\Delta u)_{L^2(\Om;d^nx)}-(\Delta w,u)_{L^2(\Om;d^nx)},
\end{eqnarray}
valid for every $u\in L^2(\Omega;d^nx)$ with $\Delta u\in L^2(\Omega;d^nx)$ 
and every $w\in H^2(\Omega)$ with $\gamma_N w                 =0$.
\end{theorem}
\begin{proof}
Consider $u\in L^2(\Omega;d^nx)$ arbitrary  
such that $\Delta u\in L^2(\Omega;d^nx)$. We shall then define a functional 
$\widehat{\gamma}_N u  \in\bigl(N^{3/2}(\partial\Omega)\bigr)^*$ in the 
following fashion: Given an arbitrary function $g\in N^{3/2}(\partial\Omega)$, 
we invoke Lemma \ref{3o-Tx} in order to find 
$w\in H^2(\Omega)$ such that $\gamma_N w =0$ and $\gamma_D w =g$. 
In addition, matters can be arranged so that  
$\|w\|_{H^2(\Omega)}\leq C\|g\|_{N^{3/2}(\partial\Omega)}$
for some finite constant $C=C(\Omega)>0$, independent of $g$. 
We then define 
\begin{eqnarray}\lb{3an-C13} 
{}_{N^{3/2}(\partial\Omega)}\langle g,\widehat{\gamma}_N u  
\rangle_{(N^{3/2}(\partial\Omega))^*}
:= (w,\Delta u)_{L^2(\Om;d^nx)}
- (\Delta w,u)_{L^2(\Om;d^nx)}.
\end{eqnarray}
We claim that the above definition is unambiguous in the sense that
the action of the functional $\widehat{\gamma}_N u  $ does not depend on 
the particular choice of $w$, with the properties listed above. 
By linearity, this comes down to proving the following claim: 
If $u$ is as before and $w\in  H^2(\Omega)$ is such that 
$\gamma_D w =\gamma_N w =0$, then 
\begin{eqnarray}\lb{3an-C14} 
(\Delta u,w)_{L^2(\Om;d^nx)}
= (u,\Delta w)_{L^2(\Om;d^nx)}.
\end{eqnarray}
However, since Theorem \ref{T-MMS} ensures that $w\in H^2_0(\Omega)$, 
formula \eqref{3an-C14} follows similarly to \eqref{Tan-C14}, by approximating 
$w$ with functions form $C^\infty_0(\Omega)$ in the norm of $H^2(\Omega)$. 

The above reasoning shows that formula \eqref{3an-C13} yields a well-defined, 
linear, and bounded operator in the context of \eqref{3an-C13}. 
By definition, this operator will satisfy \eqref{3an-C12}. 

Next, we will show that \eqref{3an-C11} is valid for each $s\geq 3/2$. 
Fix $s\geq 3/2$ and let $u\in H^s(\Omega)$ 
with $\Delta u\in L^2(\Omega;d^nx)$. Then for every  $w\in H^2(\Omega)$ with 
$\gamma_N w =0$, Green's formula \eqref{wGreen} yields
\begin{align}\lb{Bgr-T1}
(\Delta w,u)_{L^2(\Om;d^nx)} &= 
{}_{H^1(\Omega)}\langle w,\Delta u\rangle_{(H^1(\Omega))^*}
\nonumber\\ 
&= -(\nabla w,\nabla u)_{(L^2(\Omega;d^nx))^n}
+\ol{\langle\gamma_D u ,\wti\gamma_N w \rangle_{1/2}}
\nonumber\\ 
&= - (\nabla w,\nabla u)_{(L^2(\Om;d^nx))^n}.
\end{align}
On the other hand, we may once again employ \eqref{wGreen} in order to 
rewrite the last integral above as  
\begin{align}\lb{Bgr-T2}
(\nabla w,\nabla u)_{(L^2(\Om;d^nx))^n}
&= \langle\gamma_D w, \wti{\gamma}_N u  \rangle_{1/2}
-{}_{H^1(\Omega)}\langle w,\Delta u\rangle_{(H^1(\Omega))^*}
\nonumber\\ 
&= \int_{\partial\Omega}d^{n-1}\omega\,\ol{\gamma_D w}\,\wti{\gamma}_N u  
- (w,\Delta u)_{L^2(\Om;d^nx)}, 
\end{align}
given the smoothness properties of the functions involved. 
Altogether, \eqref{Bgr-T1} and \eqref{Bgr-T2} imply that 
\begin{eqnarray}\lb{Bgr-T3}
\int_{\partial\Omega}d^{n-1}\omega\,\ol{\gamma_D w}\,\wti{\gamma}_N u  
= (w,\Delta u)_{L^2(\Om;d^nx)}
- (\Delta w,u)_{L^2(\Om;d^nx)}. 
\end{eqnarray}
This formula is valid for any $w\in H^2(\Omega)$ with $\gamma_N w =0$, and 
the operator \eqref{3an-C7} maps this space onto $N^{3/2}(\partial\Omega)$. 
We may therefore deduce from \eqref{Bgr-T3} that 
$\widehat{\gamma}_N u$, originally viewed as a functional on 
$\bigl(N^{3/2}(\partial\Omega)\bigr)^*$, is given by integration 
against $\wti{\gamma}_N u\in L^2(\partial\Omega;d^{n-1}\omega)$ in the 
case where $u\in H^s(\Omega)$, $s\geq 3/2$, satisfies 
$\Delta u\in L^2(\Omega;d^nx)$. This concludes the proof of \eqref{3an-C11}. 
Moreover, the uniqueness of an operator satisfying 
\eqref{3an-C10} and \eqref{3an-C11} is guaranteed by what we have 
proved thus far and the density result in \eqref{Tan-C26}. 

It remains to prove that the Neumann trace operator \eqref{3an-C10} 
has dense range. Due to \eqref{Tan-C26}, this amounts to showing that 
\begin{eqnarray}\lb{Nh-1B.Cv}
\gamma_N[C^\infty(\ol{\Om})]
=\big\{\nu\cdot(\nabla u)|_{\partial\Omega} \,\big|\, 
u\in C^\infty(\ol{\Om})\big\}\, \mbox{ is a dense subspace of }\, 
\big(N^{3/2}(\partial\Omega)\big)^*.
\end{eqnarray}
As in the case of \eqref{Nh-1B}, this will follow as soon 
as we establish that 
\begin{eqnarray}\lb{NSaZ.1.Cv}
\mbox{if $\Phi\in\big(\bigl(N^{3/2}(\partial\Omega)\bigr)^*\big)^*$ 
vanishes on $\gamma_N[C^\infty(\overline{\Omega})]$, then necessarily 
$\Phi=0$}.
\end{eqnarray}
To justify this, we fix a functional $\Phi$ as in the first 
part of \eqref{NSaZ.1.Cv} and note that since $N^{3/2}(\partial\Omega)$ 
is a reflexive Banach space, continuously embedded 
into $H^1(\partial\Omega)$ (cf.\ Lemma \ref{L-refNN}), 
one concludes that $\Phi\in N^{3/2}(\partial\Omega)\hookrightarrow 
H^1(\partial\Omega)$. Together with Lemma \ref{3o-Tx} this implies that
there exists $w\in H^2(\Omega)$ for which $\gamma_N w=0$ and  
$\gamma_D(w)=\Phi$. As a consequence,
\begin{eqnarray}\lb{NSaZ.2.Cv}
{}_{N^{3/2}(\partial\Omega)}\langle\gamma_D w,
\gamma_N u \rangle_{(N^{3/2}(\partial\Omega))^*}=0 \, \text{ for all } \,
u\in C^\infty(\overline{\Omega}). 
\end{eqnarray}
Based on \eqref{NSaZ.2.Cv} and the integration by parts 
formula \eqref{3an-C12} in Theorem \ref{3ew-T-tr} one then concludes that
\begin{eqnarray}\lb{NSaZ.3.Cv} 
(\Delta w,u)_{L^2(\Om;d^nx)}=(w,\Delta u)_{L^2(\Om;d^nx)}, \quad
u\in C^\infty(\overline{\Omega}). 
\end{eqnarray}
Since $w\in H^2(\Omega)$, for every 
$u\in C^\infty(\overline{\Omega})$ we may write 
\begin{eqnarray}\lb{NSaZ.4.Cv} 
(\Delta w,u)_{L^2(\Om;d^nx)}
=\int_{\partial\Omega}d\omega^{n-1}\,\ol{{\gamma}_N w}\gamma_D u   
-\int_{\partial\Omega}d\omega^{n-1}\,\ol{\gamma_D w}{\gamma}_N u               
+(w,\Delta u)_{L^2(\Om;d^nx)}. 
\end{eqnarray}
Keeping in mind that $\gamma_N w=0$ and $\Phi=\gamma_D w$,
one then deduces from \eqref{NSaZ.3.Cv}--\eqref{NSaZ.4.Cv} that 
\begin{eqnarray}\lb{NSaZ.5.Cv} 
\int_{\partial\Omega}d\omega^{n-1}\,\ol{\Phi}\gamma_N u=0, \quad
u\in C^\infty(\overline{\Omega}). 
\end{eqnarray}
At this stage it remains to observe that 
\begin{eqnarray}\lb{GFF.12}
\mbox{$\gamma_N[C^\infty(\overline{\Omega})]$ is dense in 
$L^2(\partial\Omega;d^{n-1}\omega)$}
\end{eqnarray}
which, on account of \eqref{NSaZ.5.Cv}, implies that $\Phi=0$ in 
$L^2(\partial\Omega;d^{n-1}\omega)$, hence proving \eqref{NSaZ.1}.
Regarding the claim in \eqref{GFF.12}, we note that as a particular 
case of \eqref{Tan-C26} one has 
\begin{eqnarray}\lb{Tan-CDA}
C^\infty(\ol{\Omega})\hookrightarrow 
\big\{u\in H^{3/2}(\Omega) \,\big|\, \Delta u=0\mbox{ in }\Omega\big\} \,
\mbox{ densely},
\end{eqnarray}
where the latter space inherits the norm from $H^{3/2}(\Omega)$.
Upon recalling (cf.\ Lemma \ref{Neu-tr}) that $\wti\gamma_N$ maps 
$\big\{u\in H^{3/2}(\Omega) \,\big|\, \Delta u=0\mbox{ in }\Omega\big\}$
onto $L^2(\partial\Omega;d^{n-1}\omega)$ and that this version of the 
Neumann trace operator is compatible with $\gamma_N$ 
(again, see Lemma \ref{Neu-tr}), \eqref{GFF.12} follows. 
\end{proof}

For smoother domains, Theorem \ref{3ew-T-tr} takes a somewhat more 
familiar form (compare with Lions-Magenes \cite{LM72} and Grubb \cite{Gr68}, as before). We recall that 
$H^{-3/2}(\partial\Omega)=\bigl(H^{3/2}(\partial\Omega)\bigr)^*$ in the smooth setting. 

\begin{corollary}\lb{3eq-T2}
Assume that $\Omega\subset{\mathbb{R}}^n$, $n\geq 2$, is a bounded 
$C^{1,r}$ domain with $r>1/2$. Then the Neumann trace in \eqref{2.7} 
extends in a unique fashion to a linear, bounded operator 
\begin{eqnarray}\lb{3an-C10X} 
\widehat{\gamma}_N: 
\big\{u\in L^2(\Omega;d^nx) \,\big|\, \Delta u\in L^2(\Omega;d^nx)\big\}
\to  H^{-3/2}(\partial\Omega). 
\end{eqnarray}
Moreover, for this extension of the Neumann trace operator one has
the following generalized integration by parts formula
\begin{eqnarray}\lb{3an-C12X} 
{}_{H^{3/2}(\partial\Omega)}\langle{\gamma}_D w                    ,\widehat\gamma_N u 
\rangle_{(H^{3/2}(\partial\Omega))^*}
= (w,\Delta u)_{L^2(\Om;d^nx)}
- (\Delta w,u)_{L^2(\Om;d^nx)},
\end{eqnarray}
whenever $u\in L^2(\Omega;d^nx)$ satisfies $\Delta u\in L^2(\Omega;d^nx)$,  
and $w\in H^2(\Omega)$ has $\gamma_N w                 =0$. 
\end{corollary}
\begin{proof}
This is a direct consequence of Theorem \ref{3ew-T-tr} and \eqref{3an-C5}. 
\end{proof}

\begin{corollary}\lb{L-Den2}
Assume Hypothesis \ref{h2.1}. Then 
\begin{eqnarray}\lb{Nh-1B.q}
N^{3/2}(\partial\Omega)\hookrightarrow
L^2(\partial\Omega;d^{n-1}\omega)\hookrightarrow
\bigl(N^{3/2}(\partial\Omega)\bigr)^*\, \mbox{ continuously and 
densely in each case}. 
\end{eqnarray}
Furthermore, the duality paring between 
$N^{3/2}(\partial\Omega)$ and $\bigl(N^{3/2}(\partial\Omega)\bigr)^*$ 
is compatible with the natural integral paring 
in $L^2(\partial\Omega;d^{n-1}\omega)$.
\end{corollary}
\begin{proof}
By proceeding as in the first part of the proof of Corollary \ref{L-Den},
and by relying on Lemma \ref{L-refNN}, it suffices to only 
show that the second inclusion in \eqref{Nh-1B.q} is well-defined, continuous,  
and with dense range. To verify this, fix $f\in L^2(\partial\Omega;d^{n-1}\omega)$  arbitrary and let $u\in H^{3/2}(\Omega)$ be such that $(-\Delta-1)u=0$ in $\Omega$ 
and $\wti\gamma_N u=f$ (cf.\ Theorem \ref{t3.2H}). Then, due to 
\eqref{3an-C11} and \eqref{3an-C10}, one has  
$f=\widehat{\gamma}_N u\in \bigl(N^{3/2}(\partial\Omega)\bigr)^*$ and the estimate
$\|f\|_{\bigl(N^{3/2}(\partial\Omega)\bigr)^*}
\leq C\|f\|_{L^2(\partial\Omega;d^{n-1}\omega)}$ holds. Thus, the
second inclusion in \eqref{Nh-1B.q} is well-defined and continuous. 
It remains to show that this inclusion also has dense range. 
This property is, however, an immediate consequence of \eqref{GFF.12}, 
\eqref{3an-C11} and the fact that the map \eqref{3an-C10} has dense range.
\end{proof}

\begin{lemma}\lb{3U-x}
Assume Hypothesis \ref{h2.1}. Then 
\begin{eqnarray}\lb{3U-y} 
N^{3/2}(\partial\Omega)\subseteq\big(N^{1/2}(\partial\Omega)\big)^*
\, \mbox{ and } \, 
N^{1/2}(\partial\Omega)\subseteq\big(N^{3/2}(\partial\Omega)\big)^* \, 
\mbox{ densely}.
\end{eqnarray}
Moreover, in each case, the inclusion is given by canonical 
continuous injections. 
\end{lemma}
\begin{proof} 
This is a direct consequence of Corollaries~\ref{L-Den} and \ref{L-Den2}. 
\end{proof}

The following comment addresses an issue raised by G.\ Grubb:

\begin{remark} \lb{r6.14}
Under Hypothesis \ref{h2.1}, the spaces $N^{1/2}(\partial\Omega)$ and 
$N^{3/2}(\partial\Omega)$ have Hilbert space structures induced by the 
following inner products:
\begin{align}
\begin{split}
(f,g)_{N^{1/2}(\partial\Omega)} & := \sum_{j=1}^n \int_{\partial\Om} \int_{\partial\Om} 
d^{n-1} \omega (\xi) \, d^{n-1} \omega (\eta) \, 
\f{\big(\ol{(\nu_j f)(\xi)} - \ol{(\nu_j f)(\eta)}\big)
\big((\nu_j g)(\xi) - (\nu_j g)(\eta)\big)}{|\xi - \eta|^n}  \\
& \quad \, + (f,g)_{L^2(\partial\Omega; d^{n-1} \omega)}, \quad 
f, g \in N^{1/2}(\partial\Omega), 
\end{split}
\end{align}
and 
\begin{align}
\begin{split}
(f,g)_{N^{3/2}(\partial\Omega)} & := \int_{\partial\Om} \int_{\partial\Om} 
d^{n-1} \omega (\xi) \, d^{n-1} \omega (\eta) \, 
\f{\big(\ol{(\nabla_{tan} f)(\xi)} - \ol{(\nabla_{tan} f)(\eta)}\big)
\big((\nabla_{tan} g)(\xi) - (\nabla_{tan} g)(\eta)\big)}{|\xi - \eta|^n}  \\[1mm]
& \quad \, + (f,g)_{L^2(\partial\Omega; d^{n-1} \omega)},  \quad 
f, g \in N^{3/2}(\partial\Omega), 
\end{split}
\end{align}
where $\nabla_{tan}$ has been introduced in \eqref{Tan-C1}. This is a consequence 
of the definitions of $N^{1/2}(\partial\Omega)$ and $N^{3/2}(\partial\Omega)$ (cf.\ 
\eqref{Tan-C4} and \eqref{3an-C4}) and the fact that under Hypothesis \ref{h2.1},  
\begin{equation}
\|h\|^2_{H^{1/2}(\partial\Om)} \approx \int_{\partial\Om}  \int_{\partial\Om} 
d^{n-1} \omega (\xi) \, d^{n-1} \omega (\eta) \, 
\f{|h(\xi) - h(\eta)|^2}{|\xi - \eta|^n}  
+ \|h\|^2_{L^2(\partial\Om; d^{n-1} \omega)}, 
\quad h \in H^{1/2} (\partial\Om).  
\end{equation}

We wish to stress, however, that in the current paper these Hilbert space structures 
play no role. Instead, we exclusively rely on the duality structure expressed in 
Corollaries \ref{L-Den} and \ref{L-Den2}. 
\end{remark}

\section{The Minimal and Maximal Laplacians on Lipschitz Domains}
\lb{s7}

Minimal and maximal $L^2(\Omega;d^nx)$-realizations of the Laplacian on Lipschitz domains are considered in this short section. 

Given an open set $\Omega\subset{\mathbb{R}}^n$, consider the maximal Laplacian 
$-\Delta_{max}$ in $L^2(\Omega;d^nx)$ defined by 
\begin{align}
\begin{split} 
& - \Delta_{max}u := - \Delta u,     \lb{Yan-1}  \\
& \; u \in \dom (- \Delta_{max}):=\big\{v\in L^2(\Omega;d^nx) \,\big|\, 
\Delta v\in L^2(\Omega;d^nx)\big\}.
\end{split}
\end{align} 

\begin{lemma}\lb{Max-M1}
Assume Hypothesis \ref{h2.1}. Then the maximal Laplacian associated with 
$\Omega$ is a closed, densely defined operator for which 
\begin{align}\lb{Yan-2} 
H^2_0(\Omega)& \subseteq \dom ((-\Delta_{max})^*)   \no \\
& \subseteq\big\{u\in L^2(\Omega;d^nx) \,\big|\, 
\Delta u\in L^2(\Omega;d^nx),\, 
\widehat{\gamma}_D u =\widehat{\gamma}_N u  =0\big\}.
\end{align}
\end{lemma}
\begin{proof} 
If $u_j\in L^2(\Omega;d^nx)$, $j\in{\mathbb{N}}$, are functions such that
$\Delta u_j\in L^2(\Omega;d^nx)$ for each $j$ and 
\begin{eqnarray}\lb{Yan-2B}
u_j\to u,\quad 
\Delta u_j\to v\, \mbox{ in $L^2(\Omega;d^nx)$ as $j\to\infty$}, 
\end{eqnarray}
then for every $w\in C^\infty_0(\Omega)$ one has 
\begin{align}\lb{Yan-2C}
(u,\Delta w)_{L^2(\Om;d^nx)}
& =\lim_{j\to\infty}(u_j,\Delta w)_{L^2(\Om;d^nx)}
=\lim_{j\to\infty}(\Delta u_j,w)_{L^2(\Om;d^nx)} \no \\
& =\lim_{j\to\infty}(v,w)_{L^2(\Om;d^nx)}.
\end{align}
This shows that $\Delta u=v\in L^2(\Omega;d^nx)$ in the sense of 
distributions. Hence, $u\in \dom (- \Delta_{max})$ and $\Delta_{max}u=v$, 
proving that the operator \eqref{Yan-1} is closed. By \eqref{Tan-C26}, this 
operator is also densely defined. 

Consider next \eqref{Yan-2}.
If $w\in H^2_0(\Omega)$ then for every $u\in \dom (- \Delta_{max})$ 
Theorem \ref{New-T-tr} yields
\begin{eqnarray}\lb{Yan-3B}
(\Delta w,u)_{L^2(\Om;d^nx)}
- (w,\Delta u)_{L^2(\Om;d^nx)}
={}_{N^{1/2}(\partial\Omega)}\langle\gamma_N w,\widehat{\gamma}_D u 
\rangle_{(N^{1/2}(\partial\Omega))^*}=0. 
\end{eqnarray}
This shows that $w\in \dom ((-\Delta_{max})^*)$ and 
$(-\Delta_{max})^*(w) = - \Delta w$, justifying the first inclusion in \eqref{Yan-2}. 

Next, if $u\in \dom ((- \Delta_{max})^*)$ then $u\in L^2(\Omega;d^nx)$ and, 
in addition, there exists $v\in L^2(\Omega;d^nx)$ such that 
\begin{eqnarray}\lb{Yan-3}
(u,\Delta w)_{L^2(\Om;d^nx)}
= (v,w)_{L^2(\Om;d^nx)}
\end{eqnarray}
for every $w\in L^2(\Omega;d^nx)$ with $\Delta w\in L^2(\Omega;d^nx)$. 
We shall now specialize this to three distinguished classes of functions $w$. 
First, taking $w\in C^\infty_0(\Omega)$ arbitrarily, it follows
that $\Delta u=v\in L^2(\Omega;d^nx)$ in the sense of distributions. 
With this at hand, and taking $w\in H^2(\Omega)\cap H^1_0(\Omega)$ 
arbitrarily, it follows from \eqref{Tan-C12} that 
\begin{eqnarray}\lb{Yan-4} 
{}_{N^{1/2}(\partial\Omega)}\langle\gamma_N w,\widehat{\gamma}_D u 
\rangle_{(N^{1/2}(\partial\Omega))^*}
= (\Delta w,u)_{L^2(\Om;d^nx)}
- (w,\Delta u)_{L^2(\Om;d^nx)}=0. 
\end{eqnarray}
Hence, $\widehat\gamma_D u =0$ by Lemma \ref{Lo-Tx}. 
Finally, consider arbitrary functions $w\in H^2(\Omega)$ satisfying 
$\widehat\gamma_N w =0$. In this scenario, it follows from \eqref{3an-C12} that
\begin{eqnarray}\lb{Yan-5} 
{}_{N^{3/2}(\partial\Omega)}\langle\gamma_D w,\widehat\gamma_N u 
\rangle_{(N^{3/2}(\partial\Omega))^*}
= (w,\Delta u)_{L^2(\Om;d^nx)}
- (\Delta w,u)_{L^2(\Om;d^nx)}=0. 
\end{eqnarray}
Thus, by Lemma \ref{3o-Tx}, we also have $\widehat\gamma_N u =0$.
This concludes the justification of the second inclusion in \eqref{Yan-2}. 
\end{proof}

For an open set $\Omega\subset{\mathbb{R}}^n$, we also bring in 
the minimal Laplacian in $L^2(\Omega;d^nx)$, that is, 
\begin{equation}\lb{Yan-6} 
-\Delta_{min}u:= - \Delta u,    \quad u \in \dom (- \Delta_{min}):=H^2_0(\Omega). 
\end{equation}

\begin{corollary}\lb{Max-M2}
Assume Hypothesis \ref{h2.1}. Then $-\Delta_{min}$ is a densely defined, 
symmetric operator which satisfies
\begin{equation}\lb{Yan-7} 
- \Delta_{min} \subseteq (- \Delta_{max})^*\, \text{ and } \, 
- \Delta_{max} \subseteq (-\Delta_{min})^*. 
\end{equation}
Equality holds in one $($and hence in both\,$)$ inclusions in \eqref{Yan-7} if 
\begin{eqnarray}\lb{Yan-8}
H^2_0(\Omega) \, \text{ equals } \, \big\{u\in L^2(\Omega;d^nx) \,\big|\, \Delta u\in L^2(\Omega;d^nx),\, 
\widehat{\gamma}_D u = \widehat{\gamma}_N u = 0\bigr\}.
\end{eqnarray}
\end{corollary}
\begin{proof} 
The fact that $-\Delta_{min}$ is a densely defined, symmetric operator is 
obvious. As far as equality in \eqref{Yan-7} is concerned, it suffices to prove only the 
first inclusion (since the second one is a consequence of this and duality). 
This, however, is implied by Lemma \ref{Max-M1}. This lemma also shows
that equality holds if one has equality in \eqref{Yan-8}.
\end{proof}

\begin{remark} \lb{r7.3}
It was kindly pointed out to us by Jussi Behrndt and Till Micheler \cite{BM12} 
that a distributional argument and an application of Poincar\'e's inquality imply equality in \eqref{Yan-7} for any bounded domain (not necessarily Lipschitz), that is, one actually has 
\begin{equation}\lb{Yan-7a} 
- \Delta_{min} = (- \Delta_{max})^*\, \text{ and } \, 
- \Delta_{max} = (-\Delta_{min})^*, 
\end{equation}
whenever $\Omega \subset \bbR^n$, $n \geq 2$, is open and bounded. 
\end{remark}

\section{The Class of Quasi-Convex Domains}
\lb{s8}

This section is devoted to one of our principal new results, the introduction and study 
of  the class of quasi-convex domains. 

In the class of Lipschitz domains, the two spaces appearing in \eqref{Yan-8}
could be quite different. Obviously, the left-to-right inclusion always holds,
\begin{eqnarray}\lb{Yan-8A}
H^2_0(\Omega) \subseteq \big\{u\in L^2(\Omega;d^nx) \,\big|\, \Delta u\in L^2(\Omega;d^nx),\, 
\widehat{\gamma}_D u = \widehat{\gamma}_N u = 0\bigr\}.   
\end{eqnarray}
The question now arises: What extra qualities of the Lipschitz domain 
will guarantee the equality in \eqref{Yan-8A}? 
To address this issue, we need some preparations: Given $n\geq 1$, denote by 
$MH^{1/2}(\bbR^n)$ the class of pointwise multipliers of the Sobolev 
space $H^{1/2}(\bbR^n)$. That is, 
\begin{eqnarray}\label{MaS-1}
MH^{1/2}(\bbR^n):=\bigl\{f\in L^1_{\loc}(\bbR^n) \,\big|\, 
M_f\in\cB\bigl(H^{1/2}(\bbR^n)\bigr)\bigr\},
\end{eqnarray}
where $M_f$ is the operator of pointwise multiplication by $f$. This 
space is equipped with the natural norm, that is, 
\begin{eqnarray}\label{MaS-2}
\|f\|_{MH^{1/2}(\bbR^n)}:=\|M_f\|_{\cB(H^{1/2}(\bbR^n))}.
\end{eqnarray}

\begin{definition}\label{Def-MS}
Given $\delta>0$, a bounded Lipschitz domain $\Omega\subset\bbR^n$ 
is called to be of class $MH^{1/2}_\delta$, and one writes  
\begin{eqnarray}\label{MaS-3}
\dOm\in MH^{1/2}_\delta,
\end{eqnarray}
provided the following holds: There exists a finite open covering 
$\{{\mathcal O}_j\}_{1\leq j\leq N}$ of the boundary $\partial\Omega$ of 
$\Om$ such that for every $j\in\{1,...,N\}$, ${\mathcal O}_j\cap\Omega$ 
coincides with the portion of ${\mathcal O}_j$ lying in the over-graph of 
a Lipschitz function $\varphi_j:\bbR^{n-1}\to\bbR$ $($considered in a new 
system of coordinates obtained from the original one via a rigid motion\,$)$ 
which has the property that 
\begin{eqnarray}\label{MaS-4}
\varphi_j\in MH^{1/2}(\bbR^{n-1})\, \mbox{ and }\quad 
\|\varphi_j\|_{MH^{1/2}(\bbR^{n-1})}\leq\delta.
\end{eqnarray}
\end{definition}

Continuing, we consider the following classes of domains
\begin{eqnarray}\label{MaS-5}
MH^{1/2}_\infty:=\bigcup_{\delta>0}MH^{1/2}_\delta,\quad
MH^{1/2}_0:=\bigcap_{\delta>0}MH^{1/2}_\delta.
\end{eqnarray}
Finally, we also introduce the following definition: 

\begin{definition}\label{Def-MS2}
A bounded Lipschitz domain $\Omega\subset\bbR^n$ is called 
{\it square-Dini}, and one writes  
\begin{eqnarray}\label{MaS-6}
\dOm\in {\rm SD},
\end{eqnarray}
provided the following holds: There exists a finite open covering 
$\{{\mathcal O}_j\}_{1\leq j\leq N}$ of the boundary $\partial\Omega$ of 
$\Om$ such that for every $j\in\{1,...,N\}$, ${\mathcal O}_j\cap\Omega$ 
coincides with the portion of ${\mathcal O}_j$ lying in the over-graph of 
a Lipschitz function $\varphi_j:\bbR^{n-1}\to\bbR$ $($considered in a new 
system of coordinates obtained from the original one via a rigid motion\,$)$ 
which, additionally, has the property that 
\begin{eqnarray}\label{MaS-7}
\int_0^1\bigg(\frac{\omega(\nabla\varphi_j;t)}{t^{1/2}}\bigg)^2\,\frac{dt}{t}
<\infty. 
\end{eqnarray}
Here, given a $($possibly vector-valued\,$)$ function $f$ in $\bbR^{n-1}$, 
\begin{eqnarray}\label{MaS-8}
\omega(f;t):=\sup\,\{|f(x)-f(y)| \,|\, x,y\in\bbR^{n-1}, \, |x-y|\leq t\},
\quad t\in(0,1),
\end{eqnarray}
is the modulus of continuity of $f$, at scale $t$. 
\end{definition}

From the work of V. Maz'ya and T. Shaposhnikova \cite{MS85}, \cite{MS05},
it is known that if $r>1/2$ then 
\begin{eqnarray}\label{MaS-9}
C^{1,r}\subseteq{\rm SD}\subseteq MH^{1/2}_0\subseteq MH^{1/2}_\infty.
\end{eqnarray}
As pointed out in \cite{MS05}, domains of class $MH^{1/2}_\infty$ can have
certain types of vertices and edges when $n\geq 3$. 

Next, we recall that a domain is said to satisfy a uniform exterior 
ball condition provided there exists a number $r>0$ with the property that 
\begin{align} \label{UEBC}
& \mbox{for every $x\in\dOm$ there exists $y\in\bbR^n$ such that} \\
& \quad \mbox{$B(y,r)\cap\Om=\emptyset$ and $x\in\partial B(y,r)\cap\dOm$}.
\end{align}
Next, we review the class of almost-convex domains
introduced in \cite{MTV}. 

\begin{definition}\label{Def-AC}
A bounded Lipschitz domain $\Omega\subset{\mathbb{R}}^n$ is called 
an almost-convex domain provided there exists a family 
$\{\Omega_\ell\}_{\ell\in{\mathbb{N}}}$
of open sets in ${\mathbb{R}}^n$ with the following properties: 
\begin{enumerate}
\item[$(i)$] $\partial\Omega_\ell\in C^2$ and 
$\overline{\Omega_{\ell}}\subset\Omega$ for every $\ell\in{\mathbb{N}}$. 
\item[$(ii)$] $\Omega_\ell\nearrow\Omega$ as $\ell\to\infty$, in the sense
that $\overline{\Omega_{\ell}}\subset\Omega_{\ell+1}$ for each 
$\ell\in{\mathbb{N}}$ and $\bigcup_{\ell\in{\mathbb{N}}}\Omega_{\ell}=\Omega$. 
\item[$(iii)$] There exists a neighborhood $U$ of $\partial\Omega$ and, 
for each $\ell\in{\mathbb{N}}$, a $C^2$ real-valued function $\rho_{\ell}$ 
defined in $U$ with the property that $\rho_{\ell}<0$ on $U\cap\Omega_{\ell}$, 
$\rho_{\ell}>0$ in $U\backslash \overline{\Omega_{\ell}}$, and which 
vanishes on $\partial\Omega_\ell$. In addition, it is assumed that 
there exists some constant $C_1\in (1,\infty)$ such that 
\begin{eqnarray}\label{MTV3.1}
C_1^{-1}\leq |\nabla\rho_\ell(x)|\leq C_1,
\quad\forall\,x\in \partial\Omega_\ell,\quad\forall\,\ell\in{\mathbb{N}}. 
\end{eqnarray}
\item[$(iv)$] There exists $C_2\geq 0$ such that for every number
$\ell\in{\mathbb{N}}$, every point $x\in\partial\Omega_{\ell}$, 
and every vector $\xi\in{\mathbb{R}}^n$ which is tangent to 
$\partial\Omega_{\ell}$ at $x$, there holds 
\begin{eqnarray}\label{MTV3.2}
\big\langle{\rm Hess}\,(\rho_\ell)\xi,  \xi\big\rangle\geq -C_2|\xi|^2, 
\end{eqnarray}
\noindent where $\langle\dott,\dott \rangle$ is the standard inner 
product in ${\mathbb{R}}^n$ and  
\begin{eqnarray}\label{MTV3.3}
{\rm Hess}\,(\rho_\ell):=\left(\frac{\partial^2\rho_\ell}
{\partial x_i\partial x_j}\right)_{1\leq i,j\leq n},
\end{eqnarray}
\noindent is the Hessian of $\rho_{\ell}$. 
\end{enumerate}
\end{definition}

\noindent A few remarks are in order here: First, it is not difficult to see
that \eqref{MTV3.1} ensures that each domain $\Omega_\ell$ is Lipschitz, 
with Lipschitz constant bounded uniformly in $\ell$. Second, \eqref{MTV3.2}
simply says that, as quadratic forms on the tangent bundle 
$T\partial\Omega_\ell$ to $\partial\Omega_{\ell}$, we have 
\begin{eqnarray}\label{MT-SR}
{\rm Hess}\,(\rho_\ell)\geq -C_2\,I_n,
\end{eqnarray}
\noindent where $I_n$ is the identity matrix in $\bbR^n$. Hence, another equivalent 
formulation of \eqref{MTV3.2} is the following requirement: 
\begin{eqnarray}\label{MTV3.4}
\sum\limits_{i,j=1}^n\frac{\partial^2\rho_\ell}{\partial x_i\partial x_j}
\xi_i\xi_j\geq -C_2 \sum\limits_{i=1}^n\xi_i^2,\, \mbox{ whenever }\,
\rho_\ell=0\,\mbox{ and }\,\sum\limits_{i=1}^n
\frac{\partial\rho_\ell}{\partial x_i}\xi_i=0.
\end{eqnarray}
\noindent We note that since the second fundamental form $II_{\ell}$ on 
$\partial\Omega_{\ell}$ is given by 
$II_\ell={{\rm Hess}\,\rho_\ell}/{|\nabla\rho_\ell|}$,
almost-convexity is, in view of \eqref{MTV3.1}, equivalent to 
requiring that $II_\ell$ be bounded below, uniformly in $\ell$.

We now discuss some important special classes of almost-convex
domains. 

\begin{definition}\label{eu-RF}
A bounded Lipschitz domain $\Omega\subset{\mathbb{R}}^n$ satisfies 
a local exterior ball condition, henceforth referred to as LEBC,
if every boundary point $x_0\in\partial\Omega$ has an open 
neighborhood ${\mathcal{O}}$ which satisfies the following two conditions:
\begin{enumerate}
\item[$(i)$] There exists a Lipschitz function
$\varphi:{{\mathbb{R}}}^{n-1}\to{{\mathbb{R}}}$ with
$\varphi(0)=0$ and such that if $D$ is the domain above the graph of $\varphi$
then $D$ satisfies a UEBC. 
\item[$(ii)$] There exists a $C^{1,1}$ diffeomorphism $\Upsilon$ mapping 
${\mathcal{O}}$ onto the unit ball $B(0,1)$ in ${{\mathbb{R}}}^n$ and such 
that $\Upsilon(x_0)=0$, $\Upsilon({\mathcal{O}}\cap\Omega)=B(0,1)\cap D$,
$\Upsilon({\mathcal{O}}\backslash \ol{\Omega})=B(0,1)\backslash \overline{D}$.
\end{enumerate}
\end{definition}

\noindent It is clear from Definition \ref{eu-RF} that the class of 
bounded domains satisfying a LEBC is invariant under $C^{1,1}$ diffeomorphisms.
This makes this class of domains amenable to working on manifolds. 
This is the point of view adopted in \cite{MTV}, where the following 
result is also proved:  

\begin{lemma}\label{MTVp3.1}
If the bounded Lipschitz domain $\Omega\subset{\mathbb{R}}^n$ satisfies 
a LEBC then it is almost-convex.
\end{lemma}

\noindent Hence, in the class of bounded Lipschitz domains 
in ${\mathbb{R}}^n$, we have
\begin{eqnarray}\label{ewT-1}
\mbox{convex}  \, \Longrightarrow \,
\mbox{UEBC}  \, \Longrightarrow \,
\mbox{LEBC}  \, \Longrightarrow \,
\mbox{almost-convex}.
\end{eqnarray}

For a vector field $w=(w_1,w_2,\dots,w_n)$ whose components are 
distributions in an open set $\Omega\subset{\mathbb{R}}^n$,  
we define its curl, ${\rm curl} (w)$, to be the tensor field with 
$n^2$ components (entries) given by 
\begin{eqnarray}\label{akry7}
({\rm curl} (w))_{j,k}=\partial_jw_k-\partial_kw_j,\quad
j,k=1,\dots,n.
\end{eqnarray}
\noindent In addition, if $\Psi=(\Psi_{j,k})_{1\leq j,k\leq n}$ is a tensor field 
whose $n^2$ components are distributions in $\Omega$, we set 
\begin{eqnarray}\label{a-YT}
({\rm div} (\Psi))_{k}=\sum\limits_{j=1}^n\partial_j(\Psi_{j,k}-\Psi_{k,j}),
\quad k=1,\dots,n.
\end{eqnarray}
We also use the notation
\begin{equation}
\langle A, B\rangle_{\bbC^{n^2}} = \sum_{j,k =1}^n \ol{A_{k,j}} \, B_{k,j} 
= \tr(A^* B)
\end{equation}
for two tensor fields $A$ and $B$ with $n^2$ components. 
We then introduce the following definition:

\begin{definition}\label{86t8}
Assume that $\Omega\subset{\mathbb{R}}^n$ is a bounded Lipschitz domain with 
outward unit normal $\nu$. Let $w=(w_1,\dots,w_n)$ be a vector field with 
components in $L^2(\Omega;d^nx)$ such that ${\rm curl} (w)$ also has 
components in $L^2(\Omega;d^nx)$. Then $\nu\times w$ is the unique tensor 
field with $n^2$ components in 
$H^{-1/2}(\partial\Omega)=\bigl(H^{1/2}(\partial\Omega)\big)^*$ 
which satisfies the following property: If $\Psi$ is any tensor field with 
$n^2$ components in $H^1(\Omega)$ and $\psi=\gamma_D \Psi$, with 
the Dirichlet trace taken componentwise, then 
\begin{eqnarray}\label{liry9}
{}_{H^{1/2}(\partial\Omega)^{n^2}}\langle 
\psi,\nu\times w \rangle_{H^{-1/2}(\partial\Omega)^{n^2}}
=\int_\Omega d^nx \, \langle \Psi, {\rm curl} (w) \rangle_{\bbC^{n^2}}
+\int_\Omega \,d^nx \, \ol{{\rm div} (\Psi)} \cdot w.
\end{eqnarray}
\end{definition}

\noindent It is not difficult to check (using the fact that 
$\gamma_D:H^1(\Omega)\to H^{1/2}(\partial\Omega)$ is onto, and that 
$C^\infty_0(\Omega)$ is dense in the space $H^1(\Omega)$
with vanishing Dirichlet trace) that \eqref{liry9} uniquely defines 
$\nu\times w$ as a functional in $H^{-1/2}(\partial\Omega)^{n^2}$.

We shall also need the following companion of Definition \ref{86t8}:

\begin{definition}\label{86t8Y}
Assume that $\Omega\subset{\mathbb{R}}^n$ is a bounded Lipschitz domain with 
outward unit normal $\nu$. Let $w=(w_1,\dots,w_n)$ be a vector field with 
components in $L^2(\Omega;d^nx)$ such that ${\rm div} (w) \in L^2(\Omega;d^nx)$.
Then $\nu\cdot w$ is the unique functional in 
$H^{-1/2}(\partial\Omega)=\bigl(H^{1/2}(\partial\Omega)\big)^*$ 
which satisfies the following property. If $\Phi\in H^1(\Omega)$ and 
$\phi=\gamma_D \Phi$ then 
\begin{eqnarray}\label{liry9Y}
{}_{H^{1/2}(\partial\Omega)}\langle \phi, \nu\cdot w \rangle_{H^{-1/2}(\partial\Omega)}
=\int_\Omega d^nx \, {\ol \Phi} \, {\rm div} (w) 
+\int_\Omega \,d^nx \, \ol{\nabla \Phi} \cdot w.
\end{eqnarray}
\end{definition}

\noindent As before, one can check that this uniquely defines $\nu\cdot w$ as a 
functional in $H^{-1/2}(\partial\Omega)$.

Next, we note the following regularity result, which is a consequence 
of Theorem 4.1 on p.\ 1458 in \cite{MTV} (which contains a more general 
result, formulated in the language of differential forms). 

\begin{lemma}\label{jawyr}
Assume that $\Omega\subset{\mathbb{R}}^n$ is an almost-convex domain
with outward unit normal $\nu$. Then,
\begin{align}\label{kyawg}
& \big\{w\in L^2(\Omega;d^nx)^n \,\big|\, {\rm div} (w) \in L^2(\Omega;d^nx),\, 
{\rm curl} (w) \in L^2(\Omega;d^nx)^{n^2},\, \nu\times w=0\}
\nonumber\\
& \quad
=\big\{w\in H^1(\Omega)^n \,\big|\, \nu\times w=0\big\}\quad
\end{align} 
and, in addition, there exists a finite constant $C=C(\Omega)>0$ such that 
\begin{equation}\label{kTv-2}
\|w\|_{H^1(\Omega)^n}\leq 
C\big(\|w\|_{L^2(\Omega;d^nx)^n}+\|{\rm div} (w)\|_{L^2(\Omega;d^n)}
+\|{\rm curl} (w)\|_{L^2(\Omega;d^nx)^{n^2}}\big), 
\end{equation} 
whenever $\nu\times w=0$. Furthermore,  
\begin{align}\label{kyawgX}
& \big\{w\in L^2(\Omega;d^nx)^n \,\big|\, {\rm div} (w) \in L^2(\Omega;d^nx),\, 
{\rm curl} (w) \in L^2(\Omega;d^nx)^{n^2},\, \nu\cdot w=0\big\}
\nonumber\\ 
& \quad
= \big\{w\in H^1(\Omega)^n \,\big|\, \nu\cdot w=0\big\},  
\end{align}
and there exists a finite constant $C=C(\Omega)>0$ such that  
\begin{equation}\label{kTv-2X}
\|w\|_{H^1(\Omega)^n}\leq 
C\big(\|w\|_{L^2(\Omega;d^nx)^n}+\|{\rm div} (w)\|_{L^2(\Omega;d^n)}
+\|{\rm curl} (w)\|_{L^2(\Omega;d^nx)^{n^2}}\big), 
\end{equation}
whenever $\nu\cdot w=0$. 
\end{lemma}

We are now in a position to specify the class of domains in which most
of our subsequent analysis will be carried out. 

\begin{definition}\lb{d.Conv}
Let $n\in\bbN$, $n\geq 2$, and assume that $\Omega\subset{\bbR}^n$ is a 
bounded Lipschitz domain. Then $\Om$ is called a quasi-convex domain if there 
exists $\delta>0$ sufficiently small $($relative to $n$ and the Lipschitz 
character of $\Om$$)$, with the property that for every $x\in\dOm$ there 
exists an open subset $\Om_x$ of $\Omega$ such that $\dOm\cap\dOm_x$ is an 
open neighborhood of $x$ in $\dOm$, and for which one of the following two 
conditions holds: \\
$(i)$ \, $\Omega_x$ is of class $MH^{1/2}_\delta$ if $n\geq 3$, and 
of class $C^{1,r}$ for some $1/2<r<1$ if $n=2$.   \\
$(ii)$ $\Omega_x$ is an almost-convex domain. 
\end{definition}

Given Definition \ref{d.Conv}, we thus introduce the following basic assumption:

\begin{hypothesis}\lb{h.Conv}
Let $n\in\bbN$, $n\geq 2$, and assume that $\Omega\subset{\bbR}^n$ is
a quasi-convex domain. 
\end{hypothesis}

The following lemma will play a basic role in our work:
 
\begin{lemma}\lb{Bjk}
Assume Hypothesis \ref{h.Conv}. Then,
\begin{equation}\lb{Yan-9}
\dom\big(- \Delta_{D,\Om}\big)\subset H^{2}(\Omega), \quad 
\dom\big(- \Delta_{N,\Om}\big)\subset H^{2}(\Omega).   
\end{equation}
\end{lemma}
\begin{proof} 
When $\Omega$ is a domain of class $C^{1,r}$ for some $1/2<r<1$  
this result appeared in \cite{GMZ07}. Assume now that $\Omega$ is an almost-convex 
domain. Then, if $u\in\dom\big(- \Delta_{D,\Om}\big)$ or 
$u\in\dom\big(- \Delta_{N,\Om}\big)$, Lemma \ref{jawyr} applied to 
$w:=\nabla u$ yields that $w\in H^1(\Omega)^n$ so that, ultimately, 
$u\in H^{2}(\Omega)$, as desired. Next, for domains of class $MH^{1/2}_\delta$
with $\delta>0$ sufficiently small (relative to dimension and the Lipschitz
character), \eqref{Yan-9} is a consequence of work done in \cite{MS85} and
\cite{MS05}. The more general types of domains considered
here can be treated by using this and a localization argument:  
More specifically, assume that $u\in H^1_0(\Om)$ is such that 
$f:=\Delta u\in L^2(\Om;d^nx)$. Fix an arbitrary boundary point 
$x_0\in \dOm$ and let $\Om_{x_0}$ be the associated domain as in 
Hypothesis \ref{h.Conv}. Finally, fix $\psi\in C^\infty_0(\bbR^n)$ with 
$\dOm\cap \supp (\psi) \subset \dOm \cap \dOm_{x_0}$. Then 
$v:=(\psi u)|_{\Om_x}\in H^1_0(\Om_x)$ has the property that 
$\Delta v=[(\Delta\psi)u+2\nabla\psi\cdot\nabla u+\psi f]|_{\Om_{x_0}}
\in L^2(\Om_{x_0};d^nx)$. Thus, by the result recalled at the beginning 
of the current proof, $v\in H^2(\Om_{x_0})$. Using a partition of unity 
argument (and interior elliptic regularity), one deduces from this 
that $u\in H^2(\Om)$. This proves the first inclusion in \eqref{Yan-9}.

The second inclusion in \eqref{Yan-9} is a bit more delicate since, 
as opposed to the Dirichlet case, Neumann boundary conditions are not stable
under truncations by smooth functions. We shall, nonetheless, overcome this 
difficulty by proceeding as follows: Let $u\in H^1(\Om)$ be such that 
$f:=\Delta u\in L^2(\Om;d^nx)$ and $\wti\gamma_N u =0$. As before, 
fix $x_0\in\dOm$, denote by $\Om_{x_0}$ the domain associated with $x_0$ 
as in Hypothesis \ref{h.Conv}, and select $\psi\in C^\infty_0(\bbR^n)$ with 
$\dOm\cap \supp (\psi) \subset\dOm\cap\dOm_{x_0}$. 
Our goal is to show that $\psi u\in H^2(\Om_{x_o})$. From this and a 
partition of unity argument we may then once again conclude that 
$u\in H^2(\Om)$, as desired. At this point, the discussion branches out into 
several cases, which we treat separately below. 

\vskip 0.08in
\noindent{${\mbox{Case~I}}$}: {\it Assume that $\Om_{x_0}$ is a 
domain of class $C^{1,r}$, for some $r>1/2$}. In this scenario, we consider 
the function $v:=(\psi u)|_{\Om_{x_0}}\in H^1(\Om_{x_o})$ and note that, as 
before, $\Delta v\in L^2(\Om_{x_0};d^nx)$. We wish to show that 
$v\in H^2(\Om_{x_o})$. To this end, one observes that the support of 
$\wti\gamma_N v        $ is a relatively compact subset of 
$\dOm_{x_0}\cap\dOm$ and that 
\begin{align}\lb{VaK-1}
\wti\gamma_N v        |_{\dOm_{x_0}\cap\dOm}
&= [(\ga_D \psi)(\wti\gamma_N u) +(\wti\gamma_N \psi)(\ga_D u)]|_{\dOm_{x_0}\cap\dOm}
\no \\
&= [(\gamma_N \psi)(\ga_D u)]|_{\dOm_{x_0}\cap\dOm}\in H^{1/2}(\dOm_{x_o}\cap\dOm),
\end{align}
by Lemma \ref{lA.6} and the fact that $\wti\gamma_N u =0$ on $\dOm$.
Consequently, in order to conclude that $v\in H^2(\Om_{x_o})$,  
it suffices to show the following: 
\begin{eqnarray}\lb{VaK-2}
\left.
\begin{array}{l}
\Omega\subset\bbR^n\,\mbox{ bounded $C^{1,r}$ domain, for some }1/2<r<1
\\[4pt]
w\in H^1(\Om),\quad \Delta w\in L^2(\Om;d^nx),\quad
\wti\ga_N w\in H^{1/2}(\dOm)
\end{array}
\right\}  \text{ imply } \, w\in H^2(\Om).
\end{eqnarray}
To prove \eqref{VaK-2}, we shall employ an integral representation of $w$ in 
terms of layer potentials. Concretely, set $f:=(-\Delta + 1)w\in L^2(\Om;d^nx)$,  
and denote by $w_0$ the restriction to $\Om$ of the convolution of $f$ 
(extended by zero outside $\Om$, to the entire $\bbR^n$) with $E_n(-1;\cdot)$. 
Then $w_0\in H^2(\Om)$ and $(-\Delta + 1)w_0=f$ in $\Omega$. Then
$w_1:=w-w_0$ satisfies
\begin{eqnarray}\lb{VaK-3}
w_1\in H^1(\Om),\quad (-\Delta + 1)w_1=0\, \mbox{ in }\, \Om,\quad 
\wti\ga_N w_1\in H^{1/2}(\dOm), 
\end{eqnarray}
where the last membership makes use of Lemma \ref{lA.6} once again. 
Bring in the adjoint double layer potential operator $K^{\#}_{-1}$ in  
\eqref{Ksharp}. From \cite{Mi96} we know that 
\begin{eqnarray}\lb{VaK-4}
-{\textstyle\frac12}I_{\dOm}+K^{\#}_{-1}\in\cB\bigl(L^2(\dOm;d^{n-1}\omega)\bigr)
\mbox{ is an isomorphism}.
\end{eqnarray}
In addition, by Lemma \ref{K-CPT}, one concludes that 
\begin{eqnarray}\lb{VaK-5}
-{\textstyle\frac12}I_{\dOm}+K^{\#}_{-1}\in\cB\bigl(H^{1/2}(\dOm)\bigr)
\mbox{ is a Fredholm operator with index zero}.
\end{eqnarray}
Together, \eqref{VaK-4}, \eqref{VaK-5} then show that 
\begin{eqnarray}\lb{VaK-6}
-{\textstyle\frac12}I_{\dOm}+K^{\#}_{-1}\in\cB\bigl(H^{1/2}(\dOm)\bigr)
\mbox{ is an isomorphism}.
\end{eqnarray}
Thus, $w_1$ permits the representation 
\begin{eqnarray}\lb{VaK-7}
w_1=\cS_{-1}\big[\bigl(-{\textstyle\frac12}I_{\dOm}+K^{\#}_{-1}\bigr)^{-1}
(\wti\ga_N(w_1))\big]\in H^2(\Om),
\end{eqnarray}
by \eqref{goal-2} and \eqref{VaK-6}, and hence, $w=w_0+w_1\in H^2(\Om)$, 
completing the treatment of Case~I. 

\vskip 0.08in
\noindent{${\mbox{Case~II}}$}: {\it Assume that $n\geq 3$ 
and $\Om_{x_0}$ is a domain of class $MH^{1/2}_\delta$, for some sufficiently
small $\delta>0$}. In this situation we proceed as in Case~I, since
the main ingredients in the proof carried out there, that is, 
\eqref{VaK-5} and \eqref{VaK-7}, continue to hold in this setting 
due to the work in \cite{MS05}. 

\vskip 0.08in
\noindent{${\mbox{Case~III}}$}: {\it Assume that $\Om_{x_0}$ 
is an almost-convex domain}. In this situation, if $\psi$ is as before, 
we consider the vector field $w:=(\psi\nabla u)|_{\Om_{x_o}}
\in (L^2(\Om_{x_o};d^nx))^n$. This vector field satisfies 
\begin{eqnarray}\lb{VaK-8}
\begin{array}{l}
{\rm curl} (w) = (\partial_jw_k-\partial_kw_j)_{1\leq j,k\leq n}\in 
L^2(\Om_{x_0};d^nx)^{n^2},
\\[4pt]
{\rm div} (w) \in L^2(\Om;d^nx)\, \mbox{ and }\, 
\nu\cdot w=0\, \mbox{ in }\, H^{-1/2}(\dOm_{x_0}). 
\end{array}
\end{eqnarray}
Hence, Lemma \ref{jawyr} yields $w\in (H^1(\Om_{x_0}))^n$, implying 
$\psi u\in H^2(\Om_{x_o})$. 
\end{proof}

\begin{remark}\lb{B-gh}
Let $\Omega\subset\bbR^2$ be a polygonal domain with at least one
re-entrant corner. Let $\omega_1,...,\omega_N$ be the internal angles
of $\Omega$ satisfying $\pi<\omega_j<2\pi$, $1\leq j\leq N$, 
and denote by $P_1,...,P_N$ the corresponding vertices. Then the solution
to the Poisson problem 
\begin{eqnarray}\lb{bf-G}
-\Delta u=f\in L^2(\Om;d^2x),\quad u\in H^1_0(\Om),
\end{eqnarray}
permits the representation 
\begin{eqnarray}\lb{bf-G2}
u=\sum_{j=1}^N\lambda_jv_j+w,\quad\lambda_j\in\bbR, 
\end{eqnarray}
where $w\in H^2(\Om)\cap H^1_0(\Om)$ and, for each $j$, $v_j$ is a function 
exhibiting a singular behavior at $P_j$ of the following nature:  
Given $j\in\{1,...,N\}$, choose polar coordinates $(r_j,\theta_j)$ 
taking $P_j$ as the origin and so that the internal angle is spanned by 
the half-lines $\theta_j=0$ and $\theta_j=\omega_j$. Then 
\begin{eqnarray}\lb{bf-G3}
v_j(r_j,\theta_j)=\phi_j(r_j)r_j^{\pi/\omega_j}\sin(\pi\theta_j/\omega_j),
\quad 1\leq j\leq N,
\end{eqnarray}
where $\phi_j$ is a $C^\infty$-smooth cut-off function 
of small support, which is identically one near $P_j$. 

In this scenario, $v_j\in H^s(\Om)$ for every $s<1+(\pi/\omega_j)$, 
though $v_j\notin H^{1+(\pi/\omega_j)}(\Om)$. This implies that the 
best regularity statement regarding the solution of \eqref{bf-G} is 
\begin{eqnarray}\lb{bf-G4}
u\in H^s(\Om)\mbox{ for every } s<1+\frac{\pi}{\max\,\{\omega_1,...,\omega_N\}}
\end{eqnarray}
and this fails for the critical value of $s$. In particular, this
provides a quantifiable way of measuring the failure of the conclusion 
in Lemma \ref{Bjk} for Lipschitz, piecewise $C^\infty$ domains 
exhibiting inwardly directed irregularities.
\end{remark}

We wish to augment Lemma \ref{Bjk} with an analogous result 
for the Robin Laplacian $-\Delta_{\Theta,\Om}$ (cf.\ Theorem \ref{t2.3FF}) 
corresponding to the case where $\Theta$ is a nonnegative, constant function 
on $\partial\Omega$, that is, 
\begin{eqnarray}\lb{Rob-T1}
\Theta(x)\equiv\theta\geq 0,\, \mbox{ for } \, x\in\dOm. 
\end{eqnarray}
From Theorem \ref{t2.3FF} we know that assuming Hypothesis \ref{h2.1}, 
this is a self-adjoint operator in $L^2(\Om; d^n x)$, with domain  
\begin{equation}\lb{2.20B}
\dom(- \Delta_{\Theta,\Om})=\big\{u\in H^1(\Om) \,\big|\, 
\Delta u \in L^2(\Om;d^nx),
\big(\wti\gamma_N +\theta\gamma_D\big) u =0 
\text{ in $H^{-1/2}(\dOm)$}\big\}. 
\end{equation}
In the case where Hypothesis \ref{h.Conv} is assumed, the following regularity
result has been recently established in \cite{JMM08} (see also Theorem 3.2.3.1 
on p\,156 of \cite{Gr85} for the case of convex domains). 

\begin{lemma}\lb{Lt2.3}
Assume Hypothesis \ref{h.Conv} and \eqref{Rob-T1}. 
Then the Robin Laplacian satisfies
\begin{equation}\lb{2.20C}
\dom(- \Delta_{\Theta,\Om})\subset H^2(\Om).
\end{equation}
\end{lemma}

We are now ready to address the issue raised at the beginning of 
Section \ref{s8}:

\begin{theorem}\lb{T-DD1}
Assume Hypothesis \ref{h.Conv}. Then equality holds in \eqref{Yan-8}, that is, one has 
\begin{eqnarray}\lb{Yan-8B}
H^2_0(\Omega) = \big\{u\in L^2(\Omega;d^nx) \,\big|\, \Delta u\in L^2(\Omega;d^nx),\, 
\widehat{\gamma}_D u = \widehat{\gamma}_N u = 0\bigr\}.   
\end{eqnarray}
In particular, 
\begin{eqnarray}\lb{Yan-10} 
- \Delta_{min} = (- \Delta_{max})^*\, \mbox{ and }\, 
- \Delta_{max} = (- \Delta_{min})^*, 
\end{eqnarray}
with 
\begin{equation}
\dom(- \Delta_{min}) = H^2_0(\Omega).   \lb{8.46}
\end{equation}
\end{theorem}
\begin{proof}
As already pointed out before, the left-to-right inclusion in \eqref{Yan-8B}
holds under the mere assumption of Lipschitzianity for $\Omega$. 
To prove the opposite inclusion, consider $u\in L^2(\Omega;d^nx)$ 
satisfying $\Delta u\in L^2(\Omega;d^nx)$ and 
$\widehat{\gamma}_D u =\widehat{\gamma}_N u  =0$. 
Next, let $w\in H^1_0(\Omega)$ solve $\Delta w=\Delta u$ in $\Omega$. 
Then $w\in \dom (- \Delta_{D,\Om})\subset H^2(\Omega)$ by 
Lemma \ref{Bjk} and the current assumptions. Set $v:=u-w$ in $\Omega$ 
and notice that 
\begin{eqnarray}\lb{Yan-11} 
v\in L^2(\Omega;d^nx),\quad \Delta v=0\mbox{ in }\Omega,\quad
\widehat{\gamma}_D v         =0\mbox{ on }\partial\Omega.
\end{eqnarray}
We claim that these conditions imply that $v=0$. 
To show that any $v$ as in \eqref{Yan-11} necessarily vanishes, 
consider an arbitrary $f\in L^2(\Omega;d^nx)$ and let $u_f$ be the unique
solution of 
\begin{eqnarray}\lb{Yan-12} 
u_f\in H^1_0(\Omega),\quad \Delta u_f=f\, \mbox{ in } \,\Omega. 
\end{eqnarray}
Then $u_f\in \dom (- \Delta_{D,\Om})\subset H^2(\Omega)$ by 
Lemma \ref{Bjk}, so that $u_f\in H^2(\Omega)\cap H^1_0(\Omega)$. 
Then the integration by parts formula \eqref{Tan-C12} yields
\begin{align}\lb{Yan-13} 
( f,v)_{L^2(\Om;d^nx)} &= (\Delta u_f,v)_{L^2(\Om;d^nx)}
\nonumber\\ 
&= (u_f,\Delta v)_{L^2(\Om;d^nx)}
+{}_{N^{1/2}(\partial\Omega)}\langle\gamma_N u_f,\widehat{\gamma}_D v         
\rangle_{(N^{1/2}(\partial\Omega))^*}
\nonumber\\ 
&=0. 
\end{align}
Since $f\in L^2(\Omega;d^nx)$ was arbitrary, it follows that $v=0$, as
wanted. In turn, this implies that $u=w\in H^2(\Omega)$.
We recall that $u$ also satisfies $\widehat\gamma_D u =\widehat\gamma_N u =0$,  
or $\gamma_D u =\gamma_N u =0$, by Theorem \ref{New-T-tr} and 
Theorem \ref{3ew-T-tr}. Invoking Theorem \ref{T-MMS} we may then finally 
conclude that $u\in \ker (\gamma_2) = H^2_0(\Omega)$. 
\end{proof} 

It is easy to see that, for any given open set $\Omega\subseteq\bbR^n$, 
and any $z\in\bbC$, the orthogonal complement of 
$\{(-\Delta-\ol{z})\varphi \,|\, \varphi\in C^\infty_0(\Om)\}$ in $L^2(\Om;d^nx)$ 
is $\ker  (-\Delta_{max}-zI_{\Omega})$. This implies that for every
$z\in\bbC$ one has 
\begin{eqnarray}\lb{Man-1} 
L^2(\Omega;d^nx)=\ker  (-\Delta_{max}-zI_{\Omega})\oplus   
\ol{\{(-\Delta-\ol{z})\varphi \,|\, \varphi\in C^\infty_0(\Om)\}},
\end{eqnarray}
where the closure is taken in $L^2(\Om;d^nx)$. Below, we provide a more 
precise version of this result under stronger assumptions on $\Om$ and $z$. 

\begin{theorem}\lb{th-CL}
Assume Hypothesis \ref{h.Conv} and suppose that 
$z\in\bbC\backslash \si(-\Delta_{D,\Om})$. 
Then $\ran (- \Delta_{min}-\ol{z}I_{\Om})$ is closed and 
\begin{eqnarray}\lb{Man-2} 
L^2(\Omega;d^nx)=\ker  (-\Delta_{max}-zI_{\Omega})\oplus   
\big\{(-\Delta-\ol{z})u \,\big|\, u\in H^2_0(\Om)\big\}. 
\end{eqnarray}
\end{theorem}
\begin{proof}
By the previous discussion it suffices to check that 
\begin{eqnarray}\lb{Man-3}
\ol{\{(-\Delta-\ol{z})\varphi \,|\, \varphi\in C^\infty_0(\Om)\}}
= \big\{(-\Delta-\ol{z})u \,\big|\, u\in H^2_0(\Om)\big\}.
\end{eqnarray}
Pick an arbitrary 
$u\in H^2_0(\Om)$ and select a sequence $\{\varphi_j\}_{j\in\bbN}$ 
such that $\varphi_j\in C^\infty_0(\Om)$ for every $j\in\bbN$ and 
$\varphi_j\to u$ in $H^2(\Om)$ as $j\to\infty$. Then 
$(-\Delta-\ol{z})\varphi_j\to (-\Delta -\ol{z})u$ as $j\to\infty$, 
proving the right-to-left inclusion in \eqref{Man-3}.

In the opposite direction, assume that $w\in L^2(\Om;d^nx)$ has the property
that there exists a sequence $\{\varphi_j\}_{j\in\bbN}$ such that 
$\varphi_j\in C^\infty_0(\Om)$ for every $j\in\bbN$ and 
$(-\Delta-\ol{z})\varphi_j\to w$ in $L^2(\Om;d^nx)$ as $j\to\infty$. Then 
\begin{equation}\lb{Man-4}
\varphi_j=(-\Delta_{D,\Om}-\ol{z}I_{\Om})^{-1}[(-\Delta-\ol{z})\varphi_j]
\to  (-\Delta_{D,\Om}-\ol{z}I_{\Om})^{-1}w\, \mbox{ in }\, 
L^2(\Om;d^nx)
\end{equation}
as $j\to\infty$. Thus, if we set 
$u:=(- \Delta_{D,\Om}-\ol{z}I_{\Om})^{-1}w \in \dom (- \Delta_{D\Om})$, then
\begin{eqnarray}\lb{Man-5}
\varphi_j\to  u\, \mbox{ and }\, 
(-\Delta-\ol{z})\varphi_j\to (-\Delta-\ol{z})u
\, \mbox{ in }\, L^2(\Om;d^nx)\, \mbox{ as }\, j\to\infty.
\end{eqnarray}
Next, from \eqref{Man-5} and the continuity of the map in \eqref{3an-C10}
we may deduce that
\begin{eqnarray}\lb{Man-6}
0=\widehat\gamma_N \varphi_j \to \widehat\gamma_N u 
\, \mbox{ in }\, \bigl(N^{3/2}(\dOm)\bigr)^*\, \mbox{ as }\, j\to\infty.
\end{eqnarray}
This shows that $\widehat\gamma_N u =0$, hence $u\in H^2_0(\Om)$ by 
Theorem \ref{T-DD1}. Consequently, $w=(-\Delta-\ol{z})u$ for some 
$u\in H^2_0(\Om)$, proving the left-to-right inclusion in \eqref{Man-3}.
\end{proof}

\begin{remark}\lb{fdg}
An alternative proof of Theorem \ref{th-CL} is to observe that, 
for every $z\in\bbC\backslash \si(-\Delta_{D,\Om})$, the solvability 
of the Poisson problem with a homogeneous Dirichlet boundary condition
for $-\Delta-z$ implies that $\ran  (-\Delta_{max}-zI_{\Om})$
equals $L^2(\Om;d^nx)$ and hence is closed. Then Theorem \ref{T-DD1} 
along with Banach's Closed Range Mapping Theorem
(cf.\ \cite[Theorem IV.2.5.13]{Ka80}) yield that 
the range of $-\Delta_{min}-\ol{z}I_{\Om}=(-\Delta_{max}-zI_{\Om})^*$ is 
also closed. 
\end{remark}
 
\begin{remark}\lb{fdgH}
As a corollary of Theorem \ref{th-CL}, 
for every $z\in\bbC\backslash \si(-\Delta_{D,\Om})$, 
the operator $-\Delta_{min}-\ol{z}I_{\Om}$ maps its domain 
$H^2_0(\Om)$ isomorphically onto its range. 
\end{remark}

\begin{lemma}\lb{Lh-CL}
Assume Hypothesis \ref{h.Conv} and that 
$z\in\bbC\backslash \si(-\Delta_{D,\Om})$. Then 
\begin{eqnarray}\lb{MaH-1} 
H^2(\Omega) = \big[H^2(\Omega)\cap H^1_0(\Om)\big]\dotplus   
\big[H^2(\Omega)\cap\ker  (-\Delta_{max}-zI_{\Omega})\big]. 
\end{eqnarray}
\end{lemma}
\begin{proof}
Indeed, the membership of $z$ to $\bbC\backslash \si(-\Delta_{D,\Om})$
ensures that the sum in \eqref{MaH-1} is direct. In addition, if 
$u\in H^2(\Omega)$ is arbitrary and $v\in H^2(\Omega)\cap H^1_0(\Om)$ 
satisfies $(-\Delta -z)v=(-\Delta -z)u$ in $\Om$, then $u=v+(u-v)$ and
$u-v\in H^2(\Omega)\cap\ker  (-\Delta_{max}-zI_{\Omega})$. 
From this, the desired conclusion follows easily. 
\end{proof}

We recall that, given a Hilbert space $\cH$ with inner product 
$(\dott,\dott)_{\mathcal{H}}$, a linear operator 
$S: \dom (S)\subseteq \cH\to \cH$ 
is called {\it nonnegative} provided $(Su,u)_{\cH}\geq 0$ for every 
$u\in \dom (S)$. A simple argument based on polarization formula then yields (the well-known fact)
that any nonnegative operator $S$ is {\it symmetric}, that is, 
$(Su,v)_{\cH}=(u,Sv)_{\cH}$ for every $u,v\in \dom (S)$.

\begin{corollary}\lb{C-Da}
Assume Hypothesis \ref{h.Conv}. Then $-\Delta_{min}$ is a densely defined, 
closed, nonnegative $($and hence, symmetric\,$)$ operator. Furthermore,  
\begin{eqnarray}\lb{Pos-3}
\ol{- \Delta\big|_{C^\infty_0(\Om)}} = - \Delta_{min}. 
\end{eqnarray}
\end{corollary}
\begin{proof} 
The first claim in the statement is a direct consequence of Theorem \ref{T-DD1}.
As for \eqref{Pos-3}, let us temporarily denote $- \Delta_0 = - \Delta\big|_{C^\infty_0(\Om)}$. Then 
\begin{align}\lb{Pos-4}
& u\in \dom (- \Delta_0)  \no \\
& \quad \text{ if and only if } \,
\left\{
\begin{array}{l}
\mbox{there exist }v\in L^2(\Om;d^nx)\mbox{ and }
u_j\in C^\infty_0(\Om),\,j\in\bbN,\mbox{ such that }
\\[4pt]
u_j\to u\, \mbox{ and }\, \Delta u_j\to v\, \mbox{ in }\,L^2(\Om;d^nx)
\, \mbox{ as }\,j\to\infty.
\end{array}
\right.
\end{align}
Thus, if $u\in \dom (- \Delta_0)$ and $v$, $\{u_j\}_{j\in\bbN}$ are 
as in the right-hand side of \eqref{Pos-4}, then 
$\Delta u=v$ in the sense of distributions in $\Omega$, and 
\begin{eqnarray}\lb{Pos-5}
\begin{array}{l}
0=\widehat\gamma_D u_j \to \widehat\gamma_D u 
\, \mbox{ in }\, \bigl(N^{1/2}(\dOm)\bigr)^* \, 
\mbox{ as }\, j\to\infty,
\\[6pt]
0=\widehat\gamma_N u_j \to \widehat\gamma_N u 
\, \mbox{ in }\, \bigl(N^{1/2}(\dOm)\bigr)^* \, 
\mbox{ as }\, j\to\infty,
\end{array}
\end{eqnarray}
by Theorems \ref{New-T-tr} and \ref{3ew-T-tr}.
Consequently, $u\in \dom (-\Delta_{max})$ satisfies
$\widehat\gamma_D u = 0$ and $\widehat\gamma_N u = 0$. 
Hence, $u\in H^2_0(\Om)= \dom (- \Delta_{min})$ by Theorem \ref{T-DD1}
and the current assumptions on $\Omega$. This shows that 
$- \Delta_0 \subseteq - \Delta_{min}$. The converse inclusion readily follows 
from the fact that any $u\in H^2_0(\Om)$ is the limit in 
$H^2(\Om)$ of a sequence of test functions in $\Omega$. 
\end{proof} 

\begin{remark} \lb{r8.20}
By Remark \ref{r7.3}, \eqref{Yan-10} in Theorem \ref{T-DD1} and 
Corollary \ref{C-Da} actually hold for all open, bounded domains 
$\Omega \subset \bbR^n$, $n \geq 2$. 
\end{remark} 

\begin{lemma}\lb{J-Sa}
Assume Hypothesis \ref{h.Conv}. Then the domain of $-\Delta_{max}$ has
the description 
\begin{eqnarray}\lb{Pos-3W}
\dom (- \Delta_{max}) = \big[H^2(\Om)\cap H^1_0(\Om)\big]\dotplus   
\ker (- \Delta_{max}).
\end{eqnarray}
\end{lemma}
\begin{proof} 
The right-to-left inclusion in \eqref{Pos-3W} is obvious. As for the
opposite one, if $u\in \dom (- \Delta_{max})$ and 
$v\in H^2(\Om)\cap H^1_0(\Om)$ solves the inhomogeneous Dirichlet
problem $\Delta v=\Delta u$ in $\Omega$, then $u=v+(u-v)$ with 
$(u-v) \in \ker (- \Delta_{max})$. Finally, the fact that the sum 
in \eqref{Pos-3W} is direct is a consequence of the fact that
$0\notin\si(- \Delta_{D,\Om})$.
\end{proof}

\section{The Friedrichs and Krein--von Neumann Extensions: 
Some Abstract Results}
\lb{s9}

This short section is devoted to a brief summary of abstract results on the Friedrichs and 
Krein--von Neumann extensions of a closed symmetric operator $S$ in some separable, 
complex Hilbert space $\cH$.

A linear operator $S:\dom(S)\subseteq\cH\to\cH$, is called {\it symmetric}, if
$(u,Sv)_\cH=(Su,v)_\cH$, $u,v\in \dom (S)$. If $\dom(S)=\cH$, the classical 
Hellinger--Toeplitz theorem guarantees that $S\in\cB(\cH)$, in which situation $S$ is readily seen to be self-adjoint. In general, however, symmetry is a considerably weaker property than self-adjointness and a classical problem in functional analysis is that of determining all self-adjoint extensions of a given closed symmetric operator of equal and nonzero deficiency indices. In this manuscript we will be interested in this question within the class of densely defined (i.e., $\ol{\dom(S)}=\cH$), strictly positive, closed operators $S$. In this section we focus exclusively on self-adjoint extensions of $S$ that are nonnegative operators. In the latter scenario, there are two distinguished (and extremal) constructions which we will briefly review next. 

To set the stage, we recall that a linear operator $S:\dom(S)\subseteq\cH\to \cH$
is called {\it nonnegative} provided $(u,Su)_\cH\geq 0$, $u\in \dom(S)$. 
(In particular, $S$ is symmetric in this case.) $S$ is called {\it strictly positive}, if for some 
$\varepsilon >0$, $(u,Su)_\cH\geq \varepsilon \|u\|_{\cH}$, $u\in \dom(S)$. 
Next, we recall that $A \leq B$ for two self-adjoint operators in $\cH$ if 
\begin{align}
\begin{split}
& \dom\big(|A|^{1/2}\big) \supseteq \dom\big(|B|^{1/2}\big) \, \text{ and } \\ 
& \quad \big(|A|^{1/2}u, U_A |A|^{1/2}u\big)_{\cH} \leq \big(|B|^{1/2}u, U_B |B|^{1/2}u\big)_{\cH}, \quad   
u \in \dom\big(|B|^{1/2}\big),      \lb{AleqB} 
\end{split}
\end{align}
where $U_C$ denotes the partial isometry in $\cH$ in the polar decomposition of 
a densely defined closed operator $C$ in $\cH$, $C=U_C |C|$, $|C|=(C^* C)^{1/2}$. (If 
in addition, $C$ is self-adjoint, then $U_C$ and $|C|$ commute.) 
We also recall (\cite[Part II]{Fa75}, \cite[Theorem VI.2.21]{Ka80}) that if $A$ and $B$ are both self-adjoint and nonnegative in $\cH$, then 
\begin{align}
\begin{split}
& 0 \leq A\leq B \, \text{ if and only if } \,  0 \leq A^{1/2 }\leq B^{1/2},     \label{PPa-1} \\
& \quad \, \text{equivalently, if and only if } \, 
(B + a I_\cH)^{-1} \leq (A + a I_\cH)^{-1} \, \text{ for all $a>0$,} 
\end{split}
\end{align}
and
\begin{equation}
\ker(A) =\ker\big(A^{1/2}\big)
\end{equation}
(with $C^{1/2}$ the unique nonnegative square root of a nonnegative self-adjoint operator $C$ in $\cH$).

For simplicity we will always adhere in this section to the conventions that $S$ is a linear, unbounded, densely defined, nonnegative (i.e., $S\geq 0$) operator in $\cH$, and that $S$ has nonzero deficiency indices. In particular
\begin{equation}
{\rm def} (S) = \dim (\ker(S^*-z I_{\cH})) \in \bbN\cup\{\infty\}, 
\quad z\in \bbC\backslash [0,\infty),  
\lb{DEF}
\end{equation}
is well-known to be independent of $z$. 
Moreover, since $S$ and its closure $\ol{S}$ have the same self-adjoint extensions in $\cH$, we will without loss of generality assume that $S$ is closed in the remainder of this section.

The following is a fundamental result to be found in M.\ Krein's celebrated 1947 paper
 \cite{Kr47} (cf.\ also Theorems\ 2 and 5--7 in the English summary on page 492): 

\begin{theorem}\label{T-kkrr}
Assume that $S$ is a densely defined, closed, nonnegative operator in $\cH$. Then, among all 
nonnegative self-adjoint extensions of $S$, there exist two distinguished ones, $S_K$ and $S_F$, which are, respectively, the smallest and largest
$($in the sense of order between linear operators, cf.\ \eqref{AleqB}$)$ such extension. 
Furthermore, a nonnegative self-adjoint operator $\widetilde{S}$ is a self-adjoint extension 
of $S$ if and only if $\widetilde{S}$ satisfies 
\begin{equation}\label{Fr-Sa}
S_K\leq\widetilde{S}\leq S_F.
\end{equation}
In particular, \eqref{Fr-Sa} determines $S_K$ and $S_F$ uniquely. 

In addition,  if $S\geq \varepsilon I_{\cH}$ for some $\varepsilon >0$, one has 
\begin{align}
\dom (S_F) &= \dom (S) \dotplus (S_F)^{-1} \ker (S^*),     \lb{SF}  \\
\dom (S_K) & = \dom (S) \dotplus \ker (S^*),    \lb{SK}   \\
\dom (S^*) & = \dom (S) \dotplus (S_F)^{-1} \ker (S^*) \dotplus \ker (S^*)  \no \\
& = \dom (S_F) \dotplus \ker (S^*),    \lb{S*} 
\end{align}
In particular, 
\begin{equation} \label{Fr-4Tf}
\ker(S_K)= \ker\big((S_K)^{1/2}\big)= \ker(S^*) = \ran(S)^{\bot}.
\end{equation} 
\end{theorem}

Here the operator inequalities in \eqref{Fr-Sa} are understood in the sense of 
\eqref{AleqB} and hence they can  equivalently be written as
\begin{equation}
(S_F + a I_{\cH})^{-1} \le \big(\wti S + a I_{\cH}\big)^{-1} \le (S_K + a I_{\cH})^{-1} 
\, \text{ for some (and hence for all) $a > 0$.}    \lb{Res}
\end{equation}

We also refer to Birman \cite{Bi56}, \cite{Bi08}, Friedrichs \cite{Fr34}, Freudenthal \cite{Fr36}, 
Grubb \cite{Gr68}, \cite{Gr70}, Krein \cite{Kr47a}, {\u S}traus \cite{St73}, and Vi{\u s}ik \cite{Vi63} 
(see also the monographs by Akhiezer and Glazman \cite[Sect. 109]{AG81a}, 
Faris \cite[Part III]{Fa75}, and the recent book by Grubb \cite[Sect.\ 13.2]{Gr09}) for classical 
references on the subject of self-adjoint extensions of semibounded operators.  

We will call the operator $S_K$ the {\it Krein--von Neumann extension}
of $S$. See \cite{Kr47} and also the discussion in \cite{AS80}, \cite{AT03}, \cite{AT05}. It should be
noted that the Krein--von Neumann extension was first considered by von Neumann 
\cite{Ne29} in 1929 in the case where $S$ is strictly positive, that is, if 
$S \geq \varepsilon I_{\cH}$ for some $\varepsilon >0$. (His construction appears in the proof of Theorem 42 on pages 102--103.) However, von Neumann did not isolate the extremal property of this extension as described in \eqref{Fr-Sa} and \eqref{Res}. 
M.\ Krein \cite{Kr47}, \cite{Kr47a} was the first to systematically treat the general case 
$S\geq 0$ and to study all nonnegative self-adjoint extensions of $S$, illustrating the special role of the {\it Friedrichs extension} (i.e., the ``hard'' extension) $S_F$ of $S$ and the Krein--von Neumann (i.e., the ``soft'') extension $S_K$ of $S$ as extremal cases when   considering all nonnegative extensions of $S$. For a recent exhaustive treatment of 
self-adjoint extensions of semibounded operators we refer to \cite{AT02}--\cite{AT09}. 

An intrinsic description of the Friedrichs extension $S_F$ of $S\geq 0$ due to Freudenthal \cite{Fr36} 
in 1936 describes $S_F$ as the operator $S_F:\dom(S_F)\subset\cH\to\cH$ given by   
\begin{align}
& S_F u:=S^*u,   \no \\
& u \in \dom(S_F):=\big\{v\in\dom(S^*)\,\big|\,  \mbox{there exists} \, 
\{v_j\}_{j\in\bbN}\subset \dom(S),    \label{Fr-2} \\
& \quad \mbox{ with } \, \lim_{j\to\infty}\|v_j-v\|_{\cH}=0  \, 
\mbox{ and } \, ((v_j-v_k),S(v_j-v_k))_\cH\to 0 \mbox{ as }  j,k\to\infty\big\}.    \no 
\end{align}
Then, as is well-known,  
\begin{align}
& S_F \geq 0,   \label{Fr-4}  \\
& \dom\big((S_F)^{1/2}\big)=\big\{v\in\cH\,\big|\, \mbox{there exists} \, 
\{v_j\}_{j\in\bbN}\subset \dom(S),   \label{Fr-4J} \\
& \quad \mbox{ with } \, \lim_{j\to\infty}\|v_j-v\|_{\cH}=0  \, 
\mbox{ and } \, ((v_j-v_k),S(v_j-v_k))_\cH\to 0\mbox{ as }
j,k\to\infty\big\},  \no 
\end{align}
and
\begin{equation}\label{Fr-4H}
S_F=S^*|_{\dom(S^*)\cap\dom((S_{F})^{1/2})}.
\end{equation}

An intrinsic description of the Krein--von Neumann extension $S_K$ of $S\geq 0$ has been given by Ando and Nishio \cite{AN70} in 1970, where $S_K$ has been characterized as the 
operator $S_K:\dom(S_K)\subset\cH\to\cH$ given by
\begin{align} 
& S_Ku:=S^*u,   \no \\
& u \in \dom(S_K):=\big\{v\in\dom(S^*)\,\big|\,\mbox{there exists} \, 
\{v_j\}_{j\in\bbN}\subset \dom(S),    \label{Fr-2X}  \\ 
& \quad \mbox{ with} \, \lim_{j\to\infty} \|Sv_j-S^*v\|_{\cH}=0  \, 
\mbox{ and } \, ((v_j-v_k),S(v_j-v_k))_\cH\to 0 \mbox{ as } j,k\to\infty\big\}.  \no
\end{align}

For a variety of recent new results on $S_K$ in connection with Weyl-type spectral asymptotics for perturbed Laplacians on non-smooth bounded open domains in $\bbR^n$, and for a connection to an abstract buckling problem that illustrates the relevance of $S_K$ in elasticity theory, we refer to 
\cite{AGMT09} and \cite{AGMST09}. 

We continue by recording an abstract result regarding the
parametrization of all nonnegative self-adjoint extensions of a given
strictly positive, densely defined, symmetric operator. The following results were 
developed from Krein \cite{Kr47}, Vi{\u s}ik \cite{Vi63}, and Birman \cite{Bi56}, by 
Grubb \cite{Gr68}, \cite{Gr70}. Subsequent expositions are due to Faris \cite[Sect.\ 15]{Fa75} 
and Alonso and Simon \cite{AS80}. 

\begin{theorem}\label{AS-th}
Suppose that $S$ is a densely defined, closed operator in $\cH$, and 
$S \geq \varepsilon I_{\cH}$ for some $\varepsilon>0$. Then there exists a one-to-one correspondence between nonnegative self-adjoint operators 
$0 \leq B:\dom(B)\subseteq \cW\to \cW$, $\ol{\dom(B)}=\cW$, where $\cW$
is a closed subspace of $\cN_0 :=\ker(S^*)$, and nonnegative self-adjoint
extensions $S_{B,\cW}\geq 0$ of $S$. More specifically, $S_F$ is invertible, 
$S_F\geq \varepsilon I_{\cH}$, 
and one has
\begin{align} 
& \dom(S_{B,\cW})  \no \\
& \quad =\big\{f + (S_F)^{-1}(Bw + \eta)+ w \,\big|\,
f\in\dom(S),\, w\in\dom(B),\, \eta\in \cN_0 \cap \cW^{\bot}\big\},  \no \\
& S_{B,\cW} = S^*|_{\dom(S_{B,\cW})},       \label{AS-1}  
\end{align} 
where $\cW^{\bot}$ denotes the orthogonal complement of $\cW$ in $\cN_0$. In 
addition,
\begin{align}
&\dom\big((S_{B,\cW})^{1/2}\big) = \dom\big((S_F)^{1/2}\big) \dotplus 
\dom\big(B^{1/2}\big),   \\
&\big\|(S_{B,\cW})^{1/2}(u+g)\big\|_{\cH}^2 =\big\|(S_F)^{1/2} u\big\|_{\cH}^2 
+ \big\|B^{1/2} g\big\|_{\cH}^2, \\ 
& \hspace*{2.3cm} u \in \dom\big((S_F)^{1/2}\big), \; g \in \dom\big(B^{1/2}\big),  \no
\end{align}
implying,
\begin{equation}\label{K-ee}
\ker(S_{B,\cW})=\ker(B). 
\end{equation} 
Moreover, 
\begin{equation}
B \leq \wti B \, \text{ implies } \, S_{B,\cW} \leq S_{\wti B,\wti \cW},
\end{equation}
where 
\begin{align}
\begin{split}
& B\colon \dom(B) \subseteq \cW \to \cW, \quad 
 \wti B\colon \dom\big(\wti B\big) \subseteq \wti \cW \to \wti \cW,   \\
& \ol{\dom\big(\wti B\big)} = \wti\cW \subseteq \cW = \ol{\dom(B)}. 
\end{split}
\end{align}

In the above scheme, the Krein--von Neumann extension $S_K$
of $S$ corresponds to the choice $\cW=\cN_0$ and $B=0$ $($with 
$\dom(B)=\dom(B^{1/2})=\cN_0=\ker (S^*)$$)$.
In particular, one thus recovers \eqref{SK}, and \eqref{Fr-4Tf}, and also obtains
\begin{align}
&\dom\big((S_{K})^{1/2}\big) = \dom\big((S_F)^{1/2}\big) \dotplus \ker (S^*),  
\lb{SKform1}  \\
&\big\|(S_{K})^{1/2}(u+g)\big\|_{\cH}^2 =\big\|(S_F)^{1/2} u\big\|_{\cH}^2, 
\quad u \in \dom\big((S_F)^{1/2}\big), \; g \in \ker (S^*).   \lb{SKform2}
\end{align}
Finally, the Friedrichs extension $S_F$ corresponds to the choice $\dom(B)=\{0\}$ $($i.e., formally, 
$B\equiv\infty$$)$, in which case one recovers \eqref{SF}. 
\end{theorem}

The relation $B \leq \wti B$ in the case where $\wti \cW \subsetneqq \cW$ requires an explanation: In analogy to \eqref{AleqB} we mean 
\begin{equation}
\big(|B|^{1/2}u, U_B |B|^{1/2}u\big)_{\cW} \leq \big(|\wti B|^{1/2}u, U_{\wti B} |\wti B|^{1/2}u\big)_{\cW}, 
\quad u \in \dom\big(|\wti B|^{1/2}\big) 
\end{equation}
and (following \cite{AS80}) we put
\begin{equation}
\big(|\wti B|^{1/2}u, U_{\wti B} |\wti B|^{1/2}u)_{\cW} = \infty \, \text{ for } \, 
u \in \cW \backslash \dom\big(|\wti B|^{1/2}\big).
\end{equation}

We also note that under the assumptions on $S$ in Theorem \ref{AS-th}, one has 
\begin{equation}
\dim(\ker (S^*-z I_{\cH})) = \dim(\ker(S^*)) = \dim (\cN_0) = {\rm def} (S), 
\quad z\in \bbC\backslash [\varepsilon,\infty).   \lb{dim}
\end{equation}

For now, our goal is to prove the following result:
 
\begin{lemma}\lb{L-Fri1}
Assume Hypothesis \ref{h2.1}. Then the Friedrichs extension of the operator 
$-\Delta\big|_{C^\infty_0(\Om)}$, and hence that of 
$\ol{-\Delta\big|_{C^\infty_0(\Om)}}$ in $L^2(\Om;d^nx)$, is precisely the Dirichlet 
Laplacian $-\Delta_{D,\Om}$.
As a consequence, if Hypothesis \ref{h.Conv} is assumed,  
the Friedrichs extension of $-\Delta_{min}$ in \eqref{Yan-6} is 
the Dirichlet Laplacian $-\Delta_{D,\Om}$.
\end{lemma}
\begin{proof}
We recall that $\dom (- \Delta_{D,\Om})=\big\{u\in H^1_0(\Om) \,\big|\, \Delta u\in 
L^2(\Om;d^nx)\big\}$. Unravelling definitions, it can then easily be checked that,
with $\dot S$ denoting $-\Delta\big|_{C^\infty_0(\Om)}$ in $L^2(\Om; d^n x)$, we have 
$\dom (- \Delta_{D,\Om})\subseteq\dom \big(\big(\dot S\big)^*\big) 
= \dom\Big(\Big(\ol{{\dot S}}\Big)^*\Big)$. Furthermore, 
if $u\in\dom (- \Delta_{D,\Om})\hookrightarrow H^1_0(\Om)$, then there 
exist $u_j\in C^\infty_0(\Om)$, $j\in\bbN$, such that 
$u_j\to u$ in $L^2(\Om;d^nx)$ as $j\to\infty$ and
$\langle \dot S(u_j-u_k),u_j-u_k\rangle_{L^2(\Om;d^nx)}
=\|\nabla u_j-\nabla u_k\|_{L^2(\Om;d^nx)}\to 0$
as $j,k\to\infty$. Thus, if $\Big(\ol{{\dot S}}\Big)_F:\dom \Big(\Big(\ol{{\dot S}}\Big)_F\Big)\subseteq L^2(\Om;d^nx)
\to L^2(\Om;d^nx)$ is defined according to \eqref{Fr-2}, it follows that 
$-\Delta_{D,\Om}\subseteq \Big(\ol{{\dot S}}\Big)_F$. Hence, 
$-\Delta_{D,\Om} = \Big(\ol{{\dot S}}\Big)_F$ by the 
maximality of the self-adjoint operator $-\Delta_{D,\Om}$. 
\end{proof}

Concluding this section, we point out that a great variety of additional results 
for the Krein--von Neumann extension can be found, for instance, in \cite{Ad07}, 
\cite[Sect.\ 109]{AG81a}, \cite{AS80}, \cite{AN70}, \cite{Ar98}, \cite{Ar00}, \cite{AHSD01}, 
\cite{AT02}, \cite{AT03}, \cite{AT05}, \cite{AT09}, \cite{AGMT09}, \cite{AGMST09}, \cite{BC05}, 
\cite[Part III]{Fa75}, \cite[Sect.\ 3.3]{FOT11}, \cite{GM09}, \cite{Gr83}, \cite[Sect.\ 13.2]{Gr09}, 
\cite{Ha57}, \cite{HK09}, \cite{HMD04}, \cite{HSDW07}, \cite{KO77}, \cite{KO78}, \cite{MT07}, 
\cite{Ne83}, \cite{PS96}, \cite{SS03}, \cite{Si98}, \cite{Sk79}, \cite{St96}, \cite{Ts80}, \cite{Ts81}, 
\cite{Ts92}, and the references therein. We also mention the references \cite{EM05}, 
\cite{EMP04}, \cite{EMMP07} (these authors, apparently unaware of the work of von 
Neumann, Krein, Vi{\u s}hik, Birman, Grubb, {\u S}trauss, etc., in this context, 
introduced the Krein Laplacian and called it the harmonic operator; see also 
\cite{Gr06}).

\section{Trace Operators and Boundary Value Problems 
on Quasi-Convex Domains}
\lb{s10}

In this section we revisit the trace theory discussed in Section \ref{s6}, this 
time assuming that the underlying domain is either smooth or quasi-convex. 
Such a context allows for more refined results, such as the ontoness
of the trace operators in question. In the case of quasi-convex domains,  
these results are establish as a corollary of the
well-posedness of certain basic boundary value problems for the Laplacian. 

\begin{lemma}\lb{L-C1}
Assume that $\Omega\subset \bbR^n$, $n\geq 2$, is a bounded $C^{1,r}$ 
domain for some $r\in(1/2,1)$. Then the Dirichlet trace operator 
\begin{eqnarray}\lb{D-tr1}
\gamma_D:H^2(\Omega)\to  H^{3/2}(\partial\Omega), 
\end{eqnarray}
along with the Neumann trace operator 
\begin{eqnarray}\lb{D-tr2}
\gamma_N:H^2(\Omega)\to  H^{1/2}(\partial\Omega), 
\end{eqnarray}
are well-defined, bounded, and onto. In fact, each operator has a linear, 
bounded right-inverse.
\end{lemma}
\begin{proof}
That the operator \eqref{D-tr1} is well-defined and bounded follows
from the observation that for each $u\in H^2(\Omega)$ one has
\begin{align}\lb{D-tr3}
\nabla_{tan}(\gamma_D u ) = \bigg(\sum_{k=1}^n\nu_k\frac{\partial\gamma_D u }
{\partial\tau_{kj}}\bigg)_{1\leq j\leq n}
= \bigg(\gamma_D (\partial_j u)-\sum_{k=1}^n\nu_k\nu_j\gamma_D(\partial_k u)
\bigg)_{1\leq j\leq n}, 
\end{align}
and from Lemmas \ref{K-t1} and \ref{lA.6}. 
That the operator \eqref{D-tr1} has a linear, bounded right-inverse
is a consequence of Lemmas \ref{Lk-t1} and \ref{3o-Tx}.

With the goal of proving that the operator introduced in \eqref{D-tr2} is  
well-defined and bounded, one first observes that for each $u\in H^2(\Omega)$ one has
\begin{eqnarray}\lb{D-tr4}
\gamma_N u  =\nu\cdot\gamma_D(\nabla u). 
\end{eqnarray}
Then the desired conclusion follows from \eqref{2.6} and Lemma \ref{lA.6}. 
That the operator \eqref{D-tr2} also has a linear, bounded right-inverse
is a consequence of \eqref{Tan-C5} and Lemma \ref{Lo-Tx}.
\end{proof}

\begin{theorem}\lb{tH.G}
Assume Hypothesis \ref{h.Conv}. Then the Dirichlet trace operator 
\eqref{Tan-C10} is onto. More specifically, for every 
$z\in\bbC\backslash \si(-\Delta_{D,\Om})$ there exists $C=C(\Omega,z)>0$
with the property that for any functional 
$\Lambda\in\bigl(N^{1/2}(\partial\Omega)\bigr)^*$ there exists 
$u_\Lambda \in L^2(\Om;d^nx)$ with $(-\Delta -z)u_\Lambda =0$ in $\Omega$, 
such that 
\begin{eqnarray}\lb{Nh-G}
\mbox{the assignment $\bigl(N^{1/2}(\partial\Omega)\bigr)^*\ni 
\Lambda\mapsto u_\Lambda \in L^2(\Om;d^nx)$ is linear}, 
\end{eqnarray}
and 
\begin{eqnarray}\lb{Nh-1}
\widehat\gamma_D u_\Lambda = \Lambda\, \mbox{ and }\, 
\|u_\Lambda\|_{L^2(\Om;d^nx)}\leq C\|\Lambda\|_{(N^{1/2}(\partial\Omega))^*}.
\end{eqnarray}
\end{theorem}
\begin{proof}
Fix $z\in\bbC\backslash \si(-\Delta_{D,\Om})$. By Lemma \ref{Bjk}, the inverse 
$(-\Delta_{D,\Omega}-\ol{z}I_{\Om})^{-1}:L^2(\Om;d^nx)
\to H^2(\Omega)\cap H^1_0(\Om)$ is well-defined, linear, and bounded. Given 
$\Lambda\in\bigl(N^{1/2}(\partial\Omega)\bigr)^*$, consider 
the bounded linear functional 
\begin{eqnarray}\lb{Nh-2}
\Lambda \gamma_N  (-\Delta_{D,\Omega}-\ol{z}I_{\Om})^{-1}:
L^2(\Om;d^nx)\to  {\mathbb{C}}.
\end{eqnarray}
Then, by the Riesz representation theorem, there exists a unique
$u_\Lambda\in L^2(\Om;d^nx)$ with the property that 
\begin{eqnarray}\lb{Nh-3}
\Lambda\bigl(\gamma_N((-\Delta_{D,\Omega}-\ol{z}I_{\Om})^{-1}f)\bigr)
= (f,u_\Lambda)_{L^2(\Om;d^nx)},
\end{eqnarray}
and such that 
\begin{align}\lb{Nh-4}
\|u_\Lambda\|_{L^2(\Om;d^nx)} & \leq 
C\|\Lambda \gamma_N  (-\Delta_{D,\Omega}-\ol{z}I_{\Om})^{-1}\|
_{\cB(L^2(\Om;d^nx),{\mathbb{C}})}  \no \\
& \leq C\|\Lambda\|_{(N^{1/2}(\partial\Omega))^*} 
\end{align}
for some constant $C=C(\Om,z)>0$, independent of $\Lambda$. 
Furthermore, by the uniqueness of $u_\Lambda$, the assignment 
$\bigl(N^{1/2}(\partial\Omega)\bigr)^*\ni\Lambda\mapsto 
u_\Lambda \in L^2(\Om;d^nx)$ is linear.  

Specializing \eqref{Nh-3} to the case where $f:=(-\Delta-\ol{z})w$, 
with $w\in H^2(\Omega)\cap H^1_0(\Om)$ then yields 
\begin{equation}\lb{Nh-5}
\Lambda\big(\gamma_N w \big)
= ((-\Delta-\ol{z})w,u_\Lambda)_{L^2(\Om;d^nx)},
\quad w\in H^2(\Omega)\cap H^1_0(\Om). 
\end{equation}
In particular, choosing $w\in C^\infty_0(\Omega)$ yields
\begin{eqnarray}\lb{Nh-6}
0=\Lambda(0)=\Lambda\big(\gamma_N w \big)
= ((-\Delta-\ol{z})w,u_\Lambda)_{L^2(\Om;d^nx)},
\quad w\in C^\infty_0(\Omega). 
\end{eqnarray}
Thus, one concludes that $(-\Delta -z)u_\Lambda=0$ in 
$\Omega$, in the sense of distributions. Using this in \eqref{Nh-3} one then obtains  
\begin{align}\lb{Nh-7}
\Lambda\big(\gamma_N w \big) &=
((-\Delta-\ol{z})w,u_\Lambda)_{L^2(\Om;d^nx)}
- (w,(-\Delta -z)u_\Lambda)_{L^2(\Om;d^nx)}
\nonumber\\
&= 
-{}_{N^{1/2}(\partial\Omega)}\langle\gamma_N w,\widehat{\gamma}_D u_\Lambda 
\rangle_{(N^{1/2}(\partial\Omega))^*},\quad w\in H^2(\Omega)\cap H^1_0(\Om). 
\end{align}
Since $\gamma_N$ maps $H^2(\Omega)\cap H^1_0(\Om)$ onto $N^{1/2}(\dOm)$, 
\eqref{Nh-7} implies that $\widehat{\gamma}_D u_\Lambda = - \Lambda$. 
This shows that the operator \eqref{Tan-C10} has a linear, 
bounded right-inverse. In particular, it is onto.  
\end{proof}

\begin{corollary}\lb{New-CV2}
Assume Hypothesis \ref{h.Conv}. Then for every 
$z\in\bbC\backslash \si(-\Delta_{D,\Om})$, the map 
\begin{eqnarray}\lb{Tan-Bq2} 
\widehat{\gamma}_D: 
\big\{u\in L^2(\Omega;d^nx) \,\big|\, (-\Delta-z)u=0 \mbox{ in }\Omega\big\}
\to \bigl(N^{1/2}(\partial\Omega)\bigr)^*
\end{eqnarray}
is onto. 
\end{corollary}
\begin{proof}
This is a direct consequence of Theorem \ref{tH.G}.
\end{proof}

\begin{theorem}\lb{tH.A}
Assume Hypothesis \ref{h.Conv} and suppose that 
$z\in\bbC\backslash\si(-\Delta_{D,\Om})$. Then for any   
$f\in L^2(\Om;d^nx)$ and $g\in (N^{1/2}(\partial\Omega))^*$ the 
following inhomogeneous Dirichlet boundary value problem 
\begin{equation}\lb{Yan-14}
\begin{cases}
(-\Delta-z)u=f\text{ in }\,\Om,
\\[4pt]
u\in L^2(\Om;d^nx),
\\[4pt]
\widehat\ga_D u  =g\text{ on }\,\dOm,
\end{cases}
\end{equation}
has a unique solution $u=u_D$. This solution satisfies 
\begin{equation}\lb{Hh.3X} 
\|u_D\|_{L^2(\Om;d^nx)}
+\|\widehat\ga_N u_D\|_{(N^{3/2}(\partial\Omega))^*}\leq C_D
(\|f\|_{L^2(\Om;d^nx)}+\|g\|_{(N^{1/2}(\partial\Omega))^*})
\end{equation}
for some constant $C_D=C_D(\Omega,z)>0$, and the following regularity 
results hold:  
\begin{eqnarray}\lb{3.3Y}
&& g\in H^1(\partial\Omega) \, \text{ implies } \, u_D\in H^{3/2}(\Omega),
\\[4pt]
&& g\in\gamma_D\bigl(H^2(\Omega)\bigr)
 \, \text{ implies } \,  u_D\in H^2(\Omega).
\lb{3.3Ys}
\end{eqnarray} 
In particular, 
\begin{eqnarray}\lb{3.3Ybis}
g=0 \, \text{ implies } \,  u_D\in H^2(\Omega)\cap H^1_0(\Omega).
\end{eqnarray} 

Natural estimates are valid in each case. 

Moreover, 
\begin{equation}\lb{3.34Y} 
\big[\ga_N (-\Delta_{D,\Om}-{\ol z}I_\Om)^{-1}\big]^* \in 
\cB\big((N^{1/2}(\partial\Omega))^*, L^2(\Om;d^nx)\big),   
\end{equation}
and the solution of \eqref{Yan-14} is given by the formula
\begin{equation}\lb{3.35Y}
u_D=(-\Delta_{D,\Om}-zI_\Om)^{-1}f
-\big[\ga_N (-\Delta_{D,\Om}-\ol{z}I_\Om)^{-1}\big]^*g. 
\end{equation}
\end{theorem}
\begin{proof}
By Theorem \ref{tH.G} there exists a function $v\in L^2(\Om;d^nx)$ satisfying 
$\Delta v=0$ in $\Om$, $\widehat\ga_D v          =g$ 
and $\|v\|_{L^2(\Om;d^nx)}\leq C\|g\|_{(N^{1/2}(\partial\Omega))^*}$ for 
some finite $C=C(\Omega)>0$. Setting  
\begin{equation}\lb{Jh-1}
w:=(-\Delta_{D,\Omega}-zI_{\Om})^{-1}(f+zv)
\end{equation}
so that 
\begin{equation}\lb{Jh-2}
w\in \dom (- \Delta_{D,\Omega})\hookrightarrow H^2(\Om), 
\end{equation}
by our hypotheses and by Lemma \ref{Bjk}, and  
\begin{equation}\lb{Jh-3}
\|w\|_{H^2(\Om)}\leq C\|f+zv\|_{L^2(\Om,d^nx)}
\leq C\bigl(\|f\|_{L^2(\Om;d^nx)}+\|g\|_{(N^{1/2}(\partial\Omega))^*}\bigr), 
\end{equation}
for some finite constant $C=C(\Omega,z)>0$. Then $u:=v+w$ solves
\eqref{Yan-14} and satisfies \eqref{Hh.3X}. 

To prove uniqueness, we assume that $u$ satisfies the homogeneous 
version of \eqref{Yan-14}, and consider 
$w:=-z(-\Delta_{D,\Om}-zI_\Om)^{-1}u\in H^1_0(\Om)$. Then $v:=u-w$ satisfies 
\begin{eqnarray}\lb{Yan-1k} 
v\in L^2(\Omega;d^nx),\quad (-\Delta-z)v=0 \, \mbox{ in } \, \Omega,\quad
\widehat{\gamma}_D v =0 \, \mbox{ on } \, \partial\Omega.
\end{eqnarray}
We claim that these conditions imply that $v=0$. To show this, 
consider an arbitrary $f\in L^2(\Omega;d^nx)$ and let $u_f$ be the unique
solution of 
\begin{eqnarray}\lb{Yan-1kk} 
u_f\in H^1_0(\Omega),\quad (-\Delta -\ol{z})u_f=f\, \mbox{ in }\, \Omega. 
\end{eqnarray}
Then $u_f\in \dom (- \Delta_{D,\Om})\subset H^2(\Omega)$ by 
Lemma \ref{Bjk}, so that $u_f\in H^2(\Omega)\cap H^1_0(\Omega)$. 
Then the integration by parts formula \eqref{Tan-C12} yields
\begin{align}\lb{Yan-13c} 
(f,v)_{L^2(\Om;d^nx)} &= ((-\Delta -\ol{z})u_f,v)_{L^2(\Om;d^nx)}
\nonumber\\
&= (u_f,(-\Delta -z)v)_{L^2(\Om;d^nx)}
-{}_{N^{1/2}(\partial\Omega)}\langle\gamma_N u_f,\widehat{\gamma}_D v         
\rangle_{(N^{1/2}(\partial\Omega))^*}
\nonumber\\
& =0. 
\end{align}
Since $f\in L^2(\Omega;d^nx)$ was arbitrary it follows that $v=0$ and
consequently, $u\in H^1_0(\Om)$. Given that $u$ solves 
the homogeneous version of \eqref{Yan-14} and the fact that 
$z\in\bbC\backslash\si(-\Delta_{N,\Om})$, we may therefore conclude that 
$u=0$ in $\Omega$. The regularity result \eqref{3.3Y} is a consequence 
of Theorem \ref{t3.3} and the uniqueness for \eqref{Yan-14}, 
while \eqref{3.3Ys} follows from Lemma \ref{Bjk} and the uniqueness for
\eqref{Yan-14}. Finally, \eqref{3.3Ybis} is a direct 
consequence of \eqref{3.3Ys}.

Next, consider \eqref{3.34Y}. On one hand, by Lemma \ref{Bjk} and 
the current assumptions, one concludes that 
\begin{equation}\lb{Hy-1} 
(-\Delta_{D,\Om}-\ol{z}I_\Om)^{-1}:L^2(\Om;d^nx)\to  
H^2(\Omega)\cap H^1_0(\Omega)
\end{equation}
is a well-defined linear, bounded operator. On the other hand, the operator  
$\ga_N$ in \eqref{Tan-C7} is also well-defined, linear, and bounded. 
Consequently, 
\begin{equation}\lb{Hy-1b} 
\ga_N (-\Delta_{D,\Om}-{\ol z}I_\Om)^{-1}\in 
\cB\big(L^2(\Om;d^nx),N^{1/2}(\partial\Omega)\big),   
\end{equation}
and \eqref{3.34Y} follows by dualizing in \eqref{Hy-1b}. 
We note that, with $\wti\gamma_N$ as in \eqref{2.8}, 
this also implies that 
\begin{equation}\lb{Hy-2} 
\big[\ga_N (-\Delta_{D,\Om}-{\ol z}I_\Om)^{-1}\big]^*
= \big[\wti\ga_N (-\Delta_{D,\Om}-{\ol z}I_\Om)^{-1}\big]^*\in 
\cB\big((N^{1/2}(\partial\Omega))^*,L^2(\Om;d^nx)\big).
\end{equation}
This will play a role shortly. 

Next, let $u_g$ be the (unique) solution of \eqref{Yan-14} with $f=0$ 
and $g\in\bigl(N^{1/2}(\partial\Omega)\big)^*$ arbitrary. 
In addition, set $u:=\big[\ga_N (-\Delta_{D,\Om}-\ol{z}I_\Om)^{-1}\big]^*g\in 
L^2(\Om;d^nx)$. Our goal is to show that $u=u_g$. 
Since $H^1(\partial\Omega)$ is densely embedded
into $\bigl(N^{1/2}(\partial\Omega)\big)^*$ by \eqref{Nyyy}, it follows that there exists 
a sequence $\{g_j\}_{j\in{\mathbb{N}}}$ such that $g_j\in H^1(\partial\Omega)$, 
$j\in{\mathbb{N}}$, and $g_j\to g$ in 
$\bigl(N^{1/2}(\partial\Omega)\big)^*$ as $j\to\infty$. 
For each fixed $j$, we know from Theorem \ref{t3.3} that
$u_{g_j}=-\big[\wti\ga_N (-\Delta_{D,\Om}-\ol{z}I_\Om)^{-1}\big]^*g_j$.
Due to \eqref{Hy-2}, this implies that $u_{g_j}\to u$ in $L^2(\Om;d^nx)$. 
Hence, $\Delta u_{g_j}=-zu_{g_j}\to -zu$ in $L^2(\Om;d^nx)$. 
With these at hand, it is then clear that $(-\Delta-z)u=0$ in the sense 
of distributions in $\Omega$. Next, given that 
\begin{equation}\lb{Hy-3}
u_{g_j}\to u\,\mbox{ in }\,L^2(\Om;d^nx)\,\mbox{ as }\,j\to\infty,
\, \mbox{ and }\,  
\Delta u_{g_j}\to\Delta u\,\mbox{ in }\,L^2(\Om;d^nx)
\,\mbox{ as }\,j\to\infty,
\end{equation}
it follows from this and Theorem \ref{New-T-tr} that 
\begin{equation}\lb{Hy-4}
g_j=\widehat\gamma_D u_{g_j} \to \widehat\gamma_D u  
\, \mbox{ in }\, \bigl(N^{1/2}(\partial\Omega)\big)^*
\, \mbox{ as }\, j\to\infty. 
\end{equation}
However, since $g_j\to g$ in $\bigl(N^{1/2}(\partial\Omega)\big)^*$
as $j\to\infty$, one finally concludes that $\widehat\gamma_D u =g$. 
This proves that $u\in L^2(\Om;d^nx)$ satisfies
$(-\Delta -z) u=0$ in $\Omega$ and $\widehat\gamma_D u =g$ and hence  
by the uniqueness in \eqref{Yan-14} one obtains $u=u_g$. 

This concludes the proof of the fact that the unique solution of 
\eqref{Yan-14}, with data $f=0$ and $g\in\bigl(N^{1/2}(\partial\Omega)\big)^*$ 
arbitrary, is given by $-\big[\ga_N (-\Delta_{D,\Om}-\ol{z}I_\Om)^{-1}\big]^*g$.
Having established this, it is then clear that \eqref{3.35Y} solves
\eqref{Yan-14}, as originally stated. 
\end{proof}

\begin{corollary}\lb{New-CV22}
Assume Hypothesis \ref{h.Conv}. Then for every 
$z\in\bbC\backslash \si(-\Delta_{D,\Om})$, the map $\hatt \gamma_D$ in \eqref{Tan-Bq2} 
is an isomorphism $($i.e., bijective and bicontinuous\,$)$. 
\end{corollary}
\begin{proof}
This follows from Corollary \ref{New-CV2} and Theorem \ref{tH.A}. 
\end{proof}

\begin{theorem}\lb{tH.u}
Assume Hypothesis \ref{h.Conv}. Then the Neumann trace operator 
\eqref{3an-C10} is onto. More specifically, for every 
$z\in\bbC\backslash \si(-\Delta_{N,\Om})$ there exists $C=C(\Omega,z)>0$
with the property that for any functional 
$\Lambda\in\bigl(N^{3/2}(\partial\Omega)\bigr)^*$ there exists 
$u_\Lambda \in L^2(\Om;d^nx)$ with $(-\Delta -z)u_\Lambda =0$ in $\Omega$, 
such that 
\begin{eqnarray}\lb{Nh-G.a}
\mbox{the assignment $\bigl(N^{3/2}(\partial\Omega)\bigr)^*\ni 
\Lambda\mapsto u_\Lambda \in L^2(\Om;d^nx)$ is linear}, 
\end{eqnarray}
and 
\begin{eqnarray}\lb{Nh-1.a}
\widehat\gamma_N u_\Lambda = \Lambda\, \mbox{ and }\, 
\|u_\Lambda\|_{L^2(\Om;d^nx)}\leq C\|\Lambda\|_{(N^{3/2}(\partial\Omega))^*}.
\end{eqnarray}
\end{theorem}
\begin{proof}
Fix $z\in\bbC\backslash \si(-\Delta_{N,\Om})$. By Lemma \ref{Bjk}, the operator
$(-\Delta_{N,\Omega}-\ol{z}I_{\Om})^{-1}:L^2(\Om;d^nx)\to H^2(\Omega)$ 
is well-defined, linear, and bounded. 
Given $\Lambda\in\bigl(N^{3/2}(\partial\Omega)\bigr)^*$, we consider 
the bounded linear functional 
\begin{eqnarray}\lb{Nh-2.a}
\Lambda \gamma_D  (-\Delta_{N,\Omega}-\ol{z}I_{\Om})^{-1}:
L^2(\Om;d^nx)\to  {\mathbb{C}}.
\end{eqnarray}
Then, by the Riesz representation theorem, there exists a unique
$u_\Lambda\in L^2(\Om;d^nx)$ with the property that 
\begin{eqnarray}\lb{Nh-3.a}
\Lambda\bigl(\gamma_D((-\Delta_{N,\Omega}-\ol{z}I_{\Om})^{-1}f)\bigr)
=\langle f,u_\Lambda\rangle_{L^2(\Om;d^nx)},
\end{eqnarray}
and such that 
\begin{align}\lb{Nh-4.a}
\|u_\Lambda\|_{L^2(\Om;d^nx)} & \leq 
C\|\Lambda \gamma_D  (-\Delta_{N,\Omega}-\ol{z}I_{\Om})^{-1}\|
_{\cB(L^2(\Om;d^nx),{\mathbb{C}})}  \no \\
& \leq C\|\Lambda\|_{(N^{3/2}(\partial\Omega))^*},
\end{align}
for some constant $C=C(\Om,z)>0$, independent of $\Lambda$. 
Furthermore, by the uniqueness of $u_\Lambda$, the assignment 
$\bigl(N^{3/2}(\partial\Omega)\bigr)^*\ni\Lambda\mapsto 
u_\Lambda \in L^2(\Om;d^nx)$ is linear.  

Specializing \eqref{Nh-3.a} to the case where $f:=(-\Delta-\ol{z})w$, 
with $w\in H^2(\Omega)$ satisfying $\gamma_N w =0$, then yields 
\begin{equation}\lb{Nh-5.a}
\Lambda\big(\gamma_D w \big)
= ((-\Delta-\ol{z})w,u_\Lambda)_{L^2(\Om;d^nx)}
\, \mbox{ for all }\,w\in H^2(\Omega) \, \mbox{ with } \, \gamma_N w = 0.
\end{equation}
In particular, choosing $w\in C^\infty_0(\Omega)$ yields
\begin{eqnarray}\lb{Nh-6.a}
0=\Lambda(0)=\Lambda\bigl(\gamma_D w \bigr)
= ((-\Delta-\ol{z})w,u_\Lambda)_{L^2(\Om;d^nx)},
\quad w\in C^\infty_0(\Omega). 
\end{eqnarray}
Thus, one obtains $(-\Delta-z)u_\Lambda=0$ in 
$\Omega$, in the sense of distributions. Inserting this into \eqref{Nh-3.a} 
one then obtains  
\begin{align}\lb{Nh-7.a}
\Lambda\bigl(\gamma_D w \bigr) &=
((-\Delta -\ol{z})w,u_\Lambda)_{L^2(\Om;d^nx)}
- (w,(-\Delta -z)u_\Lambda)_{L^2(\Om;d^nx)}
\nonumber\\
&= 
-{}_{N^{3/2}(\partial\Omega)}\langle\gamma_D w,\widehat{\gamma}_N u_\Lambda 
\rangle_{(N^{3/2}(\partial\Omega))^*},
\end{align}
for every $w$ in $\{w\in H^2(\Omega) \,|\, \gamma_N w =0\}$. 
Since $\gamma_D$ maps the latter space onto $N^{3/2}(\dOm)$, 
\eqref{Nh-7.a} implies that $\widehat{\gamma}_N u_\Lambda = - \Lambda$. 
This shows that the operator \eqref{3an-C10} has a linear, 
bounded right-inverse. In particular, it is onto.  
\end{proof}

\begin{corollary}\lb{New-CV3}
Assume Hypothesis \ref{h.Conv}. Then for every $z \in \bbC\backslash \sigma(-\Delta_{N,\Om})$, the map 
\begin{eqnarray}\lb{Tan-FFF} 
\widehat{\gamma}_N: 
\big\{u\in L^2(\Omega;d^nx) \,\big|\, (- \Delta - z) u=0 \mbox{ in }\Omega\big\}
\to \bigl(N^{3/2}(\partial\Omega)\bigr)^*
\end{eqnarray}
is onto. 
\end{corollary}
\begin{proof}
This is a direct consequence of Theorem \ref{tH.u}.
\end{proof}

\begin{theorem}\lb{tH.G2}
Assume Hypothesis \ref{h.Conv} and suppose that 
$z\in\bbC\backslash\si(-\Delta_{N,\Om})$. Then for any   
$f\in L^2(\Om;d^nx)$ and $g\in (N^{3/2}(\partial\Omega))^*$ the 
following inhomogeneous Neumann boundary value problem 
\begin{equation}\lb{n-1H}
\begin{cases}
(-\Delta-z)u=f\text{ in }\,\Om,
\\[4pt]
u\in L^2(\Om;d^nx),
\\[4pt]
\widehat\ga_N u =g\text{ on }\,\dOm,
\end{cases}
\end{equation}
has a unique solution $u=u_N$. This solution satisfies 
\begin{equation}\lb{Hh.3f} 
\|u_N\|_{L^2(\Om;d^nx)}
+\|\widehat\ga_D u_N\|_{(N^{1/2}(\partial\Omega))^*}\leq C_N
\big(\|f\|_{L^2(\Om;d^nx)}+\|g\|_{(N^{3/2}(\partial\Omega))^*}\big)
\end{equation}
for some constant $C_N=C_N(\Omega,z)>0$, and the following regularity 
results hold:  
\begin{eqnarray}\lb{3.3f}
&& g\in L^2(\partial\Omega;d^{n-1}\omega) \, \text{ implies } \,  u_N\in H^{3/2}(\Omega),
\\[4pt]
&& g\in\gamma_N\bigl(H^2(\Om)\bigr) \, \text{ implies } \,  u_N\in H^2(\Omega). 
\lb{3.3fbis}
\end{eqnarray} 
Natural estimates are valid in each case. 

Moreover, 
\begin{equation}\lb{3.34f} 
\big[\ga_D (-\Delta_{N,\Om}-{\ol z}I_\Om)^{-1}\big]^* \in 
\cB\big((N^{3/2}(\partial\Omega))^*, L^2(\Om;d^nx)\big),   
\end{equation}
and the solution of \eqref{n-1H} is given by the formula
\begin{equation}\lb{3.a5Y}
u_N=(-\Delta_{N,\Om}-zI_\Om)^{-1}f
+\big[\ga_D (-\Delta_{N,\Om}-\ol{z}I_\Om)^{-1}\big]^*g. 
\end{equation}
\end{theorem}
\begin{proof}
By Theorem \ref{tH.u} there exists a function $v\in L^2(\Om;d^nx)$ satisfying 
$\Delta v=0$ in $\Om$, $\widehat\gamma_N v =g$ 
and $\|v\|_{L^2(\Om;d^nx)}\leq C\|g\|_{(N^{3/2}(\partial\Omega))^*}$ for 
some finite $C=C(\Omega)>0$. Setting 
\begin{equation}\lb{Jh-1f}
w:=(-\Delta_{N,\Omega}-zI_{\Om})^{-1}(f+zv)
\end{equation}
so that 
\begin{equation}\lb{Jh-2a}
w\in \dom (- \Delta_{N,\Omega})\hookrightarrow H^2(\Om), 
\end{equation}
by our hypotheses and by Lemma \ref{Bjk}, and  
\begin{equation}\lb{Jh-3f}
\|w\|_{H^2(\Om)}\leq C\|f+zv\|_{L^2(\Om,d^nx)}
\leq C\bigl(\|f\|_{L^2(\Om;d^nx)}+\|g\|_{(N^{3/2}(\partial\Omega))^*}\bigr), 
\end{equation}
for some finite constant $C=C(\Omega,z)>0$. Then $u:=v+w$ solves
\eqref{n-1H} and satisfies \eqref{Hh.3f}. 

To prove uniqueness, we assume that $u$ satisfies the homogeneous 
version of \eqref{n-1H}, and consider 
$w:=-z(-\Delta_{N,\Om}-zI_\Om)^{-1}u\in H^1_0(\Om)$. 
Then $v:=u-w$ satisfies 
\begin{eqnarray}\lb{Yan-11f} 
v\in L^2(\Omega;d^nx),\quad (-\Delta -z)v=0 \, \mbox{ in } \, \Omega,\quad
\widehat{\gamma}_N v =0 \, \mbox{ on } \, \partial\Omega.
\end{eqnarray}
We claim that these conditions imply that $v=0$. 
To show this, consider an arbitrary $f\in L^2(\Omega;d^nx)$ 
and let $u_f$ be the unique solution of 
\begin{eqnarray}\lb{Yan-12f} 
u_f\in H^1(\Omega),\quad\wti\gamma_N w =0\, \mbox{ on }\, \partial\Omega,
\quad (-\Delta -z)u_f=f\, \mbox{ in }\, \Omega. 
\end{eqnarray}
Then $u_f\in \dom (- \Delta_{N,\Om})\subset H^2(\Omega)$ by 
Lemma \ref{Bjk}, so that $u_f\in H^2(\Omega)$ and $\gamma_N w =0$. 
The integration by parts formula \eqref{3an-C12} then yields
\begin{align}\lb{Yan-13f} 
(f,v)_{L^2(\Om;d^nx)} &= ((-\Delta -z)u_f,v)_{L^2(\Om;d^nx)}
\nonumber\\
&= (u_f,(-\Delta -z)v)_{L^2(\Om;d^nx)}
+{}_{N^{3/2}(\partial\Omega)}\langle\gamma_D u_f,\widehat{\gamma}_N v           
\rangle_{(N^{3/2}(\partial\Omega))^*}
\nonumber\\ 
& =0. 
\end{align}
Since $f\in L^2(\Omega;d^nx)$ was arbitrary, it follows that $v=0$ and 
consequently, $u\in H^1(\Om)$. Given that $u$ solves 
the homogeneous version of \eqref{n-1H} and the fact that 
$z\in\bbC\backslash\si(-\Delta_{N,\Om})$, we may therefore conclude that 
$u=0$ in $\Omega$. The regularity result \eqref{3.3f} is a consequence 
of Theorem \ref{t3.XV} and the uniqueness for \eqref{n-1H}. 
Finally, \eqref{3.3fbis} follows from \eqref{3.3f} and Lemma \ref{Bjk}. 

Next, consider\eqref{3.34f}. On one hand, by Lemma \ref{Bjk} and 
the current assumptions, one concludes that 
\begin{equation}\lb{Hy-1r} 
(-\Delta_{N,\Om}-\ol{z}I_\Om)^{-1}:L^2(\Om;d^nx)\to  
\big\{w\in H^2(\Omega) \,\big|\, \gamma_N w =0\big\} 
\end{equation}
is a well-defined, linear, bounded operator. On the other hand, the operator  
$\gamma_D$ in \eqref{3an-C7} is also well-defined, linear, and bounded. 
Consequently, 
\begin{equation}\lb{Hy-1br} 
\gamma_D (-\Delta_{N,\Om}-{\ol z}I_\Om)^{-1}\in 
\cB\big(L^2(\Om;d^nx),N^{3/2}(\partial\Omega)\big),   
\end{equation}
and \eqref{3.34f} follows dualizing in \eqref{Hy-1br}. 
We note that this and Lemma \ref{3o-Tx} also imply that 
\begin{equation}\lb{Hy-2r} 
\big[\ga_D (-\Delta_{N,\Om}-{\ol z}I_\Om)^{-1}\big]^*
= \big[\widehat\ga_D (-\Delta_{D,\Om}-{\ol z}I_\Om)^{-1}\big]^*\in 
\cB\big((N^{3/2}(\partial\Omega))^*,L^2(\Om;d^nx)\big).
\end{equation}
This will play a role shortly. 

Next, let $u_g$ be the (unique) solution of \eqref{n-1H} with $f=0$ 
and $g\in\bigl(N^{3/2}(\partial\Omega)\big)^*$ arbitrary. 
In addition, set $u:=\big[\gamma_D (-\Delta_{N,\Om}-\ol{z}I_\Om)^{-1}\big]^*g\in 
L^2(\Om;d^nx)$. Our goal is to show that $u=u_g$. 
To this end, we note that \eqref{Nh-1B.q} implies the existence of a 
sequence $\{g_j\}_{j\in{\mathbb{N}}}$ such that 
$g_j\in L^2(\partial\Omega;d^{n-1}\omega)$, $j\in{\mathbb{N}}$, 
and $g_j\to g$ in $\bigl(N^{3/2}(\partial\Omega)\big)^*$ as $j\to\infty$. 
For each fixed $j$, we know from Theorem \ref{t3.XV} that
$u_{g_j}=\big[\gamma_D(-\Delta_{N,\Om}-\ol{z}I_\Om)^{-1}\big]^*g_j$. 
By \eqref{Hy-2r}, this implies that $u_{g_j}\to u$ in $L^2(\Om;d^nx)$. 
Hence, $\Delta u_{g_j}=-zu_{g_j}\to -zu$ in $L^2(\Om;d^nx)$. 
With these at hand, it is then clear that $(-\Delta-z)u=0$ in the sense 
of distributions in $\Omega$. Next, given that 
\begin{equation}\lb{Hy-3r}
u_{g_j}\to u\,\mbox{ in }\,L^2(\Om;d^nx)\,\mbox{ as }\,j\to\infty,
\, \mbox{ and }\,  
\Delta u_{g_j}\to\Delta u\,\mbox{ in }\,L^2(\Om;d^nx)
\,\mbox{ as }\, j\to\infty,
\end{equation}
it follows from this and Theorem \ref{3ew-T-tr} that 
\begin{equation}\lb{Hy-4r}
g_j=\widehat\gamma_N u_{g_j} \to \widehat\gamma_N u  
\, \mbox{ in }\, \bigl(N^{3/2}(\partial\Omega)\big)^*
\, \mbox{ as }\, j\to\infty. 
\end{equation}
However, since $g_j\to g$ in $\bigl(N^{3/2}(\partial\Omega)\big)^*$
as $j\to\infty$, one finally concludes that $\widehat\gamma_N u =g$. 
This proves that $u\in L^2(\Om;d^nx)$ satisfies
$(-\Delta -z) u=0$ in $\Omega$ and $\widehat\gamma_N u =g$ hence, 
by the uniqueness in \eqref{n-1H}, one obtains $u=u_g$. 

This concludes the proof of the fact that the unique solution 
of \eqref{n-1H}, with $f=0$ and $g\in\bigl(N^{3/2}(\partial\Omega)\big)^*$ 
arbitrary, is given by $\big[\gamma_D(-\Delta_{N,\Om}-\ol{z}I_\Om)^{-1}\big]^*g$.
Having established this, it is then clear that \eqref{3.a5Y} solves
\eqref{n-1H}, as originally stated. 
\end{proof}

\begin{corollary}\lb{New-CV33}
Assume Hypothesis \ref{h.Conv}. Then for any $z\in \bbC\backslash \sigma(- \Delta_{N,\Om})$, 
the map $\hatt \gamma_N$ in \eqref{Tan-FFF} is an isomorphism $($i.e., bijective and 
bicontinuous\,$)$. 
\end{corollary}
\begin{proof}
This is a consequence of Corollary \ref{New-CV33} and Theorem \ref{tH.G2}.
\end{proof}

\section{Dirichlet-to-Neumann Operators on Quasi-Convex Domains}
\lb{s11}

In this section we discuss spectral parameter dependent Dirichlet-to-Neumann maps, also 
known in the literature as Weyl--Titchmarsh operators (in particular, they can be viewed 
as extensions of the Poincar\'e--Steklov operator). 

Assuming Hypothesis \ref{h.Conv}, 
we introduce the Dirichlet-to-Neumann map $M_{D,N,\Om}^{(0)}(z)$ associated with 
$(-\Delta-z)$ on $\Om$, as follows,
\begin{align}\lb{3.44v}
M_{D,N,\Om}^{(0)}(z) \colon
\begin{cases}
\bigl(N^{1/2}(\dOm)\bigr)^*\to \bigl(N^{3/2}(\dOm)\bigr)^*  \\
\hspace*{1.78cm} 
f \mapsto - \widehat\ga_N u_D,
\end{cases}  
\quad z\in\bbC\backslash\si(-\Delta_{D,\Om}), 
\end{align}
where $u_D$ is the unique solution of
\begin{align}\lb{3.45v}
(-\Delta-z)u = 0 \,\text{ in }\Om, \quad u \in
L^2(\Om;d^nx), \quad \widehat\ga_D u = f \,\text{ on } \, \dOm.   
\end{align}
Assuming Hypothesis \ref{h.Conv}, we introduce the 
Neumann-to-Dirichlet map $M_{N,D,\Om}^{(0)}(z)$,  associated with 
$(-\Delta-z)$ on $\Om$, as follows,
\begin{align}\lb{3.48v}
M_{N,D,\Om}^{(0)}(z) \colon 
\begin{cases}
\bigl(N^{3/2}(\dOm)\bigr)^*\to \bigl(N^{1/2}(\dOm)\bigr)^*,\\
\hspace*{1.8cm} 
g \mapsto \widehat\ga_D u_N, 
\end{cases}  
\quad z\in\bbC\backslash\si(-\Delta_{N,\Om}), 
\end{align}
where $u_{N}$ is the unique solution of
\begin{align}\lb{3.49v}
(-\Delta-z)u = 0 \,\text{ in }\Om, \quad u \in
L^2(\Om;d^nx), \quad \widehat\ga_Nu = g \, \text{ on } \, \dOm.  
\end{align}
The following result is the natural counterpart of Theorem \ref{t3.5} in 
the setting just introduced: 

\begin{theorem}\lb{t3.5v} 
Assume Hypothesis \ref{h.Conv}. Then, with the above notation, 
\begin{equation}\lb{3.46v}
M_{D,N,\Om}^{(0)}(z) \in \cB\big((N^{1/2}(\dOm))^*,  (N^{3/2}(\dOm))^*\big), 
\quad z\in\bbC\backslash\si(-\Delta_{D,\Om}),   
\end{equation}
its action is compatible with that of $M_{D,N,\Om}^{(0)}(z)$ introduced 
in \eqref{3.44} $($thus, justifying retaining the same notation$)$, and 
\begin{equation}\lb{3.47v}
M_{D,N,\Om}^{(0)}(z)=\widehat\gamma_N
\big[\gamma_N(-\Delta_{D,\Om} - \ol{z}I_\Om)^{-1}\big]^*, 
\quad z\in\bbC\backslash\si(-\Delta_{D,\Om}). 
\end{equation}
Similarly, 
\begin{equation}\lb{3.50v}
M_{N,D,\Om}^{(0)}(z)\in\cB\big((N^{3/2}(\dOm))^*,  (N^{1/2}(\dOm))^*\big), 
\quad z\in\bbC\backslash\si(-\Delta_{N,\Om}),
\end{equation}
its action is compatible with that of $M_{N,D,\Om}^{(0)}(z)$ introduced 
in \eqref{3.48} $($once again justifying retaining the same notation$)$, and 
\begin{equation}\lb{3.52v}
M_{N,D,\Om}^{(0)}(z) = \widehat \gamma_D\big[\gamma_D
(-\Delta_{N,\Om} - \ol{z}I_\Om)^{-1}\big]^*, \quad
z\in\bbC\backslash\si(-\Delta_{N,\Om}). 
\end{equation} 
Moreover, 
\begin{equation}\lb{3.53v}  
M_{N,D,\Om}^{(0)}(z)=- M_{D,N,\Om}^{(0)}(z)^{-1},\quad 
z\in\bbC\backslash(\si(-\Delta_{D,\Om})
\cup\si(-\Delta_{N,\Om})), 
\end{equation}
and 
\begin{eqnarray}\lb{NaLa}
\begin{array}{l}
\big[M_{D,N,\Om}^{(0)}(z)\big]^* f= M_{D,N,\Om}^{(0)}(\ol{z})f,\qquad
f\in N^{3/2}(\dOm)\hookrightarrow\bigl(N^{1/2}(\dOm)\bigr)^*,
\\[4pt]
\big[M_{N,D,\Om}^{(0)}(z)\big]^* f= M_{N,D,\Om}^{(0)}(\ol{z})f,\qquad
f\in N^{1/2}(\dOm)\hookrightarrow\bigl(N^{3/2}(\dOm)\bigr)^*.
\end{array}
\end{eqnarray}
As a consequence, one also has
\begin{eqnarray}\lb{3.TTa}
&& M_{D,N,\Om}^{(0)}(z) \in \cB\big(N^{3/2}(\dOm),  N^{1/2}(\dOm)\big), 
\quad z\in\bbC\backslash\si(-\Delta_{D,\Om}),   
\\[4pt]
&& M_{N,D,\Om}^{(0)}(z)\in\cB\big(N^{1/2}(\dOm),  N^{3/2}(\dOm)\big), 
\quad z\in\bbC\backslash\si(-\Delta_{N,\Om}).
\lb{3.TTb}
\end{eqnarray}
Finally, for every $z \leq 0$ one has
\begin{equation}\lb{NDD-1}
{}_{N^{1/2}(\dOm)}\big\langle M^{(0)}_{D,N,\Om}(z)f,f\big\rangle_{(N^{1/2}(\dOm))^*}\leq 0,
\quad f\in N^{3/2}(\dOm)\hookrightarrow\bigl(N^{1/2}(\dOm)\bigr)^*,
\end{equation}
whereas for every $z < 0$ 
\begin{equation}\lb{NDD-2}
{}_{N^{3/2}(\dOm)}\big\langle M^{(0)}_{N,D,\Om}(z)f,f\big\rangle_{(N^{3/2}(\dOm))^*}\geq 0,
\quad f\in N^{1/2}(\dOm)\hookrightarrow\bigl(N^{3/2}(\dOm)\bigr)^*.
\end{equation}
\end{theorem}
\begin{proof}
The membership in \eqref{3.46v} is a consequence of Theorem \ref{tH.A}, 
whereas \eqref{3.47v} follows from \eqref{3.35Y}. In a similar fashion,  
the membership in \eqref{3.50v} is a consequence of Theorem \ref{tH.G2}, 
while \eqref{3.47v} follows from \eqref{3.a5Y}. In addition, the fact that 
the operators in \eqref{3.44v}, \eqref{3.48v} act in a compatible fashion with 
their counterparts from \eqref{3.44}, \eqref{3.48}, is implied by the 
regularity statements \eqref{3.3Y} and \eqref{3.3f}.  
Going further, formula \eqref{3.53v} follows from this compatibility result, 
\eqref{3.53}, \eqref{3.46v}, \eqref{3.50v}, and the density results 
in \eqref{Nyyy} and \eqref{Nh-1B.q}.
The duality formula \eqref{NaLa} is a consequence of Lemma 4.12 in 
\cite{GM09} (considered with $\Theta=0$) and a density argument, as 
before. Next, \eqref{3.TTa}, \eqref{3.TTb} follow from duality, \eqref{NaLa}, 
\eqref{3.46v}, \eqref{3.50v}, as well as Lemmas \ref{L-refN}, 
\ref{L-refNN} and Lemma~\ref{3U-x}. 

As far as \eqref{NDD-1} is concerned, the density and compatibility 
results mentioned above show that it suffices to prove that 
\begin{eqnarray}\lb{NDD-4}
\big\langle f,M^{(0)}_{D,N,\Om}(z)f\big\rangle_{1/2}\leq 0,\quad f\in H^{1/2}(\dOm),
\end{eqnarray}
given that $z \leq 0$. 
However, Green's formula shows that, if $f\in H^{1/2}(\dOm)$ and 
$u\in H^1(\Om)$ satisfies $(-\Delta-z)u=0$ in $\Omega$ and $\gamma_D u =f$ 
on $\dOm$ then 
\begin{align}\lb{NDD-5}
\big\langle f,M^{(0)}_{D,N,\Om}(z)f\big\rangle_{1/2} & =
-\langle\gamma_D u ,\wti\gamma_N u \rangle_{1/2}  \no \\
& =-\|\nabla u\|^2_{(L^2(\Om;d^nx))^n}+z\|u\|^2_{L^2(\Om;d^nx)}\leq 0. 
\end{align}
Finally, \eqref{NDD-2} is proved analogously. 
\end{proof}

\section{The Regularized Neumann Trace Operator 
on Quasi-Convex Domains}
\lb{s12}

In this section we introduce a regularized version of the Neumann trace operator
(cf.\ \eqref{3an-C11}) on quasi-convex domains, and study some of its 
basic properties (such as a useful variant of Green's formula; 
cf.\ \eqref{T-Green} below). 

\begin{theorem}\lb{LL.w} 
Assume Hypothesis \ref{h.Conv}. Then, for every 
$z\in\bbC\backslash\si(-\Delta_{D,\Om})$, the map 
\begin{eqnarray}\lb{3.Aw1}
\tau^N_z:\bigl\{u\in L^2(\Om;d^nx)\,\big|\,\Delta u\in L^2(\Om;d^nx)\bigr\}
\to  N^{1/2}(\partial\Omega)
\end{eqnarray}
given by 
\begin{eqnarray}\lb{3.Aw2}
\tau^N_z u :=\widehat\gamma_N u 
+M_{D,N,\Om}^{(0)}(z)\bigl(\widehat\gamma_D u \bigr), 
\quad u\in L^2(\Om;d^nx), \; \Delta u\in L^2(\Om;d^nx),
\end{eqnarray}
is well-defined, linear, and bounded when the space 
\begin{equation} 
\big\{u\in L^2(\Om;d^nx) \,\big|\, \Delta u\in L^2(\Om;d^nx)\big\} 
= \dom(-\Delta_{max})   \lb{12.dommax}
\end{equation} 
is endowed with the 
natural graph norm $u\mapsto\|u\|_{L^2(\Om;d^nx)}+\|\Delta u\|_{L^2(\Om;d^nx)}$.
Moreover, this operator satisfies the following additional properties: 

\begin{enumerate}
\item[$(i)$] For each $z\in\bbC\backslash\si(-\Delta_{D,\Om})$, the map $\tau^N_z$ in \eqref{3.Aw1}, \eqref{3.Aw2} is onto
$($i.e., $\tau^N_z (\dom (- \Delta_{max})\bigr)=N^{1/2}(\partial\Omega)$$)$. In fact, 
\begin{eqnarray}\lb{3.ON}
\tau^N_z\bigl(H^2(\Om)\cap H^1_0(\Om)\bigr)=N^{1/2}(\partial\Omega), \quad 
z\in\bbC\backslash\si(-\Delta_{D,\Om}).
\end{eqnarray}

\item[$(ii)$] One has 

\begin{eqnarray}\lb{3.Aw9}
\tau^N_z=\gamma_N(-\Delta_{D,\Om}-zI_{\Om})^{-1}(-\Delta-z),\quad
z\in\bbC\backslash \si(-\Delta_{D,\Om}).
\end{eqnarray} 

\item[$(iii)$] For each $z\in\bbC\backslash\si(-\Delta_{D,\Om})$, the 
kernel of the map $\tau^N_z$ in \eqref{3.Aw1}, \eqref{3.Aw2} is given by 
\begin{eqnarray}\lb{3.AKe}
\ker \big(\tau^N_z\big)=H^2_0(\Omega)\dotplus   \big\{u\in L^2(\Om;d^nx) \,\big|\, 
(-\Delta -z)u=0\,\mbox{ in }\,\Omega\big\}.
\end{eqnarray}
In particular, if $z\in\bbC\backslash\si(-\Delta_{D,\Om})$, then 
\begin{eqnarray}\lb{Sim-Gr} 
\tau^N_z u  =0\, \mbox{ for every }\,u\in \ker (-\Delta_{max}-zI_{\Om}).
\end{eqnarray}

\item[$(iv)$] For each $z\in\bbC\backslash \si(-\Delta_{D,\Om})$, the following 
Green formula holds for every $u,v\in\dom (- \Delta_{max})$, 
\begin{align}\lb{T-Green} 
& 
((-\Delta -z)u,  v)_{L^2(\Omega;d^nx)}
- (u,(-\Delta-\ol{z})v)_{L^2(\Omega;d^nx)}
\nonumber\\
& \quad
=-{}_{N^{1/2}(\partial\Omega)}\big\langle\tau^N_z u  ,\widehat\gamma_D v            
\big\rangle_{(N^{1/2}(\partial\Omega))^*}
+\,\ol{{}_{N^{1/2}(\partial\Omega)}\big\langle\tau^N_{\ol{z}} v,\widehat{\gamma}_D u 
\big\rangle_{(N^{1/2}(\partial\Omega))^*}}.
\end{align}
In particular, for every $u,v\in\dom (- \Delta_{max})$,  
$z\in\bbC$, and $z_0\in\bbR\backslash \si(-\Delta_{D,\Om})$, one has 
\begin{align}\lb{T-GreenX} 
& 
((-\Delta -(z+z_0))u,  v)_{L^2(\Omega;d^nx)}
- (u,  (-\Delta-(\ol{z}+z_0))v)_{L^2(\Omega;d^nx)}
\nonumber\\
& \quad
=-{}_{N^{1/2}(\partial\Omega)}\big\langle\tau^N_{z_0} u,\widehat\gamma_D v 
\big\rangle_{(N^{1/2}(\partial\Omega))^*}
+\,\ol{{}_{N^{1/2}(\partial\Omega)}\big\langle\tau^N_{z_0} v,\widehat{\gamma}_D u 
\big\rangle_{(N^{1/2}(\partial\Omega))^*}}.
\end{align}
Moreover, as a consequence of \eqref{T-Green} and \eqref{Sim-Gr}, 
for every $z\in\bbC\backslash \si(-\Delta_{D,\Om})$ one infers that 
\begin{align}\lb{T-Green2} 
& u\in\dom (- \Delta_{max}) \, \mbox{ and } \, 
v\in\ker  (-\Delta_{max}-\ol{z}I_{\Om})
\nonumber\\
& \quad 
\text{imply } \, ((-\Delta-z)u,  v)_{L^2(\Omega;d^nx)}
=-{}_{N^{1/2}(\partial\Omega)}\big\langle\tau^N_z u,\widehat\gamma_D v 
\big\rangle_{(N^{1/2}(\partial\Omega))^*}.
\end{align}
\end{enumerate}
\end{theorem}
\begin{proof}
Let$z\in\bbC\backslash \si(-\Delta_{D,\Om})$.  
Consider an arbitrary $u\in L^2(\Om;d^nx)$ satisfying 
$\Delta u\in L^2(\Om;d^nx)$, and let $v$ solve 
\begin{eqnarray}\lb{3.Aw3}
(-\Delta -z)v=0\,\mbox{ in }\,\Omega,\quad v\in L^2(\Om;d^nx),\quad 
\widehat\gamma_D v =\widehat\gamma_D u \,\mbox{ on } \,\partial\Omega. 
\end{eqnarray}
Theorems \ref{New-T-tr} and \ref{tH.A} ensure that this is possible  
and guarantee the existence of a finite constant $C=C(\Om,z)>0$ 
for which 
\begin{eqnarray}\lb{3.Aw4}
\|v\|_{L^2(\Om;d^nx)}\leq C(\|u\|_{L^2(\Om;d^nx)}+\|\Delta u\|_{L^2(\Om;d^nx)}).
\end{eqnarray}
Then $w:=u-v$ satisfies
\begin{equation}\lb{3.Aw5}
(-\Delta -z)w=(-\Delta -z)u\in L^2(\Omega;d^nx),\quad w\in L^2(\Om;d^nx), \quad   
\widehat\gamma_D w =0\, \mbox{ on } \,\partial\Omega, 
\end{equation}
so that, by \eqref{3.3Ybis}, there exists $C=C(\Om,z)>0$ such that 
\begin{eqnarray}\lb{3.Aw6}
w\in H^2(\Omega)\cap H^1_0(\Omega)\, \mbox{ and }\,  
\|w\|_{H^2(\Om)}\leq C(\|u\|_{L^2(\Om;d^nx)}+\|\Delta u\|_{L^2(\Om;d^nx)}).
\end{eqnarray}
Since $M_{D,N,\Om}^{(0)}(z)\bigl(\widehat\gamma_D u \bigr)
=-\widehat\gamma_N v        $, it follows from \eqref{3.Aw6}, the compatibility part 
of Theorem \ref{3ew-T-tr}, and Lemma \ref{Lo-Tx} that 
\begin{eqnarray}\lb{3.Aw7}
\tau^N_z u =\widehat\gamma_N u - \widehat\gamma_N v 
=\widehat\gamma_N (u-v) = \widehat\gamma_N w 
=\gamma_N w \in N^{1/2}(\partial\Omega). 
\end{eqnarray}
Moreover, 
\begin{align}\lb{3.Aw8}
\|\tau^N_z u \|_{N^{1/2}(\partial\Omega)}
&= \|\gamma_N w \|_{N^{1/2}(\partial\Omega)}\leq C\|w\|_{H^2(\Om)}
\nonumber\\
& \leq  C(\|u\|_{L^2(\Om;d^nx)}+\|\Delta u\|_{L^2(\Om;d^nx)}).
\end{align}
This shows that the operator $\tau^N_z$ in \eqref{3.Aw1}, \eqref{3.Aw2} is 
indeed well-defined and bounded. 

Incidentally, the above argument also shows that \eqref{3.Aw9} holds. 
That $\tau^N_z$ defined in \eqref{3.Aw1}, \eqref{3.Aw2} is onto, is a direct 
consequence of the fact that $\gamma_N$ in \eqref{Tan-C7} is onto 
(cf.\ Lemma \ref{Lo-Tx}). This also justifies \eqref{3.ON}.

Next, observe that due to Theorem \ref{tH.A} the sum in \eqref{3.AKe} 
is direct. Moreover, $H^2_0(\Omega)\subseteq \ker \big(\tau^N_z\big)$ by 
\eqref{3.Aw2} and $\big\{u\in L^2(\Om;d^n x) \,\big|\, 
(-\Delta -z)u=0\,\mbox{in}\,\Omega\big\}\subseteq \ker \big(\tau^N_z\big)$ 
by \eqref{3.Aw9}. This proves the right-to-left inclusion in \eqref{3.AKe}. 
To prove the opposite one, consider $u\in L^2(\Om;d^nx)$ with 
$\Delta u\in L^2(\Om;d^nx)$ for which $\tau^N_z u = 0$. 
If we now set $w:=(-\Delta_{D,\Om}-zI_{\Om})^{-1}(-\Delta-z)u$, then 
$w\in H^2(\Omega)\cap H^1_0(\Omega)$ and $\gamma_N w =\tau^N_z u =0$, 
by \eqref{3.Aw9}. Hence, $w\in  H^2_0(\Omega)$ by Theorem \ref{T-DD1}. 
Since $u=w+(u-w)$ and $(u-w)\in L^2(\Om;d^nx)$ satisfies $(-\Delta-z)(u-w)=0$ 
in $L^2(\Om;d^nx)$, the proof of \eqref{3.AKe} is complete. 

Next, consider the Green formulas in $(iv)$. Fix 
$z\in\bbR\backslash \si(-\Delta_{D,\Om})$, let $u,v\in\dom (- \Delta_{min})$,  
and set $\wti{u}:=(-\Delta_{D,\Om}-zI_{\Om})^{-1}((-\Delta -z)u)$,  
$\wti{v}:=(-\Delta_{D,\Om}-\ol{z}I_{\Om})^{-1}((-\Delta -\ol{z})v)$. Then 
\begin{align}\lb{G-G1} 
& \wti{u},\wti{v}\in H^2(\Omega)\cap H^1_0(\Omega),\quad 
(-\Delta-z)\wti{u}=(-\Delta -z)u,\quad 
(-\Delta-\ol{z})\wti{v}=(-\Delta -\ol{z})v, \no \\
& \quad \gamma_N \wti{u} = \tau^N_z u, \quad 
\gamma_N \wti{v} = \tau^N_{\ol{z}} v,    
\end{align}
by \eqref{3.Aw9}. Based on these observations and repeated applications 
of Green's formula \eqref{Tan-C12} one can then write
\begin{align}\lb{T-Gr1} 
& ((-\Delta -z)u,  v)_{L^2(\Omega;d^nx)}
- (u,  (-\Delta -\ol{z})v)_{L^2(\Omega;d^nx)}
\nonumber\\ 
& \quad = ((-\Delta-z)\wti{u},  v)_{L^2(\Omega;d^nx)}
- (u,  (-\Delta-\ol{z})\wti{v})_{L^2(\Omega;d^nx)}
\nonumber\\ 
& \quad = (\wti{u},  (-\Delta -\ol{z})v)_{L^2(\Omega;d^nx)}
- (u,  (-\Delta-\ol{z})\wti{v})_{L^2(\Omega;d^nx)}
\nonumber\\ 
& \qquad -{}_{N^{1/2}(\partial\Omega)}\langle\gamma_N \wti{u},\widehat{\gamma}_D v
\rangle_{(N^{1/2}(\partial\Omega))^*}
\nonumber\\ 
& \quad = (\wti{u}-u,  (-\Delta-\ol{z})\wti{v})_{L^2(\Omega;d^nx)}
-{}_{N^{1/2}(\partial\Omega)}\langle\tau^N_z u,\widehat{\gamma}_D v 
\rangle_{(N^{1/2}(\partial\Omega))^*}
\nonumber\\ 
& \quad = \ol{{}_{N^{1/2}(\partial\Omega)}\langle\tau^N_{\ol{z}} v,
\widehat{\gamma}_D u \rangle_{(N^{1/2}(\partial\Omega))^*}}
-{}_{N^{1/2}(\partial\Omega)}\langle\tau^N_z u,\widehat{\gamma}_D v 
\rangle_{(N^{1/2}(\partial\Omega))^*},
\end{align}
where in the last step we have used the fact that 
$(-\Delta-z)(\wti{u}-u)=0$ and 
$\widehat\gamma_D(\wti{u}-u)=-\widehat\gamma_D u $.
This justifies \eqref{T-Green} and completes the proof of the theorem.
\end{proof}

Recalling \eqref{3U-y}, we also state the following consequence of 
Theorem \ref{LL.w}: 

\begin{corollary}\lb{Lc.w} 
Assume Hypothesis \ref{h.Conv}. Then, for every 
$z_1,z_2\in\bbC\backslash\si(-\Delta_{D,\Om})$, the operator 
\begin{equation}
\big[M_{D,N,\Om}^{(0)}(z_1)-M_{D,N,\Om}^{(0)}(z_2)\big]
\in \cB\big((N^{1/2}(\dOm))^*,  (N^{3/2}(\dOm))^*\big)
\end{equation} 
satisfies  
\begin{equation}\lb{3.46vX}
\big[M_{D,N,\Om}^{(0)}(z_1)-M_{D,N,\Om}^{(0)}(z_2)\big]
\in \cB\big((N^{1/2}(\dOm))^*,  N^{1/2}(\dOm)\big).
\end{equation}
\end{corollary}
\begin{proof}
From Theorem \ref{tH.G} we know that there exists a constant $C>0$ with the
property that for every $f\in (N^{1/2}(\dOm))^*$ one can find 
$u\in\dom (- \Delta_{max})$ such that 
\begin{equation}\lb{Gfa1}
\widehat\gamma_D u =f\, \mbox{ and }\, 
\|u\|_{L^2(\Om;d^nx)}+\|\Delta u\|_{L^2(\Om;d^nx)}\leq C\|f\|_{(N^{1/2}(\dOm))^*}.
\end{equation}
With the help of \eqref{3.Aw1} we may then write 
\begin{align}\lb{3.Aw1X}
\|[M_{D,N,\Om}^{(0)}(z_1)-M_{D,N,\Om}^{(0)}(z_2)]f\|_{N^{1/2}(\dOm)}
&= \big\|\tau^N_{z_1} u   - \tau^N_{z_2} u \big\|_{N^{1/2}(\dOm)} 
\nonumber\\ 
&\leq  C\bigl[\|u\|_{L^2(\Om;d^nx)}+\|\Delta u\|_{L^2(\Om;d^nx)}\bigr]
\nonumber\\ 
&\leq  C\|f\|_{(N^{1/2}(\dOm))^*}, 
\end{align}
proving \eqref{3.46vX}.
\end{proof}

\section{The Krein Laplacian on Quasi-Convex Domains}
\lb{s13}

We now discuss the Krein--von Neumann extension $-\Delta_{K,\Om}$ of  
Laplacian $-\Delta\big|_{C_0^\infty(\Om)}$ in $L^2(\Om; d^n x)$, with $\Om\subset\bbR^n$ 
a bounded Lipschitz domain satisfying Hypothesis \ref{h.Conv}. 

In this situation, equations \eqref{SK} and \eqref{Yan-10} yield the following:  
\begin{align}\lb{Kre-Def}
\dom (-\Delta_{K,\Om}) &= 
\dom (- \Delta_{min}) \dot + \ker  ((-\Delta_{min})^*)
\nonumber\\ 
&= \dom (- \Delta_{min}) \dot + \ker (- \Delta_{max})
\nonumber\\ 
&=  H^2_0(\Om)\dotplus   
\big\{u\in L^2(\Om;d^nx) \,\big|\, \Delta u=0\mbox{ in }\Omega\big\}.
\end{align}
Nonetheless, we shall adopt a different point of view which better elucidates
the nature of the boundary condition associated with the Krein Laplacian. 
Our construction was originally inspired by the discussion in \cite{AS80}.
In Example~5.3 of that paper, the authors consider the Laplacian with 
the boundary condition
\begin{eqnarray}\lb{B.A-1} 
\frac{\partial f}{\partial n}(x)=\frac{\partial H(f)}{\partial n}(x),
\quad x\in\partial\Omega,
\end{eqnarray}
where, given $f:\Omega\to\bbR$, $H(f)$ is the harmonic extension of 
$f|_{\partial\Omega}$ to $\Omega$, and $\partial/\partial n$ denotes the 
normal derivative. The algebraic manipulations in \cite{AS80} leading up to 
\eqref{B.A-1} are somewhat formal, and Alonso and Simon mention that
{\it ``it seems to us that the Krein extension of $-\Delta$, that is, 
$-\Delta$ with the boundary condition \eqref{B.A-1}, is a natural object
and therefore worthy of further study.''}

\begin{theorem}\lb{T-Kr} 
Assume Hypothesis \ref{h.Conv} and $z\in\bbC\backslash \si(-\Delta_{D,\Om})$. 
Then the operator $- \Delta_{K,\Om,z}$ in $L^2(\Omega;d^nx)$ given by 
\begin{align}\lb{A-zz.1} 
\begin{split}
& - \Delta_{K,\Om,z} u := (- \Delta - z)u, \\
& \; u\in \dom (- \Delta_{K,\Om,z}) :=\bigl\{v\in\dom (- \Delta_{max}) \,\big|\, 
\tau^N_z v =0\bigr\},
\end{split}
\end{align}
satisfies
\begin{eqnarray}\lb{A-zz.W} 
\big(- \Delta_{K,\Om,z}\big)^* = - \Delta_{K,\Om,\ol{z}}.
\end{eqnarray}
In particular, if $z\in\bbR\backslash \si(-\Delta_{D,\Om})$ then 
$- \Delta_{K,\Om,z}$ is self-adjoint. Moreover, if 
$z\leq 0$ then $-\Delta_{K,\Om,z} \geq 0$. 

In the following, $- \Delta_{K,\Om,z}$ will be referred to as 
the Krein Laplacian in $\Omega$ $($associated with $z$$)$. 
Hence, the Krein Laplacian $- \Delta_{K,\Om}:= - \Delta_{K,\Om,0}$ 
is a self-adjoint operator in $L^2(\Om;d^nx)$ which satisfies
\begin{eqnarray}\lb{A-zz.b} 
-\Delta_{K,\Om} \geq 0, \, \mbox{ and } \, 
- \Delta_{min}\subseteq - \Delta_{K,\Om}\subseteq -\Delta_{max}. 
\end{eqnarray}
Furthermore, $\ker (- \Delta_{K,\Om})
=\big\{u\in L^2(\Om;d^nx) \,\big|\, \Delta u=0\big\}$, 
\begin{eqnarray}\lb{spec-1}
\mbox{with the possible exception of the origin, 
$- \Delta_{K,\Om}$ has a discrete, real spectrum},
\end{eqnarray}
and for any nonnegative self-adjoint extension $\wti S$ of $-\Delta\big|_{C^\infty_0(\Om)}$ 
one has
\begin{eqnarray}\lb{Ok.1} 
-\Delta_{K,\Om}\leq \wti S \leq -\Delta_{D,\Om}.
\end{eqnarray}
\end{theorem}
\begin{proof}
It is clear that the operator \eqref{A-zz.1} is densely defined. In addition, 
due to \eqref{T-Green}, this operator  
satisfies $- \Delta_{K,\Om,\ol{z}}\subseteq(- \Delta_{K,\Om,z})^*$.
Next, consider $w\in \dom (- \Delta_{K,\Om,z})^*$. Then $w\in L^2(\Om;d^nx)$,  
and there exists a function $v\in L^2(\Om;d^nx)$ such that 
\begin{eqnarray}\lb{Gdi-1}
(v,u)_{L^2(\Om;d^nx)}
= (w,(\Delta +z)u)_{L^2(\Om;d^nx)},
\quad u\in\dom (- \Delta_{K,\Om,z}).
\end{eqnarray}
Choosing $u\in C^\infty_0(\Omega)$ then shows that 
$w\in\dom (- \Delta_{max})$ and $(\Delta +\ol{z})w=v$. 
Thus, \eqref{Gdi-1} becomes
\begin{align}\lb{Gdi-2}
0 &= ((\Delta +\ol{z})w,u)_{L^2(\Om;d^nx)}
- (w,(\Delta +z)u)_{L^2(\Om;d^nx)}
\nonumber\\ 
&= {}_{N^{1/2}(\partial\Omega)}\langle\tau^N_{\ol{z}} w                     ,\widehat{\gamma}_D u 
\rangle_{(N^{1/2}(\partial\Omega))^*}, \quad 
u\in \dom (- \Delta_{K,\Om,z}), 
\end{align}
by \eqref{T-Green} and the fact that $\tau^N_z u = 0$. 
In order to continue, we remark that 
\begin{align}\lb{Gdi-4}
\widehat\gamma_D(\dom (- \Delta_{K,\Om,z})) &= 
\widehat\gamma_D(\ker  (\tau^N_z)) 
\nonumber\\ 
&= \widehat\gamma_D\bigl(
\big\{u\in L^2(\Om;d^nx) \,\big|\, (-\Delta -z)u=0\,\mbox{in}\,\Omega\big\}\bigr)
\nonumber\\ 
& = \bigl(N^{1/2}(\partial\Omega)\bigr)^*,
\end{align}
by virtue of \eqref{3.AKe} and Corollary \ref{New-CV2}. 
Using this in \eqref{Gdi-2} then yields $\tau^N_{\ol{z}} w = 0$ in 
$N^{1/2}(\partial\Omega)$ and hence, $w\in \dom (- \Delta_{K,\Om,\ol{z}})$.
This shows that $- \Delta_{K,\Om,z}^*\subseteq - \Delta_{K,\Om,\ol{z}}$, 
which completes the proof of \eqref{A-zz.W}. The subsequent commentary 
in the statement of the theorem is then justified by this. 

Next, we will show that $-\Delta_{K,\Om,z}\geq 0$ whenever
$z \leq 0$. To this end, fix such a $z$ 
and, given an arbitrary $u\in\dom (- \Delta_{K,\Om,0})$, 
consider $v\in H^2_0(\Om)$ and $w\in L^2(\Om;d^nx)$ with $(-\Delta-z)w=0$ 
such that $u=v+w$. That this is possible is ensured by \eqref{A-zz.1} and 
\eqref{3.AKe}. We may then write 
\begin{align}\lb{Kre-Pw}
(u, (-\Delta_{K,\Om,z}) u)_{L^2(\Om;d^nx)}
&= (v+w, (-\Delta-z)v)_{L^2(\Om;d^nx)}
\nonumber\\
&= (v, (-\Delta-z) v)_{L^2(\Om;d^nx)}
+((-\Delta-z)w, v)_{L^2(\Om;d^nx)}
\nonumber\\
&= 2 \Re [(v, (-\Delta-z) v)_{L^2(\Om;d^nx)}]
\nonumber\\
&= 2 [(-z)\,\|v\|^2_{L^2(\Om;d^nx)}+\|\nabla v\|^2_{(L^2(\Om;d^nx))^n}]\geq 0,
\end{align}
as required. 

Next, we note that by \eqref{A-zz.1} and \eqref{3.AKe}, the domain 
of $- \Delta_{K,\Om}$ has the description given in \eqref{Kre-Def}, 
justifying the terminology of Krein Laplacian used in connection with 
this operator. Finally, that the kernel of $- \Delta_{K,\Om}$ consists of 
all harmonic, square integrable functions in $\Omega$ follows from 
\eqref{Fr-4Tf} and \eqref{Yan-10}, whereas \eqref{spec-1} is a consequence 
of Lemma \ref{L-Fri1} and Theorem 5.1 in \cite{AS80}.
\end{proof}

We remark that the boundary condition $\tau^N_0 v = \widehat\gamma_N v 
+M_{D,N,\Om}^{(0)}(0)\big(\widehat\gamma_D v \big)=0$ for elements $v$ 
in the domain of the Krein Laplacian $-\Delta_{K,\Om}$ can be viewed as a nonlocal Robin boundary condition (cf.\ \cite{GM08}--\cite{GM09a}).

\section{Self-adjoint Extensions with the Dirichlet Laplacian 
as Reference Operator}
\lb{s14}

Having discussed the Friedrichs and the Krein--von Neumann extensions of 
$-\Delta\big|_{C^\infty_0(\Om)}$ in $L^2(\Om; d^n x)$, our goal in this section is now 
to identify {\it all} self-adjoint extensions of this operator (nonnegative or not). 

To set the stage, we first recall an abstract functional analytic result which is 
pertinent to the task at hand. The following is essentially Theorem\ II.2.1 on p.\ 448 
of \cite{Gr68}: 

\begin{theorem}\lb{T-Grubb} 
Let ${\mathcal{H}}$ be a Hilbert space, with inner product 
$(\dott,\dott)_{\mathcal{H}}$, and assume that 
\begin{eqnarray}\lb{Grubb.1}
A_0:\dom (A_0)\subseteq{\mathcal{H}}\to  {\mathcal{H}}
\end{eqnarray}
is a closed, densely defined, symmetric unbounded linear operator. 
Set $A_1:=(A_0)^*$, and let 
\begin{eqnarray}\lb{Grubb.2}
A_\beta:\dom (A_\beta)\subseteq{\mathcal{H}}\to  {\mathcal{H}}
\end{eqnarray}
be a self-adjoint extension of $A_0$ with $0\notin \si(A_\beta)$ 
$($$A_\beta$ is also called the reference operator\,$)$.
\begin{enumerate}
\item[$(i)$] Then 
\begin{eqnarray}\lb{Grubb.3A}
\begin{array}{c}
{\rm pr}_\beta:=A_\beta^{-1}A_1:\dom (A_1)\to \dom (A_\beta),
\\[4pt]
{\rm pr}_\zeta:=I_{\mathcal{H}}-{\rm pr}_\beta:\dom (A_1)\to 
\ker  (A_\beta),
\end{array}
\end{eqnarray}
are complementary projections which induce the decomposition
\begin{eqnarray}\lb{Grubb.3}
\dom (A_1)=\dom (A_\beta)\dotplus    \ker  (A_1).
\end{eqnarray}
In the sequel, we denote the above decomposition schematically 
by writing $u=u_\beta+u_\zeta$ for each $u\in \dom (A_1)$, where
$u_\beta:={\rm pr}_\beta(u)\in \dom (A_\beta)$ 
and $u_\zeta:={\rm pr}_\zeta(u)\in \ker  (A_1)$.

\item[$(ii)$] Let $V$ be an arbitrary closed subspace of $\ker  (A_1)$ 
and let $T:\dom (T)\subseteq V\to V^*$ be an arbitrary self-adjoint 
operator. Then the operator $A^{V,T}\subseteq A_1$ given by  
$A^{V,T}:\dom \big(A^{V,T}\big)\subseteq{\mathcal{H}}\to {\mathcal{H}}$, 
\begin{align}\lb{Grubb.4}
\begin{split}
& A^{V,T} u:=A_1 u,   \\
& u \in \dom \big(A^{V,T}\big):=\{v\in dom (A_1) \,|\, 
v_\zeta\in\dom (T)\mbox{ and }  \\
& \hspace*{3.1cm}
(w,A_1v)_{\cH}={}_V\langle w,Tv_\zeta\rangle_{V^*}, \, w\in V\}, 
\end{split}
\end{align}
is a self-adjoint extension of $A_0$. 

\item[$(iii)$] Conversely, let 
$\wti{A}:\dom (\wti{A})\subseteq{\mathcal{H}}\to {\mathcal{H}}$
be a self-adjoint extension of $A_0$ $($so that, necessarily, 
$\wti{A}\subseteq A_1$$)$ and define 
\begin{eqnarray}\lb{Grubb.6}
V:=\ol{\{u_\zeta \,|\, u\in \dom (\wti{A})\}} \quad \text{ $($with closure in 
${\mathcal{H}}$$)$},
\end{eqnarray}
and consider the operator $T:\dom (T)\subseteq V\to V^*$
given by 
\begin{equation}\lb{Grubb.7}
Tu_\zeta:= (\dott,A_1u)_{\mathcal{H}}\in V^*,  
\quad u_{\zeta} \in \dom (T):= \big\{v_\zeta \,\big|\, v\in \dom \big(\wti{A}\big)\big\}. 
\end{equation}
Then $T$ is a self-adjoint operator, 
and $\wti{A}=A^{V,T}$, where $A^{V,T}$ is associated with $V$ and $T$ as in item $(ii)$.
\end{enumerate} 
Consequently, the constructions in items $(ii)$, $(iii)$ establish a one-to-one 
correspondence between self-adjoint extensions $\wti{A}$ of $A_0$ and 
self-adjoint operators $T:\dom (T)\subseteq V\to V^*$ with 
$V$ a closed subspace of $\ker  (A_1)$.
\end{theorem}

\begin{remark}\lb{H-U1}
Assume that $\cH$ is a Hilbert space (with inner product 
$(\dott,\dott)_{\cH}$), and suppose that $V$ is a closed 
subspace of $\cH$. In addition, denote by $\pi_V$ the orthogonal projection 
of $\cH$ onto $V$. Then 
\begin{eqnarray}\lb{aa.zz.1}
\pi_V\pi_V^*:V^*\to  V\, \mbox{ isomorphically}, 
\end{eqnarray}
with inverse $V\ni v\mapsto (v,\dott)_{\cH}\in V^*$. 
In the context of Theorem \ref{T-Grubb}\,$(i)$, and keeping the 
identification $V^*\equiv V$ in mind given by \eqref{aa.zz.1}, an alternative
description of the domain of the operator $A^{V,T}$, originally introduced 
in \eqref{Grubb.4} is
\begin{eqnarray}\lb{aa.zz.2}
u\in\dom \big(A^{V,T}\big) \, \text{ if and only if }   
\left\{
\begin{array}{l}
u=v+w+A_\beta^{-1}(Tw+\eta),
\\[4pt]
v\in\dom (A_0),\,\,w\in \dom (T),
\\[4pt]
\quad \mbox{and }\,\,\eta\in \ker  (A_1)\cap V^\top.
\end{array}
\right.
\end{eqnarray}
See Lemma 1.4 on p.\ 444 of \cite{Gr68}. Furthermore, with the above 
identification understood, 
\begin{eqnarray}\lb{aa.zz.3}
\ker \big(A^{V,T}\big)=\ker  (T),\quad
\ran \big(A^{V,T}\big)=\ran  (T) \dotplus V^\top.
\end{eqnarray}
In particular, $A^{V,T}$ is Fredholm if and only if $T$ is Fredholm, with 
the same kernel and cokernel. If $A^{V,T}$ and hence, also $T$, is injective
then the inverse satisfies
\begin{eqnarray}\lb{aa.zz.4}
\big(A^{V,T}\big)^{-1}=A_\beta^{-1}+T^{-1}\pi_V,\, \mbox{ defined on }\, 
\ran \big(A^{V,T}\big).
\end{eqnarray}
A related result is \cite[Theorem 2.1]{BGW09}.
\end{remark}

Theorem \ref{T-Grubb} provides a universal parametrization of all 
self-adjoint extensions of $A_0$, and our aim is to 
implement this abstract scheme 
in the case of $-\Delta\big|_{C^\infty_0(\Om)}$. In the theorem below 
we choose the Dirichlet Laplacian as the reference operator. 

\begin{theorem}\lb{CC.w} 
Assume Hypothesis \ref{h.Conv} and let $z\in\bbR\backslash\si(-\Delta_{D,\Om})$.
Suppose that $X$ is a closed subspace of $\bigl(N^{1/2}(\partial\Omega)\bigr)^*$
and denote by $X^*$ the conjugate dual space of $X$. In addition, consider a  
self-adjoint operator 
\begin{eqnarray}\lb{4.Aw1}
L:\dom (L)\subseteq X\to  X^*, 
\end{eqnarray} 
and define the linear operator 
$- \Delta^D_{X,L,z}:\dom \big(- \Delta^D_{X,L,z}\big)\subset L^2(\Om;d^nx) \to  L^2(\Om;d^nx)$ by 
\begin{align}\lb{4.Aw3}
\begin{split}
& -\Delta^D_{X,L,z} u:=(-\Delta-z)u,    \\ 
& \; u \in \dom \big(- \Delta^D_{X,L,z}\big):= \big\{
v\in \dom (- \Delta_{max})\,\big|\, \widehat\gamma_D v \in \dom (L),\,
\tau^N_z v  \big|_{X}= - L\big(\widehat\gamma_D v \big)\big\}.     \\ 
\end{split}
\end{align}
Above, $\tau^N_z$ is the map introduced in \eqref{3.Aw1}, \eqref{3.Aw2},
and the boundary condition $\tau^N_z u  \big|_{X}=-L\bigl(\widehat\gamma_D u \bigr)$ 
is interpreted as
\begin{eqnarray}\lb{4.Aw4B}
{}_{N^{1/2}(\dOm)}\big\langle\tau^N_z u  ,f\big\rangle_{(N^{1/2}(\dOm))^*}
= - \ol{{}_{X}\langle f,L(\widehat\gamma_D u )\rangle_{X^*}}, 
\quad  f \in X.
\end{eqnarray}
Then 
\begin{equation}\lb{4.Aw4}
- \Delta^D_{X,L,z}\,\mbox{ is self-adjoint in $L^2(\Om;d^nx)$}, 
\end{equation}
and
\begin{equation}
-\Delta_{min}-zI_{\Om}\subseteq -\Delta^D_{X,L,z}\subseteq -\Delta_{max}-zI_{\Om}.
\end{equation}

Conversely, if 
\begin{eqnarray}\lb{4.Aw5}
\wti S:\dom \big(\wti S\big)\subseteq L^2(\Om;d^nx)\to  L^2(\Om;d^nx) 
\end{eqnarray}
is a self-adjoint operator with the property that 
\begin{eqnarray}\lb{4.Aw6}
-\Delta_{min}-zI_{\Om}\subseteq \wti S\subseteq -\Delta_{max}-zI_{\Om}, 
\end{eqnarray}
then there exist $X$, a closed subspace of $\bigl(N^{1/2}(\dOm)\bigr)^*$, 
and $L:\dom (L)\subseteq X\to X^*$, a self-adjoint operator, such that 
\begin{eqnarray}\lb{4.Aw7}
\wti S = -\Delta^D_{X,L,z}.
\end{eqnarray}

In the above scheme, the operator $\wti S$ and the pair $X,L$ correspond 
uniquely to each other. In fact, 
\begin{eqnarray}\lb{4.Aw8}
\dom (L)=\widehat\gamma_D\big(\dom \big(\wti S\big)\big),\quad
X=\ol{\widehat\gamma_D\big(\dom \big(\wti S)\big)} \quad  \text{ $\big($with closure in 
$\bigl(N^{1/2}(\dOm)\bigr)^*$$\big)$.}
\end{eqnarray}
\end{theorem}
\begin{proof}
Let $z\in\bbR\backslash \si(-\Delta_{D,\Om})$. 
We will employ Theorem \ref{T-Grubb} in the following context:
\begin{eqnarray}\lb{BB.s1}
A_0:=-\Delta_{min}-zI_{\Om},\quad A_1:=-\Delta_{max}-zI_{\Om},\quad 
A_\beta:=-\Delta_{D,\Om}-zI_{\Om}.
\end{eqnarray}
In particular, $0\in\bbC\backslash \si(A_\beta)$. 

Throughout the proof, we shall make use of the following results and notation:
\begin{align}\lb{BB.s2}
& u\in \dom (- \Delta_{max})  
\, \text{ implies }  \begin{cases} 
u_\beta:=(-\Delta_{D,\Om}-zI_{\Om})^{-1}
(-\Delta -z)u\in H^2(\Omega)\cap H^1_0(\Omega),
\\
u_\zeta:=u-(-\Delta_{D,\Om}-zI_{\Om})^{-1}(-\Delta -z)u
\in \ker  (-\Delta_{max}-zI_{\Om}),
\\
\widehat\gamma_N u_\zeta = \tau^N_z u  \in N^{1/2}(\partial\Omega),
\\
\widehat\gamma_D u_\zeta 
=\widehat\gamma_D u \in \bigl(N^{1/2}(\partial\Omega)\bigr)^*.
\end{cases}   
\end{align}
We now assume that a closed subspace $X$ of 
$\bigl(N^{1/2}(\partial\Omega)\bigr)^*$ has been specified, that $L$ 
is a self-adjoint operator as in \eqref{4.Aw1}, and that $- \Delta^D_{X,L,z}$ 
is defined as in \eqref{4.Aw3}. We will show that $- \Delta^D_{X,L,z}$ is 
a self-adjoint operator. To this end we define
\begin{eqnarray}\lb{BB.s3}
V:=\bigl\{u\in\ker  (-\Delta_{max}-zI_{\Om}) \,\big|\, 
\widehat\gamma_D u \in X\bigr\}, 
\end{eqnarray}
and observe that 
\begin{align}\lb{V-cL}
\begin{split}
& \mbox{$V=\ker  (-\Delta_{max}-zI_{\Om})\cap \widehat\gamma_D^{-1}(X)$} \\
& \quad \mbox{is a closed subspace of $\ker (-\Delta_{max}-zI_{\Om})$},
\end{split}
\end{align}
(where $\widehat\gamma_D^{-1}(X)$ denotes the pre-image of $X$ 
under $\widehat\gamma_D$), and that the restriction of $\widehat\gamma_D$ 
to $V$ satisfies  
\begin{eqnarray}\lb{V-cL2}
\widehat\gamma_D\in\cB(V,X)\, \mbox{ is an isomorphism}, 
\end{eqnarray}
by \eqref{BB.s3} and Corollary \ref{New-CV2}. 

Next, we also introduce the operator $T:\dom (T)\subseteq V\to  V^*$ by setting 
\begin{align}\lb{BB.s4}
\begin{split}
& Tu:= {}_X\langle\widehat\gamma_D(\cdot), L(\widehat\gamma_D u )\rangle_{X^*},   \\
& u \in \dom (T):= \{v\in V \,|\, \widehat\gamma_D v \in\dom (L)\}. 
\end{split}
\end{align}
For every $u,v\in\dom (T)$, using the self-adjointness of $L$, we may write
\begin{align}\lb{BU.q1}
{}_V\langle v,Tu\rangle_{V^*}
= {}_X\langle\widehat\gamma_D v, L(\widehat\gamma_D u )\rangle_{X^*} 
& = \ol{{}_X\langle\widehat\gamma_D u, L(\widehat\gamma_D v)\rangle_{X^*}} 
\no \\
& =\ol{{}_V\langle u,Tv\rangle_{V^*}}.
\end{align}
This shows that $T$ is symmetric, that is, $T\subseteq T^*$. To prove the 
converse inclusion, consider $u\in \dom (T^*)$. Then 
$u\in V$ and there exists $\Lambda\in V^*$ such that 
\begin{eqnarray}\lb{BU.q2}
{}_V\langle w,\Lambda\rangle_{V^*}=\ol{{}_V\langle u,Tw\rangle_{V^*}}
= \ol{{}_X\langle \widehat\gamma_D u ,L(\widehat\gamma_D w)\rangle_{X^*}}, 
\quad w\in \dom (T).
\end{eqnarray}
Our goal is to show that $u\in \dom (T)$ which, by \eqref{BB.s4}, comes
down to proving that $\widehat\gamma_D u \in \dom (L)$, or equivalently,
that $\widehat\gamma_D u \in \dom (L^*)$. In turn, the veracity of the 
latter condition is established as soon as we show that there exists a finite constant 
$C>0$ with the property that 
\begin{eqnarray}\lb{BU.q3}
|{}_X\langle \widehat\gamma_D u ,L(f)\rangle_{X^*}|
\leq C\|f\|_X,  \quad f\in \dom (L).
\end{eqnarray}
Since \eqref{V-cL2} entails that 
\begin{eqnarray}\lb{V-cL3}
\widehat\gamma_D:\dom (T)\to  \dom (L)
\, \mbox{ boundedly, with a bounded inverse}, 
\end{eqnarray}
the estimate \eqref{BU.q3} is going to be implied by 
\begin{eqnarray}\lb{BU.q4}
|{}_X\langle \widehat\gamma_D u ,L(\widehat\gamma_D w)\rangle_{X^*}|
\leq C\|w\|_{V}, \quad w\in \dom (T).
\end{eqnarray}
This, however, for the choice $C=\|\Lambda\|_{V^*}$, is a direct consequence
of \eqref{BU.q2}. In summary, the above reasoning yields that 
$u\in \dom (T^*)$, completing the proof of the fact that $T$ is self-adjoint. 

With this at hand, Theorem \ref{T-Grubb} will imply that $- \Delta^D_{X,L,z}$ 
is self-adjoint as soon as we establish that
\begin{align}\lb{BB.s5}
\begin{split}
& \dom \big(- \Delta^D_{X,L,z}\big) = \{u\in\dom (- \Delta_{max}) \,|\, 
u_\zeta\in\dom (T),    \\
& \quad (w,(-\Delta -z)u)_{L^2(\Om;d^nx)}
={}_V\langle w,Tu_\zeta\rangle_{V^*}, \, w\in V\}.
\end{split}
\end{align} 
We note that for each $u\in\dom (- \Delta_{max})$, 
\begin{eqnarray}\lb{BB.s6}
u_\zeta\in\dom (T)  \, \text{ if and only if } \, 
\widehat\gamma_D u \in \dom (L)  
\end{eqnarray}
by \eqref{BB.s2}. Consequently, it remains to show that 
\begin{equation}\lb{BB.s7}
\tau^N_z u  \big|_{X}= - L (\widehat\gamma_D u)
\, \text{ if and only if } \,
(w,(-\Delta -z)u)_{L^2(\Om; d^n x)}={}_V\langle w,Tu_\zeta\rangle_{V^*}, \quad w\in V, 
\end{equation}
whenever $u\in\dom (- \Delta_{max})$ has 
$\widehat\gamma_D u \in \dom (L)$. Fix such a $u$ 
and recall from \eqref{BB.s2} that 
$\widehat\gamma_D u_\zeta = \widehat\gamma_D u $. Unraveling definitions, 
the task at hand becomes showing that 
\begin{align}\lb{BB.s8}
& \tau^N_z u  \big|_{X}= - L (\widehat\gamma_D u)
\\
& \quad
\text{if and only if } \, 
(w,(-\Delta -z)u)_{L^2(\Om;d^nx)}
= {}_{X}\langle\widehat\gamma_D w,L(\widehat\gamma_D u )\rangle_{X^*}, \quad 
w\in V.
\nonumber
\end{align}
One observes that for every 
$w\in V\subseteq\ker  (-\Delta_{max}-zI_{\Om})$ 
one has (recalling that $z\in\bbR$)
\begin{eqnarray}\lb{BB.s9}
(w,(-\Delta -z)u)_{L^2(\Om;d^nx)}=-
\ol{{}_{N^{1/2}(\dOm)}\langle\tau^N_z u  ,
\widehat\gamma_D w \rangle_{(N^{1/2}(\dOm))^*}}
\end{eqnarray}
by \eqref{T-Green2}. Then \eqref{BB.s8} readily follows from this 
and \eqref{V-cL2}. This concludes the proof of \eqref{4.Aw4}. 

Conversely, assume that $\wti S$, as in \eqref{4.Aw5}, is a   
self-adjoint operator which satisfies \eqref{4.Aw6}. If we define 
\begin{eqnarray}\lb{BB.s3R}
V:=\ol{\{u_\zeta \,\big|\, u\in\dom \big(\wti S\big)\big\}} \quad 
\text{ $\big($with closure in $L^2(\Omega;d^nx)$$\big)$}, 
\end{eqnarray}
then $V$ is a closed subspace of $\ker  (-\Delta_{max}-zI_{\Om})$
for which \eqref{V-cL2} continues to hold. Next, one introduces the 
operator $T:\dom (T)\subseteq V\to  V^*$ by setting 
\begin{align}\lb{BB.s4R}
\begin{split} 
& Tu_\zeta:= (\dott,(-\Delta -z)u)_{L^2(\Omega;d^nx)}\in V^*,   \\
& u_{\zeta} \in \dom (T):= \big\{v_\zeta \,\big|\, v\in\dom \big(\wti S\big)\big\}. 
\end{split}
\end{align}
Then Theorem \ref{T-Grubb} ensures that $T$ is a self-adjoint operator. 

Next, consider $X$, $\dom (L)$ as in \eqref{4.Aw8}, 
and introduce an operator $L$ as in \eqref{4.Aw1} by requiring that
\begin{equation}\lb{BB.s10}
{}_{X}\langle \widehat\gamma_D v,L(\widehat\gamma_D u )\rangle_{X^*}
= {}_{V}\langle v,Tu\rangle_{V^*}, \quad 
u\in \dom (T), \; v\in V.
\end{equation}
Since 
\begin{eqnarray}\lb{V-cL4}
\widehat\gamma_D:V\to X\, \mbox{ and }\, 
\widehat\gamma_D:\dom (T)\to \dom (L)
\, \mbox{ isomorphically}, 
\end{eqnarray}
the requirement in \eqref{BB.s10} uniquely defines $L$. 
Furthermore, for every $u,v\in \dom (T)$, 
using the self-adjointness of $T$ we may write
\begin{align}\lb{BU.y1}
{}_X\langle\widehat\gamma_D v, L(\widehat\gamma_D u )\rangle_{X^*}
& = {}_V\langle v,Tu\rangle_{V^*}= \ol{{}_V\langle u,Tv\rangle_{V^*}}  
=\ol{{}_X\langle\widehat\gamma_D u ,  L(\widehat\gamma_D v)\rangle_{X^*}}.
\end{align}
Together with \eqref{V-cL4}, 
this shows that $L$ is symmetric, that is, $L\subseteq L^*$. To prove the 
converse inclusion, consider $f\in \dom (L^*)$. Then 
$f\in X$ and there exists $\Lambda\in X^*$ such that 
\begin{equation}\lb{BU.y2}
{}_X\langle\widehat\gamma_D w,\Lambda\rangle_{X^*}
=  \ol{{}_X\langle f,L(\widehat\gamma_D w)\rangle_{X^*}}, \quad 
w\in \dom (T).
\end{equation}
Let $u\in V$ be such that $\widehat\gamma_D u =f$. Upon recalling 
\eqref{BB.s10}, the above formula becomes
\begin{equation}\lb{BU.y3}
{}_X\langle\widehat\gamma_D w,\Lambda\rangle_{X^*}
=\ol{{}_V\langle u,T w\rangle_{V^*}}, \quad 
w\in \dom (T).
\end{equation}
In particular, 
\begin{align}\lb{BU.y4}
|{}_V\langle u,T w\rangle_{V^*}|
& \leq  \|\Lambda\|_{X^*}\|\widehat\gamma_D w \|_{(N^{1/2}(\dOm))^*}
\nonumber\\
& \leq  \|\Lambda\|_{X^*}\|w\|_{L^2(\Om;d^nx)}, \quad 
w\in \dom (T).
\end{align}
This shows that $u\in\dom (T^*)=\dom (T)$ and hence, 
$f=\widehat\gamma_D u \in \dom (L)$, by \eqref{V-cL4}. Altogether, 
the above argument proves that $L$ is a self-adjoint operator.

Next, we will prove that 
\begin{eqnarray}\lb{BU.y5}
\dom \big(\wti S\big)\subseteq \dom (- \Delta^D_{X,L,z}).
\end{eqnarray}
We note that if $u\in \dom \big(\wti S\big)$ then
$\widehat\gamma_D u \in \dom (L)$, by definition. 
Thus, as far as \eqref{BU.y5} is concerned, it remains to verify that
\eqref{4.Aw4B} holds. To see that this is indeed the case, given 
an arbitrary $f\in X$, pick $w\in V\subseteq\ker  (-\Delta_{max}-zI_{\Om})$ 
such that $\widehat\gamma_D w =f$. Then for each $u\in\dom \big(\wti S\big)$
we may write
\begin{align}\lb{BU.y6}
{}_{N^{1/2}(\dOm)}\langle\tau^N_z u  ,f\rangle_{(N^{1/2}(\dOm))^*} &= 
{}_{N^{1/2}(\dOm)}\langle\tau^N_z u  ,\widehat\gamma_D w 
\rangle_{(N^{1/2}(\dOm))^*} 
\nonumber\\
&= - ((-\Delta -z)u,w)_{L^2(\Om;d^nx)}
\nonumber\\
&= - \ol{{}_V\langle w,Tu_\zeta\rangle_{V^*}}
\nonumber\\
&= 
- \ol{{}_{X}\langle\widehat\gamma_D w,L(\widehat\gamma_D u_\zeta)\rangle_{X^*}}
\nonumber\\ 
&= - \ol{{}_{X}\langle f,L(\widehat\gamma_D u )\rangle_{X^*}}.
\end{align}
Above, the second equality is a consequence of \eqref{T-Green2}, 
the third equality follows from \eqref{BB.s4R}, the fourth equality 
is implied by \eqref{BB.s10}, and the fifth equality is derived with 
the help of the last line in \eqref{BB.s2}. This shows that \eqref{4.Aw4B}
holds, thus completing the proof of \eqref{BU.y5}. 

Since both $\wti S$ and $-\Delta^{D}_{X,L,z}$ are self-adjoint, 
\eqref{BU.y5} implies $\wti S=-\Delta^{D}_{X,L,z}$. 
\end{proof}

Several distinguished self-adjoint extensions of $-\Delta\big|_{C^\infty_0(\Om)}$ are 
singled out below:

\begin{corollary}\lb{rE.1}
In the context of Theorem \ref{CC.w}, for every 
$z_0\in\bbR\backslash \si(-\Delta_{D,\Om})$ one has the following facts:
\begin{eqnarray}\lb{4.Aw10}
X:=\dom (L):=\{0\}\, \mbox{ and }\,  
L:=0 \, \text{ imply } \, - \Delta^D_{X,L,z_0} + z_0I_{\Om} = - \Delta_{D,\Om},
\end{eqnarray}
the Dirichlet Laplacian, and 
\begin{align}\lb{4.Aw9}
\left.
\begin{array}{c}
X:=\bigl(N^{1/2}(\partial\Omega)\bigr)^*\mbox{ and }
L:= - M^{(0)}_{D,N,\Om}(z_0)
\\[4pt]
\mbox{ with }\dom (L):=N^{3/2}(\partial\Omega)
\end{array}
\right\}   \text{ imply } \, - \Delta^D_{X,L,z_0} + z_0I_{\Om} = - \Delta_{N,\Om},
\end{align}
the Neumann Laplacian. Furthermore, 
\begin{eqnarray}\lb{4.AK}
X:=\dom (L):=\bigl(N^{1/2}(\partial\Omega)\bigr)^*,\, \mbox{ and }\, 
L:=0  \, \text{ imply } \,  -\Delta^D_{X,L,z_0} = - \Delta_{K,\Om,z_0},
\end{eqnarray}
the Krein Laplacian introduced in \eqref{A-zz.1}. 
\end{corollary}
\begin{proof}
For the case of the Dirichlet Laplacian $- \Delta_{D,\Om}$, the regularity 
result in \eqref{3.3Ybis} of Theorem \ref{tH.A} shows that 
$\dom (- \Delta_{N,\Om})=H^2(\Omega)\cap H^1_0(\Omega)$. Thus, 
\begin{eqnarray}\lb{4.Ai1}
\dom (L):=\widehat\gamma_D\bigl(\dom (- \Delta_{D,\Om})\bigr)=\{0\}, 
\end{eqnarray}
and further, $X$, the closure of $\dom (L)$ in 
$\bigl(N^{1/2}(\partial\Omega)\bigr)^*$ is the zero vector.
Hence, trivially, $L:=0$ is self-adjoint and the boundary condition 
\eqref{4.Aw4B} is always satisfied. Finally, the requirement
for $u\in\dom (- \Delta_{max})$ that 
$\widehat\gamma_D u \in \dom (L)=\{0\}$ forces, by Theorem \ref{tH.A}, 
that $u\in H^2(\Omega)\cap H^1_0(\Omega)$. Consequently, in this case, 
$\dom (- \Delta^D_{X,L,z_0})=H^2(\Omega)\cap H^1_0(\Omega)$, 
and since $- \Delta_{min}\subseteq - \Delta^D_{X,L,z_0} + z_0I_{\Om}$, this 
proves \eqref{4.Aw10}. 

For the Neumann Laplacian $- \Delta_{N,\Om}$, the regularity result 
in \eqref{3.3fbis} of Theorem \ref{tH.G2} shows that, under the current 
geometrical assumptions,  
$\dom(-\Delta_{N,\Om})= \big\{u\in H^2(\Omega) \,\big|\, \gamma_N u = 0\big\}$.
Hence, by Lemma \ref{3o-Tx}, 
\begin{eqnarray}\lb{4.AK1}
\dom (L):=\widehat\gamma_D\bigl(\dom (- \Delta_{N,\Om})\bigr)
=N^{3/2}(\partial\Omega).
\end{eqnarray}
Moreover, by Lemma \ref{3U-x}, we have (with the closure below 
taken in $\big(N^{1/2}(\partial\Omega)\big)^*$)  
\begin{eqnarray} \lb{4.AK2}
X:= \ol{\dom (L)}=\big(N^{1/2}(\partial\Omega)\big)^*,
\end{eqnarray}
so that $X^*=N^{1/2}(\partial\Omega)$, by Lemma \ref{L-refN}.
In this context, the operator 
\begin{eqnarray}\lb{4.AK3}
L:\dom (L)\subseteq X\to  X^*,\quad 
L= - M^{(0)}_{D,N,\Om}(z_0)
\end{eqnarray}
is well-defined, linear, and symmetric by Theorem \ref{t3.5v}. 
Proving that $L$ is in fact self-adjoint then requires establishing 
the following regularity result:
\begin{eqnarray}\lb{4.AK4}
f\in \bigl(N^{1/2}(\partial\Omega)\bigr)^*\, \mbox{ has }\, 
M^{(0)}_{D,N,\Om}(z_0)f\in N^{1/2}(\partial\Omega), \, \text{ implying } \,  
f\in N^{3/2}(\partial\Omega).
\end{eqnarray}
This, however, is a consequence of the fact that the action of   
$M^{(0)}_{D,N,\Om}(z_0)\in\cB\big((N^{1/2}(\dOm))^*, (N^{3/2}(\dOm))^*\big)$
is compatible with that of the operator 
$M^{(0)}_{D,N,\Om}(z_0)\in\cB\big(N^{3/2}(\dOm), N^{1/2}(\dOm)\big)$
and the latter is an isomorphism. This justifies \eqref{4.AK4}
which completes the proof of the fact that $L$ in 
\eqref{4.AK3} is indeed a self-adjoint operator. 

Upon recalling that $\tau^N_{z_0} u     =\widehat\gamma_N u 
+M_{D,N,\Om}^{(0)}(z_0)\bigl(\widehat\gamma_D u \bigr)$, the boundary condition 
$\tau^N_{z_0} u     \big|_{X}= - L(\widehat\gamma_D u )$ reduces to $\widehat\gamma_N u =0$.
Now consider the requirement for $u\in\dom (- \Delta_{max})$ that 
$\widehat\gamma_D u \in \dom (L)=N^{3/2}(\partial\Omega)$. Our claim 
is that this is automatically satisfied whenever $u$ fulfills the boundary 
condition \eqref{4.Aw4B}. Indeed, as just discussed, the latter entails 
$\widehat\gamma_N u =0$ and hence $u$ belongs
to $\big\{u\in H^2(\Omega) \,\big|\, \gamma_N u = 0\big\}$ by the regularity result 
\eqref{3.3fbis} of Theorem \ref{tH.G2}. Then Lemma \ref{3o-Tx} yields
$\widehat\gamma_D u \in N^{3/2}(\partial\Omega)$, justifying the claim  
and hence \eqref{4.Aw9}.

Next, consider the case where  
\begin{eqnarray}\lb{4.AK.a}
X:=\dom (L):=\bigl(N^{1/2}(\partial\Omega)\bigr)^*
\, \mbox{ and }\, L:=0.
\end{eqnarray}
Trivially, $L$ is self-adjoint. Moreover, for a function 
$u\in\dom (- \Delta_{max})$, the boundary condition 
$\tau^N_{z_0} u     \big|_{X}= - L(\widehat\gamma_D u )$ reduces to 
$\tau^N_{z_0} u =0$, whereas the requirement that $u$ satisfies 
$\widehat\gamma_D u \in \dom (L)$ becomes superfluous, 
by Theorem \ref{New-T-tr}. Hence, $\dom \big(- \Delta^D_{X,L,z_0}\big)
=\dom \big(- \Delta_{K,\Om,z_0}\big)$ and ultimately, 
$- \Delta^D_{X,L,z_0} = - \Delta_{K,\Om,z_0}$ since they are both
contained in $- \Delta_{max} - z_0 I_{\Om}$.  
This yields \eqref{4.AK} and completes the proof of the corollary. 
\end{proof}

We wish to augment Corollary \ref{rE.1} by elaborating on connections 
with the Robin Laplacian (cf.\ Theorem \ref{t2.3FF}). 

\begin{corollary}\lb{rE.1R}
Suppose that $\Om \subset \bbR^n$, $n\geq 2$, is a bounded $C^{1,r}$ domain  
with $r>1/2$, and assume \eqref{Rob-T1}.  
Then, in the context of Theorem \ref{CC.w}, for every 
$z_0\in\bbR\backslash \si(-\Delta_{D,\Om})$ one has
\begin{equation}\lb{4.Aw9G}
\left.
\begin{array}{c}
X:=\bigl(N^{1/2}(\partial\Omega)\bigr)^*\mbox{ and }
L:= - M^{(0)}_{D,N,\Om}(z_0) + \Theta 
\\[4pt]
\mbox{ with }\dom (L):=H^{3/2}(\partial\Omega)
\end{array}
\right\} \text{ imply } \, - \Delta^D_{X,L,z_0} + z_0I_{\Om} 
= - \Delta_{\Theta,\Om},
\end{equation}
the Robin Laplacian.
\end{corollary}
\begin{proof}
The argument is analogous to the proof of Corollary \ref{rE.1}.
This time, we make use of Lemma \ref{Lt2.3} and the fact that, 
by \eqref{Tan-C5} and Lemma \ref{Lk-t1}, one has  
$N^{3/2}(\partial\Omega)\hookrightarrow N^{1/2}(\partial\Omega)$ boundedly.
\end{proof}

Our next theorem elucidates the circumstances under which a given 
self-adjoint extension of the Laplacian is a nonnegative operator. 
Before stating this result recall that, given a reflexive Banach space 
${\mathcal{V}}$ and a linear unbounded operator 
$R:{\rm dom}\,(R)\subset{\mathcal{V}}\to{\mathcal{V}}^\ast$, we 
say that $R\geq 0$ provided 
\begin{equation}\label{TTii}
{}_{{\mathcal{V}}}\langle u, Ru\rangle_{{\mathcal{V}}^*}\geq 0, 
\qquad u \in{\rm dom}\,(R)\subset {\mathcal{V}}.
\end{equation}

\begin{theorem}\lb{Th.Pos} 
Retaining the context of Theorem \ref{CC.w}, let  
$z_0 \leq 0$. Then 
\begin{eqnarray}\lb{Pis-1}
-\Delta^D_{L,X,z_0}\geq 0 \, \text{ if and only if } \,  L\geq 0.
\end{eqnarray}
\end{theorem}
\begin{proof}
A direct comparison of \eqref{AS-1} and \eqref{aa.zz.2} reveals that, 
if $T$ is as in \eqref{BB.s4}, then 
\begin{eqnarray}\lb{Pis-2}
-\Delta^D_{L,X,z_0}\geq 0 \, \text{ if and only if } \,  T\geq 0, 
\, \text{ equivalently, if and only if } \,  L\geq 0.  
\end{eqnarray} 
\end{proof}

\section{Self-Adjoint Extensions with the Neumann Laplacian 
as Reference Operator}
\lb{s15}

Having used the Dirichlet Laplacian as the reference operator in the previous section, we now 
illustrate the construction of self-adjoint extensions of $-\Delta\big|_{C^\infty_0(\Om)}$ in 
$L^2(\Om; d^n x)$, with the 
(shifted) Neumann Laplacian as the reference operator. (The shift 
$- \Delta_{N,\Om} \longrightarrow - \Delta_{N,\Om} - z_0 I_{\Omega}$ for some $z_0\in\bbR$ becomes necessary since $0 \in \sigma(- \Delta_{N,\Om})$.)

As a preamble, we now state and prove a result which is the counterpart of Theorem \ref{LL.w} pertaining to the regularized version of the Dirichlet trace operator on quasi-convex domains: 

\begin{theorem}\lb{LL.wF} 
Assume Hypothesis \ref{h.Conv}. Then, for every 
$z\in\bbC\backslash\si(-\Delta_{N,\Om})$, the map  
\begin{eqnarray}\lb{3.Aw1F}
\tau^D_z:\big\{u\in L^2(\Om;d^nx) \,\big|\, \Delta u\in L^2(\Om;d^nx)\big\}
\to N^{3/2}(\partial\Omega)
\end{eqnarray}
given by 
\begin{eqnarray}\lb{3.Aw2F}
\tau^D_z u :=\widehat\gamma_D u 
-M_{N,D,\Om}^{(0)}(z) (\widehat\gamma_N u), 
\quad u\in L^2(\Om;d^nx),\,\,\Delta u\in L^2(\Om;d^nx),
\end{eqnarray}
is well-defined, linear, and bounded when the space 
\begin{equation}
\big\{u\in L^2(\Om;d^nx) \,\big|\, \Delta u\in L^2(\Om;d^nx)\big\} 
= \dom(-\Delta_{max})    \lb{15.dommax}
\end{equation} 
is endowed with the 
natural graph norm $u\mapsto\|u\|_{L^2(\Om;d^nx)}+\|\Delta u\|_{L^2(\Om;d^nx)}$.
Moreover, this operator satisfies the following additional properties: 

\begin{enumerate}
\item[$(i)$] For each $z\in\bbC\backslash\si(-\Delta_{N,\Om})$, the map $\tau^D_z$ 
in \eqref{3.Aw1F}, \eqref{3.Aw2F} is onto
$($i.e., $\tau^D_z(\dom (- \Delta_{max}))=N^{3/2}(\partial\Omega)$$)$. In fact, 
\begin{eqnarray}\lb{3.ON2}
\tau^D_z\big(\big\{u\in H^2(\Om) \,\big|\, \gamma_N u = 0\big\}\big)
=N^{1/2}(\partial\Omega),\quad z\in\bbC\backslash\si(-\Delta_{N,\Om}).
\end{eqnarray}

\item[$(ii)$] For each $z\in\bbC\backslash \big[\si(-\Delta_{D,\Om})\cup
\si(-\Delta_{N,\Om})\big]$, one has
\begin{eqnarray}\lb{3.NewT}
\tau^D_z=-M_{N,D,\Om}^{(0)}(z)\,\tau^N_z.
\end{eqnarray}

\item[$(iii)$] One has 

\begin{eqnarray}\lb{3.Aw9F}
\tau^D_z=\gamma_D(-\Delta_{N,\Om}-zI_{\Om})^{-1}(-\Delta-z),\quad
z\in\bbC\backslash \si(-\Delta_{N,\Om}).
\end{eqnarray} 

\item[$(iv)$] For each $z\in\bbC\backslash\si(-\Delta_{N,\Om})$, the 
kernel of the map $\tau^D_z$ in \eqref{3.Aw1F}, \eqref{3.Aw2F} is 
given by 
\begin{eqnarray}\lb{3.AKeF}
\ker \big(\tau^D_z\big)=H^2_0(\Omega)\dotplus    \big\{u\in L^2(\Om;d^nx) \,\big|\, 
(-\Delta -z)u=0 \mbox{ in } \Omega\big\}.
\end{eqnarray}
In particular,
\begin{eqnarray}\lb{Sim-GrF} 
\tau^D_z u     =0\, \mbox{ for every }\,u\in\ker  (-\Delta_{max}-zI_{\Om}).
\end{eqnarray}

\item[$(v)$] For each $z\in\bbC\backslash\si(-\Delta_{N,\Om})$, the 
following Green formula holds for every $u,v\in\dom (- \Delta_{max})$,  
\begin{align}\lb{T-GreenF} 
& 
((-\Delta -z)u,  v)_{L^2(\Omega;d^nx)}
- (u,  (-\Delta -\ol{z})v)_{L^2(\Omega;d^nx)}
\nonumber\\ 
& \quad
=-{}_{N^{3/2}(\partial\Omega)}\big\langle\tau^D_z u,\widehat\gamma_N v        
\big\rangle_{(N^{3/2}(\partial\Omega))^*}
+\,\ol{{}_{N^{3/2}(\partial\Omega)}\big\langle\tau^D_{\ol{z}} v,\widehat{\gamma}_N u  
\big\rangle_{(N^{3/2}(\partial\Omega))^*}}.
\end{align}
In particular, for every $u,v\in \dom (- \Delta_{max})$, $z\in\bbC$, and 
$z_0\in\bbR\backslash\si(-\Delta_{N,\Om})$, one has 
\begin{align}\lb{T-GreenFX} 
& 
((-\Delta -(z+z_0))u,  v)_{L^2(\Omega;d^nx)}
- (u,  (-\Delta -(\ol{z}+z_0))v)_{L^2(\Omega;d^nx)}
\nonumber\\ 
& \quad
=-{}_{N^{3/2}(\partial\Omega)}\big\langle\tau^D_{z_0} u,\widehat\gamma_N v        
\big\rangle_{(N^{3/2}(\partial\Omega))^*}
+\,\ol{{}_{N^{3/2}(\partial\Omega)}\big\langle\tau^D_{z_0} v,\widehat{\gamma}_N u  
\big\rangle_{(N^{3/2}(\partial\Omega))^*}}.
\end{align}
Moreover, as a consequence of \eqref{T-GreenF} and \eqref{Sim-GrF}, 
for every $z\in\bbC\backslash \si(-\Delta_{N,\Om})$ one infers that 
\begin{align}\lb{T-Green2F} 
& u\in \dom (- \Delta_{max})\mbox{ and }
v\in\ker  (-\Delta_{max}-\ol{z}I_{\Om})
\nonumber\\ 
& \quad
\text{imply } \, ((-\Delta -z)u,  v)_{L^2(\Omega;d^nx)}
=-{}_{N^{3/2}(\partial\Omega)}\big\langle\tau^D_z u,\widehat\gamma_N v        
\big\rangle_{(N^{3/2}(\partial\Omega))^*}. 
\end{align} 
\end{enumerate}
\end{theorem}
\begin{proof}
Let $z\in\bbC\backslash\sigma(-\Delta_{N,\Om})$. Consider an arbitrary $u\in L^2(\Om;d^nx)$ 
satisfying $\Delta u\in L^2(\Om;d^nx)$ and let $v$ solve 
\begin{eqnarray}\lb{3.Aw3F}
(-\Delta -z)v=0\,\mbox{ in }\,\Omega,\quad v\in L^2(\Om;d^nx),\quad 
\widehat\gamma_N v =\widehat\gamma_N u \,\mbox{ on } \, \partial\Omega. 
\end{eqnarray}
Theorems \ref{3ew-T-tr} and \ref{tH.G2} ensure that this is possible  
and guarantee the the existence of a finite constant $C=C(\Om,z)>0$ for which 
\begin{equation}\lb{3.Aw4F}
\|v\|_{L^2(\Om;d^nx)}\leq C\big(\|u\|_{L^2(\Om;d^nx)}+\|\Delta u\|_{L^2(\Om;d^nx)}\big).
\end{equation}
Then $w:=u-v$ satisfies
\begin{equation}\lb{3.Aw5F}
(-\Delta -z)w=(-\Delta -z)u\in L^2(\Omega;d^nx),\quad w\in L^2(\Om;d^nx), \quad  
\widehat\gamma_N w =0\, \mbox{ on }\, \partial\Omega, 
\end{equation}
so that, by \eqref{3.3fbis}, there exists $C=C(\Om,z)>0$ such that 
\begin{eqnarray}\lb{3.Aw6F}
w\in H^2(\Omega)\, \mbox{ and }\,  
\|w\|_{H^2(\Om)}\leq C\big(\|u\|_{L^2(\Om;d^nx)}+\|\Delta u\|_{L^2(\Om;d^nx)}\big).
\end{eqnarray}
Since $M_{N,D\Om}^{(0)}(z)\bigl(\widehat\gamma_N u \bigr)
=\widehat\gamma_D v$, it follows from \eqref{3.Aw6F}, the compatibility part 
of Theorem \ref{New-T-tr}, and Lemma \ref{3o-Tx} that 
\begin{eqnarray}\lb{3.Aw7F}
\tau^D_z u =\widehat\gamma_D u -\widehat\gamma_D v            
=\widehat\gamma_D(u-v)=\widehat\gamma_D w                   
=\gamma_D w \in N^{3/2}(\partial\Omega). 
\end{eqnarray}
Moreover, 
\begin{align}\lb{3.Aw8F}
\big\|\tau^D_z u \big\|_{N^{3/2}(\partial\Omega)}
&= \|\gamma_D w \|_{N^{3/2}(\partial\Omega)}\leq C\|w\|_{H^2(\Om)}
\nonumber\\ 
& \leq  C(\|u\|_{L^2(\Om;d^nx)}+\|\Delta u\|_{L^2(\Om;d^nx)}).
\end{align}
This shows that the operator $\tau^D_z$ in \eqref{3.Aw1F}, \eqref{3.Aw2F} is 
indeed well-defined and bounded. 

Incidentally, the above argument also shows that \eqref{3.Aw9F} holds. 
That $\tau^D_z$ defined in \eqref{3.Aw1F}, \eqref{3.Aw2F} is onto, is a direct 
consequence of the fact that $\gamma_D$ in \eqref{3an-C7} is onto 
(cf.\ Lemma \ref{3o-Tx}). This also justifies \eqref{3.ON2}.
Next, \eqref{3.NewT} is a direct consequence of \eqref{3.53v}. 

As previously remarked, the sum in \eqref{3.AKe} 
is direct. Moreover, $H^2_0(\Omega)\subseteq \ker  (\tau^D_z)$ by 
\eqref{3.Aw2F} and $\big\{u\in L^2(\Om;d^nx) \,\big|\, 
(-\Delta -z)u=0\,\mbox{in}\,\Omega\big\}\subseteq \ker \big(\tau^D_z\big)$ 
by \eqref{3.Aw9F}. This proves the right-to-left inclusion in \eqref{3.AKeF}. 
To prove the opposite one, consider $u\in L^2(\Om;d^nx)$ with 
$\Delta u\in L^2(\Om;d^nx)$ for which $\tau^D_z u =0$. 
If we now set $w:=(-\Delta_{N,\Om}-zI_{\Om})^{-1}(-\Delta-z)u$, then 
$w\in H^2(\Omega)$, $\gamma_N w =0$ and $\gamma_D w =\tau^D_z u =0$  
by \eqref{3.Aw9F}. Hence, $w\in H^2_0(\Omega)$ by Theorem \ref{T-DD1}. 
Since $u=w+(u-w)$ and $u-w\in L^2(\Om;d^nx)$ is harmonic, the proof of 
\eqref{3.AKeF} is complete. 

Next, consider the Green formulas in $(v)$. Fix
$z\in\bbC\backslash \si(-\Delta_{N,\Om})$, let $u,v\in\dom (- \Delta_{min})$,  
and set $\wti{u}:=(-\Delta_{N,\Om}-z)^{-1}(-\Delta -z)u$, 
$\wti{v}:=(-\Delta_{N,\Om}-\ol{z})^{-1}(-\Delta -\ol{z})v$. Then 
\begin{align}\lb{G-G1F} 
& \wti{u},\wti{v}\in H^2(\Omega),\quad \gamma_N \wti{u} =\gamma_N \wti{v} =0,\no \\
& (-\Delta-z)\wti{u}=(-\Delta-z)u,\quad (-\Delta-\ol{z})\wti{v}=(-\Delta -\ol{z})v, 
\\
& \gamma_D \wti{u} =\tau^D_z u, \quad 
\gamma_D \wti{v} = \tau^D_{\ol{z}} v \no 
\end{align}
by \eqref{3.Aw9F}. Based on these observations 
and repeated applications of Green's formula \eqref{3an-C12} one can then write 
\begin{align}\lb{T-Gr1F} 
& ((-\Delta-z)u,  v)_{L^2(\Omega;d^nx)}
- (u,  (-\Delta -\ol{z})v)_{L^2(\Omega;d^nx)}
\nonumber\\ 
& \quad = ((-\Delta-z)\wti{u},  v)_{L^2(\Omega;d^nx)}
- (u,  (-\Delta-\ol{z})\wti{v})_{L^2(\Omega;d^nx)}
\nonumber\\ 
& \quad = (\wti{u},  (-\Delta-\ol{z})v)_{L^2(\Omega;d^nx)}
- (u,  (-\Delta-\ol{z})\wti{v})_{L^2(\Omega;d^nx)}
\nonumber\\ 
& \qquad -{}_{N^{3/2}(\partial\Omega)}\langle\gamma_D \wti{u},\widehat{\gamma}_N v           
\rangle_{(N^{3/2}(\partial\Omega))^*}
\nonumber\\ 
& \quad = ([\wti{u}-u],  (-\Delta-z)\wti{v})_{L^2(\Omega;d^nx)}
-{}_{N^{3/2}(\partial\Omega)}\langle\tau^D_0 u,\widehat{\gamma}_N v           
\rangle_{(N^{3/2}(\partial\Omega))^*}
\nonumber\\ 
& \quad = \ol{{}_{N^{3/2}(\partial\Omega)}\big\langle\tau^D_0 v,\widehat{\gamma}_N u  
\big\rangle_{(N^{3/2}(\partial\Omega))^*}}
-{}_{N^{3/2}(\partial\Omega)}\big\langle\tau^D_0 u,\widehat{\gamma}_N v 
\big\rangle_{(N^{3/2}(\partial\Omega))^*},
\end{align}
where in the last step we have used the fact that $(-\Delta-z)(\wti{u}-u)=0$ 
and $\widehat\gamma_N (\wti{u}-u)=-\widehat\gamma_N u $.
This justifies \eqref{T-GreenF} and completes the proof of the theorem.
\end{proof}

We are now ready to implement Theorem \ref{T-Grubb} in the case where the
reference operator is the (shifted) Neumann Laplacian. Specifically, we have the 
following result:

\begin{theorem}\lb{CC.wF} 
Assume Hypothesis \ref{h.Conv} and let  
$z\in\bbR\backslash \si(-\Delta_{N,\Om})$. Suppose that $X$ is a closed 
subspace of $\bigl(N^{3/2}(\partial\Omega)\bigr)^*$ and denote by $X^*$ the 
conjugate dual space of $X$. In addition, consider a self-adjoint operator 
\begin{eqnarray}\lb{4.Aw1F}
L:\dom (L)\subseteq X\to  X^*, 
\end{eqnarray} 
and define the linear operator 
$- \Delta^N_{X,L,z}:\dom (- \Delta^N_{X,L,z})\subseteq L^2(\Om;d^nx)\to 
 L^2(\Om;d^nx)$ by 
\begin{align}\lb{4.Aw3F}
\begin{split}
& -\Delta^N_{X,L,z} u:=(-\Delta -z)u,   \\
& \; u \in \dom (- \Delta^N_{X,L,z}):=\big\{
v \in \dom (- \Delta_{max}) \,\big|\, \widehat\gamma_N v \in \dom (L),\,
\tau^D_z v|_{X}= - L (\widehat\gamma_N v)\big\}, 
\end{split}
\end{align}
where $\tau^D_z$ is the map \eqref{3.Aw1F}, \eqref{3.Aw2F}, 
and the boundary condition 
$\tau^D_z u|_{X}= - L (\widehat\gamma_N u)$ is interpreted as
\begin{eqnarray}\lb{4.Aw4BF}
{}_{N^{3/2}(\dOm)}\big\langle\tau^D_z u,f\big\rangle_{(N^{3/2}(\dOm))^*}
= - \ol{{}_{X}\big\langle f,L(\widehat\gamma_N u )\big\rangle_{X^*}}, \quad 
f\in X.
\end{eqnarray}
Then 
\begin{equation}\lb{4.Aw4F}
- \Delta^N_{X,L,z}\, \mbox{ is self-adjoint in $L^2(\Om;d^nx)$}, 
\end{equation}
and 
\begin{equation} 
-\Delta_{min}-zI_{\Om}\subseteq -\Delta^N_{X,L,z}\subseteq
-\Delta_{max}-zI_{\Om}. 
\end{equation}

Conversely, if 
\begin{eqnarray}\lb{4.Aw5F}
\wti S:\dom \big(\wti S\big)\subseteq L^2(\Om;d^nx)\to  L^2(\Om;d^nx) 
\end{eqnarray}
is a self-adjoint operator with the property that 
\begin{eqnarray}\lb{4.Aw6F}
-\Delta_{min}-zI_{\Om}\subseteq \wti S\subseteq -\Delta_{max}-zI_{\Om}, 
\end{eqnarray}
then there exists $X$, a closed subspace of $\bigl(N^{3/2}(\dOm)\bigr)^*$, 
and $L:\dom (L)\subseteq X\to X^*$, a self-adjoint operator, such that 
\begin{eqnarray}\lb{4.Aw7F}
\wti S= -\Delta^N_{X,L,z}.
\end{eqnarray}

In the above scheme, the operator $\wti S$ and the pair $X,L$ correspond 
uniquely to each other. In fact, 
\begin{eqnarray}\lb{4.Aw8F}
\dom (L)=\widehat\gamma_N\big(\dom \big(\wti S\big)\big),\quad
X=\ol{\widehat\gamma_N\big(\dom \big(\wti S\big)\big)} \quad \text{ $\big($with closure in 
$\big(N^{3/2}(\dOm)\big)^*$$\big)$}.
\end{eqnarray}
\end{theorem}
\begin{proof}
Fix $z\in{\mathbb{R}}\backslash \si(-\Delta_{N,\Om})$. 
We will employ Theorem \ref{T-Grubb} in the following context:
\begin{eqnarray}\lb{BB.s1F}
A_0:=-\Delta_{min}-zI_{\Om},\quad 
A_1:=-\Delta_{max}-zI_{\Om},\quad A_\beta:=-\Delta_{N,\Om}-zI_{\Om}.
\end{eqnarray}
Throughout the proof, we shall make use of the following results and notation:
\begin{align}\lb{BB.s2F}
& u\in \dom (- \Delta_{max})  
\, \text{ implies }  \left\{
\begin{array}{l}
u_\beta:=(-\Delta_{N,\Om}-zI_{\Om})^{-1}(-\Delta -z)u\in H^2(\Omega),\,\,
\widehat\gamma_N u_\beta = 0,
\\[6pt]
u_\zeta:=u-(-\Delta_{N,\Om}-zI_{\Om})^{-1}(-\Delta -z)u
\in \ker  (-\Delta_{max}-zI_{\Om}),
\\[6pt]
\widehat\gamma_D u_\zeta = \tau^D_z u     \in N^{3/2}(\partial\Omega),
\\[6pt]
\widehat\gamma_N u_\zeta 
=\widehat\gamma_N u \in \bigl(N^{3/2}(\partial\Omega)\bigr)^*.
\end{array}
\right.   
\end{align}
We now assume that a closed subspace $X$ of 
$\bigl(N^{3/2}(\partial\Omega)\bigr)^*$ has been specified, that $L$ 
is a self-adjoint operator as in \eqref{4.Aw1F}, and that $- \Delta^N_{X,L}$ 
is defined as in \eqref{4.Aw3F}. We will show that $- \Delta^N_{X,L}$ is 
a self-adjoint operator. To this end, we define
\begin{eqnarray}\lb{BB.s3F}
V:=\big\{u\in\ker  (-\Delta_{max}-zI_{\Om}) \,\big|\, \widehat\gamma_N u \in X\big\}, 
\end{eqnarray}
and observe that 
\begin{eqnarray}\lb{V-cLF}
\mbox{$V=\ker (-\Delta_{max}-zI_{\Om})\cap \widehat\gamma_N^{-1}(X)$
is a closed subspace of $\ker (- \Delta_{max})$},
\end{eqnarray}
(where $\widehat\gamma_N^{-1}(X)$ denotes the pre-image of $X$ 
under $\widehat\gamma_N$), and the restriction of $\widehat\gamma_N$ to $V$ 
satisfies  
\begin{eqnarray}\lb{V-cL2F}
\widehat\gamma_N\in\cB(V,X)\, \mbox{ is an isomorphism}, 
\end{eqnarray}
by \eqref{BB.s3F} and Corollary \ref{New-CV3}. 

Next, we also introduce the operator $T:\dom (T)\subseteq V\to  V^*$ by setting 
\begin{align}\lb{BB.s4F}
\begin{split}
& Tu:= {}_X\langle\widehat\gamma_D(\cdot), L(\widehat\gamma_N u )\rangle_{X^*},  \\
& u\in\dom (T):= \{v\in V \,|\, \widehat\gamma_N v \in\dom (L)\}. 
\end{split}
\end{align}
For every $u,v\in\dom (T)$, using the self-adjointness of $L$, we may write
\begin{align}\lb{BU.q1F}
{}_V\langle v,Tu\rangle_{V^*}
= {}_X\langle\widehat\gamma_N v, L(\widehat\gamma_N u )\rangle_{X^*}  
& = \ol{{}_X\langle\widehat\gamma_N u, L(\widehat\gamma_N v)\rangle_{X^*}} 
\no \\
& =\ol{{}_V\langle u,Tv\rangle_{V^*}}.
\end{align}
This shows that $T$ is symmetric, that is, $T\subseteq T^*$. To prove the 
converse inclusion, consider $u\in \dom (T^*)$. Then 
$u\in V$ and there exists $\Lambda\in V^*$ such that 
\begin{equation}\lb{BU.q2F}
{}_V\langle w,\Lambda\rangle_{V^*}=\ol{{}_V\langle u,Tw\rangle_{V^*}}
= \ol{{}_X\langle \widehat\gamma_N u ,L(\widehat\gamma_N w)\rangle_{X^*}}, 
\quad w\in \dom (T).
\end{equation}
Our goal is to show that $u\in\dom (T)$ which, by \eqref{BB.s4F}, comes
down to proving that $\widehat\gamma_N u \in \dom (L)$, or equivalently,
that $\widehat\gamma_N u \in \dom (L^*)$. In turn, the veracity of the 
latter condition is established as soon as we show that there exists a finite constant 
$C>0$ with the property that 
\begin{equation}\lb{BU.q3F}
\big|{}_X\langle \widehat\gamma_N u ,L(f)\rangle_{X^*}\big|
\leq C\|f\|_X, \quad f\in \dom (L).
\end{equation}
Since \eqref{V-cL2F} entails that 
\begin{eqnarray}\lb{V-cL3F}
\widehat\gamma_N:\dom (T)\to  \dom (L)
\, \mbox{ boundedly, with a bounded inverse}, 
\end{eqnarray}
the estimate \eqref{BU.q3F} is going to be implied by 
\begin{equation}\lb{BU.q4F}
\big|{}_X\langle\widehat\gamma_N u ,L(\widehat\gamma_N w)\rangle_{X^*}\big|
\leq C\|w\|_{V}, \quad w\in \dom (T).
\end{equation}
This, however, for the choice $C=\|\Lambda\|_{V^*}$, is a direct consequence
of \eqref{BU.q2F}. In summary, the above reasoning yields that 
$u\in \dom (T^*)$, completing the proof of the fact that
$T$ is self-adjoint. 

With this at hand, Theorem \ref{T-Grubb} will imply that $-\Delta^N_{X,L,z}$ 
is self-adjoint as soon as we establish that
\begin{align}\lb{BB.s5F}
\begin{split} 
\dom (- \Delta^N_{X,L,z}) &= \big\{u\in \dom (- \Delta_{max}) \,\big|\, u_\zeta\in \dom (T), \\
& \quad \;\;\,
(w,(-\Delta -z)u)_{L^2(\Om;d^nx)}
={}_V\langle w,Tu_\zeta\rangle_{V^*}, \, w\in V\big\}.
\end{split} 
\end{align} 
We note that for each $u\in\dom (- \Delta_{max})$, 
\begin{eqnarray}\lb{BB.s6F}
u_\zeta\in\dom (T) \, \text{ if and only if } \,  
\widehat\gamma_N u \in \dom (L) 
\end{eqnarray}
by \eqref{BB.s2F}. Consequently, it remains to show that 
\begin{equation}\lb{BB.s7F}
\tau^D_z u     \big|_{X}= - L\bigl(\widehat\gamma_D u \bigr)  \, \text{ if and only if } \, 
(w,(-\Delta -z)u)_{L^2(\Om; d^n x)}={}_V\langle w,Tu_\zeta\rangle_{V^*}, \quad 
w\in V, 
\end{equation}
whenever $u\in\dom (- \Delta_{max})$ satisfies  
$\widehat\gamma_N u \in \dom (L)$. Fix such a $u$ 
and recall from \eqref{BB.s2F} that 
$\widehat\gamma_N u_\zeta = \widehat\gamma_N u $. Unraveling definitions, 
the task at hand becomes showing that 
\begin{align}\lb{BB.s8F}
& \tau^D_z u \big|_{X}= - L\bigl(\widehat\gamma_N u \bigr)
\\
&\quad
\text{if and only if } \, 
(w,(-\Delta -z)u)_{L^2(\Om;d^nx)}
= {}_{X}\langle\widehat\gamma_N w,L(\widehat\gamma_N u )\rangle_{X^*}, 
\quad w\in V.
\nonumber
\end{align}
One observes, however, that for every 
$w\in V\subseteq\ker  (-\Delta_{max}-zI_{\Om})$ one has (recalling $z\in\bbR$) 
\begin{eqnarray}\lb{BB.s9F}
(w,(-\Delta -z)u)_{L^2(\Om;d^nx)} = -
\ol{{}_{N^{3/2}(\dOm)}\big\langle\tau^D_z u,
\widehat\gamma_N w \big\rangle_{(N^{3/2}(\dOm))^*}}
\end{eqnarray}
by \eqref{T-Green2F}. Then \eqref{BB.s8F} readily follows from this 
and \eqref{V-cL2F}. This concludes the proof of \eqref{4.Aw4F}. 

Conversely, assume that $\wti S$, as in \eqref{4.Aw5F}, is a self-adjoint 
operator which satisfies \eqref{4.Aw6F}. If we define 
\begin{eqnarray}\lb{BB.s3RF}
V:=\ol{\big\{u_\zeta \,\big|\, u\in\dom \big(\wti S\big)\big\}} \quad \text{ $\big($with closure in 
$L^2(\Omega;d^nx)$$\big)$}, 
\end{eqnarray}
then $V$ is a closed subspace of $\ker  (-\Delta_{max}-zI_{\Om})$
for which \eqref{V-cL2F} continues to hold. Next, one introduces the 
operator $T:\dom (T)\subseteq V\to  V^*$ 
by setting 
\begin{align}\lb{BB.s4RF}
\begin{split} 
& Tu_\zeta:= (\dott,(-\Delta -z)u)_{L^2(\Omega;d^nx)}\in V^*,  \\
& u_{\zeta} \in \dom (T):= \big\{v_\zeta \,\big|\, v\in\dom \big(\wti S\big)\big\}. 
\end{split}
\end{align}
Then Theorem \ref{T-Grubb} ensures that $T$ is a self-adjoint operator. 

Next, consider $X$, $\dom (L)$ as in \eqref{4.Aw8F}, 
and introduce an operator $L$ as in \eqref{4.Aw1F} by requiring that
\begin{eqnarray}\lb{BB.s10F}
{}_{X}\langle \widehat\gamma_N v,L(\widehat\gamma_N u )\rangle_{X^*}
= {}_{V}\langle v,Tu\rangle_{V^*}, \quad u\in \dom (T), \; v\in V.
\end{eqnarray}
Since 
\begin{equation}\lb{V-cL4F}
\widehat\gamma_N:V\to  X\, \mbox{ and }\, 
\widehat\gamma_N:\dom (T)\to  \dom (L)
\, \mbox{ isomorphically}, 
\end{equation}
the requirement in \eqref{BB.s10F} uniquely defines $L$.  
Furthermore, for every $u,v\in \dom (T)$, 
using the self-adjointness of $T$ we may write
\begin{align}\lb{BU.y1F}
{}_X\langle\widehat\gamma_N v, L(\widehat\gamma_N u )\rangle_{X^*}
& = {}_V\langle v,Tu\rangle_{V^*}= \ol{{}_V\langle u,Tv\rangle_{V^*}}  
=\ol{{}_X\langle\widehat\gamma_N u, L(\widehat\gamma_N v)\rangle_{X^*}}.
\end{align}
Together with \eqref{V-cL4F}, this shows that $L$ is symmetric, that is, 
$L\subseteq L^*$. To prove the 
converse inclusion, consider $f\in \dom (L^*)$. Then 
$f\in X$ there exists $\Lambda\in X^*$ such that 
\begin{equation}\lb{BU.y2F}
{}_X\langle\widehat\gamma_N w,\Lambda\rangle_{X^*}
= \ol{{}_X\langle f,L(\widehat\gamma_N w)\rangle_{X^*}}, 
\quad w\in \dom (T).
\end{equation}
Let $u\in V$ be such that $\widehat\gamma_N u =f$. Upon recalling 
\eqref{BB.s10F}, the above formula becomes
\begin{equation}\lb{BU.y3F}
{}_X\langle\widehat\gamma_N w,\Lambda\rangle_{X^*}
=\ol{{}_V\langle u,T w\rangle_{V^*}}, \quad 
w\in \dom (T).
\end{equation}
In particular, 
\begin{align}\lb{BU.y4F}
|{}_V\langle u,T w\rangle_{V^*}|
& \leq  \|\Lambda\|_{X^*}\|\widehat\gamma_D w\|_{(N^{1/2}(\dOm))^*}
\nonumber\\ 
& \leq  \|\Lambda\|_{X^*}\|w\|_{L^2(\Om;d^nx)}, \quad 
w\in \dom (T).
\end{align}
This shows that $u\in\dom (T^*)=\dom (T)$ and hence, 
$f=\widehat\gamma_D u \in \dom (L)$, by \eqref{V-cL4F}. Altogether, 
the above argument proves that $L$ is a self-adjoint operator.

Next, we will prove that 
\begin{eqnarray}\lb{BU.y5F}
\dom \big(\wti S\big)\subseteq \dom (- \Delta^D_{X,L,z}).
\end{eqnarray}
We note that if $u\in \dom \big(\wti S\big)$ then
$\widehat\gamma_D u \in \dom (L)$, by definition. 
Thus, as far as \eqref{BU.y5F} is concerned, it remains to verify that
\eqref{4.Aw4BF} holds. To see that this is indeed the case, given 
an arbitrary $f\in X$, pick $w\in V\subseteq\ker  (-\Delta_{max}-zI_{\Om})$ 
such that $\widehat\gamma_N w = f$. Then for each $u\in \dom \big(\wti S\big)$
we may write
\begin{align}\lb{BU.y6F}
{}_{N^{3/2}(\dOm)}\langle\tau^D_z u,f\rangle_{(N^{3/2}(\dOm))^*} &=  
{}_{N^{3/2}(\dOm)}\langle\tau^D_z u,\widehat\gamma_N w \rangle_{(N^{3/2}(\dOm))^*} 
\nonumber\\ 
&= - ((-\Delta -z)u,w)_{L^2(\Om;d^nx)}
\nonumber\\ 
&= - \ol{{}_V\langle w,Tu_\zeta\rangle_{V^*}}
\nonumber\\ 
&=  
- \ol{{}_{X}\langle\widehat\gamma_N w,L(\widehat \gamma_N u_\zeta)\rangle_{X^*}}
\nonumber\\ 
&= - \ol{{}_{X}\langle f,L(\widehat\gamma_N u )\rangle_{X^*}}.
\end{align}
Above, the second equality is a consequence of \eqref{T-Green2F}, 
the third equality follows from \eqref{BB.s4RF}, the fourth equality 
is implied by \eqref{BB.s10F}, and the fifth equality is derived with 
the help of the last line in \eqref{BB.s2F}. This shows that \eqref{4.Aw4BF}
holds, thus completing the proof of \eqref{BU.y5F}. 

Since both $\wti S$ and $-\Delta^{N}_{X,L,z}$ are self-adjoint, 
\eqref{BU.y5F} implies $\wti S=-\Delta^{N}_{X,L,z}$. 
\end{proof}

The following lemma is useful in the statement of Corollary \ref{rE.1F} below.

\begin{lemma}\lb{L-Kr2} 
Assume Hypothesis \ref{h.Conv}. Then the operator-valued map  
\begin{eqnarray}\lb{A-Kq.1} 
\bbC\backslash \si(-\Delta_{D,\Om})\ni z\mapsto - \Delta_{K,\Om,z}
\end{eqnarray}
extends naturally $($i.e., with preservations of properties stated
in Theorem \ref{T-Kr}$)$ to the larger domain 
\begin{eqnarray}\lb{A-Kq.2} 
\bbC\backslash \big[\si(-\Delta_{D,\Om})\cap \si(-\Delta_{N,\Om})\big]
\ni z\mapsto - \Delta_{K,\Om,z}. 
\end{eqnarray}
\end{lemma}
\begin{proof}
Consider the linear operator in $L^2(\Omega;d^nx)$ given by 
\begin{align}\lb{A-zz.1Bis} 
\begin{split} 
& - \wti\Delta_{K,\Om,z}(u):=(- \Delta - z)u,    \\
& \; u \in \dom (- \wti\Delta_{K,\Om,z}):=\big\{v\in \dom (- \Delta_{max}) \,\big|\, 
\tau^D_z v = 0\big\}, 
\end{split} 
\end{align}
In a similar fashion to Theorem \ref{T-Kr} it can be shown that 
$- \wti\Delta_{K,\Om,z}$ satisfies 
\begin{eqnarray}\lb{A-zz.bFW} 
\big(- \wti \Delta_{K,\Om,z}\big)^* = - \wti \Delta_{K,\Om,\ol{z}}.
\end{eqnarray}
Moreover, if $z\in\bbR\backslash \si(-\Delta_{N,\Om})$, then 
$- \Delta_{K,\Om,z}$ is self-adjoint, and if 
$z < 0$ then $- \Delta_{K,\Om,z} \geq 0$. Finally, 
\begin{eqnarray}\lb{A-zz.bF} 
- \Delta_{min} \subseteq - \Delta_{K,\Om,z} + zI_{\Om}\subseteq - \Delta_{max}. 
\end{eqnarray}
Since, by \eqref{3.AKe} and \eqref{3.AKeF},
$\ker (\tau_z^D) =  \ker (\tau_z^N)$ whenever 
$z\in\bbC\backslash \big[\si(-\Delta_{D,\Om})\cup \si(-\Delta_{N,\Om})\big]$, 
one obtains that $- \wti\Delta_{K,\Om,z} = - \Delta_{K,\Om,z}$ 
for $z\in\bbC\backslash \big[\si(-\Delta_{D,\Om})\cup \si(-\Delta_{N,\Om})\big]$.
Since by \eqref{3.Aw2} and the representation of $M_{D,N,\Om}^{(0)}(z)$ in \eqref{3.47v}, 
$\tau_z^N$ is analytic for $z \in\bbC\backslash \si(-\Delta_{D,\Om})$, and 
similarly, by \eqref{3.Aw2F} and the representation of $M_{N,D,\Om}^{(0)}(z)$ in \eqref{3.52v}, 
$\tau_z^D$ is analytic for $z \in\bbC\backslash \si(-\Delta_{N,\Om})$, it follows 
that the map \eqref{A-Kq.1} extends to 
$\bbC\backslash \big[\si(-\Delta_{D,\Om})\cap \si(-\Delta_{N,\Om})\big]$, as claimed in 
\eqref{A-Kq.2}.
\end{proof}

Analogously to the proof of Corollary \ref{rE.1}, it is then 
straightforward to establish the following result:

\begin{corollary}\lb{rE.1F}
In the context of Theorem \ref{CC.wF}, for every 
$z_0\in\bbR\backslash \si(-\Delta_{N,\Om})$ one has the following facts:
\begin{eqnarray}\lb{4.Aw10F}
X:=\dom (L):=\{0\}\, \mbox{ and }\,  
L:=0 \, \text{ imply } \, - \Delta^N_{X,L,z_0} + z_0I_{\Om} = - \Delta_{N,\Om},
\end{eqnarray}
the Neumann Laplacian, and 
\begin{eqnarray}\lb{4.Aw9F}
\left.
\begin{array}{c}
X:=\bigl(N^{3/2}(\partial\Omega)\bigr)^*\mbox{ and }
L:= M^{(0)}_{N,D,\Om}(z_0)
\\[4pt]
\mbox{ with }\dom (L):=N^{1/2}(\partial\Omega)
\end{array}
\right\} \text{ imply } \, - \Delta^N_{X,L,z_0} + z_0I_{\Om} = - \Delta_{D,\Om},
\end{eqnarray}
the Dirichlet Laplacian. Furthermore, 
\begin{eqnarray}\lb{4.AKF}
X:=\dom (L):=\bigl(N^{3/2}(\partial\Omega)\bigr)^*\, \mbox{ and }\, 
L:=0 \, \text{ imply } \,  -\Delta^N_{X,L,z_0} = - \Delta_{K,\Om,z_0},
\end{eqnarray}
the Krein Laplacian $($initially introduced in \eqref{A-zz.1}, and further
extended in Lemma \ref{L-Kr2}$)$.
\end{corollary}

\section{Krein-Type Resolvent Formulas}
\lb{s16}

Having catalogued all self-adjoint extensions of the Laplacian, 
we establish in this section a variety of Krein-type resolvent formulas, 
which express the difference of two resolvents for two different self-adjoint
extensions of the Laplacian as $R^* M R$, where $M$ is a suitable Weyl--Titchmarsh 
operator on the boundary and $R$ is closely related to one of the 
resolvents in question. 

Krein-type resolvent formulas have been studied in a great variety of contexts, far too 
numerous to account for all in this paper. For instance, they are of fundamental 
importance in connection with the spectral and inverse spectral theory of ordinary and 
partial differential operators. Abstract versions of Krein-type resolvent formulas (see 
also the brief discussion at the end of our introduction), connected to boundary value spaces 
(boundary triples) and self-adjoint extensions of closed symmetric operators 
with equal (possibly infinite) deficiency spaces, have received enormous 
attention in the literature. In particular, we note that Robin-to-Dirichlet 
maps in the context of ordinary differential operators reduce to the 
celebrated (possibly, matrix-valued) Weyl--Titchmarsh function, the basic 
object of spectral analysis in this context.  Since it is impossible to 
cover the literature in this paper, we refer to the rather extensive 
recent bibliography in \cite{GM08} and \cite{GM09}. Here we mention, for instance, 
\cite[Sect.\ 84]{AG81a}, \cite{ABMN05}, \cite{ADKK07}, \cite{AB09}, \cite{AP04}, \cite{AT03}, \cite{AT05}, 
\cite{BL07}, \cite{BMN08}, \cite{BMT01}, \cite{BT04}, \cite{BMN00}, \cite{BMN02}, \cite{BGW09}, 
\cite{BHMNW09}, \cite{BM04}, \cite{BMNW08}, \cite{BGP08}, 
\cite{DHMS00}, \cite{DHMS06}, \cite{DM91}, \cite{DM95}, \cite{GKMT01}, 
\cite{GLMZ05}, \cite{GMT98}, \cite{GMZ07}, \cite{GT00}, \cite[Ch.\ 3]{GG91}, 
\cite{Gr08a}, \cite[Ch.\ 13]{Gr09}, \cite{Ko00}, \cite{KO77}, \cite{KO78}, \cite{KS66}, 
\cite{Ku09}, \cite{KK04}, 
\cite{LT77}, \cite{Ma92}, \cite{MM02},  \cite{Ma04}, \cite{MPP07}, 
\cite{Ne83}, \cite{Pa06}, \cite{Pa87}, \cite{Pa02}, \cite{Po01}, \cite{Po03}, \cite{Po04}, \cite{Po08}, 
\cite{PR09}, \cite{Ry07}, \cite{Ry09}, \cite{Ry10}, \cite{Sa65}, \cite{St50}, \cite{St54}, \cite{St70a}, 
\cite{TS77}, and the references cited therein. We add, however, that the case of infinite 
deficiency indices in the context of partial differential operators 
(in our concrete case, related to the deficiency indices of the operator 
closure of $-\Delta\upharpoonright_{C^\infty_0(\Om)}$ in $\LOm$), is much less 
studied and the results obtained in this section, especially, under the 
assumption of quasi-convex domains are new.

We start with a couple of preliminary results, contained in the next two lemmas: 

\begin{lemma}\lb{l.Nak}
Assume Hypothesis \ref{h.Conv} and suppose that $z_0 \in\bbR$, 
$z_0, (z+z_0)\in\bbC\backslash\si(-\Delta_{D,\Om})$.
Let $X$ be a closed subspace of $\bigl(N^{1/2}(\partial\Omega)\bigr)^*$ and 
denote by $X^*$ the conjugate dual space of $X$. In addition, consider a 
self-adjoint operator $L$ as in \eqref{4.Aw1} and define the self-adjoint  
operator $- \Delta^D_{X,L,z_0}$ as in \eqref{4.Aw3} 
corresponding to $z=z_0$. 
Then the following resolvent relation holds on $L^2(\Omega;d^nx)$, 
\begin{align}
\bigl(- \Delta^D_{X,L,z_0}-zI_{\Om}\bigr)^{-1} 
&= \big(- \Delta_{D,\Om}-(z+z_0)I_{\Om}\bigr)^{-1}   \no
\\
& \quad - \bigl[\widehat\gamma_D(- \Delta^D_{X,L,z_0}-\ol{z}I_{\Om})^{-1}\bigr]^*
\bigl[\tau_{z_0}^N(- \Delta_{D,\Om}-(z+z_0)I_{\Om})^{-1}\bigr],    \lb{Na1BF} \\
& \hspace*{6.35cm} z\in\bbC\backslash \si(-\Delta^D_{X,L,z_0}).    \no
\end{align} 
\end{lemma}
\begin{proof}
We first assume that $z \in \bbC\backslash\bbR$. 

To set the stage, we recall that 
\begin{eqnarray}\lb{Nak.2}
&& (-\Delta_{D,\Om}-(z+z_0)I_{\Om})^{-1}:L^2(\Om;d^nx)\to 
\dom (- \Delta_{max}),
\\[4pt]
&& \tau_{z_0}^N:\dom (- \Delta_{max})\to  N^{1/2}(\dOm),
\lb{Nak.3}
\\[4pt]
&& (-\Delta^D_{X,L,z_0}-\ol{z}I_{\Om})^{-1}:L^2(\Om;d^nx)\to 
\dom (- \Delta_{max}),
\lb{Nak.4}
\\[4pt]
&& \widehat\gamma_D:\dom (- \Delta_{max})\to 
\bigl(N^{1/2}(\dOm)\bigr)^*,
\lb{Nak.5}
\end{eqnarray}
are bounded, linear operators. These facts imply 
\begin{eqnarray}\lb{Nak.6}
&& \tau_{z_0}^N(-\Delta_{D,\Om}-(z+z_0)I_{\Om})^{-1}
\in\cB\big(L^2(\Om;d^nx),N^{1/2}(\dOm)\big),
\\[4pt]
&& \widehat\gamma_D(-\Delta^D_{X,L,z_0}-\ol{z}I_{\Om})^{-1}
\in\cB\big(L^2(\Om;d^nx),\bigl(N^{1/2}(\dOm)\bigr)^*\big),
\lb{Nak.8}
\\[4pt]
&& \big[\widehat\gamma_D(-\Delta^D_{X,L,z_0}-\ol{z}I_{\Om})^{-1}\big]^*
\in\cB\big(N^{1/2}(\dOm),L^2(\Om;d^nx)\big),
\lb{Nak.9}
\end{eqnarray}
which ensures that the composition of operators appearing on the 
right-hand side of \eqref{Na1BF} is meaningful. 
Next, let $u_1,v_1\in L^2(\Om;d^nx)$ be arbitrary and define
\begin{align}\lb{Na2F}
\begin{split}
u & :=(-\Delta^D_{X,L,z_0}-\ol{z}I_\Om)^{-1}u_1\in\dom(- \Delta^D_{X,L,z_0})
\subset \dom(- \Delta_{max}),
\\[4pt]
v & :=(-\Delta_{D,\Om}-(z+z_0)I_\Om)^{-1}v_1\in\dom(- \Delta_{D,\Om})
=H^2(\Om)\cap H^1_0(\Om).
\end{split} 
\end{align}
Checking \eqref{Na1BF} then reduces to showing that the 
following identity holds:
\begin{align}
&\big(u_1,\big(-\Delta^D_{X,L,z_0}-zI_{\Om}\big)^{-1}v_1\big)_{L^2(\Om;d^nx)} 
-(u_1,
(-\Delta_{D,\Om}-(z+z_0)I_{\Om})^{-1}v_1)_{L^2(\Om;d^nx)}
\no \\ 
&\quad =\big(u_1,\big[\widehat\gamma_D(-\Delta^D_{X,L,z_0}-zI_{\Om})^{-1}\big]^*
\big[\tau^N_{z_0}(-\Delta_{D,\Om}-(z+z_0)I_{\Om})^{-1}\big]v_1\big)_{L^2(\Om;d^nx)}.
\end{align}
We note that according to \eqref{Na2F} one has,
\begin{align}
(u_1,(-\Delta_{D,\Om}-(z+z_0)I_{\Om})^{-1}v_1)_{L^2(\Om;d^nx)}
 & = ((-\Delta_{D,\Om}-(\ol{z}+z_0)I_{\Om})u,v)_{L^2(\Om;d^nx)},
\nonumber\\
\big(v_1,\big(-\Delta^D_{X,L,z_0}-zI_{\Om}\big)^{-1}v_1\big)_{L^2(\Om;d^nx)}
& =\big(\big(\big(-\Delta^D_{X,L,z_0}-zI_{\Om}\big)^{-1}\big)^*
u_1,v_1\big)_{L^2(\Om;d^nx)} 
\nonumber
\\
& =\big(\big(-\Delta^D_{X,L,z_0}-\ol{z}I_{\Om}\big)^{-1}u_1,v_1\big)_{L^2(\Om;d^nx)}
\nonumber
\\ 
& = (u,(-\Delta_{D,\Om}-(z+z_0)I_{\Om})v)_{L^2(\Om;d^nx)},
\end{align}
and 
\begin{align}
& \big(u_1,\big[\widehat\gamma_D(-\Delta^D_{X,L,z_0}-zI_{\Om})^{-1}\big]^*
\big[\tau^N_{z_0}(-\Delta_{D,\Om}-(z+z_0)I_{\Om})^{-1}\big]v_1\big)_{L^2(\Om;d^nx)} 
\nonumber\\[4pt]
&\quad 
=\ol{{}_{N^{1/2}(\dOm)}\big\langle
\tau^N_{z_0}(-\Delta_{D,\Om}-(z+z_0)I_{\Om})^{-1}v_1,
\widehat\gamma_D(-\Delta^D_{X,L,z_0}-\ol{z}I_{\Om})^{-1}u_1
\big\rangle_{(N^{1/2}(\dOm))^*}}
\nonumber\\[4pt]
&\quad =\ol{{}_{N^{1/2}(\dOm)}\big\langle\tau^N_{z_0} v,\widehat\gamma_D u 
\big\rangle_{(N^{1/2}(\dOm))^*}}.
\end{align}
With this, matters have been reduced to proving that
\begin{align}\lb{Na3F}
& \big(\big(-\Delta^D_{X,L,z_0}-\ol{z}I_{\Om}\big)u,v\big)_{L^2(\Om;d^nx)} 
- (u,(-\Delta_{D,\Om}-(z+z_0)I_{\Om})v)_{L^2(\Om;d^nx)} 
\nonumber\\ 
& \quad 
=\ol{{}_{N^{1/2}(\dOm)}\big\langle\tau^N_{z_0} v,\widehat\gamma_D u 
\big\rangle_{(N^{1/2}(\dOm))^*}},
\end{align}
or equivalently, to 
\begin{align}\lb{Nak.10}
& ((-\Delta-z_0)u,v)_{L^2(\Om;d^nx)} 
- (u,(-\Delta-z_0)v)_{L^2(\Om;d^nx)} 
\nonumber\\
& \quad 
=\ol{{}_{N^{1/2}(\dOm)}\big\langle\tau^N_{z_0} v,\widehat\gamma_D u 
\big\rangle_{(N^{1/2}(\dOm))^*}},
\end{align}
since $\big(-\Delta^D_{X,L,z_0}-\ol{z}I_{\Om}\big)u=(-\Delta-(\ol{z}+z_0))u$ 
and $(-\Delta_{D,\Om}-(z+z_0)I_{\Om})v=(-\Delta-(z+z_0))v$.
However, \eqref{Nak.10} is a consequence of \eqref{T-Green} and the
fact that $\widehat\gamma_D v =0$. 

In order to remove the additional hypothesis that $z\in\bbC\backslash\bbR$, it now suffices to note 
that due to the assumptions $z_0 \in\bbR$ and $z_0, (z+z_0)\in\bbC\backslash\si(-\Delta_{D,\Om})$, 
and due to the self-adjointness of $- \Delta^D_{X,L,z_0}$, both sides of \eqref{Na1BF} 
extend to all $z\in\bbC\backslash \si(-\Delta^D_{X,L,z_0})$ by analytic continuation with 
respect to $z$.  
\end{proof}

From \eqref{Nak.6} we know that 
\begin{eqnarray}\lb{Nak.12}
\bigl[\tau_{z_0}^N(-\Delta^D_{X,L,z_0}-zI_{\Om})^{-1}\bigr]^*
\in\cB\big(\big(N^{1/2}(\dOm)\big)^*,L^2(\Om;d^nx)\big).
\end{eqnarray}
Moreover, \eqref{Nak.8} yields
\begin{equation}\lb{Nak.12B}
\big[\widehat\gamma_D(-\Delta^D_{X,L,z_0}-zI_{\Om})^{-1}\big]^*
\in\cB\big(N^{1/2}(\dOm),L^2(\Om;d^nx)\big).
\end{equation}
In the lemma below we further clarify the nature of the ranges of 
these operators.  

\begin{lemma}\lb{l.Nak2}
Retain the hypotheses and conventions made in Lemma \ref{l.Nak} and let  
$z\in\bbC\backslash \si(-\Delta^D_{X,L,z_0})$. Then  
\begin{equation}\lb{Nak.11}
\bigl[\tau_{z_0}^N(-\Delta^D_{X,L,z_0}-zI_{\Om})^{-1}\bigr]^*
\in\cB\big(\bigl(N^{1/2}(\dOm)\bigr)^*,
\ker  (-\Delta_{max}-(\ol{z}+z_0)I_{\Om})\big),
\end{equation}
and
\begin{eqnarray}\lb{Nak.11B}
\bigl[\widehat\gamma_D(-\Delta^D_{X,L,z_0}-zI_{\Om})^{-1}\bigr]^*
\in\cB\big(N^{1/2}(\dOm),\ker  (-\Delta_{max}-(\ol{z}+z_0)I_{\Om})\big).
\end{eqnarray}
\end{lemma}
\begin{proof}
Granted \eqref{Nak.12}, it suffices to show that if 
$f\in \bigl(N^{1/2}(\dOm)\bigr)^*$ then, in the sense of distributions,  
\begin{eqnarray}\lb{Nak.13}
(-\Delta-(\ol{z}+z_0))
\bigl[\tau_{z_0}^N(-\Delta^D_{X,L,z_0}-zI_{\Om})^{-1}\bigr]^*f=0 \, 
\mbox{ in } \, \Omega.
\end{eqnarray}
To this end, pick an arbitrary $\varphi\in C^\infty_0(\Omega)$ and write 
\begin{align}\lb{Nak.14}
& \big(\big[\tau_{z_0}^N(-\Delta^D_{X.L,z_0}-zI_{\Om})^{-1}\big]^*f, 
(-\Delta-(z+z_0))\varphi\big)_{L^2(\Om;d^nx)}
\nonumber\\[1mm] 
& \quad 
=\ol{{}_{N^{1/2}(\dOm)}\big\langle 
\tau_{z_0}^N(-\Delta^D_{X,L,z_0}-zI_{\Om})^{-1}
(-\Delta-(z+z_0))\varphi,  f\big\rangle_{(N^{1/2}(\dOm))^*}}.
\end{align}
To continue, we notice that 
\begin{eqnarray}\lb{Nak.14C}
\varphi\in H^2_0(\Omega) = \dom (- \Delta_{min})
\subset \dom (- \Delta^D_{X,L,z_0}).
\end{eqnarray}
Thus, we have 
$(-\Delta^D_{X,L,z_0}-zI_{\Om})^{-1}(-\Delta-(z+z_0))\varphi=\varphi$ and 
also $\tau_{z_0}^N\varphi=0$ by \eqref{3.AKe}. Consequently, the right-hand
side of \eqref{Nak.14} vanishes and \eqref{Nak.13} follows. 

As far as \eqref{Nak.11B} is concerned, given \eqref{Nak.12B}, 
it suffices to show that if $f\in N^{1/2}(\dOm)$ then, in the sense of 
distributions,  
\begin{equation}\lb{Nak.13B}
(-\Delta-(\ol{z}+z_0))
\bigl[\widehat\gamma_D(-\Delta^D_{X,L,z_0}-zI_{\Om})^{-1}\bigr]^*f=0\, 
\mbox{ in }\,\Omega.
\end{equation}
To verify this, select an arbitrary $\varphi\in C^\infty_0(\Omega)$ and write 
\begin{align}\lb{Nak.14B}
& \big(\bigl[\widehat\gamma_D(-\Delta^D_{X,L,z_0}-zI_{\Om})^{-1}\bigr]^*f, 
(-\Delta-(z+z_0))\varphi\big)_{L^2(\Om;d^nx)}
\nonumber\\
& \quad 
={}_{N^{1/2}(\dOm)}\big\langle f,   
\widehat\gamma_D(-\Delta^D_{X,L,z_0}-zI_{\Om})^{-1}
(-\Delta-(z+z_0))\varphi\big\rangle_{(N^{1/2}(\dOm))^*}.
\end{align}
Because of \eqref{Nak.14C} one obtains  
\begin{eqnarray}\lb{Nak.14D}
\widehat\gamma_D(-\Delta^D_{X,L,z_0}-zI_{\Om})^{-1}(-\Delta-(z+z_0))\varphi
=\widehat\gamma_D \varphi = 0.
\end{eqnarray}
Thus, the right-hand side of \eqref{Nak.14B} vanishes, 
proving \eqref{Nak.13B}.  
\end{proof}

Lemmas \ref{l.Nak} and \ref{l.Nak2} have been inspired by \cite[Lemmas 6, 7]{Na01}, where the special case of Dirichlet and Neumann Laplacians on $\Omega$ a cubic box in $\bbR^n$ was studied. 

In the context of Lemma \ref{l.Nak}, let us define the boundary operator
\begin{eqnarray}\lb{UU.1}
M^D_{X,L,z_0}(z):=\big(\widehat\gamma_D
\big[\widehat\gamma_D\big(-\Delta^D_{X,L,z_0}-zI_{\Om}\big)^{-1}\big]^*\big)^*, 
\quad z \in \bbC \backslash \sigma\big(-\Delta^D_{X,L,z_0}\big), 
\end{eqnarray}
so that 
\begin{eqnarray}\lb{UU.2}
M^D_{X,L,z_0}(z)\in\cB\big(N^{1/2}(\dOm),\bigl(N^{1/2}(\dOm)\bigr)^*\big)
\end{eqnarray}
by Theorem \ref{New-T-tr} and \eqref{Nak.11B}. We are now ready to prove
a version of Krein 's resolvent formula for arbitrary self-adjoint 
extensions of the Laplacian. Concretely, we have the following result:
 
\begin{theorem}\lb{Th.Nak}
Assume Hypothesis \ref{h.Conv} and suppose that $z_0 \in\bbR$, 
$z_0, (z+z_0)\in\bbC\backslash\si(-\Delta_{D,\Om})$.
Let $X$ be a closed subspace of $\bigl(N^{1/2}(\partial\Omega)\bigr)^*$ and 
denote by $X^*$ the conjugate dual space of $X$. In addition, consider a  
self-adjoint operator $L$ as in \eqref{4.Aw1} and define the self-adjoint 
operator $- \Delta^D_{X,L,z_0}$ as in \eqref{4.Aw3} 
corresponding to $z=z_0$. 
Then the following Krein  formula holds on $L^2(\Omega;d^nx)$, 
\begin{align} 
&\big(-\Delta^D_{X,L,z_0}-zI_{\Om}\big)^{-1} 
= \big(-\Delta_{D,\Om}-(z+z_0)I_{\Om}\big)^{-1}  \no
\\
& \quad 
+\big[\tau^N_{z_0}\big(-\Delta_{D,\Om}-(\ol{z}+z_0)I_{\Om}\big)^{-1}\big]^* 
M^D_{X,L,z_0}(z)
\big[\tau^N_{z_0}\big(-\Delta_{D,\Om}-(z+z_0)I_{\Om}\big)^{-1}\big],    \lb{NaK2F}  \\
& \hspace*{9.1cm}    z\in\bbC\backslash \si(-\Delta^D_{X,L,z_0}). 
\nonumber
\end{align}
\end{theorem}
\begin{proof} 
Applying $\widehat\gamma_D$ from the left to both sides of \eqref{Na1BF} yields
\begin{equation}\lb{NaK3F}
\widehat\gamma_D\big(-\Delta^D_{X,L,z_0}-\ol{z}I_{\Om}\big)^{-1}
=\widehat\gamma_D\big[\widehat\gamma_D\big(-\Delta^D_{X,L,z_0}-zI_{\Om}\big)^{-1}\big]^*
\big[\tau^N_{z_0}\big(-\Delta_{D,\Om}-(z+z_0)I_{\Om}\big)^{-1}\big], 
\end{equation}
since $\widehat\gamma_D\big(-\Delta_{D,\Om}-(z+z_0)I_{\Om}\big)^{-1}=0$. 
Thus, upon recalling \eqref{UU.1}, \eqref{NaK3F} then becomes 
\begin{eqnarray}\lb{NaK4F}
\widehat\gamma_D\big(-\Delta^D_{X,L,z_0}-zI_{\Om}\big)^{-1}
=M^D_{X,L,z_0}(\ol{z})^*\tau^N_{z_0}\big(-\Delta_{D,\Om}-(z+z_0)I_{\Om}\big)^{-1}, 
\end{eqnarray}
as operators in $\in\cB\big(L^2(\Om;d^nx),\bigl(N^{1/2}(\dOm)\bigr)^*\big)$.
Taking adjoints in \eqref{NaK4F}, written with $\ol{z}$ in place of $z$, 
then leads to 
\begin{eqnarray}\lb{NaK5F}
\big[\widehat\gamma_D\big(-\Delta^D_{X,L,z_0}-\ol{z}I_{\Om}\big)^{-1}\big]^* 
=\big[\tau^N_{z_0}\big(-\Delta_{D,\Om}-(\ol{z}+z_0)I_{\Om}\big)^{-1}\big]^* 
M^D_{X,L,z_0}(z).
\end{eqnarray}
Inserting this into \eqref{Na1BF}, one arrives at \eqref{NaK2F}.  
\end{proof}

\begin{remark}\lb{R-sco}
Formula \eqref{NaK2F} relates the resolvent of an arbitrary 
self-adjoint extension of $-\Delta\big|_{C^\infty_0(\Om)}$ in $L^2(\Om; d^n x)$ 
to that of the Dirichlet Laplacian $-\Delta_{D,\Om}$ in a transparent way
which makes it possible to extract information about the spectrum of 
the extension from information about the associated operator-valued 
Weyl--Titchmarsh $M$-function. More details on this will appear elsewhere. 
\end{remark}

Next we study some properties of the Weyl--Titchmarsh $M$-function \eqref{UU.1} in 
more detail and show that it satisfies a natural symmetry condition. More specifically, we have the following result:

\begin{theorem}\lb{Th.DH}
Retain the hypotheses and conventions made in Theorem \ref{Th.Nak} and let  
$z\in\bbC\backslash \si(-\Delta^D_{X,L,z_0})$. Then, as operators in 
$\cB\big(N^{1/2}(\dOm),\bigl(N^{1/2}(\dOm)\bigr)^*\big)$, one has
\begin{eqnarray}\lb{UU.3}
M^D_{X,L,z_0}(z)^*=M^D_{X,L,z_0}(\ol{z}).
\end{eqnarray}
As a consequence, 
\begin{equation}\lb{UU.1H}
M^D_{X,L,z_0}(z)=\widehat\gamma_D
\big[\widehat\gamma_D\big(-\Delta^D_{X,L,z_0}-\ol{z}I_{\Om}\big)^{-1}\big]^*
\in\cB\big(N^{1/2}(\dOm),\bigl(N^{1/2}(\dOm)\bigr)^*\big). 
\end{equation} 
\end{theorem}
\begin{proof}
Let $z_0 \in\bbR$, $z_0, (z+z_0) \in\bbC\backslash \si(-\Delta_{D,\Om})$ and fix 
arbitrary $f,g\in N^{1/2}(\dOm)$. By Theorem \ref{LL.w}\,$(i)$, 
it is then possible to find $u,w\in \dom (- \Delta_{max})$ 
such that $\tau^N_{z_0} u =f$ and $\tau^N_{z_0} w =g$. In fact, by 
\eqref{3.ON}, we can actually pick $u,w\in H^2(\Om)\cap H^1_0(\Om)$ 
(this is going to be of relevance shortly). One then has 
\begin{align}\lb{UU.5}
& 
{}_{N^{1/2}(\dOm)}\big\langle f,M^D_{X,L,z_0}(\ol{z})^*g\big\rangle_{(N^{1/2}(\dOm))^*}
=\ol{{}_{N^{1/2}(\dOm)}\big\langle g,
M^D_{X,L,z_0}(\ol{z})f\big\rangle_{(N^{1/2}(\dOm))^*}}
\nonumber\\ 
& \quad
={}_{N^{1/2}(\dOm)}\big\langle f,\widehat\gamma_D
\big[\widehat\gamma_D\big(-\Delta^D_{X,L,z_0}-\ol{z}I_{\Om}\big)^{-1}\big]^*g
\big\rangle_{(N^{1/2}(\dOm))^*}
\nonumber\\
& \quad
={}_{N^{1/2}(\dOm)}\big\langle \tau^N_{z_0} u     ,\widehat\gamma_D
\big[\widehat\gamma_D\big(-\Delta^D_{X,L,z_0}-\ol{z}I_{\Om}\big)^{-1}\big]^*g
\big\rangle_{(N^{1/2}(\dOm))^*}
\nonumber\\
& \quad
=-\big((-\Delta-(\ol{z}+z_0))u,
\big[\widehat\gamma_D\big(-\Delta^D_{X,L,z_0}-\ol{z}I_{\Om}\big)^{-1}\big]^*g
\big)_{L^2(\Om;d^nx)},
\end{align}
by \eqref{T-GreenX}, \eqref{Nak.11B} and the fact that $\widehat\gamma_D u =0$.
Continuing, we employ taking adjoints and conjugation and bring in $w$ in order to write
\begin{align}\lb{UU.6}
&  -\big( (-\Delta-(\ol{z}+z_0))u,
\big[\widehat\gamma_D\big(-\Delta^D_{X,L,z_0}-\ol{z}I_{\Om}\big)^{-1}\big]^*g
\big)_{L^2(\Om;d^nx)}
\nonumber\\ 
& \quad
=-\ol{{}_{N^{1/2}(\dOm)}\big\langle \tau^N_{z_0} w                      ,
\widehat\gamma_D\big(-\Delta^D_{X,L,z_0}-\ol{z}I_{\Om}\big)^{-1}
(-\Delta-(\ol{z}+z_0))u\big\rangle_{(N^{1/2}(\dOm))^*}}
\nonumber\\ 
& \quad
=-\big((-\Delta-(\ol{z}+z_0))
\big(-\Delta^D_{X,L,z_0}-\ol{z}I_{\Om}\big)^{-1}
(-\Delta-(\ol{z}+z_0))u,w\big)_{L^2(\Om;d^nx)}
\nonumber\\ 
& \qquad 
+\big(\big(-\Delta^D_{X,L,z_0}-\ol{z}I_{\Om}\big)^{-1}
(-\Delta-(\ol{z}+z_0))u,(-\Delta-(z+z_0))w\big)_{L^2(\Om;d^nx)},
\end{align}
by \eqref{T-GreenX} and the fact that $\widehat\gamma_D w =0$. Since
\begin{equation} 
(-\Delta-(\ol{z}+z_0))\big(-\Delta^D_{X,L,z_0}-\ol{z}I_{\Om}\big)^{-1}
(-\Delta-(\ol{z}+z_0))u=(-\Delta-(\ol{z}+z_0))u, 
\end{equation} 
we may summarize the above calculation as follows: 
\begin{align}\lb{UU.7}
&  
{}_{N^{1/2}(\dOm)}\big\langle f,M^D_{X,L,z_0}(\ol{z})^*g\big\rangle_{(N^{1/2}(\dOm))^*}
=-\big((-\Delta-(\ol{z}+z_0))u,w\big)_{L^2(\Om;d^nx)}
\\ 
& \quad 
+\big( \big(-\Delta^D_{X,L,z_0}-\ol{z}I_{\Om}\big)^{-1}
(-\Delta-(\ol{z}+z_0))u,(-\Delta-(z+z_0))w\big)_{L^2(\Om;d^nx)}.
\nonumber
\end{align}
At this stage, we recall once more that $u,w\in H^2(\Om)\cap H^1_0(\Om)$. 
Consequently, we may integrate by parts in the first pairing in the 
right-hand side of \eqref{UU.7} (without boundary terms) and obtain 
\begin{align}\lb{UU.8}
& 
{}_{N^{1/2}(\dOm)}\big\langle f,M^D_{X,L,z_0}(\ol{z})^*g\big\rangle_{(N^{1/2}(\dOm))^*}
\no \\ 
& \quad 
=(z+z_0)(u,w)_{L^2(\Om;d^nx)}
- (\nabla u,\nabla w)_{(L^2(\Om;d^nx))^n}
\\ 
& \qquad 
+ \big(\big(-\Delta^D_{X,L,z_0}-\ol{z}I_{\Om}\big)^{-1}
(-\Delta-(\ol{z}+z_0))u,(-\Delta-(z+z_0))w\big)_{L^2(\Om;d^nx)}.
\nonumber
\end{align}
At this point, the right-hand side of \eqref{UU.8} is invariant 
under interchanging $u$, $w$, replacing $z$ by $\ol{z}$, and taking 
conjugation (here, we have also used the fact that 
the adjoint of $(-\Delta^D_{X,L,z_0}-\ol{z}I_{\Om}\big)^{-1}$ 
is the operator $(-\Delta^D_{X,L,z_0}-zI_{\Om}\big)^{-1}$).
Given the relationship between $u,w$ on the one hand, and $f,g$ 
on the other, this observation then translates into 
\begin{equation}\lb{UU.9}
{}_{N^{1/2}(\dOm)}\big\langle f,M^D_{X,L,z_0}(\ol{z})^*g\big\rangle_{(N^{1/2}(\dOm))^*}
=\ol{{}_{N^{1/2}(\dOm)}\big\langle g,
M^D_{X,L,z_0}(z)^*f\big\rangle_{(N^{1/2}(\dOm))^*}}.
\end{equation}
Since the right-hand side of \eqref{UU.9} is given by 
\begin{equation}\lb{UU.9B}
\ol{{}_{N^{1/2}(\dOm)}\big\langle g,M^D_{X,L,z_0}(z)^*f\big\rangle_{(N^{1/2}(\dOm))^*}}
={}_{N^{1/2}(\dOm)}\big\langle f,M^D_{X,L,z_0}(z)g\big\rangle_{(N^{1/2}(\dOm))^*},
\end{equation}
this proves \eqref{UU.3}. 
\end{proof}

Next, we recall \eqref{3.46vX} and \eqref{Nak.11B} and state the following fact:
 
\begin{lemma}\lb{LA.s} 
Under the assumptions of Theorem \ref{Th.Nak}, one has 
\begin{eqnarray}\lb{Ggg-1}
\tau^N_{z_0}\bigl[\widehat\gamma_D(-\Delta^D_{X,L,z_0}-\ol{z}I_{\Om})^{-1}\bigr]^*
=\bigl[M^{(0)}_{D,N,\Om}(z_0)-M^{(0)}_{D,N,\Om}(z+z_0)\bigl]M^D_{X,L,z_0}(z)
\end{eqnarray}
as operators in $\cB\big(N^{1/2}(\dOm),  N^{1/2}(\dOm)\big)$. In addition, 
\begin{eqnarray}\lb{Gaa-1}
\widehat\gamma_D(-\Delta^D_{X,L,z_0}-\ol{z}I_{\Om})^{-1}:
L^2(\Om;d^nx)\to \dom (L)
\end{eqnarray}
is a well-defined, bounded, linear operator, when $\dom (L)$ is equipped 
with the natural graph norm, that is, 
$g\mapsto\|g\|_{(N^{1/2}(\dOm))^*}+\|L(g)\|_{X^*}$.
\end{lemma}
\begin{proof}
Pick arbitrary $f\in N^{1/2}(\dOm)$, $g\in \bigl(N^{1/2}(\dOm)\bigr)^*$, 
and let $v\in\dom (- \Delta_{max})$ satisfy 
\begin{eqnarray}\lb{Ggg-2}
(-\Delta-(\ol{z}+z_0))v=0\, \mbox{ in }\, \Omega,\quad 
\widehat\gamma_D v =g. 
\end{eqnarray}
Then, since 
$\big[\widehat\gamma_D(-\Delta^D_{X,L,z_0}-\ol{z}I_{\Om})^{-1}\big]^* f 
\in\ker (- \Delta_{max}-(z+z_0)I_{\Om})$ by \eqref{Nak.11B},
we may write
\begin{eqnarray}\lb{Ggg-3}
&& {}_{N^{1/2}(\dOm)}\big\langle\tau^N_{z_0}
\big[\widehat\gamma_D(-\Delta^D_{X,L,z_0}-\ol{z}I_{\Om})^{-1}\big]^*f
,  g\big\rangle_{(N^{1/2}(\dOm))^*}
\nonumber\\[4pt]
&&\qquad =\ol{{}_{N^{1/2}(\dOm)}\big\langle\tau^N_{z_0} v              ,  
\widehat\gamma_D\big[\widehat\gamma_D(-\Delta^D_{X,L,z_0}-\ol{z}I_{\Om})^{-1}\big]^*f
\big\rangle_{(N^{1/2}(\dOm))^*}}
\nonumber\\[4pt]
&&\qquad =\ol{{}_{N^{1/2}(\dOm)}\big\langle\tau^N_{z_0} v              ,  
M^D_{X,L,z_0}(z)f\big\rangle_{(N^{1/2}(\dOm))^*}},
\end{eqnarray}
by \eqref{T-GreenX} and \eqref{UU.1H}. Next, one observes that 
\begin{eqnarray}\lb{Ggg-4}
\tau^N_{z_0} v =\widehat\gamma_N v +M_{D,N,\Om}^{(0)}(z_0)g
=\big[-M_{D,N\Om}^{(0)}(\ol{z}+z_0)+M_{D,N,\Om}^{(0)}(z_0)\big]g
\end{eqnarray}
by \eqref{3.Aw2} and \eqref{3.44v}. Thus, 
\begin{align}\lb{Ggg-5}
& \ol{{}_{N^{1/2}(\dOm)}\big\langle\tau^N_{z_0} v              ,  
M^D_{X,L,z_0}(z)f\big\rangle_{(N^{1/2}(\dOm))^*}}
\nonumber\\
& \quad =\ol{{}_{N^{1/2}(\dOm)}\big\langle
\big[-M_{D,N\Om}^{(0)}(\ol{z}+z_0)+M_{D,N,\Om}^{(0)}(z_0)\big]g,  
M^D_{X,L,z_0}(z)f\big\rangle_{(N^{1/2}(\dOm))^*}}
\nonumber\\
& \quad ={}_{N^{1/2}(\dOm)}\big\langle
\big[-M_{D,N\Om}^{(0)}(\ol{z}+z_0)+M_{D,N,\Om}^{(0)}(z_0)\big]^*
M^D_{X,L,z_0}(z)f,  g\big\rangle_{(N^{1/2}(\dOm))^*}
\nonumber\\
& \quad ={}_{N^{1/2}(\dOm)}\big\langle
\big[-M_{D,N\Om}^{(0)}(z+z_0)+M_{D,N,\Om}^{(0)}(z_0)\big]
M^D_{X,L,z_0}(z)f,  g\big\rangle_{(N^{1/2}(\dOm))^*},
\end{align}
by \eqref{NaLa}, implying \eqref{Ggg-1}. 

Next, consider the claim made about the operator \eqref{Gaa-1}.
Let $f\in X$ and $u\in L^2(\Om;d^nx)$ be arbitrary. Then 
$(-\Delta^D_{X,L,z_0}-\ol{z}I_{\Om})^{-1}u\in\dom (- \Delta^D_{X,L,z_0})$ so 
that $\widehat\gamma_D(-\Delta^D_{X,L,z_0}-\ol{z}I_{\Om})^{-1}u\in\dom (L)$
and 
\begin{eqnarray}\lb{Gaa-2}
&& \left|
{}_{X}\big\langle f,L \widehat\gamma_D(-\Delta^D_{X,L,z_0}-\ol{z}I_{\Om})^{-1}u 
\big\rangle_{X^*}\right|
\nonumber\\[4pt]
&& \qquad
=\left|{}_{N^{1/2}(\dOm)}\big\langle
\tau^N_{z_0}(-\Delta^D_{X,L,z_0}-\ol{z}I_{\Om})^{-1}u,f
\big\rangle_{(N^{1/2}(\dOm))^*}\right|
\nonumber\\[4pt]
&& \qquad
\leq C\|\tau^N_{z_0}(-\Delta^D_{X,L,z_0}-\ol{z}I_{\Om})^{-1}u 
\|_{N^{1/2}(\dOm)}\|f\|_{(N^{1/2}(\dOm))^*}
\nonumber\\[4pt]
&& \qquad
\leq C\|u\|_{L^2(\Om;d^nx)}\|f\|_{X},
\end{eqnarray}
by \eqref{4.Aw4B} and Theorem \ref{LL.w}. Since $u\in L^2(\Om;d^nx)$ was 
arbitrary, this implies that
\begin{eqnarray}\lb{Gaa-3}
\|L \widehat\gamma_D(-\Delta^D_{X,L,z_0}-\ol{z}I_{\Om})^{-1}u 
\|_{X^*}\leq C\|f\|_{X}.
\end{eqnarray}
Together with \eqref{Nak.8}, this shows that the operator \eqref{Gaa-1} is
indeed well-defined, linear, and bounded. 
\end{proof}

We continue to study the properties of the boundary ``transition'' operator 
\eqref{UU.1}, with the goal of establishing a Herglotz property. We recall that 
an operator-valued function 
$M(z)\in\cB({\mathcal{X}},{\mathcal{X}}^*)$, with  
$z\in\bbC_+:=\{z\in\bbC \,|\, \Im(z)>0\}$ and ${\mathcal{X}}$ a 
reflexive Banach space, is called an operator-valued Herglotz function if 
$M(\cdot)$ is analytic on $\bbC_+$ and
\begin{equation}\lb{4.41F}
\Im (M(z)):=(M(z) - M(z)^*)/(2i)\geq 0,\quad z\in\bbC_+.
\end{equation} 
It is customary to extend $M(\cdot)$ to $\bbC_-:=\{z\in\bbC \,|\, \Im(z)<0\}$ by defining 
$M(z) = M(\ol z)^*$, $z\in\bbC_-$, but the latter is generally not an analytic 
continuation of $M |_{\bbC_+}$.

\begin{theorem}\lb{Th.DH2}
Retain the hypotheses and conventions made in Theorem \ref{Th.Nak}.
Then the assignment  
\begin{eqnarray}\lb{UU.4}
\bbC_{+}\ni z\mapsto M^D_{X,L,z_0}(z)
\in\cB\big(N^{1/2}(\dOm),\bigl(N^{1/2}(\dOm)\bigr)^*\big)
\end{eqnarray}
is an operator-valued Herglotz function whenever 
\begin{eqnarray}\lb{UU.4X}
X=\bigl(N^{1/2}(\dOm)\bigr)^*.
\end{eqnarray}
\end{theorem}
\begin{proof}
For the duration of this proof we suppose that \eqref{UU.4X} holds. We first establish 
the analytical dependence of $M^D_{X,L,z_0}(z)$ on 
$z\in\bbC\backslash\sigma\big(-\Delta^D_{X,L,z_0}\big)$: Based on 
\eqref{UU.1} and the resolvent equation for $-\Delta^D_{X,L,z_0}$, for 
$z\in\bbC\backslash \sigma\big(-\Delta^D_{X,L,z_0}\big)$ and $w\in\bbC$ sufficiently close to $0$, one obtains
\begin{equation}
\f{1}{w} \big[M^D_{X,L,z_0}(z+w) -  M^D_{X,L,z_0}(z)\big] 
= \big(\hatt \gamma_D\big[\hatt\gamma_D \big(-\Delta^D_{X,L,z_0} -z I_{\Om}\big)^{-1} \big(-\Delta^D_{X,L,z_0} - (z+w) I_{\Om}\big)^{-1}\big]^*\big)^*.   \lb{16.M1}
\end{equation}
Based on \eqref{Nak.8} and the fact that 
\begin{equation}
\big(-\Delta^D_{X,L,z_0} - (z+w) I_{\Om}\big)^{-1} 
\in \cB\big(L^2(\Om;d^nx)\big),     \lb{16.M2}
\end{equation}
one infers that
\begin{equation}
T_w:=\hatt \gamma_D \big(-\Delta^D_{X,L,z_0} -z I_{\Om}\big)^{-1} 
\big(-\Delta^D_{X,L,z_0} - (z+w) I_{\Om}\big)^{-1}  
\in \cB\big(L^2(\Om;d^nx), \big(N^{1/2}(\dOm)\big)^*\big)    \lb{16.M3}
\end{equation}
and 
\begin{equation}
\lim_{w\to 0} T_w = \hatt \gamma_D \big(-\Delta^D_{X,L,z_0} -z I_{\Om}\big)^{-2} 
\, \text{ in $\cB\big(L^2(\Om;d^nx), \big(N^{1/2}(\dOm)\big)^*\big)$}.   \lb{16.M4}
\end{equation} 
Consequently,
\begin{equation}
T_w^* \in \cB\big(N^{1/2}(\dOm), L^2(\Om;d^nx)\big) \, \text{ and } \, 
\lim_{w\to 0} T_w^* 
\, \text{ exists in $\cB\big(N^{1/2}(\dOm), L^2(\Om;d^nx)\big)$}.   \lb{16.M5}
\end{equation}
Next, we claim that actually,
\begin{equation}
T_w^* \in \cB\big(N^{1/2}(\dOm), \dom(-\Delta_{max})\big)    \lb{16.M6}
\end{equation} 
and 
\begin{equation} 
\lim_{w\to 0} T_w^* 
\, \text{ exists in $\cB\big(N^{1/2}(\dOm), \dom(-\Delta_{max})\big)$}.  \lb{16.M7}
\end{equation}
Given \eqref{16.M5} and \eqref{Nak.11B}, \eqref{16.M6} follows once we establish 
that 
\begin{equation}
(-\Delta - (\ol z + z_0) I_{\Om}) T_w^* 
= \big[\hatt \gamma_D \big(-\Delta^D_{X,L,z_0} - (z+w) I_{\Om}\big)^{-1}\big]^*.   
\lb{16.M8}
\end{equation}  
To prove \eqref{16.M8}, we fix $f \in N^{1/2}(\dOm)$ and write for arbitrary 
$\varphi \in C^\infty_0(\Om)$,
\begin{align}
& {}_{\cD'(\Om)}\langle (-\Delta - (\ol z + z_0) I_{\Om}) T_w^* f, \varphi\rangle_{\cD(\Om)} 
= (T_w^* f, (-\Delta - (z + z_0) I_{\Om}) \varphi)_{L^2(\Om; d^n x}  \no \\
& \quad 
= {}_{N^{1/2}(\dOm)}\langle f, T_w (-\Delta - (z + z_0) I_{\Om}) 
\varphi\rangle_{(N^{1/2}(\dOm))^*}   \no \\
& \quad  
= {}_{N^{1/2}(\dOm)}\big\langle f, \hatt \gamma_D 
\big(-\Delta^D_{X,L,z_0} - (z+w) I_{\Om}\big)^{-1} 
\varphi\big\rangle_{(N^{1/2}(\dOm))^*}    \no \\
& \quad 
= \big(\big[\hatt \gamma_D \big(-\Delta^D_{X,L,z_0} - (z+w) I_{\Om}\big)^{-1}\big]^* f, 
\varphi\big)_{L^2(\Om; d^n x)}, 
\end{align}
using \eqref{16.M3} in the next to last step (above, 
${}_{\cD'(\Om)}\langle \dott , \dott \rangle_{\cD(\Om)}$ denotes the 
distributional pairing, here considered to be linear in the second argument
and antilinear in the first).  

To prove \eqref{16.M7} one can argue as follows: First, one notes that if 
$T_j \in \cB(X, \dom(-\Delta_{max}))$, $j \in\bbN$, where $X$ is a Banach space 
and $\dom(-\Delta_{max})$ is equipped with the graph norm (cf.\ \eqref{15.dommax}), 
\begin{align}
\begin{split}
& \text{$\lim_{j\to\infty} T_j$ exists in $\cB(X,\dom(-\Delta_{max}))$}   \\
& \quad \text{if and only if $\lim_{j\to\infty} T_j$ and 
$\lim_{j\to\infty} \Delta T_j$ exist in $\cB\big(X, L^2(\Om; d^n x)\big)$}.  \lb{16.MC}
\end{split} 
\end{align}
While the left-to-right implication in \eqref{16.MC} is clear from the graph norm 
employed on $\dom(-\Delta_{max})$, we indicate the proof of the right-to-left implication next, that is, we assume  
\begin{equation}
\text{$\lim_{j\to\infty} T_j = T$ in $\cB\big(X, L^2(\Om; d^n x)\big)$ and 
$\lim_{j\to\infty} \Delta T_j = S$ in $\cB\big(X, L^2(\Om; d^n x)\big)$.}  
\end{equation}
Let $f \in X$, then $Tf \in L^2(\Om; d^n x) \hookrightarrow \cD'(\Om)$, and for 
any $\varphi \in C_0^\infty(\Om)$,
\begin{align}
& {}_{\cD'(\Om)}\langle \Delta (Tf), \varphi \rangle_{\cD(\Om)} 
= (Tf, \Delta \varphi)_{L^2(\Om; d^n x)} 
= \lim_{j\to\infty} (T_j f, \Delta \varphi)_{L^2(\Om; d^n x)}  
= \lim_{j\to\infty} {}_{\cD'(\Om)}\langle T_j f, \Delta \varphi)\rangle_{\cD(\Om)}  \no \\
& \quad  
= \lim_{j\to\infty} {}_{\cD'(\Om)}\langle \Delta (T_j f), \varphi \rangle_{\cD(\Om)} 
= \lim_{j\to\infty} (\Delta (T_j f), \varphi)_{L^2(\Om; d^n x)}  
= (S f, \varphi)_{L^2(\Om; d^n x)}    \no \\
& \quad =  {}_{\cD'(\Om)}\langle S f,  \varphi \rangle_{\cD(\Om)},  
\end{align}
implying 
\begin{equation}
\text{$\Delta T = S$ on $L^2(\Om; d^n x)$ and hence 
$T \in \cB(X,\dom(-\Delta_{max}))$.}
\end{equation}
Moreover, this yields 
\begin{equation}
\lim_{j\to \infty} \|T_j - T\|_{\cB(X,\dom(-\Delta_{max}))} = 0. 
\end{equation}
Next, using again the equivalence of norms
\begin{equation}
\|T\|_{\cB(X,\dom(-\Delta_{max}))} \approx   \|T\|_{\cB(X,L^2(\Om; d^n x))} +
 \|\Delta T\|_{\cB(X,L^2(\Om; d^n x))},  \lb{16.Meq} 
\end{equation}
and the fact that we may change $\Delta T$ into $(\Delta + z I_{\Om})T$, 
$z\in\bbC$, in \eqref{16.Meq} (at the expense of altering the comparability constants), 
one finally concludes that 
$\big(-\Delta^D_{X,L,z_0} -z I_{\Om}\big)^{-1} \in \cB\big(L^2(\Om; d^n x)\big)$ can actually be improved to 
$\big(-\Delta^D_{X,L,z_0} -z I_{\Om}\big)^{-1} 
\in \cB\big(L^2(\Om; d^n x), \dom(-\Delta_{max})\big)$ and that the assignment 
$\bbC \backslash \sigma\big(-\Delta^D_{X,L,z_0}\big) \ni z 
\mapsto \big(-\Delta^D_{X,L,z_0} -z I_{\Om}\big)^{-1} 
\in \cB\big(L^2(\Om; d^n x), \dom(-\Delta_{max})\big)$ 
is continuous. 

With the help of \eqref{16.M6} and \eqref{16.M7} one concludes from 
Theorem \ref{Th.DH} that
\begin{equation}
\hatt \gamma_D T_w^* \in \cB\big(N^{1/2}(\dOm), \big(N^{1/2}(\dOm)\big)^*\big)
\, \text{ and hence also } \, 
\big(\hatt \gamma_D T_w^*\big)^* \in 
\cB\big(N^{1/2}(\dOm), \big(N^{1/2}(\dOm)\big)^*\big),     \lb{16.MTw}
\end{equation}
and 
\begin{equation}
\lim_{w\to 0} \big(\hatt \gamma_D T_w^*\big)^* 
\, \text{ exists in $\cB\big(N^{1/2}(\dOm), \big(N^{1/2}(\dOm)\big)^*\big)$}.  \lb{16.MT}
\end{equation}
Consequently, employing \eqref{16.MTw} and \eqref{16.MT} in \eqref{16.M1} one obtains that
\begin{equation}
\lim_{w\to 0} \f{1}{w} \big[M^D_{X,L,z_0}(z+w) -  M^D_{X,L,z_0}(z)\big] 
\, \text{ exists in $\cB\big(N^{1/2}(\dOm), \big(N^{1/2}(\dOm)\big)^*\big)$}.  
\end{equation}
Thus, $\bbC\backslash \sigma\big(-\Delta^D_{X,L,z_0}\big) \ni z 
\mapsto M^D_{X,L,z_0}(z) \in 
\cB\big(N^{1/2}(\dOm), \big(N^{1/2}(\dOm)\big)^*\big)$ is analytic. 

Next, for an arbitrary $f\in N^{1/2}(\dOm)$ set 
\begin{eqnarray}\lb{UU.4X1}
u:= \big[\widehat\gamma_D(-\Delta^D_{X,L,z_0}-\ol{z}I_{\Om})^{-1}\big]^*f
\in \ker  (-\Delta_{max}-(z+z_0)I_{\Om}).
\end{eqnarray}
For each $v\in \dom (- \Delta^D_{X,L,z_0})$ we may then write 
\begin{align}\lb{UU.4X2}
& {}_{N^{1/2}(\dOm)}\langle f,\widehat\gamma_D v \rangle_{(N^{1/2}(\dOm))^*}
\nonumber\\
& \quad
={}_{N^{1/2}(\dOm)}\big\langle f,  \widehat\gamma_D
\big(-\Delta^D_{X,L,z_0}-\ol{z}I_{\Om}\big)^{-1} \big(-\Delta^D_{X,L,z_0}-\ol{z}I_{\Om}\big)v
\big\rangle_{(N^{1/2}(\dOm))^*}
\nonumber\\
& \quad
=\big(\big[\widehat\gamma_D(-\Delta^D_{X,L,z_0}-\ol{z}I_{\Om})^{-1}\big]^*f
, \big(-\Delta^D_{X,L,z_0}-\ol{z}I_{\Om}\big)v\big)_{L^2(\Om;d^nx)}
\nonumber\\[4pt]
& \quad
= (u,  (-\Delta-(\ol{z}+z_0))v)_{L^2(\Om;d^nx)}
\\
& \quad
={}_{N^{1/2}(\dOm)}\big\langle \tau^N_{z_0} u,\widehat\gamma_D v            
\big\rangle_{(N^{1/2}(\dOm))^*}
-\ol{{}_{N^{1/2}(\dOm)}\big\langle \tau^N_{z_0} v,\widehat\gamma_D u 
\big\rangle_{(N^{1/2}(\dOm))^*}},
\nonumber
\end{align}
by \eqref{T-GreenX} and \eqref{UU.4X1}. Granted \eqref{UU.4X} and given that 
$v$ belongs to $\dom \big(- \Delta^D_{X,L,z_0}\big)$, we may then transform the 
last term in \eqref{UU.4X2} into 
\begin{equation}\lb{UU.4X3}
\ol{{}_{N^{1/2}(\dOm)}\big\langle \tau^N_{z_0} v,\widehat\gamma_D u 
\big\rangle_{(N^{1/2}(\dOm))^*}}
= - \ol{{}_{N^{1/2}(\dOm)}\langle L(\widehat\gamma_D v),\widehat\gamma_D u 
\rangle_{(N^{1/2}(\dOm))^*}}.
\end{equation}
When inserted into \eqref{UU.4X2}, this yields 
\begin{equation}\lb{UU.4X4}
{}_{N^{1/2}(\dOm)}\big\langle \tau^N_{z_0} u - f,
\widehat\gamma_D v \big\rangle_{(N^{1/2}(\dOm))^*}
= - \ol{{}_{N^{1/2}(\dOm)}\langle L(\widehat\gamma_D v),\widehat\gamma_D u 
\rangle_{(N^{1/2}(\dOm))^*}}. 
\end{equation}
Since $\dom (L)=\widehat\gamma_D \dom \big(- \Delta^D_{X,L,z_0}\big)$ and
$L$ is self-adjoint, the above formula yields 
\begin{equation}\lb{UU.4X5}
\widehat\gamma_D u \in \dom (L)\, \mbox{ and }\,  
- L(\widehat\gamma_D u )=\tau^N_{z_0} u -f.
\end{equation}
Thus, $M^D_{X,L,z_0}(z)f=\widehat\gamma_D u $, \eqref{UU.4X5}, 
\eqref{T-GreenX}, and the fact that 
$\Delta u=-(z+z_0)u$ imply 
\begin{align}\lb{UU.4X6}
& 
{}_{N^{1/2}(\dOm)}\big\langle f, \Im\big(M^D_{X,L,z_0}(z)\big)f
\big\rangle_{(N^{1/2}(\dOm))^*}
\nonumber\\
& \quad 
=\f{1}{2i}\big[{}_{N^{1/2}(\dOm)}\big\langle f, M^D_{X,L,z_0}(z) f\big\rangle
_{(N^{1/2}(\dOm))^*} \no \\
& \hspace*{1.3cm}   
-\ol{{}_{N^{1/2}(\dOm)}\big\langle f, M^D_{X,L,z_0}(z) f\big\rangle
_{(N^{1/2}(\dOm))^*}}\big]  
\nonumber\\ 
& \quad 
=\f{1}{2i}\big[{}_{N^{1/2}(\dOm)}
\big\langle \tau^N_{z_0} u + L(\widehat\gamma_D u ),  \widehat\gamma_D u 
\big\rangle_{(N^{1/2}(\dOm))^*}
\nonumber\\ 
& \hspace*{1.3cm} 
-\ol{{}_{N^{1/2}(\dOm)}
\big\langle \tau^N_{z_0} u + L(\widehat\gamma_D u ),  \widehat\gamma_D u 
\big\rangle_{(N^{1/2}(\dOm))^*}}\big]  
\nonumber\\[4pt]
& \quad 
=\f{1}{2i}\big[{}_{N^{1/2}(\dOm)}\big\langle \tau^N_{z_0} u     ,\widehat\gamma_D u 
\big\rangle_{(N^{1/2}(\dOm))^*}
-\ol{{}_{N^{1/2}(\dOm)}\big\langle \tau^N_{z_0} u     ,\widehat\gamma_D u 
\big\rangle_{(N^{1/2}(\dOm))^*}}\big]  
\nonumber\\ 
& \quad 
=\f{1}{2i}\big[( \Delta u,u)_{L^2(\Om;d^nx)}
- (u,\Delta u)_{L^2(\Om;d^nx)}\big]  
\nonumber\\ 
& \quad 
=\f{1}{2i}\big[(-(z+z_0) u,u)_{L^2(\Om;d^nx)}
- (u,-(z+z_0)u)_{L^2(\Om;d^nx)}\big]  
\nonumber\\ 
& \quad 
=\Im(z)\|u\|^2_{\LOm}\geq 0.
\end{align} 
\end{proof}

Our next theorem provides a transparent connection between 
the inverse of $M^D_{X,L,z_0}(z)$ and $M^{(0)}_{D,N,\Om}(z+z_0)$, 
$M^{(0)}_{D,N,\Om}(z_0)$.  

\begin{theorem}\lb{Th.DH3}
Retain the hypotheses and conventions made in Theorem \ref{Th.Nak}.
In addition, assume that \eqref{UU.4X} holds. Then the map 
\begin{eqnarray}\lb{UU.4Y1}
M^D_{X,L,z_0}(z):N^{1/2}(\dOm)\to \dom (L)
\end{eqnarray}
is well-defined, linear, and bounded, provided $\dom (L)$ is equipped 
with the natural graph norm, that is, $\dom (L)\ni g\mapsto 
\|g\|_{(N^{1/2}(\dOm))^*}+\|L(g)\|_{N^{1/2}(\dOm)}$ $($in which case $\dom (L)$ 
becomes a Banach space\,$)$. Its inverse is given by 
\begin{equation}\lb{UU.4Y2}
M^D_{X,L,z_0}(z)^{-1}= L-M^{(0)}_{D,N,\Om}(z+z_0)+M^{(0)}_{D,N,\Om}(z_0)
:\dom (L)\to  N^{1/2}(\dOm).
\end{equation}
\end{theorem}
\begin{proof}
Fix $f\in N^{1/2}(\dOm)$ arbitrary, and define $u$ as
in \eqref{UU.4X1}. Then $\widehat\gamma_D u \in \dom (L)$ and 
\begin{eqnarray}\lb{UU.4Y3}
(-\Delta-(z+z_0))u=0\mbox{ in }\Omega,\quad u\in L^2(\Om;d^nx),\quad
\tau^N_{z_0} u + L(\widehat\gamma_D u)=f, 
\end{eqnarray}
because of the argument carried out in \eqref{UU.4X1}--\eqref{UU.4X5}.
Then $M^D_{X,L,z_0}(z)f=\widehat\gamma_D u$ so that 
\begin{align}\lb{UU.w1}
& \big\|M^D_{X,L,z_0}(z)f\big\|_{(N^{1/2}(\dOm))^*}
+ \big\|L(M^D_{X,L,z_0}(z)f)\big\|_{N^{1/2}(\dOm)}
\nonumber\\ 
&  \quad
=\|\widehat\gamma_D u \|_{(N^{1/2}(\dOm))^*}
+\|L(\widehat\gamma_D u )\|_{N^{1/2}(\dOm)}
\nonumber\\
& \quad
=\|\widehat\gamma_D u \|_{(N^{1/2}(\dOm))^*}
+ \big\|\tau^N_{z_0} u - f \big\|_{N^{1/2}(\dOm)}
\nonumber\\
& \quad
\leq C\|u\|_{L^2(\Om;d^nx)}+C\|f\|_{N^{1/2}(\dOm)}
\nonumber\\[4pt]
& \quad
\leq C\|f\|_{N^{1/2}(\dOm)}.
\end{align}
Above, we have used the boundedness of the operators \eqref{Tan-C10} and
\eqref{3.Aw1}. The estimate \eqref{UU.w1} proves that the map  
\eqref{UU.4Y1} is bounded. Furthermore, the fact that
$M^D_{X,L,z_0}(z)f=\widehat\gamma_D u $ also entails 
\begin{align}\lb{UU.4Y4}
& \big[L-M^{(0)}_{D,N,\Om}(z+z_0)+M^{(0)}_{D,N,\Om}(z_0)\big]M^D_{X,L,z_0}(z)f
\nonumber\\
& \quad
= L(\widehat\gamma_D u )-M^{(0)}_{D,N,\Om}(z+z_0)(\widehat\gamma_D u)
+M^{(0)}_{D,N,\Om}(z_0)(\widehat\gamma_D u)
\nonumber\\
& \quad
=f-\tau^N_{z_0} u -M^{(0)}_{D,N,\Om}(z+z_0)(\widehat\gamma_D u)
+M^{(0)}_{D,N,\Om}(z_0)(\widehat\gamma_D u)
\nonumber\\
& \quad
=f-\widehat\gamma_N u +M_{D,N,\Om}^{(0)}(z_0) (\widehat\gamma_D u)  
-M^{(0)}_{D,N,\Om}(z+z_0)(\widehat\gamma_D u)
+M^{(0)}_{D,N,\Om}(z_0)(\widehat\gamma_D u)
\nonumber\\
& \quad =f,
\end{align}
since $M^{(0)}_{D,N,\Om}(z+z_0)(\widehat\gamma_D u)=-\widehat\gamma_N u$.
Conversely, assume that $g\in \dom (L)$ is given and 
let $w$ be such that 
\begin{equation}\lb{UU.4Y5}
(-\Delta-(z+z_0))w=0\mbox{ in }\Omega,\quad w\in L^2(\Om;d^nx),\quad
\widehat\gamma_D w =g\in \bigl(N^{1/2}(\dOm)\bigr)^*.
\end{equation}
Then $\big[L-M^{(0)}_{D,N,\Om}(z+z_0)+M^{(0)}_{D,N,\Om}(z_0)\big]g=f$, 
where we have set $f:=\tau^N_{z_0} w + L g \in N^{1/2}(\dOm)$.
Next, $M^D_{X,L,z_0}(z)f=\widehat\gamma_D u $, where $u$ is as in \eqref{UU.4X1}.
However, because of \eqref{UU.4Y3}, the function $v:=u-w$ satisfies
\begin{eqnarray}\lb{UU.4Y6}
(-\Delta-(z+z_0))v=0\mbox{ in }\Omega,\quad v\in L^2(\Om;d^nx),\quad
\widehat\gamma_D v=0,
\end{eqnarray}
and hence necessarily $v=0$, by the uniqueness part in Theorem \ref{tH.A}. 
This yields $u=w$ which implies 
$\widehat\gamma_D u =\widehat\gamma_D w=g$, and hence, $M^D_{X,L,z_0}(z)f=g$. 
\end{proof}

Having established Theorem \ref{Th.DH3}, we now proceed to state 
an alternative version of our earlier Krein formula (cf.\ Theorem \ref{Th.Nak})
on the space $L^2(\Omega;d^nx)$. 

\begin{corollary}\lb{C.Nak}
In addition to the hypotheses made in Theorem \ref{Th.Nak} assume that 
$X=\bigl(N^{1/2}(\dOm)\bigr)^*$. Then the following Krein  formula holds 
on $L^2(\Omega;d^nx)$:
\begin{align} \lb{NaK-C1}
& \big(-\Delta^D_{X,L,z_0}-zI_{\Om}\big)^{-1} 
= \big(-\Delta_{D,\Om}-(z+z_0)I_{\Om}\big)^{-1}
\nonumber\\ 
& \quad
+\big[\tau^N_{z_0}\big(-\Delta_{D,\Om}-(\ol{z}+z_0)I_{\Om}\big)^{-1}\big]^* 
 \big[L-M^{(0)}_{D,N,\Om}(z+z_0)+M^{(0)}_{D,N,\Om}(z_0)\big]^{-1}
\nonumber\\
& \qquad \times 
 \big[\tau^N_{z_0}\big(-\Delta_{D,\Om}-(z+z_0)I_{\Om}\big)^{-1}\big].
\end{align}
In particular, for the resolvents of the Krein  Laplacian, one has
\begin{align}\lb{NaK-C2}
& \big(-\Delta_{K,\Om,z_0}-zI_{\Om}\big)^{-1} 
= \big(-\Delta_{D,\Om}-(z+z_0)I_{\Om}\big)^{-1}   \no \\
& \quad 
+\big[\tau^N_{z_0}\big(-\Delta_{D,\Om}-(\ol{z}+z_0)I_{\Om}\big)^{-1}\big]^* 
\big[M^{(0)}_{D,N,\Om}(z_0)-M^{(0)}_{D,N,\Om}(z+z_0)\big]^{-1}    \\
& \qquad \times 
\big[\tau^N_{z_0}\big(-\Delta_{D,\Om}-(z+z_0)I_{\Om}\big)^{-1}\big],   \no 
\end{align}
as operators in $\cB\bigl(L^2(\Omega;d^nx)\bigr)$.
\end{corollary}
\begin{proof}
Formula \eqref{NaK-C1} is a direct consequence of Theorem \ref{Th.Nak} and 
Theorem \ref{Th.DH3}. Formula \eqref{NaK-C2} follows from this and 
\eqref{4.AK}.
\end{proof}

Naturally, \eqref{NaK-C2} (with $z_0=0$) resembles the abstract relation \eqref{3.8c} between 
the Krein--von Neumann extension $S_K$ and the Friedrichs extension $S_F$ of $S$, but the 
actual methods of proof are quite different in either case. 

\begin{remark}\lb{r.her1}
Under the assumptions made in Corollary \ref{C.Nak}, 
the operator-valued function given by 
\begin{equation}
z\mapsto \big[M^{(0)}_{D,N,\Om}(z_0)-M^{(0)}_{D,N,\Om}(z+z_0)\big]^{-1}
\in\cB\big(N^{1/2}(\dOm),(N^{1/2}(\dOm))^*\big),
\end{equation} 
appearing in \eqref{NaK-C2}, has the Herglotz property. Indeed, by 
Theorem \ref{Th.DH2} and Theorem \ref{Th.DH3}, this is the case for  
the operator-valued function $M^D_{X,L,z_0}(z)=
\big[L-M^{(0)}_{D,N,\Om}(z+z_0)+M^{(0)}_{D,N,\Om}(z_0)\big]^{-1}$ 
and the map in question corresponds to this when $L=0$. 
\end{remark} 

Our last theorem in this section is a Krein formula involving the resolvents 
of certain self-adjoint extensions of the Laplacian. To set the stage, 
we first prove the following lemma. 

\begin{lemma} \lb{l.Na-1}
Assume Hypothesis \ref{h.Conv} and suppose that 
$z_0\in\bbR\backslash\si(-\Delta_{D,\Om})$.
Let $X_1$ be a closed subspace of $X_2:=\bigl(N^{1/2}(\partial\Omega)\bigr)^*$.
In addition, consider two self-adjoint operators 
\begin{eqnarray}\lb{AG-1}
L_j:\dom (L_j)\subseteq X_j\to  X_j^*, \quad j=1,2,
\end{eqnarray} 
with the property that 
\begin{eqnarray}\lb{AG-e}
\dom (L_1)\hookrightarrow \dom (L_2) \, \mbox{ boundedly}.
\end{eqnarray} 
Associated with $L_j$, $j=1,2$, define the self-adjoint  
operators $- \Delta^D_{X_j,L_j,z_0}$, $j=1,2$, as in \eqref{4.Aw3} 
corresponding to $z=z_0$. In addition, let 
$z\in\bbC\backslash \big(\si(-\Delta^D_{X_1,L_1,z_0})\cup
\si(-\Delta^D_{X_2,L_2,z_0})\big)$.
Then the following resolvent relation holds on $L^2(\Omega;d^nx)$, 
\begin{align}\lb{N-Bbb1}
& \big(-\Delta^D_{X_2,L_2,z_0}-zI_\Om\big)^{-1} 
=\big(-\Delta^D_{X_1,L_1,z_0}-zI_\Om\big)^{-1}  
\\
& \quad 
+ \bigl[\widehat\gamma_D\big(-\Delta^D_{X_2,L_2,z_0}-\ol{z}I_\Om\big)^{-1}\bigr]^*
\big(\tau^N_{z_0} + L_2 \widehat\gamma_D\big)
\big(-\Delta^D_{X_1,L_1,z_0}-zI_\Om\big)^{-1}.   \no 
\end{align}
\end{lemma}
\begin{proof}
First, we note that 
\begin{align}\lb{N-Bbb2}
& \widehat\gamma_D\big(-\Delta^D_{X_1,L_1,z_0}-zI_\Om\big)^{-1}
:L^2(\Omega;d^nx)\to 
\dom (L_1)\hookrightarrow \dom (L_2)   \\
& \hspace*{4.25cm} \hookrightarrow 
X_2=\big(N^{1/2}(\dOm)\big)^*.
\end{align}
Together with \eqref{Gaa-1}, \eqref{AG-1}, Lemma \ref{L-refN}, \eqref{3.Aw1}
and \eqref{Nak.11B}, this ensures that 
the composition of the operators appearing on the right-hand side of 
\eqref{N-Bbb1} is well-defined. 
Next, let $\phi_1,\phi_2\in L^2(\Om;d^nx)$ be arbitrary and define
\begin{align}
\begin{split}
\psi_1 & :=(-\Delta^D_{X_1,L_1,z_0}-zI_\Om)^{-1}\phi_1 
\in\dom(- \Delta^D_{X_1,L_1,z_0}) \subset \dom(-\Delta_{max}),
\\[6pt]
\psi_2 & :=(-\Delta^D_{X_2,L_2,z_0}-\ol{z}I_\Om)^{-1}\phi_2 
\in\dom(- \Delta^D_{X_2,L_2,z_0}) \subset \dom(- \Delta_{max}).
\end{split} \lb{Na2Z}
\end{align}
Our goal is to show that the following identity holds:
\begin{align}\lb{Hff-1}
&\big(\phi_2, \big(-\Delta^D_{X_2,L_2,z_0}-zI_\Om\big)^{-1}\phi_1
\big)_{L^2(\Om;d^nx)}   
-\big(\phi_2,\big(-\Delta^D_{X_1,L_1,z_0}-zI_\Om\big)^{-1}\phi_1
\big)_{L^2(\Om;d^nx)}
\no \\
&\quad =\big(\phi_2,
\bigl[\widehat\gamma_D\big(-\Delta^D_{X_2,L_2,z_0}-\ol{z}I_\Om\big)^{-1}\bigr]^*
\big(\tau^N_{z_0} + L_2 \widehat\gamma_D\big)  
\big(-\Delta^D_{X_1,L_1,z_0}-zI_\Om\big)^{-1}\phi_1\big)_{L^2(\Om;d^nx)}. 
\end{align}
We note that according to \eqref{Na2Z} one has,
\begin{align}
& \big(\phi_2,\big(-\Delta^D_{X_1,L_1,z_0}-zI_\Om\big)^{-1}\phi_1
\big)_{L^2(\Om;d^nx)}   \no \\ 
& \quad = \big(\big(-\Delta^D_{X_2,L_2,z_0}-\ol{z}I_\Om\big)\psi_2,
\psi_1\big)_{L^2(\Om;d^nx)},
\\
& \big(\phi_2,\big(-\Delta^D_{X_2,L_2,z_0}-zI_\Om\big)^{-1}\phi_1
\big)_{L^2(\Om;d^nx)}  \no \\
& \quad =\big(\big(\big(-\Delta^D_{X_2,L_2,z_0}-zI_\Om\big)^{-1}\big)^*\phi_2,\phi_1
\big)_{L^2(\Om;d^nx)} 
\nonumber
\\
& \quad =\big(\big(-\Delta^D_{X_2,L_2,z_0}-\ol{z}I_\Om\big)^{-1}\phi_2,\phi_1
\big)_{L^2(\Om;d^nx)}
\nonumber
\\
& \quad =\big(\psi_2,\big(-\Delta^D_{X_1,L_1,z_0}-zI_\Om\big)\psi_1
\big)_{L^2(\Om;d^nx)},
\end{align}
and 
\begin{align}
&\big(\phi_2,
\bigl[\widehat\gamma_D\big(-\Delta^D_{X_2,L_2,z_0}-\ol{z}I_\Om\big)^{-1}\bigr]^*
\big(\tau^N_{z_0} + L_2 \widehat\gamma_D\big)  
\big(-\Delta^D_{X_1,L_1,z_0}-zI_\Om\big)^{-1}\phi_1\big)_{L^2(\Om;d^nx)} 
\nonumber
\\[1mm] 
&\quad =\ol{{}_{N^{1/2}(\dOm)}\big\langle
\big(\tau^N_{z_0} + L_2 \widehat\gamma_D\big)
\big(-\Delta^D_{X_1,L_1,z_0}-zI_\Om\big)^{-1}\phi_1, 
\widehat\gamma_D}  
\ol{\big(-\Delta^D_{X_2,L_2,z_0}-\ol{z}I_\Om\big)^{-1}\phi_2
\big\rangle_{(N^{1/2}(\dOm))^*}}
\nonumber
\\[1mm]
&\quad =\ol{{}_{N^{1/2}(\dOm)}\big\langle
\big[\tau^N_{z_0} \psi_1 + L_2 \widehat\gamma_D \psi_1\big], 
\widehat\gamma_D \psi_2 \big\rangle_{(N^{1/2}(\dOm))^*}}.
\end{align}
Thus, matters have been reduced to proving that
\begin{align}\lb{Na3Z}
& \big(\psi_2,\big(-\Delta^D_{X_1,L_1,z_0}-zI_\Om\big)\psi_1 \big)_{L^2(\Om;d^nx)}  
-\big(\big(-\Delta^D_{X_2,L_2,z_0}-\ol{z}I_\Om\big)\psi_2,\psi_1
\big)_{L^2(\Om;d^nx)} 
\no \\[1mm]
& \quad 
=\ol{{}_{N^{1/2}(\dOm)}\big\langle
\big[\tau^N_{z_0} \psi_1 + L_2 \widehat\gamma_D \psi_1\big],  
\widehat\gamma_D \psi_2 \big\rangle_{(N^{1/2}(\dOm))^*}}. 
\end{align}
Using \eqref{T-GreenX} for the left-hand side of \eqref{Na3Z} one obtains
\begin{align}
& \big(\psi_2,\big(-\Delta^D_{X_1,L_1,z_0}-zI_\Om\big)\psi_1
\big)_{L^2(\Om;d^nx)}  
-\big(\big(-\Delta^D_{X_2,L_2,z_0}-\ol{z}I_\Om\big)\psi_2,\psi_1
\big)_{L^2(\Om;d^nx)}
\nonumber 
\\
&\quad = (\psi_2,
(-\Delta-(\ol{z}+z_0))\psi_1)_{L^2(\Om;d^nx)} 
- ((-\Delta-(z+z_0))\psi_2,\psi_1)_{L^2(\Om;d^nx)} 
\nonumber 
\\
& \quad ={}_{N^{1/2}(\partial\Omega)}\big\langle\tau^N_{z_0} \psi_1,
\widehat\gamma_D \psi_2 \big\rangle_{(N^{1/2}(\partial\Omega))^*}
-\,\ol{{}_{N^{1/2}(\partial\Omega)} \big\langle\tau^N_{z_0} \psi_2,
\widehat{\gamma}_D \psi_1 \big\rangle_{(N^{1/2}(\partial\Omega))^*}}
\nonumber
\\
& \quad ={}_{N^{1/2}(\partial\Omega)}\big\langle\tau^N_{z_0} \psi_1,
\widehat\gamma_D \psi_2 \big\rangle_{(N^{1/2}(\partial\Omega))^*}
+ {}_{X_2}\langle\widehat{\gamma}_D \psi_1,L_2 \widehat{\gamma}_D \psi_2
\rangle_{X_2^*}
\nonumber
\\
& \quad ={}_{N^{1/2}(\partial\Omega)}\big\langle\tau^N_{z_0} \psi_1,
\widehat\gamma_D \psi_2 \big\rangle_{(N^{1/2}(\partial\Omega))^*}
+ \ol{{}_{X_2}\langle\widehat{\gamma}_D \psi_2,
L_2 \widehat{\gamma}_D \psi_1 \rangle_{X_2^*}}
\nonumber \\
& \quad ={}_{N^{1/2}(\partial\Omega)}\big\langle\tau^N_{z_0} \psi_1,
\widehat\gamma_D \psi_2 \big\rangle_{(N^{1/2}(\partial\Omega))^*}  
+ {}_{N^{1/2}(\partial\Omega)}
\langle L_2 \widehat{\gamma}_D \psi_1,\widehat{\gamma}_D \psi_2 
\rangle_{(N^{1/2}(\partial\Omega))^*}
\nonumber  \\
& \quad 
={}_{N^{1/2}(\partial\Omega)}\big\langle \big[\tau^N_{z_0} \psi_1 
+ L_2 \widehat{\gamma}_D \psi_1\big],
\widehat\gamma_D \psi_2 \big\rangle_{(N^{1/2}(\partial\Omega))^*},
\end{align}
since $\psi_2\in\dom (- \Delta^D_{X_2,L_2,z_0})$ and 
\begin{align}
\psi_1\in \dom (- \Delta^D_{X_1,L_1,z_0})  \, \text{ implies } \, 
\widehat\gamma_D(\psi_1)\in & \dom (L_1)\hookrightarrow
\dom (L_2)  \no \\
& \hookrightarrow X_2=\big(N^{1/2}(\dOm)\big)^*.  \lb{Hn-d}
\end{align}
This justifies \eqref{Na3Z}, completing the proof of the lemma. 
\end{proof}

We are now prepared to state and prove the Krein-type  formula alluded 
to earlier, expressing the difference of the resolvents of certain  
self-adjoint extensions of $-\Delta\big|_{C_0^{\infty}(\Om)}$ in $L^2(\Om; d^n x)$. 

\begin{theorem}\lb{T.Na-q}
Assume Hypothesis \ref{h.Conv} and suppose that 
$z_0\in\bbR\backslash\si(-\Delta_{D,\Om})$.
In addition, consider two bounded, self-adjoint operators 
\begin{eqnarray}\lb{AG-1.1}
L_j:\bigl(N^{1/2}(\partial\Omega)\bigr)^*\to  N^{1/2}(\partial\Omega),
\quad j=1,2.
\end{eqnarray} 
Associated with $L_j$, $j=1,2$, define the self-adjoint 
operators $- \Delta^D_{X_j,L_j,z_0}$, $j=1,2$, as in \eqref{4.Aw3} 
corresponding to $z=z_0$ and $X_1=X_2=\bigl(N^{1/2}(\partial\Omega)\bigr)^*$.
In addition, let $z\in\bbC\backslash \big(\si(-\Delta^D_{X_1,L_1,z_0})\cup
\si(-\Delta^D_{X_2,L_2,z_0})\big)$.
Then the following resolvent relation holds on $L^2(\Omega;d^nx)$, 
\begin{align}\lb{N-Bbb1.1}
&  
\big(-\Delta^D_{X_2,L_2,z_0}-zI_\Om\big)^{-1} 
=\big(-\Delta^D_{X_1,L_1,z_0}-zI_\Om\big)^{-1}  
\\
& \quad
+\big[\widehat\gamma_D\big(-\Delta^D_{X_1,L_1,z_0}-zI_\Om\big)^{-1}\big]^*
M^D_{L_1,L_2,z_0}(z)
\big[\widehat\gamma_D\big(-\Delta^D_{X_1,L_1,z_0}-zI_\Om\big)^{-1}\big].    \no
\end{align}
Here
\begin{align}\lb{N-Bbb1.2}
& M^D_{L_1,L_2,z_0}(z):= -  
(L_2-L_1)\Big[I_\Om -\big[M^{(0)}_{D,N,\Om}(z_0)-M^{(0)}_{D,N,\Om}(z+z_0)
+L_2\big]^{-1} (L_2-L_1)\Big]
\no \\
& \quad
= - (L_2-L_1)\big[M^{(0)}_{D,N,\Om}(z_0)-M^{(0)}_{D,N,\Om}(z+z_0)+L_2\big]^{-1} 
\no \\
& \qquad \times 
\big[M^{(0)}_{D,N,\Om}(z_0)-M^{(0)}_{D,N,\Om}(z+z_0) + L_1\big] 
\end{align}
has the property that 
\begin{eqnarray}\lb{N-Bbb1.3}
\bbC_{+}\ni z\mapsto M^D_{L_1,L_2,z_0}(z)
\in \cB\bigl((N^{1/2}(\dOm))^*,N^{1/2}(\dOm)\bigr)
\end{eqnarray}
is an operator-valued Herglotz function which satisfies 
\begin{eqnarray}\lb{N-Bbb1.4}
\big[M^D_{L_1,L_2,z_0}(z)\big]^*=M^D_{L_1,L_2,z_0}(\ol{z}). 
\end{eqnarray}
\end{theorem}
\begin{proof}
Applying $\tau^N_{z_0}-L_1 \widehat\ga_D$ to both sides of \eqref{N-Bbb1} 
and using the fact that 
\begin{eqnarray}\lb{Na-v}
\big(\tau^N_{z_0} + L_j\widehat\ga_D\big)
\big(-\Delta^D_{X_j,L_j,z_0}-zI_\Om\big)^{-1}=0,\quad j=1,2,
\end{eqnarray}
one obtains
\begin{align}\lb{Na-v1}
& (L_2-L_1)\big[\widehat\ga_D\big(-\Delta^D_{X_2,L_2,z_0}-zI_\Om\big)^{-1}\big] 
\no \\
& \quad 
= \big(\tau^N_{z_0} + L_1 \widehat\ga_D\big)
\bigl[\widehat\gamma_D\big(-\Delta^D_{X_2,L_2,z_0}-\ol{z}I_\Om\big)^{-1}\bigr]^*
(L_1-L_2)  
\big[\widehat\gamma_D\big(-\Delta^D_{X_1,L_1,z_0}-zI_\Om\big)^{-1}\big]
\nonumber\\ 
& \quad 
=\bigl[M^{(0)}_{D,N,\Om}(z_0)-M^{(0)}_{D,N,\Om}(z+z_0) + L_1\bigl]
M^D_{X_2,L_2,z_0}(z) (L_1-L_2)  
\big[\widehat\gamma_D\big(-\Delta^D_{X_1,L_1,z_0}-zI_\Om\big)^{-1}\big], 
\end{align}
where in the last step we have also made use of \eqref{Ggg-1} and \eqref{UU.1H}.
By taking adjoints of both sides in \eqref{Na-v1} one arrives at 
\begin{align}\lb{Na-v2}
& \big[\widehat\ga_D\big(-\Delta^D_{X_2,L_2,z_0}-zI_\Om\big)^{-1}\big]^*(L_2-L_1) 
\\
& \quad 
= \big[\widehat\gamma_D\big(-\Delta^D_{X_1,L_1,z_0}-zI_\Om\big)^{-1}\big]^*
(L_1-L_2)M^D_{X_2,L_2,z_0}(\ol{z})  \no \\
& \qquad \times 
\bigl[M^{(0)}_{D,N,\Om}(z_0)-M^{(0)}_{D,N,\Om}(\ol{z}+z_0) + L_1\bigl], 
\nonumber
\end{align}
using \eqref{UU.3} and \eqref{NaLa}. 
Inserting this identity (written with $\ol{z}$ in place of $z$) 
into \eqref{N-Bbb1} then yields
\begin{align}\lb{Na-v3}
& \big(-\Delta^D_{X_2,L_2,z_0}-zI_\Om\big)^{-1} 
=\big(-\Delta^D_{X_1,L_1,z_0}-zI_\Om\big)^{-1}  
\nonumber\\
& \quad - \big[\widehat\gamma_D\big(-\Delta^D_{X_1,L_1,z_0}
-\ol{z}I_\Om\big)^{-1}\big]^*
(L_2-L_1)M^D_{X_2,L_2,z_0}(z)   \\
& \qquad \times 
\bigl[M^{(0)}_{D,N,\Om}(z_0)-M^{(0)}_{D,N,\Om}(z+z_0) + L_1\bigl]
\big[\widehat\gamma_D\big(-\Delta^D_{X_1,L_1,z_0}-zI_\Om\big)^{-1}\big].  \no 
\end{align}
Making reference to \eqref{UU.4Y2}, we may then write 
\begin{align}\lb{Na-v4}
& \big(-\Delta^D_{X_2,L_2,z_0}-zI_\Om\big)^{-1} 
=\big(-\Delta^D_{X_1,L_1,z_0}-zI_\Om\big)^{-1}  
\nonumber\\
& \quad 
- \big[\widehat\gamma_D\big(-\Delta^D_{X_1,L_1,z_0}-\ol{z}I_\Om\big)^{-1}\big]^*
(L_2-L_1)M^D_{X_2,L_2,z_0}(z)
\nonumber\\ 
& \qquad 
\times \big[M^D_{X_2,L_2,z_0}(z)^{-1} - (L_2-L_1)\big]
\big[\widehat\gamma_D\big(-\Delta^D_{X_1,L_1,z_0}-zI_\Om\big)^{-1}\big], 
\end{align}
or equivalently,  
\begin{align}\lb{Na-v5}
& \big(-\Delta^D_{X_2,L_2,z_0}-zI_\Om\big)^{-1} 
=\big(-\Delta^D_{X_1,L_1,z_0}-zI_\Om\big)^{-1}  
\nonumber\\
& \quad - 
\big[\widehat\gamma_D\big(-\Delta^D_{X_1,L_1,z_0}-\ol{z}I_\Om\big)^{-1}\big]^*  
(L_2-L_1)[I - M^D_{X_2,L_2,z_0}(z)(L_2-L_1)]  \no \\
& \qquad \times 
\big[\widehat\gamma_D\big(-\Delta^D_{X_1,L_1,z_0}-zI_\Om\big)^{-1}\big].
\end{align}
Now \eqref{N-Bbb1.1} follows from this and \eqref{UU.4Y2}. 

Finally, the advertised properties for $M^D_{L_1,L_2,z_0}(z)$ are 
consequences of the formula 
\begin{equation}\lb{Na-v6}
M^D_{L_1,L_2,z_0}(z)= - (L_2-L_1)\big[I_\Om - M^D_{X_2,L_2,z_0}(z)(L_2-L_1)\big]
\end{equation}
and the corresponding properties of $M^D_{X_2,L_2,z_0}(z)$.
\end{proof}

\begin{remark}\label{R-k1}
An alternative expression of $M^D_{L_1,L_2,z_0}(z)$ is given by 
\begin{equation}\lb{N-ag1}
M^D_{L_1,L_2,z_0}(z)= M^D_{X_1,L_1,z_0}(z)^{-1}
\big[M^D_{X_2,L_2,z_0}(z)-M^D_{X_1,L_1,z_0}(z)\big]M^D_{X_1,L_1,z_0}(z)^{-1}.
\end{equation}
This follows from the second equality in \eqref{N-Bbb1.2} and \eqref{UU.4Y2}.
\end{remark}

\noindent {\bf Acknowledgments.}
We are indebted to Yury Arlinskii, Gerd Grubb, Konstantin Makarov, Mark Malamud, 
Vladimir Maz'ya, and Eduard Tsekanovskii for many helpful discussions and very valuable correspondence on various topics of this paper. We are particularly grateful 
to Gerd Grubb for constructive remarks which led to a strengthening of some of the results in Section \ref{s6}. 


\end{document}